\documentclass[11pt, a4paper, headings=standardclasses, openany, bibliography=totocnumbered]{scrbook}

\usepackage{MA_Titlepage}
\usepackage[automark, autooneside=false]{scrlayer-scrpage}  % header design
\usepackage[usenames,dvipsnames]{xcolor}
\usepackage[T1]{fontenc}
\usepackage[utf8]{inputenc}
\usepackage{lmodern}
\usepackage[UKenglish]{babel}
\usepackage{amsmath, amssymb, amsfonts, amsthm, mathtools}
\usepackage{extarrows}
\usepackage{mathrsfs}
\usepackage{pgfplots}
\usetikzlibrary{arrows}
\usepackage{tikz-cd}
\usepackage{graphicx}
\usepackage{faktor}
\usepackage{young}
\usepackage[centertableaux]{ytableau}
\usepackage{verbatim} %for comments
\usepackage{comment}
\usepackage{wasysym}
\usepackage{enumitem}
\usepackage{microtype}
\usepackage[style=alphabetic, backend=biber]{biblatex}
\usepackage{csquotes}
\usepackage{cancel}
\usepackage{color}      % colors
\usepackage{xparse}
\usepackage{longtable}
\usepackage{hyperref}
\usepackage{cleveref}	%for references
\usepackage{braids}
\usepackage{adjustbox}
\usepackage{forest} %for branching graphs
\usepackage{dynkin-diagrams}
\usepackage{rank-2-roots}
\definecolor{light-gray}{gray}{0.85}

\binoppenalty=\maxdimen         % breaking of binary operators
\relpenalty=\maxdimen           % breaking of relation symbols

\newcommand{\opn}[1]{\operatorname{#1}}

\newcommand\abstractname{Abstract}  %%% here
\makeatletter
\if@titlepage
\newenvironment{abstract}{%
	\titlepage
	\null\vfil
	\@beginparpenalty\@lowpenalty
	\begin{center}%
		\bfseries \abstractname
		\@endparpenalty\@M
\end{center}}%
{\par\vfil\null\endtitlepage}
\else
\newenvironment{abstract}{%
	\if@twocolumn
	\section*{\abstractname}%
	\else
	\small
	\begin{center}%
		{\bfseries \abstractname\vspace{-.5em}\vspace{\z@}}%
	\end{center}%
	\quotation
	\fi}
{\if@twocolumn\else\endquotation\fi}
\fi
\makeatother

\newtheorem{theo}{Theorem}
\newtheorem{lemm}[theo]{Lemma}

\newtheorem{koro}[theo]{Corollary}
\newtheorem*{thmA}{Theorem A}
\newtheorem*{thmB}{Theorem B}
\newtheorem*{thmC}{Theorem C}
\theoremstyle{definition}
\newtheorem{defi}[theo]{Definition}
\newtheorem{beis}[theo]{Example}
\newtheorem{rema}[theo]{Remark}
\newtheorem{obse}[theo]{Observation}

\newtheorem{warn}[theo]{Warning}
\newtheorem{fact}[theo]{Fact}

\DeclareMathOperator{\Hom}{Hom}
\DeclareMathOperator{\End}{End}

\numberwithin{theo}{chapter} % Bessere Nummerierung
\DeclareMathOperator{\C}{\mathbb{C}}

\DeclareMathOperator{\sgn}{sgn}
\DeclareMathOperator{\id}{id}

\newcommand{\bp}{b_{n,{+}}}
\newcommand{\bn}{b_{n,{-}}}
\DeclareMathOperator{\Sp}{Sp}

\DeclareMathOperator{\SO}{SO}
\DeclareMathOperator{\SL}{SL}
\DeclareMathOperator{\U}{\mathbf{U}}
            
\DeclareMathOperator{\modd}{\mathbf{mod}}

\DeclareMathOperator{\Rep}{\mathbf{Rep}}

\DeclareMathOperator{\TL}{TL}
\DeclareMathOperator{\ATL}{ATL}

\newcommand{\dotTL}{\opn{TL}(B_n)}
\DeclareMathOperator{\Ug}{U}
\DeclareMathOperator{\sl2}{\mathfrak{sl}_2}
\DeclareMathOperator{\sln}{\mathfrak{sl}_n}

\DeclareMathOperator{\h}{\mathfrak{h}}
\DeclareMathOperator{\Kar}{Kar}
\DeclareMathOperator{\Add}{Add}
\DeclareMathOperator{\kVect}{\mathbf{Vect}}

\DeclareMathOperator{\nice}{ss}
\DeclareMathOperator{\lie}{\mathfrak{g}}
\DeclareMathOperator{\gl}{\mathfrak{gl}}

\newcommand{\Uq}{\Ug_q(\sl2)}
\newcommand{\Uzwei}{\Ug_q(\gl_{2m})}
\newcommand{\Vq}{\Ug_q'(\gl_1\times \gl_1)}
\DeclareMathOperator{\M}{M}

\newcommand{\RES}{\opn{res}}
\newcommand{\IND}{\opn{ind}}

\DeclareMathOperator{\GL}{GL}
\DeclareMathOperator{\op}{op}
\DeclareMathOperator{\mZ}{\mathbb{Z}}
\DeclareMathOperator{\mQ}{\mathbb{Q}}
\newcommand{\mZn}{\mathbb{Z}/n\mathbb{Z}}

\DeclareMathOperator{\mN}{\mathbb{N}_0}
\DeclareMathOperator{\mU}{\mathbf{1}}

\DeclareMathOperator{\mR}{\mathbb{R}}

\newcommand{\Hecke}{\mathcal{H}}

\newcommand{\shift}[1]{\langle#1\rangle} % grading shift by some argument like n, 1, -1, etc.

\DeclareMathOperator*{\tlcap}{{\,\pretlcap\,}}
\DeclareMathOperator*{\tlcup}{{\,\pretlcup\,}}
\DeclareMathOperator*{\dtldcap}{{\,\dpretldcap\,}}

\DeclareMathOperator*{\tlcupcap}{\,\pretlcupcap\,}

\DeclareMathOperator*{\tlline}{\pretlline}
 
\newcommand{\cO}{{\mathcal O}} %category O
\newcommand{\Ko}{\opn{K}_0} %Grothendieck

\newcommand{\wo}{w_o} %longest word
\DeclareMathOperator{\so}{\mathfrak{so}} %symmetric orthogonal lie algebra
\DeclareMathOperator{\spp}{\mathfrak{sp}} %symmetric orthogonal lie algebra
\newcommand{\blank}{{-}}

\DeclareMathOperator{\Br}{Br}
\DeclareFieldFormat[article,unpublished]{citetitle}{\mkbibitalic{#1}\isdot}
\DeclareFieldFormat[article,unpublished]{title}{\mkbibitalic{#1}\isdot}
\DeclareFieldFormat[article,book]{number}{\mkbibbold{#1}}
\DeclareFieldFormat[article]{volume}{\mkbibbold{#1}}
\usetikzlibrary{calc}
\usetikzlibrary{tqft}
\newcommand{\cbox}[1]{ \vcenter{ \hbox{#1} } }
\usetikzlibrary{patterns,decorations.pathreplacing}
\tikzset{
	tldiagram/.style={thick, scale=0.35}
}
\newcommand{\tlcoord}[2]
{
	(2*#2 , 3*#1)
}
\newcommand{\makecdots}[2]
{
	\node at (2*#2, 3*#1 + 1.5) {$\cdots$};
}
\newcommand{\lineup}{-- ++(0,3)}
\newcommand{\halflineup}{-- ++(0,1.5)}

\newcommand{\linedown}{-- ++(0,-3)}

\newcommand{\smalllineup}{-- ++(0,2)}
\newcommand{\smalllinedown}{-- ++(0,-2)}
\newcommand{\linewave}[2]{
	.. controls +(0,1.5*#1) and +(0,1.5*-#1) .. ++(2*#2, 3*#1)
}
\newcommand{\dlinewave}[2]{
	.. controls +(0,1.5*#1) and +(0,1.5*-#1) .. ++(2*#2, 3*#1)  \onedot
}

\newcommand{\dlineup}{\halflineup \onedot \halflineup}

\newcommand{\capright}{arc (180:0:1)}
\newcommand{\capleft}{arc (0:180:1)}
\newcommand{\cupright}{arc (180:360:1)}
\newcommand{\cupleft}{arc (360:180:1)}

\newcommand{\onedot}{node {$\bullet$}}

\newcommand{\dcapright}{arc (180:90:1) node {$\bullet$} arc (90:0:1)}

\newcommand{\dcupright}{arc (180:270:1) node {$\bullet$} arc (270:360:1)}

\newcommand{\dcupleft}{arc (360:270:1) node {$\bullet$} [fill=none] arc (270:180:1)}

\newcommand{\dcaprightsmall}{arc (180:90:1) node {\scalebox{0.7}{\textbullet}} arc (90:0:1)}

\newcommand{\dcuprightsmall}{arc (180:270:1) node {\scalebox{0.7}{\textbullet}} arc (270:360:1)}

\newcommand{\xcapright}[1]{arc (180:0:#1)}
\newcommand{\xcapleft}[1]{arc (0:180:#1)}
\newcommand{\xcupright}[1]{arc (180:360:#1)}
\newcommand{\xcupleft}[1]{arc (360:180:#1)}
\newcommand{\biggerdrawdots}[3]
{
	\foreach \i in {#2, ..., #3}
	{
		\draw[fill] (2*\i, 3*#1) circle (0.4em);
	}
}
\newcommand{\maketlboxnormal}[2]{
	++(-1,-0.75)
	+(-0.5, -0.3)
	[fill=blue!20!white, rounded corners]
	rectangle
	+(#1*2 + 0.5, 1.5 + 0.3)
	node at ++(#1,0.75) {#2}
}
\newcommand{\maketlboxred}[2]{  %for type B, + Jones-Wenzl 
	++(-1,-0.75)
	+(-0.5, -0.3)
	[fill=red!20!white, rounded corners]
	rectangle
	+(#1*2 + 0.5, 1.5 + 0.3)
	node at ++(#1,0.75) {#2}
}
\newcommand{\maketlboxgreen}[2]{
	++(-1,-0.75)
	+(-0.5, -0.3)
	[fill=ForestGreen!20!white, rounded corners]
	rectangle
	+(#1*2 + 0.5, 1.5 + 0.3)
	node at ++(#1,0.75) {#2}
}
\newcommand{\maketlboxblue}[2]{ %for type B, - Jones-Wenzl 
	++(-1,-0.75)
	+(-0.5, -0.3)
	[fill=cyan!20!white, rounded corners]
	rectangle
	+(#1*2 + 0.5, 1.5 + 0.3)
	node at ++(#1,0.75) {#2}
}

\newcommand{\maketlellipse}[2]{
	++ (#1 - 1,0)
	[fill=green!20!white]
	ellipse (#1 + 0.5 and 1.2)
	node {#2}
}

\newcommand{\poscrossing}[2]{
	\draw \tlcoord{#1}{#2+2} \linewave{1}{-1} ;
	\draw[white, double=black] \tlcoord{#1}{#2} \linewave{1}{1};
}
\newcommand{\negcrossing}[2]{
\draw \tlcoord{#1}{#2} \linewave{1}{1};
\draw[white, double=black] \tlcoord{#1}{#2+2} \linewave{1}{-1} ;
	}
\newcommand{\dnegcrossing}[2]{
	\draw \tlcoord{#1}{#2} \linewave{1}{1};
	\draw[white, double=black] \tlcoord{#1}{#2+2} \linewave{1}{-1};
	\draw \tlcoord{#1+3}{#2} \onedot;
	\draw \tlcoord{#1}{#2} \onedot; 
}
\newcommand{\uppernegcrossing}[2]{
	\draw \tlcoord{#1}{#2} \linewave{1}{1};
	\draw[white, double=black] \tlcoord{#1}{#2+2} \linewave{1}{-1};
	\draw \tlcoord{#1+3}{#2} \onedot;
}
\newcommand{\lowernegcrossing}[2]{
	\draw \tlcoord{#1}{#2}\linewave{1}{1};
	\draw[white, double=black] \tlcoord{#1}{#2+2} \linewave{1}{-1};
	\draw \tlcoord{#1}{#2} \onedot; 
}

\newcommand{\negtwista}[2]{
	\draw [ultra thick] \tlcoord{0+3*#1}{0+2*#2} -- \tlcoord{0.45+3*#1}{0.4+2*#2};
	\draw [ultra thick] \tlcoord{0.55+3*#1}{0.6+2*#2} -- \tlcoord{1+3*#1}{1+2*#2};
	\draw  \tlcoord{0+3*#1}{1+2*#2} -- \tlcoord{1+3*#1}{0+2*#2};
}

\newcommand{\negtwist}[2]{
	\negtwista{#1}{#2}
	\negtwist{1+#1}{#2}
}

\newcommand{\dpretldcap}
{
	\begin{tikzpicture}[xscale=0.092,yscale=0.12]
		\draw (0,0) -- +(0,1) \dcaprightsmall -- +(0,-1);
	\end{tikzpicture}
}

\newcommand{\pretlcap}
{
	\begin{tikzpicture}[xscale=0.092,yscale=0.12]
		\draw (0,0) -- +(0,1) \capright -- +(0,-1);
	\end{tikzpicture}
}
\newcommand{\pretlcup}
{
	\begin{tikzpicture}[xscale=0.092,yscale=0.12]
		\draw (0,0) -- +(0,-1) \cupright -- +(0,1);
	\end{tikzpicture}
}
\newcommand{\pretlcupcap}
{
	\begin{tikzpicture}[scale=0.092]
		\draw \tlcoord{1}{0} \cupright;
		\draw \tlcoord{0}{0} \capright;
	\end{tikzpicture}
}
\newcommand{\predtlcupcap} %has to get fixed
{
	\begin{tikzpicture}[scale=0.092]
		\draw \tlcoord{1}{0} \dcuprightsmall;
		\draw \tlcoord{0}{0} \dcaprightsmall;
	\end{tikzpicture}
}
\newcommand{\tlcircle}
{
	\begin{tikzpicture}[scale=0.11]
		\draw (0,0) circle (1);
	\end{tikzpicture}
}
\newcommand{\pretlline}
{
	\vert
}

\setlist[enumerate]{label= \roman*)}

\setlist{
	listparindent=\parindent,
	parsep=0pt
}

\authornew{Zbigniew Wojciechowski}
\geburtsdatum{14th August 1997}
\geburtsort{Miko\l \'ow, Poland}
\date{25th May 2022}

\betreuer{Advisor: Prof.\ Dr.\ Catharina Stroppel}
\zweitgutachter{Second Advisor: Prof.\ Dr.\ Paul Wedrich}

\institut{Mathematisches Institut}
\title{Categorified Jones--Wenzl Projectors and Generalizations}
\ausarbeitungstyp{Master's Thesis  Mathematics}

\begin{document}
	\frontmatter
	%erzeugt Titelseite
	
	\maketitle
	
	%erzeugt Inhaltsverzeichnis
	
		\begin{abstract}
		This preprint comprises the first four out of five chapters of the Master's thesis I wrote 2022 under the supervision of Catharina Stroppel and Paul Wedrich at the University of Bonn titled "Categorified Jones--Wenzl projectors and Generalizations". It can serve as introductory literature for researchers or graduate students familiar with Schur--Weyl duality between the symmetric group and the general linear Lie algebra or their quantum counter parts, and who are seeking to learn a type $B$/$D$ analogue of the famous type $A$ story. Chapter 1 serves as an introduction to the type $B$ combinatorics and blob-diagrammatics. Chapters 2,3, and 4 contain proofs of three main theorem each - the first one is an explicit/computational proof of the easiest case of quantum coideal Schur--Weyl duality involving non-trivial quantizations of weight spaces, the second one shows the generalization of Jones--Wenzl projectors to the type $B$ Temperley--Lieb algebra and the third one uses  generalized Reidemeister moves and a functional analytic argument to prove a version of the fact that powers of the type $D$ full-twist converge towards the type $D$ Jones--Wenzl projector. 
	\end{abstract}
	\tableofcontents

	\chapter{Introduction}

\section*{Motivation}
Symmetry is everywhere! One of the easiest and most prominent examples is the permutation symmetry of a finite set $\{1, \ldots, n \}$ consisting of $n$ elements. Concretely this is the set of $n!$ many symmetries, which are the possible ways to permute the elements, and has the structure of a group denoted $S_n=S(\{1, \ldots, n \})$. Every permutation has a geometric interpretation as a reflection or rotation of the $n$ vertices of the $(n-1)$-simplex $\Delta^{n-1}$, which can be seen as the $(n-1)$-dimensional version of a equilateral triangle from a symmetries point of view. The $(n-1)$-simplex has two friends, which also exist for all $n$: the hypercube and the hyperoctahedron (also known as cross-polytope) which have $2^n$ (respectively $2n$) vertices and both generalize a square to $n$ dimensions. These two convex regular polytopes are dual to each other and hence have the same symmetries. Each of their symmetries can be expressed as a unique permutation $\sigma\in S(\{\pm1, \ldots, \pm n\})$, which satisfies $\sigma(-i)=-\sigma(i)$, giving a total of $2^n n!$ symmetries of the hypercube respectively the hyperoctahedron. As a special example the square has $8=2^2\cdot 2!$ symmetries, which are compositions of the two reflections $s_0\coloneqq(1,-1)$ and $s_1\coloneqq(1,2)(-1,-2)$, see figure \eqref{pictures for geometric symmetry}. The four symmetries with an even sign (viewed as permutations of the $4$ vertices) form an index $2$ subgroup, which is generated by $s_0'\coloneqq s_0s_1s_0$ and $s_1$.
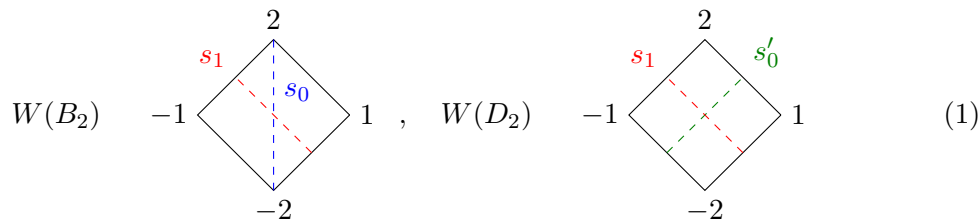
\begin{equation} \label{pictures for geometric symmetry} 
	W(B_2) \quad  \cbox{
	\begin{tikzpicture}
		\draw (1,0) -- (0,1) -- (-1,0) -- (0,-1) -- (1,0);
		\draw[blue, dashed] (0,1) -- (0,-1);
		\draw[red, dashed] (0.5,-0.5) -- (-0.5,0.5) node[anchor=south east] {\textcolor{red}{$s_1$}};
		\draw (1,0) node[anchor=west] {$1$};
		\draw (0,1) node[anchor=south] {$2$};
		\draw (-1,0) node[anchor=east] {$-1$};
		\draw (0,-1) node[anchor=north] {$-2$};
		\draw (0,0.3) node[anchor=west] {\textcolor{blue}{$s_0$}};
	\end{tikzpicture}
} \, , \quad W(D_2) \quad  \cbox{
	\begin{tikzpicture}
		\draw (1,0) -- (0,1) -- (-1,0) -- (0,-1) -- (1,0);
		\draw[black!50!green, dashed] (-0.5,-0.5) -- (0.5,0.5) node[anchor=south west] {\textcolor{black!50!green}{$s_0'$}};
		\draw[red, dashed] (0.5,-0.5) -- (-0.5,0.5) node[anchor=south east] {\textcolor{red}{$s_1$}};
		\draw (1,0) node[anchor=west] {$1$};
		\draw (0,1) node[anchor=south] {$2$};
		\draw (-1,0) node[anchor=east] {$-1$};
		\draw (0,-1) node[anchor=north] {$-2$};
	\end{tikzpicture}
}
\end{equation}
All these groups appear naturally in Lie theory as Weyl groups of semisimple algebraic groups. The symmetric group $S_n$ is the Weyl group $W(A_{n-1})$ of $\SL_n(\C)$ of Dynkin type $A_{n-1}$. The full group of hyperoctahedral symmetries is the Weyl group $W(B_n)$ of $\SO_{2n+1}(\C)$ of Dynkin type $B_n$ and agrees with the Weyl group $W(C_n)$ of $\Sp_{2n}(\C)$ of Dynkin type $C_n$, which is a shadow of Langlands duality for their Lie algebras $\so_{2n+1}$ and $\spp_{2n}$. The even sign subgroup of the hyperoctahedral group is the Weyl group $W(D_n)$ of $\SO_{2n}(\C)$ of Dynkin type $D_n$.
There are various questions concerning $W(B_n)$ motivated by the well-understood (representation) theory of the symmetric group $S_n=W(A_{n-1})$:
\begin{enumerate}
	\item The isomorphism classes of irreducible representations of $S_n$ correspond one-to-one with the set of partitions $\{\lambda\dashv n\}$ of $n$. Is there a similar picture for $W(B_n)$? \label{question 1}
	\item By Schur--Weyl duality the group algebra $\C\!S_n$ has as a quotient the Temperley--Lieb algebra $\TL_n=\End_{\sl2}(L(1)^{\otimes n})$, where $L(1)=\C^2$ is the standard representation of $\sl2(\C)$. Can this quotient be generalized to give a type $B$ Temperley--Lieb algebra $\TL(B_n)$? If yes, does this algebra come with a diagrammatic description extending the type $A$ diagrammatics? \label{question 2}
	\item The Temperley--Lieb algebra $\TL_n$ contains the Jones--Wenzl projector $a_{n-1}$, which projects onto the irreducible summand $L(n)$ of $L(1)^{\otimes n}$ containing the vector of maximal weight. The element $a_{n-1}$ can be described inherently in the Temperley--Lieb algebra as the unique non-zero idempotent annihilating all generators of $\TL_n$. If question \ref{question 2} has a positive answer, can a type $B$ Jones--Wenzl projector $b_n\in \TL(B_n)$ be constructed and if so does it have a representation theoretic interpretation? \label{question 3}
\end{enumerate}
In the classical (non-quantized) setting these questions have a long history with many results spread throughout the literature. Remarkably however, there is  no source which treats these questions all in one framework. The first goal of this thesis is to provide such a common setting and in particular bring together algebraic, combinatorial and diagrammatic approaches. We try to give a self-contained treatment of the mentioned subjects. Several of the provided proofs are new and developed to fit into the setup, although the results might be well-known. Based on this setup we treat the quantized setting. Hereby results from the last ten years are incorporated. Most prominently we discuss the new representation theoretic connection of the type $B$ Temperley--Lieb algebra to the coideal subalgebra $\Ug'_q(\gl_1\times\gl_1)$ from \cite{stroppel2018}. As a consequence we deduce new results on the (surprisingly rich!) representation theory of these coideal subalgebras and develop a theory similar to the very important and influential theory of $\Ug_q(\sl2)$. %We discuss now the content of the thesis in more detail.
\section*{Overview and main results} 

We start the thesis off with a presentation of the combinatorics of the type $B$ and $D$ Coxeter groups in \Cref{Chapter 1 section 1}. There we prove first the existence of an embedding $W(D_n)\hookrightarrow W(B_n)$ as an index $2$ subgroup, which allows us to treat type $B$ and $D$ simultaneously throughout the thesis. This treatment is also very convenient later in the quantized version. Next we discuss a diagrammatic description for $W(B_n)$ in terms of signed permutations, which we draw as string diagrams decorated with dots. This diagrammatic description of $W(B_n)$ gives an isomorphism to the semidirect product $(\mZ/2\!\mZ)^n\rtimes S_n$ obtained through gluing $S_n$ with $(\mZ/2\!\mZ)^n$ along the permutation action of the symmetric group on $(\mZ/2\!\mZ)^n$, and realizes the group $W(B_n)$ as a wreath product of $S_n$ with the cyclic group $\mZ/2\!\mZ$. In this way we combine insightful diagrammatics with known combinatorics for these groups. We use our description to describe the longest elements of $W(D_n)$ and $W(B_n)$ diagrammatically. The longest elements are obtained as products of multiplicative type $B$ Jucys--Murphy elements, which were treated in e.g.\ \cite{ogievetsky2012jucysmurphy} and also appear prominently in our proofs. We loosely follow \cite[\S10]{stroppel2018} and \cite{ogievetsky2012jucysmurphy}, moreover we use the language and general techniques for Coxeter groups, see e.g.\ \cite{humphreys1990}. In \Cref{Chapter 1 section 2} we present an answer to question \ref{question 1} by giving a concrete bijection between bipartitions $(\lambda,\mu)$ of $n$ and irreducible Specht modules denoted $S(\lambda, \mu)$ defined via idempotents in the group algebra $\C\! W(B_n)$. This bijection is well-known and follows for instance by the realiziation of $W(B_n)$ as the wreath product via the Wigner-Mackey little group argument from \cite[8.2]{serre77} from the fact that partitions label the irreducible representations of $S_n$. Our approach takes a different route and does not argue with classical character theory nor Schur--Weyl duality, which are the two classical approaches to the representation theory of $W(B_n)$. As a benefit we obtain very explicit and visually appealing descriptions of the irreducible representations. Just like in the mentioned proof of the bijection, the realization of $W(B_n)$ as wreath product is the key step in the classical approaches. The first approach is through the theory of irreducible characters for wreath products and can be found e.g.\ in \cite{geck00}, where it was applied to $W(B_n)$. The second approach -- Schur--Weyl duality for wreath products, see e.g. \cite[Theorem 9]{mazorczuk2015} -- gives a different way to classify the irreducible representations. In this version of Schur--Weyl duality the groups $W(B_n)$ and $\GL_a(\C)\times \GL_b(\C)$ act on tensor powers $V^{\otimes n}$ of $V=\C^a \oplus \C^b$. Since their actions commute and centralize each other, the double centralizer theorem gives a decomposition
\begin{equation} \label{Schur--Weyl decomp}
	V^{\otimes n}=\bigoplus_{(\lambda,\mu)}L(\lambda, \mu)\boxtimes S(\lambda, \mu)
\end{equation}
into bimodules, where $(\lambda, \mu)$ runs over pairs of partitions $(\lambda, \mu)$ such that $\lambda$ has at most $a$ rows, $\mu$ has at most $b$ rows and such that $n=|\lambda|+|\mu|$. Here $L(\lambda, \mu)$ respectively $S(\lambda, \mu)$ are irreducible finite dimensional representations of $\GL_a(\C)\times \GL_b(\C)$, respectively $W(B_n)$. This approach yields an abstract labeling of irreducible representations of $W(B_n)$ by pairs of partitions. However in contrast to $L(\lambda, \mu)$, which is just the tensor product $L(\lambda)\boxtimes L(\mu)$ of the irreducible Schur modules $L(\lambda)$ of $\GL_a(\C)$ and $L(\mu)$ of $\GL_b(\C)$, this procedure does not give a concrete description of $S(\lambda, \mu)$ in contrast to our more explicit approach.
We conclude the discussion of irreducible representations of $W(B_n)$ in \Cref{Chapter 1 section 3}. There we consider induction and restriction functors relating the representation categories $\Rep(W(B_n))$ for varied $n\in \mN$ and explain the type $B$ branching graph, which encodes a great deal of information about the Specht modules $S(\lambda,\mu)$ -- most notably their dimensions. Our treatment enhances the results in \cite{ogievetsky2012jucysmurphy} by a diagrammatic description in terms of idempotents.

\Cref{Chapter 1 section 4} gives an answer to question \ref{question 2} through a methodical review of \cite{green1998}, which historically marks a starting point of a substantial treatment of Temperley--Lieb algebras for classical types. There a family of type $B$ and $D$ Temperley--Lieb algebras has been studied. The type $B$ algebras have dimension ${{2n}\choose {n}}$, contain a type $A_{n-1}$ Temperley--Lieb algebra as a subalgebra, and most importantly come with a diagrammatic description in terms of decorated Temperley--Lieb diagrams. The type $D$ Temperley--Lieb algebras are smaller, however there are two different notions of such algebras. Our preferred definition is by design a subalgebra of $\TL(B_n)$, which is $\frac{1}{2}{{2n}\choose {n}}$-dimensional, while the notion from \cite{green1998} is larger and $\left(\frac{n+3}{2(n+1)} {{2n} \choose n}-1\right)$-dimensional.
The difficulty when working with type $B$ and $D$ Temperley--Lieb algebras is their dimension, which allows only for a few small examples: $\TL(B_2), \TL(D_2), \TL(D_3)$. 

\Cref{Chapter 2} will be a change in flavor and return to representation theory of $W(B_n)$ via a quantized version of the Schur--Weyl duality \eqref{Schur--Weyl decomp}. However instead of working with the group $\GL_n(\C)\times \GL_n(\C)$ or the Lie algebra $\gl_n \times \gl_n$, we work with the quantum symmetric pair 
\[
	(\Ug_q(\gl_{2m}), \Ug_q'(\gl_m \times \gl_m)).
\] 
Loosely speaking, a quantum symmetric pair $(\Ug_q(\lie), \Ug_q^{\iota})$ consists of a quantum group $\Ug_q(\lie)$ and a subalgebra $\Ug_q^{\iota}\subseteq\Ug_q(\lie)$, which is a $q$-deformation of the universal enveloping algebra $U(\lie^{\iota})$ of a fixed-point subalgebra $\lie^{\iota}\subseteq \lie$ taken with respect to an involution $\iota$ of $\lie$.  The most important difficulty of working with the subalgebra $\Ug_q^{\iota}$ is that it is not a Hopf subalgebra of $\Ug_q(\lie)$, but instead only a (right) coideal, i.e.\ the comultiplication $\Delta\colon \Ug_q(\lie) \rightarrow \Ug_q(\lie) \otimes \Ug_q(\lie)$ satisfies
\[
	\Delta(\Ug_q^{\iota})\subseteq \Ug_q^{\iota} \otimes \Ug_q(\lie) \quad \text{instead of} \quad  \Delta(\Ug_q^{\iota})\subseteq \Ug_q^{\iota} \otimes \Ug_q^{\iota}.
\] 
This turns off many techniques usually applicable for Hopf algebras, since the category $\Rep(\Ug_q^{\iota})$ is not monoidal. Instead it is only possible to tensor a module of $\Ug_q^{\iota}$ with a representation of the entire quantum group $\Ug_q(\lie)$ from the right. The subalgebra $\Ug_q^{\iota}\coloneqq\Ug_q'(\gl_m \times \gl_m)$ we are considering was treated in detail in \cite{stroppel2018} and \cite{watanabe2021}. A substantial treatment of coideal subalgebras, including a classification was given in \cite{letzter02}. Despite many recent developments in connection with reflection equations generalizing Yang--Baxter equations, their representation theory is still poorly understood. This thesis should be seen as a contribution to the most basic case: the representations of $\Vq)$. One might expect that the representation theory becomes easy in such a small example, but surprisingly it is already highly non-trivial to find explicit decompositions of representations of the coideal. The first of our main results gives such a explicit decomposition.
\begin{thmA}[\Cref{Theorem decomposition of tensor power of V}]
	Let $V$ be the standard quantized version of the vector representation $\C^2$, viewed as module over $\Ug_q(\gl_2)$. The $n$-th tensor power $V^{\otimes n}$ admits a non-trivial decomposition into $1$-dimensional subrepresentations, when viewed as a module over the coideal subalgebra $\Vq$. Let $B=q^{-1}EK^{-1}+F\in \Vq$. The decomposition is given by
	\begin{equation} \label{super cute decomposition}
		V^{\otimes n}=\bigoplus_{k=0}^n {n \choose k}L([n-2k]),
	\end{equation}
	where $L([k])$ denotes a one-dimensional subrepresentation of $V^{\otimes n}$ on which $B$ acts by the quantum integer $[k]=\frac{q^k-q^{-k}}{q-q^{-1}}$. All direct summands in the decomposition \eqref{super cute decomposition} can be described explicitly. 
\end{thmA}
Using this theorem we will show the Schur--Weyl duality \eqref{Schur--Weyl decomp} for $\Vq$ including the isomorphism
\[
\TL(B_n)\cong \End_{\Vq}(V^{\otimes n})
\]
of $\C(q)$-algebras, which gives a representation theoretic interpretation of the type $B$ Temperley--Lieb algebra. The decomposition \eqref{super cute decomposition} gives an answer to question \ref{question 3}, which is surprising coming from the type $A$ story. In type $A$ there is one Jones--Wenzl projector $a_{n-1}\in \TL_n\cong \End_{\Ug_q(\sl2)}(V^{\otimes n})$, which projects onto the unique $(n+1)$-dimensional irreducible submodule $L(q^n)$, see e.g.\ \cite{flath95} and the references therein. 
In contrast the algebra $\TL(B_n)$ has two (and not just one!) Jones--Wenzl projectors $\bp$ and $\bn$ projecting onto the subrepresentations $L([n])$ respectively $L(-[n])$ occurring with multiplicity $1$.
Their definition, their algebraic and diagrammatic properties, as well as their representation theoretic interpretation is the second main result.
\begin{thmB}[\Cref{Theorem on Existence of JW projectors of type B}, \Cref{characterizing properties of type B projector} and \Cref{type D projector recursive description}]
	There exist two special idempotents $\bp$ and $\bn$ in $\TL(B_n)$, which project onto $L([n])$ respectively $L(-[n])$. They are uniquely characterized by the property that they annihilate all cup-cap generators of $\TL(B_n)$ and satisfy
	\[
		\bp s_0 = \bp = s_0 \bp \quad \text{respectively} \quad \bn s_0 = (-1)\cdot\bn = s_0 \bn
	\]
	where $s_0\in\TL(B_n)$ is the dotted strand. Writing these idempotents as
	\[
		\bp = \cbox
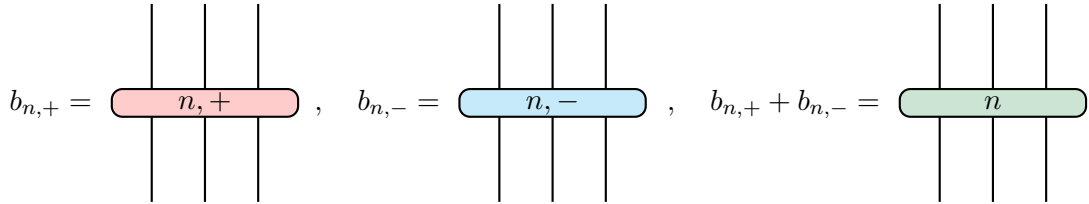
{
			\begin{tikzpicture}[tldiagram, yscale=1/2]
				% lines
				\draw \tlcoord{-2.5}{0} \lineup \lineup \lineup \lineup \lineup;
				\draw \tlcoord{-2.5}{1} \lineup \lineup \lineup \lineup \lineup;
				\draw \tlcoord{-2.5}{2} \lineup \lineup \lineup \lineup \lineup;
				% boxes
				\draw \tlcoord{0}{0} \maketlboxred{3}{$n, {+}$};
				=
			\end{tikzpicture}
		} \, , \quad \bn =\cbox{
		\begin{tikzpicture}[tldiagram, yscale=1/2]
			% lines
			\draw \tlcoord{-2.5}{0} \lineup \lineup \lineup \lineup \lineup;
			\draw \tlcoord{-2.5}{1} \lineup \lineup \lineup \lineup \lineup;
			\draw \tlcoord{-2.5}{2} \lineup \lineup \lineup \lineup \lineup;
			% boxes
			\draw \tlcoord{0}{0} \maketlboxblue{3}{$n, {-}$};
		\end{tikzpicture}
	} \, \, , \quad \bp +\bn = \cbox{
	\begin{tikzpicture}[tldiagram, yscale=1/2]
		% lines
		\draw \tlcoord{-2.5}{0} \lineup \lineup \lineup \lineup \lineup;
		\draw \tlcoord{-2.5}{1} \lineup \lineup \lineup \lineup \lineup;
		\draw \tlcoord{-2.5}{2} \lineup \lineup \lineup \lineup \lineup;
		% boxes
		\draw \tlcoord{0}{0} \maketlboxgreen{3}{$n$};
	\end{tikzpicture}
} 
	\] 
	gives a diagrammatic interpretation of their algebraic properties. The sum $d_n\coloneqq \bp + \bn$ is the unique non-zero idempotent contained in $\TL(D_n)$ annihilating all cup-cap generators of $\TL(D_n)$. It can be expressed by a recursive formula and all its coefficients are defined over $\mZ[[q]][q^{-1}]$.
\end{thmB}
\Cref{Chapter 3} is all about these Jones--Wenzl projectors and their properties. 
In the light of this theorem, purely algebraic idempotent formulas from \cite[\S3]{tomdieck1998} turn into idempotents with a representation theoretic meaning. Our description in terms of red, blue, and green boxes should then be seen as an interesting generalization of the graphical calculus for ordinary Jones--Wenzl projectors (see e.g.\ \cite{flath95} or \cite{kauffman1994}). It turns out that our formulas quantize the extremal projectors from \cite[\S2.8]{wedrich2018}, where the important subtleties for the quantization come from the coideal. The connection to \cite[\S2.8]{wedrich2018} will become more clear, when we explain in \Cref{affine section} how to view the Temperley--Lieb algebra $\TL(B_n)$ as a quotient of the extended affine Temperley--Lieb algebra used in their paper. 

Behind this thesis is the ultimate goal to find categorifications of the described structure. In particular we like to provide the categorical basement here for a categorification of these Jones--Wenzl projectors. In type $A$ three independent approaches were developed to categorify Jones--Wenzl projectors. One, \cite{cooper10}, using a cobordism category based on Temperley--Lieb combinatorics in type $A$, one using Lie theory and based on (categorified) Hecke algebra actions in \cite{stroppel2012} and finally one based on a categorified realization of the Jones--Wenzl projector as an iterated full-twist in \cite{rozansky2010}. The results presented so far in this thesis should be viewed as the combinatorial framework of a possible generalization of the first two approaches. In \Cref{Chapter on Infinite Braids} we propose a generalization for the combinatorics underlying \cite{rozansky2010} in types $B$ and $D$ by defining type $B$ and $D$ full-twists. The most important result of \Cref{Chapter on Infinite Braids} is the following theorem.
\begin{thmC}[\Cref{infinite type D braid}]
	Let $[\delta_n]\in\TL(D_n)$ be the image of the type $D$ full twist $\delta_n$ under the quotient map $\mZ[[q]][q^{-1}]B(D_n)\twoheadrightarrow \TL(D_n)$. Powers of $[\delta_n]$ converge in the $q$-adic norm on $\TL(D_n)$ towards the Jones--Wenzl projector $d_n$:
	\begin{equation} \label{super duper convergence}
	\lim_{m\to\infty}[\delta_n]^m= d_n = \quad \cbox{
		\begin{tikzpicture}[tldiagram, yscale=2/3]
			% lines
			\draw \tlcoord{-1.5}{0} \lineup \lineup \lineup;
			\draw \tlcoord{-1.5}{1} \lineup \lineup \lineup;
			\draw \tlcoord{-1.5}{2} \lineup  \lineup \lineup;
			% boxes
			\draw \tlcoord{0}{0} \maketlboxgreen{3}{$n$};
		\end{tikzpicture}
	}.
	\end{equation}
\end{thmC}
Since the Jones--Wenzl projectors are not defined over $\mZ$ the involved combinatorics in all approaches to the categorified projectors is non-trivial. Most formulas have to be interpreted in some kind of completion involving (infinite) Laurent series. In \Cref{chapter 5} we describe this phenomenon in detail illustrated by examples. The combinatorial difficulties are set in parallel to the homological constructions from \cite{rozansky2010}. As a byproduct we give more detailed and self-contained proofs of some results from \cite{rozansky2010} and write everything in such a way to fit within the framework of completed Grothendieck groups  from \cite{achar2011}. This allows to replace the missing integrality properties with a quite powerful theory involving semi-infinite complexes. 

\section*{Acknowledgments}

The Master's thesis as printed here would never have been possible without the help and encouragement of my family, friends, colleagues, and many faculty members of the University of Bonn. First, thanks to my family who has given me moral support throughout the years of my studies. Thanks to Thomas Perron--Frobenius--Häßel, Laura Khaze, Simon Korswird, Elbrus Mayer, Benjamin Nettesheim, Marena Richter, Jessica Schega, Nicolas Schmitt, Jonathan Wiebusch, Lilian Witters and all my other uncountable friends from the student council, which played a huge role throughout my studies and gave me emotional support. Then, thanks to my fellow students and now PhD students Lukas Bonfert, Jonas Nehme, Timm Peerenboom, and Liao Wang for their numerous notes, comments, and discussions on the topics of the thesis. I hope that each of us will successfully complete the PhD and that we will continue to work together in the future. Thanks to Jacob Matherne for comments on my English and an epic online seminar on Cluster Algebras in spring 2021 in Oregon. Thanks to Joanna Meinel for fun walks, discussions on the branching graph which inspired one section of the thesis, and many critical comments. 
%Additionally I want to thank Monica Vazirani for her series of talks on the ``Representation Theory and Combinatorics of the Symmetry Group and Related Structures'' and the organizers and speakers of the ``Learning Seminar on Categorification" for this source of interesting representation theory.
Thanks to Bart Vlaar for a few extremely helpful discussions on quantum symmetric pairs and providing me with intuition for these mysterious mathematical objects. Thanks to Paul Wedrich for a great seminar on link homology, including many discussions on our favorite (categorified) idempotents, and for proofreading this thesis. Lastly, thanks to Catharina Stroppel for many years of informative and enjoyable representation theory classes, interesting topic suggestions for bachelor's and master's theses and supervising them, and finally the energy with which she brings the representation theory research group together.
%
%Translated with www.DeepL.com/Translator (free version)
%
%
%Sollten übersetzt werden + mehr Leute dazukommen.
%Jonathan Wiebusch für Hilfe bei Konstruktion von Affine Temperley-Lieb Bildern mit Geogebra und lustiger Konstruktion mithilfe von Ellipsen (!), Thomas für moralische Unterstützung und dass er mir Perron-Frobenius gezeigt hat, Jonas für Diskussionen über Khovanov Algebra und mehr. Lukas für Diskussionen über $\mathfrak{gl}(1|1)$, Timm Peerenboom für mathematische Diskussionen bezüglich Signed Permutations und zahlreiche Korrekturvorschläge, Liao für Diskussionen über Representations von Typ B Weyl Gruppe, Stroppel für alles, Jendrik Stelzner ohne den ich bis heute nicht wüsste, wie man Lineare Algebra betreibt und dem die ursprüngliche Idee für die Latex Bilder entstammt, Elbrus Mayer für geometrische Intuition für Reidemeister-Bewegungen, Joanna Meinel für Diskussionen über Jucys--Murphy Elemente und Spaziergänge, Paul Wedrich für Diskussionen über Jones--Wenzl Idempotente und Cobordismen.
%Alle anderen, die ich vergessen hab. Meine Schwester Marta für ihre Unterstützung im Hintergrund, alle Leute die mit mir über Algebra reden oder gesprochen haben. Jonas und Lukas für Diskussionen über Jucys--Murphys.

%--------------------------------------------------------------

	\mainmatter

	\chapter{Type \texorpdfstring{$B$}{B} and \texorpdfstring{$D$}{D} combinatorics and diagrammatics} \label{Chapter 1}

In this chapter we collect facts about Coxeter groups, braid groups and Temperley--Lieb algebras of type $B$ and $D$, collected from many different sources throughout the literature.
While we assume that the reader knows the abstract theory of semisimple Lie algebras (see e.g.\ \cite{humphreys2012}), Coxeter groups (see e.g.\ \cite{humphreys1990}) and understands the type $A$ case very well, we do not require any special knowledge about types $B$ and $D$ or their combinatorics. We strive for this chapter to be as self-contained as possible and to provide the reader, who intends to learn the theory, with many visual examples coming from our diagrammatic treatment of the subject.
\Cref{Chapter 1 section 1} is all about the interplay between the Coxeter groups of types $A$, $B$ and $D$ and their compatible diagrammatics. In \Cref{Chapter 1 section 2} we focus on the representation theory of type $B$ Coxeter groups, which, while parallel to the representation theory of symmetric groups, is more difficult to work out explicitly since the number of irreducible representations grows quickly with $n$.
There we define the type $B$ Specht modules from idempotents in the group algebra $\C \! W(B_n)$ and give a sketch of the proof of the classification theorem. Afterwards we discuss induction and restriction functors in \Cref{Chapter 1 section 3} and explain their behavior when applied to the Specht modules from the second section via the branching rule.
In the final \Cref{Chapter 1 section 4} we present the type $B$ and $D$ Temperley--Lieb algebras, which are quotients of the Hecke algebras of their respective Coxeter groups and come with their own intriguing diagrammatics. 

%\begin{center}
%	\begin{tikzpicture}[xscale=-1]
%		\dynkin[
%		labels={n-1,n-2,n-3,3,2,0,1},
%		scale=2.5]{D}
%	\end{tikzpicture}
%\end{center}

We fix a nonnegative integer $n\geq0$. Our notation for the type $A,B,D$ Dynkin diagrams is
\begin{equation} \label{Dynkin diagrams}
	% https://tikzcd.yichuanshen.de/#N4Igdg9gJgpgziAXAbVABwnAlgFyxMJZABgBpiBdUkANwEMAbAVxiRAEEB9YMAWgEYAviEGl0mXPkIoATOSq1GLNvxFiQGbHgJEAzPOr1mrRCBlrxWqUQCsBxcbZ9Voy5J0oALPaPLTAHX8GKAgcBFcNCW1pElIZBV8TEAAhTkIIzXcYuX4EpSTiC0irD2Q5eMN8lSLM6L04vMdTcwyo6y8GyqaQQODQ8PVa9uQ7Coc-cAEattKyT0aJgBE06ZLs0l0FgtWsojl5romXQZmY-QPxpJaTtaJvC8S2XpCwnbqUOweq02c34YB2HzfcCBADmMAAjsc3O9kICxo8fmDIYVWrcUICvt0wMiIS0FDAoOCEChQAAzABOEAAtkh+NQcBAkHIQAALGB0KBOAisCKUmnMhlMxDeNkcrk-HlFfm0kVCpD6MWc7lgXnqGVIABs8sQgJADCwqrYISYACMGKxqOzlZKjQy6FgGCq1eSqbK9YykAAOK3i53St1IACcOp9Sol4ClfMDiBDIE9iH4ZHD-ujAsTLIT-FF1ojkCNadl-EVWZzfttLpAGsToqzdhTFYD6f49az2obkYLFEEQA
	\begin{tikzcd}
		A_{n-1} &  & s_1 \arrow[r, no head]             & s_2 \arrow[r, no head] & \ldots \arrow[r, no head]                    & s_{n-1}, &  & n\geq1, \\
		&  & s_0 \arrow[d, no head, Leftarrow] &                      &                                              &     &  &        \\
		B_n     &  & s_1 \arrow[r, no head]             & s_2                    & \ldots \arrow[l, no head] \arrow[r, no head] & s_{n-1}, &  & n\geq0, \\
		&  & s_0' \arrow[rd, no head]            &                      &                                              &     &  &        \\
		D_n     &  & s_1 \arrow[r, no head]             & s_2 \arrow[r, no head] & \ldots \arrow[r, no head]                    & s_{n-1}, &  & n\geq2,
	\end{tikzcd}
\end{equation}
where $A_0$ and $B_0$ are the empty diagrams and $D_2=A_1\times A_1$ consists of two vertices without edges. We denote by $W(A_{n-1}), W(B_n)$, and $W(D_n)$ their corresponding Coxeter groups with generators as indicated by the labels of the vertices in their respective Dynkin diagrams. More precisely,
$W(A_{n-1})$ is the group given by generators $\{s_i\mid 1\leq i \leq n-1\}$ subject to the type $A$ relations
\begin{enumerate}[label = {(A\arabic*)}, align=left]
	\item $s_i^2=e$
	\item $s_is_{i+1}s_i=s_{i+1}s_is_{i+1}$ for all $1\leq i \leq n-2$
	\item $s_is_j=s_js_i$ whenever $|i-j|\geq2$.
\end{enumerate}
The Coxeter group $W(B_n)$ has the same generators as $W(A_{n-1})$ and one additional generator $s_0$ subject to the type $A$ relations and the additional type $B$ relations
 \begin{enumerate}[label = {(B\arabic*)}, align=left]
 	\item $s_0^2=e$
 	\item $s_1s_0s_1s_0=s_0s_1s_0s_1$
 	\item $s_0s_i=s_is_0$ whenever $i\geq2$.
 \end{enumerate}
Lastly the group $W(D_n)$ has the same generators as $W(A_{n-1})$ and one additional generator $s_0'$ subject to the type $A$ relations and the additional type $D$ relations
\begin{enumerate}[label = {(D\arabic*)}, align=left]
	\item $(s_0')^2=e$
	\item $s_0's_2s_0'=s_2s_0's_2$ \label{Weyl group D ii}
	\item $s_0's_i=s_is_0'$ for all $i\neq2$. \label{Weyl group D iii}
\end{enumerate}

We denote by $B(A_{n-1})$, $B(B_n)$ and $B(D_n)$ the braid groups of their respective type with generators $\{\sigma_1, \ldots, \sigma_{n-1}\}$, $\{\sigma_0, \sigma_1, \ldots, \sigma_{n-1}\}$ and $\{\sigma_0', \sigma_1, \ldots, \sigma_{n-1}\}$, respectively, subject to the same relations as in their respective Coxeter counterparts with the exception of the quadratic relation $\sigma_i^2=e$, which is left out for all generators.

\section{Coxeter and braid groups of types \texorpdfstring{$B$}{B} and \texorpdfstring{$D$}{D}} \label{Three families of groups}
\label{Chapter 1 section 1}

This section is a gentle introduction to the Coxeter groups of type $B$ and $D$ and their respective braid groups. 
%In order to understand the Coxeter groups $W(B_n)$ and $W(D_n)$ in terms of the well-understood type $A$ case, we keep track of their parabolic subgroup generated by $s_1,\ldots, s_{n-1}$. 
Some Goals of this section are to prove the existence of inclusions
\[
	W(A_{n-1})\hookrightarrow W(D_n)\hookrightarrow W(B_n),
\] 
to show the correctness of the diagrammatics for $W(B_n)$ in terms of dotted/signed permutations and to define the multiplicative Jucys--Murphy elements $J_1,\ldots,J_n\in W(B_n)$. We start with giving the easiest examples for the Coxeter groups $W(B_n)$ and $W(D_n)$. 
\begin{beis}
	It is well-known that the Coxeter group $W(A_{n-1})$ is isomorphic to $S_n$, the group of permutations of $\{1,\ldots, n\}$ via mapping $s_i$ to the simple transposition $(i, i+1)$ for $1\leq i \leq n-1$ (see e.g.\ \cite{turaev2008} for an elementary proof).
	As degenerate special cases we have $W(B_0)=W(A_0)=\{e\}$, $W(A_1)=\{e,s_1\}\cong S_2$, and $W(B_1)=\{e,s_0\}\cong S_2$. The group $W(A_2)$ (respectively $W(B_2)$) is a dihedral group with $6$ (respectively $8$) elements. The smallest possible Coxeter group of type $D_n$ is $W(D_2)$, which is generated by the self-inverse elements $s_0'$ and $s_1$ which commute with each other, and hence isomorphic to $W(A_1)\times W(A_1)\cong S_2 \times S_2$. The next bigger example is $W(D_3)$, which is isomorphic to $W(A_3)\cong S_4$.
\end{beis} 
%	\begin{comment}
%		\begin{defi}
%		The Coxeter group of type $B_n$ $s_0,s_1,\ldots,s_{n-1}$ and relations 
%		\begin{enumerate}
%		\item $s_i^2= e$ for all $i=0,1,\ldots,n-1$
%		\item $s_0s_1s_0s_1=s_1s_0s_1s_0$
%		\item $s_is_{i+1}s_i=s_{i+1}s_is_{i+1}$ for all $i=1,\ldots, n-1$
%		\item $s_is_j=s_js_i$ for all $|i-j|\geq2$.
%		\end{enumerate}
%		\end{defi}
%		contains the Coxeter group of type $D_n$, which is given by generators $s_0', s_1, \ldots, s_{n-1}$ and relations
%		\begin{enumerate}
%		\item[i)] $(s_0')^2=e$ and $s_i^2= e$ for all $i=1,\ldots,n-1$.
%		\item[ii')] $s_0's_1=s_1s_0'$
%		\item[iii)] $s_is_{i+1}s_i=s_{i+1}s_is_{i+1}$ for all $i=1,\ldots, n-1$
%		\item[iii')] $s_0's_2s_0'=s_2s_0's_2$ 
%		\item[iv)] $s_is_j=s_js_i$ for all $|i-j|\geq2$ and $s_0's_j=s_js_0'$ for all $i\geq2$,
%		\end{enumerate}
%	\end{comment}
	As we will see the cardinalities of $W(B_n)$ and $W(D_n)$ are given by
	\[
	|W(A_{n-1})|=n! \, , \quad |W(D_n)|=2^{n-1}n! \, , \quad |W(B_n)|=2^n n! 
	\] 
	and can be already observed in these small examples. This will ultimately follow from \Cref{lemma inclusion of d in b} and \Cref{Weyl group of B_n is dotted group}.
	The proof of \Cref{lemma inclusion of d in b} will use the realization of the above Coxeter groups as Weyl groups of their associated root systems, i.e.\ the root system of the complex semisimple Lie algebras associated to their respective Dynkin type
	\[
		\lie(A_{n-1})\coloneqq \sln \, , \quad \lie(D_n)\coloneqq \so_{2n} \, , \quad \lie(B_n)\coloneqq \so_{2n+1}
	\]
	as given e.g.\ in \cite{bourbaki02}.
	This realization is proven for abstract root systems for instance in \cite{hofmann2007}. The lemma gives a way to relate the groups $W(A_{n-1})$, $W(D_n)$ and $W(B_n)$ via inclusion morphisms. It can be found for example in \cite[Lemma 10.2]{stroppel2018}.
	\begin{lemm} \label{lemma inclusion of d in b}
		Let $n\geq 2$. We have inclusions of groups
		\[
			\begin{tikzcd}[row sep=tiny]
				\phantom{s_i}W(A_{n-1}) \arrow[r, hook]   & W(D_n) \arrow[r, hook]        & W(B_n)\phantom{s_i}    \\
				\phantom{W(A_{n-1})}s_i \arrow[r, mapsto] & s_i \arrow[r, mapsto]  & s_i\phantom{W(A_{n-1})}       \\
				& s_0'  \arrow[r, mapsto] & s_0s_1s_0.
			\end{tikzcd}
		\]
	The second map identifies $W(D_n)$ with a normal index $2$ subgroup of $W(B_n)$.
	\end{lemm}
	\begin{proof}
	The first map is well-defined and maps the Coxeter generators of $W(A_{n-1})$ to the generators of a parabolic subgroup of $W(D_n)$ isomorphic to $W(A_{n-1})$, hence it is an isomorphism onto its image. 
	
	The second map, which we denote by $\iota$, is well-defined, since
	\[
		\iota(s_0')\iota(s_1)= s_0s_1s_0s_1=s_1s_0s_1s_0=\iota(s_1)\iota(s_0')
	\]
	and
	\begin{align*}
		\iota(s_0')\iota(s_2)\iota(s_0')&= s_0s_1s_0s_2s_0s_1s_0 \\
		&=s_0s_1s_2s_0s_0s_1s_0=s_0s_1s_2s_1s_0=s_0s_2s_1s_2s_0=s_2s_0s_1s_0s_2 =\iota(s_2)\iota(s_0')\iota(s_2).
	\end{align*}
	Also clearly the element $s_0s_1s_0$ commutes with all $s_i$ for $i\geq3$ and is self-inverse, since $s_0$ and $s_1$ are. All relations for the Coxeter generators other than $s_0'$ hold by the inclusion $W(A_{n-1})\subseteq W(B_n)$, $s_i\mapsto s_i$ for $1\leq n-1$, which is proven the same way as the inclusion $W(A_{n-1})\subseteq W(D_n)$. 
	
	For the injectivity of $\iota$ we argue using the commutative square
	\[
		\begin{tikzcd}
			W(D_n) \arrow[r, hook] \arrow[d, "\iota"] & \GL(\h^{*}(D_n)) \arrow[d, dashed, "\cong"] \\
			W(B_n) \arrow[r, hook]           & \GL(\h^{*}(B_n)),                   
		\end{tikzcd}
	\]
	where $\h^{*}(D_n)$ respectively $\h^{*}(B_n)$ denote the dual of a Cartan subalgebra $\h$ of $\lie(D_n)=\so_{2n}$ respectively $\lie(B_n)=\so_{2n+1}$. Both vector spaces are $n$-dimensional and come with bases $\alpha_0',\alpha_1,\ldots,\alpha_{n-1}$ and $\alpha_0,\alpha_1,\ldots,\alpha_{n-1}$ consisting of simple roots (c.f.\ \Cref{remark about roots for b and d}).
	The action of a simple reflection on $\h^{*}$ is defined in terms of the corresponding Cartan matrices. Concretely the action of $W(B_n)$ on $\h^{*}(B_n)$ is given by 
	\[
		s_i(\alpha_j)= \begin{cases}
			-\alpha_j, & \text{for $i=j$,} \\  
			2 \alpha_0 + \alpha_1 ,& \text{for $i=0$ and $j=1$,} \\
			\alpha_i + \alpha_j, & \text{for $i\neq 0$ and $j=i\pm 1$,} \\
			\alpha_j, &\text{otherwise.}  
		\end{cases}
	\]
	for $0\leq i,j\leq n-1$. The action of $W(D_n)$ on $\h^{*}(D_n)$ is given by
	\begin{align*}
		s_i(\alpha_j)&= \begin{cases}
			-\alpha_j, & \text{for $i=j$,} \\
			\alpha_i + \alpha_j, & \text{for $j=i\pm 1$,} \\
			\alpha_j, &\text{otherwise,}  
		\end{cases} \\
		s_0'(\alpha_j)&= \begin{cases}
			\alpha_0' + \alpha_2, & \text{for $j=2$,} \\
			\alpha_j, &\text{otherwise,}  
		\end{cases} \\
		s_i(\alpha_0')&= \begin{cases}
			\alpha_0' + \alpha_2 & \text{for $i=2$,} \\
			\alpha_0', &\text{otherwise,}  
		\end{cases} \\
		s_0'(\alpha_0')&= -\alpha_0'
	\end{align*}
	for $1\leq i,j \leq n-1$. 
	Sending $\alpha_0'\mapsto 2\alpha_0+\alpha_1$ and $\alpha_i\mapsto \alpha_i$ for $i\geq1$ gives a vector space isomorphism $\h^{*}(D_n)\cong\h^{*}(B_n)$, which is invariant under the action of $W(D_n)$. More precisely this isomorphism translates the above action of $W(D_n)$ on $\h^{*}(D_n)$ into the action of $W(D_n)$ on $\h^{*}(B_n)$ via restriction $W(D_n)\rightarrow W(B_n) \rightarrow \GL(\h^{*}(B_n))$. 
	Hence by conjugating with this isomorphism we obtain the vertical isomorphism
	\[
		\GL(\h^{*}(D_n))\xrightarrow{\cong} \GL(\h^{*}(B_n))
	\]
	making the diagram commute. We conclude that $\iota$ is injective from the commutativity of the diagram. 
	
	The group $W(D_n)$ is a normal subgroup of $W(B_n)$, since it is stable under conjugation with the Coxeter generators of $W(B_n)$. The quotient $W(B_n)/W(D_n)$ is generated by $s_0$, hence has order either $1$ or $2$. However the number of $s_0$'s occurring in any (reduced) expression of $w\in W(B_n)$ is constant modulo $2$ by Matsumoto's theorem, which shows that the latter is the case.
	\end{proof}
	\begin{warn}
		The inclusion of $W(D_n)$ into $W(B_n)$ is very useful for group theoretic arguments, but not for arguments about Coxeter groups. For instance this inclusion does not preserve the length of group elements.
	\end{warn}
	\begin{beis}
		We visualize the inclusions from \Cref{lemma inclusion of d in b} for $n=2$ via the (weak) Bruhat graphs/Cayley graphs 
		\begin{equation*}
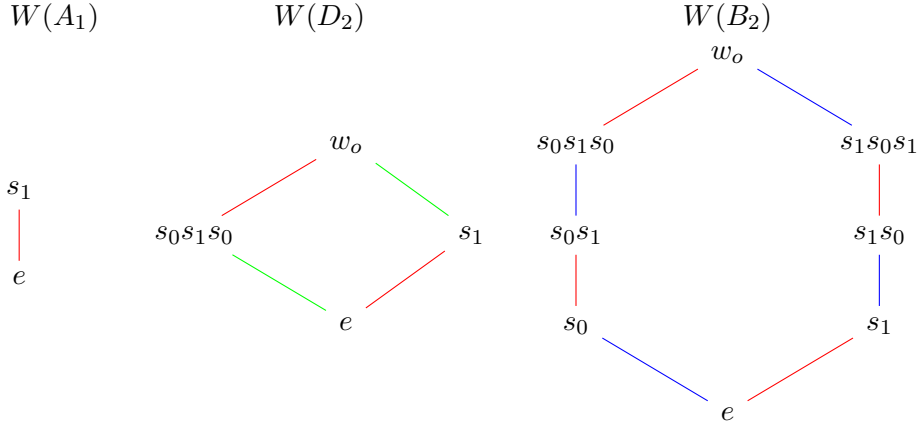
			\begin{array}{c c c}
				W(A_1)& W(D_2)& W(B_2)\\
				\begin{tikzcd}
					s_1 \arrow[d, no head, red] & \\
					e            
				\end{tikzcd}  & \begin{tikzcd}
				& \wo                     &                \\
				s_0s_1s_0 \arrow[ru, no head, red] &                         & s_1 \arrow[lu, no head, green] \\
				& e \arrow[ru, no head, red] \arrow[lu, no head, green] &               
			\end{tikzcd}  &
		\begin{tikzcd}
			& \wo \arrow[rd, no head, blue] &                              \\
			s_0s_1s_0 \arrow[d, no head, blue] \arrow[ru, no head, red] &                         & s_1s_0s_1 \arrow[d, no head, red] \\
			s_0s_1 \arrow[d, no head, red]                        &                         & s_1s_0 \arrow[d, no head, blue]    \\
			s_0 \arrow[rd, no head, blue]                          &                         & s_1                          \\
			& e \arrow[ru, no head, red]   &                             
		\end{tikzcd}
			\end{array}
		\end{equation*}
		of their respective Coxeter groups. The element $\wo$ is the longest element of both $W(B_2)$ and $W(D_2)$ and is given by $s_0's_1=s_0s_1s_0s_1=s_0s_1s_0s_1=s_1s_0'$. 
		%The inclusions from \Cref{lemma inclusion of d in b} are
%		\[
%			W(A_1)=\{e,s_1\}\subseteq W(D_2)=\{e,s_1,s_0s_1s_0, s_0s_1s_0s_1\}\subseteq W(B_2).
%		\]
		One should note that the longest element $\wo\in W(B_n)$ is not always contained in $W(D_n)$. We will see in \Cref{longest element in type B} a concrete diagrammatic description of the longest element and that it is already contained in $W(D_n)$ if and only if $n$ is even in \Cref{longest element in type D}.
	\end{beis}
	\begin{rema} \label{remark about roots for b and d}
		The isomorphism $\h^{*}(D_n)\cong \h^{*}(B_n)$ which we used in the proof of \Cref{lemma inclusion of d in b} identifies the roots of $\so_{2n}(\C)$ with the long roots of $\so_{2n+1}(\C)$. For instance if $n=2$ the long roots of the $B_2$ root system
		\begin{equation} \label{B2 root system}
			\begin{tikzpicture}
				\draw [->,line width=1pt] (0,0) -- (1,1);
				\draw [->,line width=1pt] (0,0) -- (1,0);
				\draw [->,line width=1pt] (0,0) -- (0,1);
				\draw [->,line width=1pt] (0,0) -- (-1,1);
				\draw [->,line width=1pt] (0,0) -- (-1,0);
				\draw [->,line width=1pt] (0,0) -- (-1,-1);
				\draw [->,line width=1pt] (0,0) -- (0,-1);
				\draw [->,line width=1pt] (0,0) -- (1,-1);
				\draw (2,1.3) node{$2\alpha_0+\alpha_1$};
				\draw (1.5,0) node{$\alpha_0$};
				\draw (0,1.3) node{$\alpha_0+\alpha_1$};
				\draw (-1.5,1.3) node{$\alpha_1$};
			\end{tikzpicture}
		\end{equation}
		give the $D_2=A_1\times A_1$ root system. For the general root systems of type $B_n$ and $D_n$ see \eqref{full root system type B} and \eqref{full root system type D}.
	\end{rema}
	\begin{rema} \label{braid group of type d inside type b}
		Note that the embedding from \Cref{lemma inclusion of d in b} cannot be directly generalized to an embedding of the corresponding braid groups $\overline{\iota}\colon B(D_n)\hookrightarrow B(B_n)$.
		Recall that $B(D_n)$ and $B(B_n)$ are generated by $\sigma_0', \sigma_1, \ldots, \sigma_{n-1}$ and $\sigma_0, \sigma_1, \ldots, \sigma_{n-1}$, respectively. Both obvious candidates for an image of $\sigma_0'$
		\begin{equation*}
			\overline{\iota}(\sigma_0')\coloneqq\sigma_0 \sigma_1 \sigma_0 \quad \text{and} \quad \overline{\iota}(\sigma_0')\coloneqq\sigma_0 \sigma_1 \sigma_0^{-1} \,
		\end{equation*}
		do not give a well-defined map of groups. However the first choice satisfies the type $D$ relation $\sigma_0'\sigma_1=\sigma_1\sigma_0'$ from \ref{Weyl group D iii}, while the second choice satisfies the braid relation $\sigma_0'\sigma_2\sigma_0'=\sigma_2\sigma_0'\sigma_2$ from \ref{Weyl group D ii}.
		This leads to the conclusion that the braid group $B(D_n)$ of type $D_n$ should be mapped into the quotient $B_1(B_n)$ of $B(B_n)$ obtained by enforcing the additional quadratic relation $s_0^2\coloneqq\sigma_0^2=e$. We summarize this picture in a diagram, which explains which groups live where and where they map:
		\begin{equation*}
			% https://tikzcd.yichuanshen.de/#N4Igdg9gJgpgziAXAbVABwnAlgFyxMJZABgBpiBdUkANwEMAbAVxiRACEAKdgfTAEoQAX1LpMufIRQBGclVqMWbXtO59BIsdjwEis6fPrNWiDpwAi64aJAZtkogCZSB6kaWmA6haubb4nSlkZ0o3RRMQb14Bay0JXRlSR0Nw5U4AQR5gMABaaSENGzt4oOdksOM2b0zsvILheRgoAHN4IlAAMwAnCABbJDIQHAgkWRAACxg6KDYcAHcISemEP26+pGchkcQxpZnEMCYGBmocOiwGNnGICABrWJA1-sRN4aQAZlWe57G3xAAWah7WYLPYrGxPD6nbaAobnS6ma53B6QxAAVmhG1O8KuN3uX3W6MxiAAbATniTie9sRdccjyUgMVsBtQ4OMsB0cCy4bTEXiGkIgA
			\begin{tikzcd}
				B(B_n) \arrow[r, two heads] & B_1(B_n) \arrow[r, two heads]                          & W(B_n)                     \\
				& B(D_n) \arrow[u] \arrow[r, two heads]                       & W(D_n) \arrow[u, hook]     \\
				& B(A_{n-1}) \arrow[u] \arrow[r, two heads] \arrow[luu] & W(A_{n-1}) \arrow[u, hook]
			\end{tikzcd}
		\end{equation*}
		We expect that the map $B(D_n)\rightarrow B_1(B_n)$ is injective, however we have no proof of this.
	\end{rema}
\begin{rema}[Big picture]
		For every $n$ we have have an inclusion of each ``$n$-version'' of the ``$(n+1)$-version'' of that group as a parabolic subgroup. 
		So we have the following three-dimensional commutative diagram: 
	\begin{equation*}
		% https://tikzcd.yichuanshen.de/#N4Igdg9gJgpgziAXAbVABwnAlgFyxMJZAJgBoAGAXVJADcBDAGwFcYkQAhACg4H1gwAagCMAXwCUIUaXSZc+QigDMFanSat2fYT35Cxk6bOx4CRACyqaDFm0QgA6roEiJUmSAwmFRFcTU2mvbcACJ6roYeXvJmKJb+1hp2jlxhLgbuxjGKyCrmAUlaXACCem5GnnKmOZb5ibbsTqUC5VFVPijkpEoFDcHOYK1Z1UTC3b1BnLw6fC2Rwx0k4-WTTnxg85XeschjAKwTyaG8G5lb2URkByvJTmGnFdEjKGMA7IdFzWAAtBmP7TsyO8bo0SnpfkNzs8lsIPvYADrwtAAC3oYBwEAAtsBiqIzk9FioAGxwkCIlFojHY3FSNQwKAAc3gRFAADMAE5YpB7GgYpBKCocrmIIm8iBISwgHD0LCMdjIiAQADWZyFmP5YqQYg8aolmsQZClMrl9gVytVnPViB5UvFiHIgstSC6tq1NGRMHoUHYOAA7hAPV6EI7hWNXQb3Z7vfY-QGo8GdU7rfrRSBA9Gpf70wm2UmVOHJemfVn4xbQ4a+YhhA7E6H85XhGHpbL5YqVSGrdX9QAODtapt2gCcvONrfNfcQ3f1w7TUeLcaDZc7LobYaLMZLi4nwgrduE+fXmYXUBzIF1iFePcjXvn2aXWt3SFeE8v4ZdjCwYGSUHocA93onKdwzDD8v3YH8-3pe9EBnStDVA79f3-aCu3DfMEPApCoO3Ac9RADD7Ag5Dt0fZN8M-RDIIA2tO3rO1UwIkAiOwyhRCAA
		\begin{tikzcd}[row sep = small]
			&                                                                                & B(B_{n+1}) \arrow[r, two heads]                   & B_1(B_{n+1}) \arrow[r, two heads]                      & W(B_{n+1})                  \\
			&                                                                                & \phantom{A}                                       &                                                        &                             \\
			&                                                                                &                                                   & B(D_{n+1}) \arrow[uu] \arrow[r, two heads]             & W(D_{n+1}) \arrow[uu, hook] \\
			B(B_{n}) \arrow[r, two heads] \arrow[rruuu, dashed] & B_1(B_{n}) \arrow[r, two heads] \arrow[rruuu, dashed]                          & W(B_n) \arrow[rruuu, dashed]                      &                                                        &                             \\
			&                                                                                &                                                   & B(A_{n}) \arrow[uu] \arrow[luuuu] \arrow[r, two heads] & W(A_{n}) \arrow[uu, hook]   \\
			& B(D_n) \arrow[uu] \arrow[r, two heads] \arrow[rruuu, dashed]                   & W(D_n) \arrow[uu, hook] \arrow[rruuu, dashed]     &                                                        &                             \\
			&                                                                                &                                                   & \phantom{A}                                            &                             \\
			& B(A_{n-1}) \arrow[uu] \arrow[r, two heads] \arrow[luuuu] \arrow[rruuu, dashed] & W(A_{n-1}) \arrow[uu, hook] \arrow[rruuu, dashed] &                                                        &                            
		\end{tikzcd}
	\end{equation*}
	In \Cref{theorem induction and restriction behaviour for type B} we will study the behavior of induction and restriction functors applied to irreducible representations of $W(B_n)$ coming from the inclusion $W(B_n)\subseteq W(B_{n+1})$ of parabolic subgroup, which is the top right diagonal arrow in the diagram. Observe that while there is only one way to view $W(B_n)$ as a parabolic subgroup of $W(B_m)$ for $n\leq m$, there are many non-parabolic subgroups $U\subseteq W(B_m)$ isomorphic to $W(B_n)$. For instance one can conjugate the obvious parabolic subgroup with some element $g\in W(A_{n-1})$. The diagrammatics for $W(B_n)$ which we introduce later will make this phenomenon very visible and show many non-trivially embedded, non-parabolic subgroups of $W(B_m)$ isomorphic to $W(B_m)$.
\end{rema}
We take a break from the abstract theory and discuss a well-known map from the braid group $B(B_n)$ to a group of diagrams. Notably the proof includes the diagrammatic identity \eqref{picture type B relation}, which gives an interpretation of the type $B$ relation
\begin{equation} \label{type B relation formula}
	\sigma_1\sigma_0\sigma_1\sigma_0 = \sigma_0\sigma_1\sigma_0\sigma_1.
\end{equation}
\begin{lemm} \label{embedding bn into an}
	There is a group homomorphism $\iota\colon B(B_n)\rightarrow B(A_n)$ defined on the generators via
	\begin{equation*}
		\sigma_0 \mapsto \sigma_1^2\, , \quad \sigma_i \mapsto \sigma_{i+1} \quad \text{ for } i=1\ldots,n-1.
	\end{equation*}
	Diagrammatically this homomorphism maps the generators $\sigma_1, \ldots, \sigma_{n-1}$ of $B(B_n)$ to standard type $A$ generators, while the additional generator $\sigma_0$ is mapped to a twist around an additional fixed strand on the left.
\end{lemm}
\begin{proof}
	We have to show that the map is well-defined. The only non-trivial relation to check is the type $B$ relation \eqref{type B relation formula}, which follows from the braid identity
	\begin{equation} \label{picture type B relation}
		\iota(\sigma_1\sigma_0\sigma_1\sigma_0) \quad = \quad \cbox{
			\begin{tikzpicture}[tldiagram, yscale=2/3]
				%dots 
				%\drawdots{0}{0}{1}
				%lines
				\draw \tlcoord{0}{0.5} \lineup;
				\draw[white, double=black] \tlcoord{0}{1} \linewave{1}{-1} \linewave{2}{2} \lineup \lineup;
				\draw[white, double=black] \tlcoord{1}{0.5} \lineup \lineup \lineup;
				\draw[white, double=black] \tlcoord{0}{2} \lineup \lineup \linewave{2}{-2} \linewave{2}{2};
				\draw[white, double=black] \tlcoord{4}{0.5} \lineup \lineup;
				\draw[white, double=black] \tlcoord{5}{2} \linewave{1}{-1};
				%\draw  \tlcoord{0}{0} \linewave{1}{1};
				%\draw[white, double=black] \tlcoord{0}{1} \linewave{1}{-1} \capright;
			\end{tikzpicture}
		} \quad = \quad
		\cbox{
			\begin{tikzpicture}[tldiagram, yscale=2/3]
				%dots 
				%\drawdots{0}{0}{1}
				%lines
				\draw \tlcoord{-1}{0.5} \lineup \lineup;
				\draw \tlcoord{-1}{1} \linewave{1}{1};
				\draw[white, double=black] \tlcoord{-1}{2} \linewave{2}{-2} \linewave{2}{2} \lineup \lineup;
				\draw[white, double=black] \tlcoord{1}{0.5} \lineup \lineup \lineup;
				\draw[white, double=black] \tlcoord{0}{2} \lineup \lineup \linewave{2}{-2} \linewave{1}{1};
				\draw[white, double=black] \tlcoord{4}{0.5} \lineup;
				%\draw  \tlcoord{0}{0} \linewave{1}{1};
				%\draw[white, double=black] \tlcoord{0}{1} \linewave{1}{-1} \capright;
			\end{tikzpicture}
		} \quad =\quad \iota(\sigma_1\sigma_0\sigma_1\sigma_0) \, .
	\end{equation}
	Diagrammatically the right braid is obtained from the left one by moving one of the loops going around the leftmost strand through the other. In formulas we have
	\[
		\iota(\sigma_1\sigma_0\sigma_1\sigma_0)=\sigma_2\sigma_1^2\sigma_2\sigma_1^2=\sigma_2\sigma_1 \sigma_2\sigma_1\sigma_2\sigma_1
	\]
	and 
	\[
		\iota(\sigma_0\sigma_1\sigma_0\sigma_1)=\sigma_1^2\sigma_2\sigma_1^2\sigma_2=\sigma_1\sigma_2 \sigma_1\sigma_2\sigma_1\sigma_2,
	\]
	where we used the braid identity $\sigma_1\sigma_2\sigma_1=\sigma_2\sigma_1\sigma_2$. By applying the braid identity twice more we see
	\[
		\sigma_2\sigma_1 \sigma_2\sigma_1\sigma_2\sigma_1 = \sigma_1 \sigma_2\sigma_1^2\sigma_2\sigma_1=\sigma_1\sigma_2 \sigma_1\sigma_2\sigma_1\sigma_2 . \qedhere
	\]
\end{proof}
\Cref{embedding bn into an} is not just a solitary statement. Historically it opened the door to studying links in the solid torus, as we outline in the next remark. 
\begin{rema}
	According to \cite{lambropoulou93} the group homomorphism $\iota$ is an inclusion and identifies the braid group of type $B_n$ with all braids on $n$ strands, which can possibly twist around an additional left strand. Going around the left strand can be interpreted as going upwards around the ``hole'' in the solid torus $S^1 \times [0,1]^2$. Framed links $L$ in the solid torus can be obtained by closing up such a type $B$ braid, makes it possible to define link invariants via defining them on $\bigcup_{n\geq 0} B(B_n)$ and checking Markov moves. This point of view has been successfully applied in \cite{Lam94}, where a Ocneanu trace was defined and used to obtain an link invariant. However historically this topology has not yet been connected with a representation-theoretic picture. \Cref{Chapter 2} is a step towards developing such a picture, where the connection between topology and representation theory is made by considering intertwiners between tensor powers of the standard representation for the coideal subalgebra $\Vq$ of $\Ug_q(\gl_2)$. 
	For further reference on the purely topological point of view see also \cite{kimphd}, \cite{lambropoulou2004knot} and the references therein.
\end{rema}

Motivated by the representation theory from the second chapter, we introduce a diagrammatic notation for elements in the quotient braid group $B_1(B_n)$ of $B(B_n)$, c.f.\ \eqref{braid group of type d inside type b}. Since $\sigma_0$ becomes self-inverse in the quotient, we will not draw the fixed strand on the left like in \Cref{embedding bn into an}, but instead draw $s_0\coloneqq[\sigma_0]\in B_1(B_n)$ diagrammatically as a dot on the first strand. The relation \eqref{type B relation formula} becomes then
\begin{equation*}
	\cbox{
		\begin{tikzpicture}[tldiagram, yscale=2/3]
			%dots
			%\drawdots{0}{0}{1}
			%\drawdots{2}{0}{1}
			%crossings
			\lowernegcrossing{0}{0}
			\lowernegcrossing{1}{0}
			%nodes
%			\node at \tlcoord{0}{0} [anchor = north] {$\scriptstyle 1$};
%			\node at \tlcoord{0}{1} [anchor = north] {$\scriptstyle 2$};
%			\node at \tlcoord{2}{0} [anchor = south] {$\scriptstyle 1$};
%			\node at \tlcoord{2}{1} [anchor = south] {$\scriptstyle 2$};
		\end{tikzpicture}
	} \quad = \quad \cbox{
		\begin{tikzpicture}[tldiagram,yscale=2/3]
			%dots
			%\drawdots{0}{0}{1}
			%\drawdots{2}{0}{1}
			%crossings
			\uppernegcrossing{0}{0}
			\uppernegcrossing{1}{0}
			%nodes
%			\node at \tlcoord{0}{0} [anchor = north] {$\scriptstyle 1$};
%			\node at \tlcoord{0}{1} [anchor = north] {$\scriptstyle 2$};
%			\node at \tlcoord{2}{0} [anchor = south] {$\scriptstyle 1$};
%			\node at \tlcoord{2}{1} [anchor = south] {$\scriptstyle 2$};
		\end{tikzpicture}
	},
\end{equation*}
while the additional relation $s_0^2=e$ is drawn as two consecutive dots canceling each other:
\begin{equation} \label{two dots cancel each other}
	\cbox{
		\begin{tikzpicture}[tldiagram, yscale=2/3]
			%dots
			%\drawdots{0}{0}{0}
			%\drawdots{2}{0}{0}
			%nodes
			%\node at \tlcoord{0}{0} [anchor = north] {$\scriptstyle 1$};
			%\node at \tlcoord{2}{0} [anchor = south] {$\scriptstyle 1$};
			
			%lines
			\draw \tlcoord{0}{0} \dlineup \dlineup;
			%big dots
			%\biggerdrawdots{0.66}{0}{0}
			%\biggerdrawdots{1.33}{0}{0}
		\end{tikzpicture}
	} \quad = \quad \cbox{
		\begin{tikzpicture}[tldiagram,yscale=2/3]
			%dots
			%\drawdots{0}{0}{0}
			%\drawdots{2}{0}{0}
			%nodes
			%\node at \tlcoord{0}{0} [anchor = north] {$\scriptstyle 1$};
			%\node at \tlcoord{2}{0} [anchor = south] {$\scriptstyle 1$};
			%lines
			\draw \tlcoord{0}{0} \lineup \lineup;
		\end{tikzpicture}
	} \, .
\end{equation}
This diagrammatics motivates the next definition of the group of dotted permutations, which is a version of the symmetric group with many dots everywhere. We will prove that this group is isomorphic to the group $W(B_n)$ in the subsequent lemma.
%This diagrammatics is not new and can be for instance found in \cite{elias20}[Exercise 1.12].

\begin{defi}
	The \textbf{group of dotted permutations} $G_n$ is the set of all tuples $x=(\sigma, \varepsilon)$, where $\sigma\in S_n$ and $\varepsilon\in\{0,1\}^n$. We call $\varepsilon$ a \textbf{dot-configuration}.
	Each element $x=(\sigma, \varepsilon)$ of $G_n$ is drawn as the string diagram $\sigma$ on $n$ strands, where every strand is decorated with at most one dot. The placement of dots is specified by the dot-configuration $\varepsilon$, where each $1$ corresponds to at dot placed on the bottom of the string diagram. The multiplication is given by stacking string diagrams on top of each other, with two dots on the same strand canceling each other as in \eqref{two dots cancel each other}.
\end{defi}
\begin{beis}
	The element $x= (s_1s_2, (0,1,1))\in G_3$ is given by the diagrams
	\[
	x \quad = \quad \cbox{
		\begin{tikzpicture}[tldiagram, yscale=2/3]
			\draw \tlcoord{0}{0} \linewave{2}{1};
			\draw \tlcoord{0}{1} \onedot \linewave{2}{1};
			\draw \tlcoord{0}{2} \onedot \linewave{2}{-2};
		\end{tikzpicture}
	} \quad = \quad \cbox{
		\begin{tikzpicture}[tldiagram, yscale=2/3]
			\draw \tlcoord{0}{0} \linewave{2}{1};
			\draw \tlcoord{0}{1} \dlinewave{2}{1};
			\draw \tlcoord{0}{2} \dlinewave{2}{-2};
		\end{tikzpicture}
	}.
	\]
	As visualized dots can be moved arbitrarily along the strands. 
	For example in $G_3$ the product of the elements
	\[
	x\quad = \quad \cbox{
		\begin{tikzpicture}[tldiagram, yscale=2/3]
			\draw \tlcoord{0}{0} \linewave{2}{1};
			\draw \tlcoord{0}{1} \onedot \linewave{2}{1};
			\draw \tlcoord{0}{2} \onedot \linewave{2}{-2};
		\end{tikzpicture}
	} \quad \text{and} \quad y\quad = \quad \cbox{
		\begin{tikzpicture}[tldiagram, yscale=2/3]
			\draw \tlcoord{0}{0} \linewave{2}{2};
			\draw \tlcoord{0}{1} \onedot \linewave{2}{-1};
			\draw \tlcoord{0}{2} \onedot \linewave{2}{-1};
		\end{tikzpicture}
	}
	\]
	is given by
	\[
	x \cdot y \quad = \quad \cbox{
		\begin{tikzpicture}[tldiagram, yscale=2/3]
			\draw \tlcoord{0}{0} \linewave{2}{2} \onedot \linewave{2}{-2};
			\draw \tlcoord{0}{1} \onedot \linewave{2}{-1} \linewave{2}{1};
			\draw \tlcoord{0}{2} \onedot \linewave{2}{-1} \onedot \linewave{2}{1};
		\end{tikzpicture}
	} \quad = \quad \cbox{
		\begin{tikzpicture}[tldiagram, yscale=4/3]
			\draw \tlcoord{0}{0} \lineup \onedot \lineup;
			\draw \tlcoord{0}{1} \lineup \onedot \lineup;
			\draw \tlcoord{0}{2} \lineup \lineup;
		\end{tikzpicture}
	} \quad = \quad (e, (1,1,0)) \, .
	\]
\end{beis}
\begin{rema}
	Observe that the group $G_n$ has by definition cardinality $2^n n!$, which will turn out to be a lower bound for the cardinality of $W(B_n)$ and an important step in proving that $W(B_n)$ and $G_n$ are isomorphic. 
\end{rema}

Although the group $G_n$ might look artificial at first sight, it appears in our daily life as the next ``application'' shows, which is due to G.\ D.\ James. and S.\ Lambropoulou from \cite{Lam94}.
\begin{rema} 
	Consider a shoe cabinet consisting of $n$ ordered shelves $\{1,\ldots, n\}$ filled with $n$ different pairs of shoes labelled by $\{1, 2, \ldots, n\}$, where $i_l$ denotes the left and $i_r$ the right shoe of the $i$th pair $\{i_l, i_r\}$. However the pairs of shoes might not be in their correct shelf, also the left and right shoe might be swapped. For instance
	\begin{equation*}
		\ytableausetup{boxsize=3.7em}
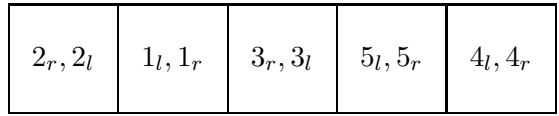
		 \begin{ytableau}
			2_r, 2_l & 1_l, 1_r & 3_r, 3_l & 5_l, 5_r & 4_l, 4_r 
		\end{ytableau} 
	\end{equation*}
	would be such a constellation. (This is not the shoe cabinet of the author, since he neither owns a shoe cabinet nor $5$ pairs of shoes. Also numbers are not shoes but rather decategorifications of sets.) 
	%The second part of this joke is by Timm Peerenboom, don't blame the author.)
	The group $G_n$ acts faithfully and transitively on such constellations of shoes by permuting the pairs, where a dot on the $i$-th strand (see \eqref{image of Jucys-Murphy element in Type B Weyl Group}) swaps the right and left shoe on the $i$-th shelf. The movement of a dot on a string just corresponds to the fact that it does not matter whether we first move the pair of shoes and then swap the left and right shoe or the other way around.
\end{rema}
\ytableausetup{smalltableaux}
\begin{lemm} \label{Weyl group of B_n is dotted group}
	The assignments 
	\begin{align*}
		\Phi\colon W(B_n) &\longrightarrow G_n \\
		s_0 &\longmapsto \cbox
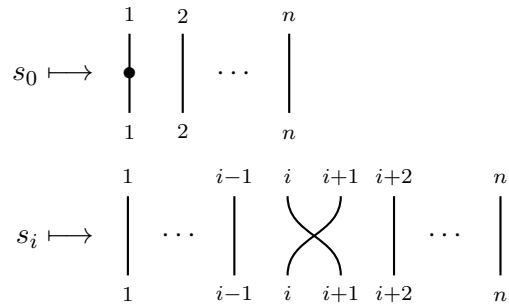
{
		\begin{tikzpicture}[tldiagram]
			% dots
			%\drawdots{0}{0}{0}
			%\drawdots{0}{2}{5}
			%\drawdots{0}{7}{7}
			%\drawdots{1}{0}{0}
			%\drawdots{1}{2}{5}
			%\drawdots{1}{7}{7}
			% lines
			\draw \tlcoord{0}{0} \dlineup;
			\draw \tlcoord{0}{1} \lineup;
			\draw \tlcoord{0}{3} \lineup;
			%\draw \tlcoord{0}{7} \lineup;
			% dots
			\makecdots{0}{2}
			%\makecdots{0}{6}
			% cross
			%\draw \tlcoord{0}{3} -- \tlcoord{1}{4};
			%\draw \tlcoord{0}{4} -- \tlcoord{1}{3};
			% bottom nodes
			\node at \tlcoord{0}{0} [anchor = north] {$\scriptstyle 1\vphantom{k}$};
			\node at \tlcoord{0}{1} [anchor = north] {$\scriptstyle {2}$};
			%\node at \tlcoord{0}{3} [anchor = north] {$\scriptstyle k$};
			%\node at \tlcoord{0}{4} [anchor = north] {$\scriptstyle k+1$};
			%\node at \tlcoord{0}{5} [anchor = north] {$\scriptstyle k+2$};
			\node at \tlcoord{0}{3} [anchor = north] {$\scriptstyle n\vphantom{k}$};
			% top nodes
			\node at \tlcoord{1}{0} [anchor = south] {$\scriptstyle 1\vphantom{+}$};
			\node at \tlcoord{1}{1} [anchor = south] {$\scriptstyle 2$};
			%\node at \tlcoord{1}{3} [anchor = south] {$\scriptstyle k\vphantom{+}$};
			%\node at \tlcoord{1}{4} [anchor = south] {$\scriptstyle k+1$};
			%\node at \tlcoord{1}{5} [anchor = south] {$\scriptstyle k+2$};
			\node at \tlcoord{1}{3} [anchor = south] {$\scriptstyle n\vphantom{+}$};
		\end{tikzpicture}
	}\\
		s_i &\longmapsto \cbox{
		\begin{tikzpicture}[tldiagram]
			% dots
			%\drawdots{0}{0}{0}
			%\drawdots{0}{2}{5}
			%\drawdots{0}{7}{7}
			%\drawdots{1}{0}{0}
			%\drawdots{1}{2}{5}
			%\drawdots{1}{7}{7}
			% lines
			\draw \tlcoord{0}{0} \lineup;
			\draw \tlcoord{0}{2} \lineup;
			\draw \tlcoord{0}{5} \lineup;
			\draw \tlcoord{0}{7} \lineup;
			% dots
			\makecdots{0}{1}
			\makecdots{0}{6}
			% cross
			\draw \tlcoord{0}{3} \linewave{1}{1};
			\draw \tlcoord{0}{4} \linewave{1}{-1};
			% bottom nodes
			\node at \tlcoord{0}{0} [anchor = north] {$\scriptstyle 1\vphantom{k}$};
			\node at \tlcoord{0}{2} [anchor = north] {$\scriptstyle {i-1}$};
			\node at \tlcoord{0}{3} [anchor = north] {$\scriptstyle i$};
			\node at \tlcoord{0}{4} [anchor = north] {$\scriptstyle i+1$};
			\node at \tlcoord{0}{5} [anchor = north] {$\scriptstyle i+2$};
			\node at \tlcoord{0}{7} [anchor = north] {$\scriptstyle n\vphantom{k}$};
			% top nodes
			\node at \tlcoord{1}{0} [anchor = south] {$\scriptstyle 1\vphantom{+}$};
			\node at \tlcoord{1}{2} [anchor = south] {$\scriptstyle i-1$};
			\node at \tlcoord{1}{3} [anchor = south] {$\scriptstyle i\vphantom{+}$};
			\node at \tlcoord{1}{4} [anchor = south] {$\scriptstyle i+1$};
			\node at \tlcoord{1}{5} [anchor = south] {$\scriptstyle i+2$};
			\node at \tlcoord{1}{7} [anchor = south] {$\scriptstyle n\vphantom{+}$};
		\end{tikzpicture}
	}
	\end{align*}
	define an isomorphism of groups between $W(B_n)$ and the group of dotted permutations. The image of the subgroup $W(D_n)\subseteq W(B_n)$ under $\Phi$ consists of all dotted permutations, which have an even number of dots.
\end{lemm}
\begin{proof}
	First note the map is well-defined since the images satisfy the Coxeter relations. Notable is the type $B$ relation \eqref{type B relation formula} which translates to 
	\begin{equation*} %pictures hinschreiben
				\Phi(s_0s_1s_0s_1)\quad = \quad \cbox{
					\begin{tikzpicture}[tldiagram]
						\draw \tlcoord{0}{0}   \linewave{1}{1} \linewave{1}{-1} \onedot;
						\draw \tlcoord{0}{1} \linewave{1}{-1} \onedot \linewave{1}{1};
					\end{tikzpicture}
				} \quad = \quad
				\cbox{
					\begin{tikzpicture}[tldiagram]
						\draw \tlcoord{-1}{0} \lineup \onedot \lineup;
						\draw \tlcoord{-1}{1} \lineup \onedot \lineup;
					\end{tikzpicture}
				}
				\quad 
				= 
				\quad 
				\cbox{
					\begin{tikzpicture}[tldiagram]
						\draw \tlcoord{0}{0} \onedot  \linewave{1}{1} \linewave{1}{-1} ;
						\draw \tlcoord{0}{1} \linewave{1}{-1}  \onedot \linewave{1}{1} ;
					\end{tikzpicture}
				}
				\quad =\quad \Phi(s_1s_0s_1s_0).
			\end{equation*}
		The relation $s_0s_i=s_is_0$ for $i\geq2$ translates to the fact that a dot on the left strand can be moved independently from a crossing not involving the left strand. 
		The only thing left to show is the bijectivity of $\Phi$. First we prove the surjectivity and then prove independently that $|W(B_n)|\leq 2^n n!$. Both steps combined give that $\Phi$ is one-to-one.
		For the first step note that every element $w=(\sigma, \varepsilon)$ of $G_n$ can be written as a product $w=(\sigma,\underline{0})(e,\varepsilon)$.
		Every undotted permutation $(\sigma,\underline{0})$ is in the image of $\Phi$, since the image contains the simple transpositions. Moreover any element $(e, \varepsilon)$ is contained in the image of $\Phi$, since
		the simply dotted identity diagram
		\begin{equation} \label{image of Jucys-Murphy element in Type B Weyl Group}
			\cbox{
				\begin{tikzpicture}[tldiagram]
					% dots
					%\drawdots{0}{0}{0}
					%\drawdots{0}{2}{5}
					%\drawdots{0}{7}{7}
					%\drawdots{1}{0}{0}
					%\drawdots{1}{2}{5}
					%\drawdots{1}{7}{7}
					% lines
					\draw \tlcoord{0}{0} \lineup;
					\draw \tlcoord{0}{2} \lineup;
					\draw \tlcoord{0}{4} \lineup;
					\draw \tlcoord{0}{6} \lineup;
					% dots
					\makecdots{0}{1}
					\makecdots{0}{5}
					% cross
					\draw \tlcoord{0}{3} \dlineup;
					%\draw \tlcoord{0}{4} -- \tlcoord{1}{3};
					% bottom nodes
					\node at \tlcoord{0}{0} [anchor = north] {$\scriptstyle 1\vphantom{k}$};
					\node at \tlcoord{0}{2} [anchor = north] {$\scriptstyle {i-1}$};
					\node at \tlcoord{0}{3} [anchor = north] {$\scriptstyle i$};
					\node at \tlcoord{0}{4} [anchor = north] {$\scriptstyle i+1$};
					%\node at \tlcoord{0}{5} [anchor = north] {$\scriptstyle i+2$};
					\node at \tlcoord{0}{6} [anchor = north] {$\scriptstyle n\vphantom{k}$};
					% top nodes
					\node at \tlcoord{1}{0} [anchor = south] {$\scriptstyle 1\vphantom{+}$};
					\node at \tlcoord{1}{2} [anchor = south] {$\scriptstyle i-1$};
					\node at \tlcoord{1}{3} [anchor = south] {$\scriptstyle i\vphantom{+}$};
					\node at \tlcoord{1}{4} [anchor = south] {$\scriptstyle i+1$};
					%\node at \tlcoord{1}{4} [anchor = south] {$\scriptstyle i+2$};
					\node at \tlcoord{1}{6} [anchor = south] {$\scriptstyle n\vphantom{+}$};
				\end{tikzpicture}
			}
		\end{equation}
		is the image of $(s_{i-1}\cdots s_2s_1)s_0(s_1s_2\cdots s_{i-1})$.
		
		The only thing left to show is that $|W(B_n)|\leq 2^nn!$. This fact follows inductively by proving
		\[
			|W(B_n)/W(B_{n-1})|\leq 2n
		\]
		and observing that $W(B_0)$ is the trivial group.
		A straightforward calculation shows that the right cosets of the $2n$ elements
		\begin{gather*}
			e\, , \quad s_{n-1}\, , \quad s_{n-2}s_{n-1}\, , \quad \ldots\, , \quad  s_{1}\cdots s_{n-1}, \\
			s_0(s_{1}\cdots s_{n-1}), \, s_1s_0(s_1\cdots s_{n-1}), \, (s_2s_1)s_0(s_1\cdots s_{n-1}), \, \ldots, \, (s_{n-1}\cdots s_1) s_0 (s_1 \cdots s_{n-1}).
		\end{gather*}
		are permuted by the left action with the Coxeter generators of $W(B_n)$. For the details see for instance \cite{dipper}. Hence we conclude $|W(B_n)/W(B_{n-1})|\leq 2n$. This shows that $|W(B_n)|\leq 2^n n!=|G_n|$. Together with the surjectivity this completes the proof of the isomorphism.
	
		All signed permutations with an even number of dots form an index $2$ subgroup and contain the image $\Phi(W(D_n))$, since they contain the elements $\Phi(s_0')=\Phi(s_0s_1s_0)$ and $\Phi(s_i)$ for $1\leq i\leq n$. This shows that $\Phi(W(D_n))$ is precisely this subgroup by cardinality reasons.
\end{proof}
We will use this lemma going forward and identify $W(B_n)$ with $G_n$ and $W(D_n)$ with its special index $2$ subgroup unless specified otherwise.

The elements \eqref{image of Jucys-Murphy element in Type B Weyl Group} play a special role in the representation theory of $W(B_n)$, so we define them properly in the next definition which is taken from \cite{ogievetsky2012jucysmurphy}.
\begin{defi} \label{defi mult jucys-murphys}
	The (multiplicative) \textbf{Jucys--Murphy elements} $J_1,\ldots,J_{n}$ are the recursively defined elements
	\[
		J_1=\sigma_0, \quad J_i=\sigma_{i-1}J_{i-1}\sigma_{i-1}
	\]
	in the braid group $B(B_n)$. Their images in the Coxeter group $W(B_n)$ are also called Jucys--Murphy elements.
\end{defi}
\begin{beis} \label{Pictures for n=3 Jucys-Murphys}
	The images of the Jucys--Murphy elements for $n=3$ under the map $\iota$ from \Cref{embedding bn into an} are given by 
	\[
		\iota(J_1) = \cbox{
		\begin{tikzpicture}[tldiagram, yscale=2/3]
				%dots 
				%\drawdots{0}{0}{1}
				%lines
				\draw \tlcoord{0}{0.5} \lineup;
				\draw[white, double=black] \tlcoord{0}{1} \linewave{1}{-1} \linewave{1}{1};
				\draw[white, double=black] \tlcoord{1}{0.5} \lineup;
				\draw \tlcoord{0}{2} \lineup \lineup;
				\draw \tlcoord{0}{3} \lineup \lineup;
				%\draw  \tlcoord{0}{0} \linewave{1}{1};
				%\draw[white, double=black] \tlcoord{0}{1} \linewave{1}{-1} \capright;
		\end{tikzpicture}
	} \, , \quad 
\iota(J_2) = \cbox{
	\begin{tikzpicture}[tldiagram, yscale=2/3]
		%dots 
		%\drawdots{0}{0}{1}
		%lines
		\draw \tlcoord{0}{0.5} \lineup;
		\draw \tlcoord{0}{1} \lineup;
		\draw[white, double=black] \tlcoord{0}{2} \linewave{1}{-2} \linewave{1}{2};
		\draw[white, double=black] \tlcoord{1}{0.5} \lineup;
		\draw[white, double=black] \tlcoord{1}{1} \lineup;
		\draw \tlcoord{0}{3} \lineup \lineup;
		%\draw  \tlcoord{0}{0} \linewave{1}{1};
		%\draw[white, double=black] \tlcoord{0}{1} \linewave{1}{-1} \capright;
	\end{tikzpicture}
} \, , \quad \iota(J_3) = \cbox{
\begin{tikzpicture}[tldiagram, yscale=2/3]
	%dots 
	%\drawdots{0}{0}{1}
	%lines
	\draw \tlcoord{0}{0.5} \lineup;
	\draw \tlcoord{0}{1} \lineup;
	\draw \tlcoord{0}{2} \lineup;
	\draw[white, double=black] \tlcoord{0}{3} \linewave{1}{-3} \linewave{1}{3};
	\draw[white, double=black] \tlcoord{1}{0.5} \lineup;
	\draw[white, double=black] \tlcoord{1}{1} \lineup;
	\draw[white, double=black] \tlcoord{1}{2} \lineup;
	%\draw  \tlcoord{0}{0} \linewave{1}{1};
	%\draw[white, double=black] \tlcoord{0}{1} \linewave{1}{-1} \capright;
\end{tikzpicture}
} \, .
	\]
	In particular assuming the injectivity of $\iota\colon B(B_n)\rightarrow B(A_n)$ observe that the Jucys--Murphy elements commute pairwise and each $J_i$ commutes with all generators $\sigma_k\in B(B_n)$ except possibly $\sigma_{i-1}$ and $\sigma_{i}$.
\end{beis}
The multiplicative Jucys--Murphy elements $J_i\in W(B_n)$ are closely related to the longest element $\wo\in W(B_n)$. The next remark gives a concrete expression for $\wo$ utilizing these elements, which will even turn out to be reduced.
\begin{rema}[Longest element of $W(B_n)$] \label{longest element in type B} 
	The images of the Jucys--Murphy elements in the Coxeter group $W(B_n)$ are self-inverse and generate a normal subgroup isomorphic to $(\mZ/2\!\mZ)^{n}$. 
	By viewing $W(B_n)$ as the Weyl group of the $B_n$ root system one can show that $\prod_{i=1}^n J_i$ is the longest element $\wo$ of $W(B_n)$ and that this expression is already reduced. Indeed the type $B_n$ root system has $2n^2$ roots
	\begin{equation} \label{full root system type B}
		\{ \pm \varepsilon_i \mid i=1,\ldots, n\} \cup \{\pm \varepsilon_i \pm \varepsilon_j \mid 1\leq i < j \leq n\} \subseteq \h^{*}(\so_{2n+1})
	\end{equation}
	where the simple roots from the proof of \Cref{lemma inclusion of d in b} are (c.f.\ \cite{bourbaki02}) given by 
	\[
		\alpha_0=\varepsilon_n, \quad \alpha_1=\varepsilon_{n-1} -\varepsilon_{n}, \quad \ldots, \quad \alpha_2=\varepsilon_2- \varepsilon_3, \quad  \alpha_{n-1}=\varepsilon_1-\varepsilon_2 \, .
	\]
	The number of Coxeter generators in the definition of $\prod_{i=1}^n J_i$ is $\sum_{i=1}^{n} (2i-1)=n^2$, which is the number of positive roots. A tedious but straightforward calculation shows that the element $\prod_{i=1}^n J_i$ acts by $-\id$ on the simple roots and hence sends all the $n^2$ positive roots to negative roots. So indeed we have $\wo=\prod_{i=1}^n J_i$ and this expression is already reduced.
	The multiplicative Jucys--Murphy elements were used in \cite{ogievetsky2012jucysmurphy} to construct/classify all the irreducible representations of the $2$-parameter Hecke algebra $\Hecke_{v,q}(B_n)$ and will return when we discuss full twists in the final chapter. 
	%These elements also have applications in type $B$ knot theory, for instance in \cite{lambropoulou2004knot} they were used to define a type $B$ version of the Ocneanu trace.  
\end{rema}
The next natural question is: When does the longest element $\wo\in W(B_n)$ live already in the index two subgroup $W(D_n)$, and if so, is it already the longest element of $W(D_n)$? The next remark answers this question, utilizing both a diagrammatic approach and an argument via roots.
\begin{rema}[Longest element in $W(D_n)$] \label{longest element in type D}
	Under the isomorphism from \Cref{Weyl group of B_n is dotted group} the longest element $\wo\in W(B_n)$ corresponds  to the identity diagram with a dot on every strand:
	\[
		\wo \quad = \quad \prod_{i=1}^n J_i \quad = \quad \cbox{
			\begin{tikzpicture}[tldiagram]
				% dots
				%\drawdots{0}{0}{0}
				%\drawdots{0}{2}{5}
				%\drawdots{0}{7}{7}
				%\drawdots{1}{0}{0}
				%\drawdots{1}{2}{5}
				%\drawdots{1}{7}{7}
				% lines
				\draw \tlcoord{0}{0} \dlineup;
				\draw \tlcoord{0}{1} \dlineup;
				\draw \tlcoord{0}{3} \dlineup;
				%\draw \tlcoord{0}{7} \lineup;
				% dots
				\makecdots{0}{2}
				%\makecdots{0}{6}
				% cross
				%\draw \tlcoord{0}{3} -- \tlcoord{1}{4};
				%\draw \tlcoord{0}{4} -- \tlcoord{1}{3};
				% bottom nodes
%				\node at \tlcoord{0}{0} [anchor = north] {$\scriptstyle 1\vphantom{k}$};
%				\node at \tlcoord{0}{1} [anchor = north] {$\scriptstyle {2}$};
%				%\node at \tlcoord{0}{3} [anchor = north] {$\scriptstyle k$};
%				%\node at \tlcoord{0}{4} [anchor = north] {$\scriptstyle k+1$};
%				%\node at \tlcoord{0}{5} [anchor = north] {$\scriptstyle k+2$};
%				\node at \tlcoord{0}{3} [anchor = north] {$\scriptstyle n\vphantom{k}$};
%				% top nodes
%				\node at \tlcoord{1}{0} [anchor = south] {$\scriptstyle 1\vphantom{+}$};
%				\node at \tlcoord{1}{1} [anchor = south] {$\scriptstyle 2$};
%				%\node at \tlcoord{1}{3} [anchor = south] {$\scriptstyle k\vphantom{+}$};
%				%\node at \tlcoord{1}{4} [anchor = south] {$\scriptstyle k+1$};
%				%\node at \tlcoord{1}{5} [anchor = south] {$\scriptstyle k+2$};
%				\node at \tlcoord{1}{3} [anchor = south] {$\scriptstyle n\vphantom{+}$};
			\end{tikzpicture}
		} \, .
	\]
	In particular, by \Cref{Weyl group of B_n is dotted group} this element is contained in $W(D_n)$ if and only if $n$ is even. Moreover it is then already the longest element of $W(D_n)$. Indeed the identification from \Cref{lemma inclusion of d in b} identifies the positive roots of $D_n$ with a subset of the positive roots for the $B_n$ root system (with respect to the chosen simple roots) an hence maps all positive type $D_n$ roots to negative roots. The $D_n$ root system consists of the $2n^2-2n$ roots
	\begin{equation} \label{full root system type D}
		\{ \pm \varepsilon_i \pm \varepsilon_j \mid 1\leq i < j \leq n\} \subseteq \h^{*}(\so_{2n})
	\end{equation}
	where the simple roots from the proof of \Cref{lemma inclusion of d in b} are given by 
	\[
	\alpha_0'=\varepsilon_{n-1}+\varepsilon_n, \quad \alpha_1=\varepsilon_{n-1} -\varepsilon_{n}, \quad \ldots, \quad \alpha_2=\varepsilon_2- \varepsilon_3, \quad  \alpha_{n-1}=\varepsilon_1-\varepsilon_2 \, ,
	\]
	see e.g.\ \cite{bourbaki02}.
	In particular the longest element for type $D_n$ has length $n^2-n$.  A similar reasoning as in \Cref{longest element in type B} shows that the longest element of $D_n$ for odd $n$ is given by
	\[
	\wo \quad = \quad \cbox{
		\begin{tikzpicture}[tldiagram]
			% dots
			%\drawdots{0}{0}{0}
			%\drawdots{0}{2}{5}
			%\drawdots{0}{7}{7}
			%\drawdots{1}{0}{0}
			%\drawdots{1}{2}{5}
			%\drawdots{1}{7}{7}
			% lines
			\draw \tlcoord{0}{-1} \lineup;
			\draw \tlcoord{0}{0} \dlineup;
			\draw \tlcoord{0}{1} \dlineup;
			\draw \tlcoord{0}{3} \dlineup;
			%\draw \tlcoord{0}{7} \lineup;
			% dots
			\makecdots{0}{2}
			%\makecdots{0}{6}
			% cross
			%\draw \tlcoord{0}{3} -- \tlcoord{1}{4};
			%\draw \tlcoord{0}{4} -- \tlcoord{1}{3};
			% bottom nodes
			%				\node at \tlcoord{0}{0} [anchor = north] {$\scriptstyle 1\vphantom{k}$};
			%				\node at \tlcoord{0}{1} [anchor = north] {$\scriptstyle {2}$};
			%				%\node at \tlcoord{0}{3} [anchor = north] {$\scriptstyle k$};
			%				%\node at \tlcoord{0}{4} [anchor = north] {$\scriptstyle k+1$};
			%				%\node at \tlcoord{0}{5} [anchor = north] {$\scriptstyle k+2$};
			%				\node at \tlcoord{0}{3} [anchor = north] {$\scriptstyle n\vphantom{k}$};
			%				% top nodes
			%				\node at \tlcoord{1}{0} [anchor = south] {$\scriptstyle 1\vphantom{+}$};
			%				\node at \tlcoord{1}{1} [anchor = south] {$\scriptstyle 2$};
			%				%\node at \tlcoord{1}{3} [anchor = south] {$\scriptstyle k\vphantom{+}$};
			%				%\node at \tlcoord{1}{4} [anchor = south] {$\scriptstyle k+1$};
			%				%\node at \tlcoord{1}{5} [anchor = south] {$\scriptstyle k+2$};
			%				\node at \tlcoord{1}{3} [anchor = south] {$\scriptstyle n\vphantom{+}$};
		\end{tikzpicture}
	}  \, .
	\]
	One can inductively show that an expression for the longest element $\wo$ of $W(D_n)$ is given by 
	\[
		\wo=(s_0's_1s_2\cdots s_{n-1})^{n-1} \, ,
	\]
	and this expression is independent of whether $n$ is even or odd. Moreover it is of length $n^2-n$, which is the number of positive type $D_n$ roots. Hence this expression is reduced. 
	In particular observe that in type $D_n$ for $n$ even, the longest element $\wo\in W(D_n)$ lives in the center of group $W(B_n)$ and in particular in the center of $W(D_n)$. If $n$ is odd however $\wo$ commutes with all generators, except $s_0'$ and $s_1$, which it conjugates to one another
	\begin{equation*}
		s_0'\wo=\wo s_1 \quad , \quad s_1 \wo = \wo s_0' \quad , \quad \wo s_0' \wo =s_1 \, .
	\end{equation*}
\end{rema}
Now that we established the longest elements, we can go back to discussing diagrammatic descriptions for $W(B_n)$. The next lemma describes $W(B_n)$ as a subgroup of $W(A_{n-1})$. This description will be helpful, when we count the conjugacy classes of $W(B_n)$. 

\begin{lemm} \label{sign permutation are dot permutations}
	Consider the group $S_{2n}$ of permutations of $\{-n,\ldots-1,1,\ldots,n\}$. The assignments
	\begin{align*}
		\Psi\colon W(B_n) &\longrightarrow S_{2n} \\
		s_0 &\longmapsto (1,-1) =\cbox{
			\begin{tikzpicture}[tldiagram, xscale=3/4, yscale=3/4]
				% dots
				%\drawdots{0}{0}{0}
				%\drawdots{0}{2}{5}
				%\drawdots{0}{7}{7}
				%\drawdots{1}{0}{0}
				%\drawdots{1}{2}{5}
				%\drawdots{1}{7}{7}
				% lines
				\draw \tlcoord{0}{0} \lineup;
				\draw \tlcoord{0}{2} \lineup;
				\draw \tlcoord{0}{5} \lineup;
				\draw \tlcoord{0}{7} \lineup;
				% dots
				\makecdots{0}{1}
				\makecdots{0}{6}
				% cross
				\draw \tlcoord{0}{3}  \linewave{1}{1};
				\draw \tlcoord{0}{4}  \linewave{1}{-1};
				% bottom nodes
				\node at \tlcoord{0}{0} [anchor = north] {$\scriptstyle -n\vphantom{k}$};
				\node at \tlcoord{0}{2} [anchor = north] {$\scriptstyle {-2}$};
				\node at \tlcoord{0}{3} [anchor = north] {$\scriptstyle -1$};
				\node at \tlcoord{0}{4} [anchor = north] {$\scriptstyle 1$};
				\node at \tlcoord{0}{5} [anchor = north] {$\scriptstyle 2$};
				\node at \tlcoord{0}{7} [anchor = north] {$\scriptstyle n\vphantom{k}$};
				% top nodes
				\node at \tlcoord{1}{0} [anchor = south] {$\scriptstyle -n\vphantom{+}$};
				\node at \tlcoord{1}{2} [anchor = south] {$\scriptstyle -2$};
				\node at \tlcoord{1}{3} [anchor = south] {$\scriptstyle -1\vphantom{+}$};
				\node at \tlcoord{1}{4} [anchor = south] {$\scriptstyle 1$};
				\node at \tlcoord{1}{5} [anchor = south] {$\scriptstyle 2$};
				\node at \tlcoord{1}{7} [anchor = south] {$\scriptstyle n\vphantom{+}$};
			\end{tikzpicture}
		}\\
		s_i &\longmapsto (i,i+1)(-i,-i+1)=\cbox{
			\begin{tikzpicture}[tldiagram, xscale=3/4, yscale=3/4]
				% dots
				%\drawdots{0}{0}{0}
				%\drawdots{0}{2}{5}
				%\drawdots{0}{7}{7}
				%\drawdots{1}{0}{0}
				%\drawdots{1}{2}{5}
				%\drawdots{1}{7}{7}
				% lines
				\draw \tlcoord{0}{0} \lineup;
				\draw \tlcoord{0}{2} \lineup;
				\draw \tlcoord{0}{5} \lineup;
				\draw \tlcoord{0}{7} \lineup;
				\draw \tlcoord{0}{10} \lineup;
				\draw \tlcoord{0}{12} \lineup;
				% dots
				\makecdots{0}{1}
				\makecdots{0}{6}
				\makecdots{0}{11}
				% cross
				\draw \tlcoord{0}{3} \linewave{1}{1};
				\draw \tlcoord{0}{4} \linewave{1}{-1};
				\draw \tlcoord{0}{8} \linewave{1}{1};
				\draw \tlcoord{0}{9} \linewave{1}{-1};
				% bottom nodes
				\node at \tlcoord{0}{0} [anchor = north] {$\scriptstyle -n\vphantom{k}$};
				%\node at \tlcoord{0}{2} [anchor = north] {$\scriptstyle {-i-2}$};
				\node at \tlcoord{0}{3} [anchor = north] {$\scriptstyle -i-1$};
				\node at \tlcoord{0}{4} [anchor = north] {$\scriptstyle -i$};
				%\node at \tlcoord{0}{5} [anchor = north] {$\scriptstyle -i+1$};
				%\node at \tlcoord{0}{7} [anchor = north] {$\scriptstyle i-1\vphantom{k}$};
				\node at \tlcoord{0}{8} [anchor = north] {$\scriptstyle i$};
				\node at \tlcoord{0}{9} [anchor = north] {$\scriptstyle i+1$};
				%\node at \tlcoord{0}{10} [anchor = north] {$\scriptstyle i+2$};
				\node at \tlcoord{0}{12} [anchor = north] {$\scriptstyle n$};
				% top nodes
				\node at \tlcoord{1}{0} [anchor = south] {$\scriptstyle -n\vphantom{+}$};
				%\node at \tlcoord{1}{2} [anchor = south] {$\scriptstyle -i-2$};
				\node at \tlcoord{1}{3} [anchor = south] {$\scriptstyle -i-1\vphantom{+}$};
				\node at \tlcoord{1}{4} [anchor = south] {$\scriptstyle -i$};
				%\node at \tlcoord{1}{5} [anchor = south] {$\scriptstyle -i+1$};
				%\node at \tlcoord{1}{7} [anchor = south] {$\scriptstyle i-1\vphantom{+}$};
				\node at \tlcoord{1}{8} [anchor = south] {$\scriptstyle i\vphantom{+}$};
				\node at \tlcoord{1}{9} [anchor = south] {$\scriptstyle i+1$};
				%\node at \tlcoord{1}{10} [anchor = south] {$\scriptstyle i+2$};
				\node at \tlcoord{1}{12} [anchor = south] {$\scriptstyle n\vphantom{+}$};
			\end{tikzpicture}
		}
	\end{align*}
define an injective group homomorphism $W(B_n)\rightarrow S_{2n}$. The image of $\Psi$ is the subgroup $S^{\mathrm{sign}}_n\subseteq S_{2n}$ of signed permutations. This subgroup consists of all elements $\sigma$ which preserve the sign, i.e.\ satisfy the condition $\sigma(-i)=-\sigma(i)$ for all $i=1,\ldots,n$.
\end{lemm}
\begin{proof}
	First we have to check that the images $\Psi(s_i)$ satisfy the relations of $W(B_n)$. The only non-trivial relation is the type $B$ relation. It translates to the combinatorial identity
	\[
		\Psi(s_0s_1s_0s_1) \, = \,\cbox{
		\begin{tikzpicture}[tldiagram, xscale=2/3, yscale=2/3] %alternatively use \straightline{1}{1} instead of linewave
			\draw \tlcoord{0}{0} \linewave{1}{1} \linewave{1}{1} \linewave{1}{1} \lineup;
			\draw \tlcoord{0}{1} \linewave{1}{-1} \lineup \linewave{1}{1} \linewave{1}{1};
			\draw \tlcoord{0}{2} \linewave{1}{1}  \lineup \linewave{1}{-1} \linewave{1}{-1};
			\draw \tlcoord{0}{3} \linewave{1}{-1} \linewave{1}{-1} \linewave{1}{-1} \lineup;
		\end{tikzpicture}
	} \, = \,(1,-1)(2,-2) \, = \, \cbox{
	\begin{tikzpicture}[tldiagram, xscale=2/3, yscale=2/3]
		\draw \tlcoord{0}{0} \lineup \linewave{1}{1} \linewave{1}{1} \linewave{1}{1};
		\draw \tlcoord{0}{1} \linewave{1}{1} \linewave{1}{1} \lineup \linewave{1}{-1};
		\draw \tlcoord{0}{3} \lineup \linewave{1}{-1} \linewave{1}{-1} \linewave{1}{-1};
		\draw \tlcoord{0}{2} \linewave{1}{-1} \linewave{1}{-1} \lineup \linewave{1}{1};
	\end{tikzpicture}
} \, = \,\Psi(s_1s_0s_1s_0) \, .
	\]
	Similarly one can check the other relations, giving that $\Psi$ is a well-defined group homomorphism.
	The bijectivity of $\Psi$ restricted to its image follows from the commutative triangle
	\[
		\begin{tikzcd}
			W(B_n) \arrow[r, "\Phi"] \arrow[rd, "\Psi"'] & G_n \arrow[d, "\cong"] \\
			& S^{\text{sign}}_n, 
		\end{tikzcd}
	\]
	where $\Phi$ is the isomorphism from \Cref{Weyl group of B_n is dotted group} and the vertical map identifies dotted permutations with signed permutations. Concretely we assign to $(\sigma, \varepsilon)\in G_n$ the unique signed permutation $\sigma'$ which satisfies
	\[
		\sigma'(i)=\begin{cases}
			\sigma(i), & \text{if $\varepsilon_i=0$} \\
			-\sigma(i), & \text{if $\varepsilon_i=1$,}
		\end{cases}
	\]
	 for all $i\in\{1,\ldots, n\}$. This assignment clearly gives a bijection making the diagram commute. This completes the proof, since $\Phi$ is a an isomorphism.
\end{proof}
\begin{beis}
	The Jucys--Murphy element $J_{i}\in W(B_n)$ corresponds under the isomorphism $\Psi\colon W(B_n)\rightarrow S^{\mathrm{sign}}_n$ to the signed permutation $(i,-i)$, while the longest element $\wo$ corresponds to $-\id$, the product of all those transpositions. In particular the Jucys--Murphy elements $J_i$ have the same cycle type in $S_{2n}$, which implies that they are all conjugates viewed as elements in $S_{2n}$. This is not surprising since they are already conjugate in $W(B_n)$ by their construction. 
\end{beis}
\section{Irreducible representations of $W(B_n)$} \label{Chapter 1 section 2}

In this section we present the representation theory of the Coxeter group $W(B_n)$ over the complex numbers. The results of the theorems are known for many years and can be drawn from the abstract theory of wreath products (see for instance \cite{jameskerber81}). However our goal is to mimic the well understood representation theory of the symmetric groups, which can be found in \cite{fulton2013} and point out the differences, whenever they appear.

First we study the conjugacy classes of $W(B_n)$. The combinatorial tool to describe those are pairs of partitions, which we call bipartitions.
\begin{defi}
	A \textbf{bipartition of} $n$ is an ordered pair $(\lambda, \mu)$ of partitions, such that $|\lambda|+|\mu|=n$. We often write bipartitions as tuples of the Young diagrams corresponding to the partitions.
\end{defi}
\begin{beis}
	The number of bipartitions $b(n)$ of $n$ is given by
	\begin{equation*}
		b(n)=\sum_{k=0}^n a(k)\cdot a(n-k),
	\end{equation*}
	where $a(l)$ is the number of partitions of $l\geq0$.
	For example we see that
	\begin{equation*}
		b(3)=3\cdot1 + 2\cdot 1 + 1\cdot 2 + 1 \cdot 3=10 
	\end{equation*} 
	Here is a side-by-side comparison of the two sequences $a(n)$ and $b(n)$ for $n=0,\ldots,10$:
	\begin{equation*}
		\begin{array}{c|ccccccccccc}
			a(n) & 1 & 1 & 2 & 3 & 5 & 7 & 11 &15 &22 &30 & 42\\
			\hline
			b(n) & 1 & 2 & 5 & 10 & 20 & 36 & 65 & 110 & 185& 300  &481
		\end{array}
	\end{equation*}
	Most notably the number of bipartitions grows much faster than the number of partitions. This is a problem, since small representation theoretic examples for $W(B_n)$ are already much bigger than for the symmetric group $S_n$. More interesting facts about the sequence $b(n)$ can be found under \cite{oeis}.
\end{beis}
The next theorem is the connection between bipartitions of $n$ and conjugacy classes (and hence irreducible complex representations) of $W(B_n)$. Throughout the proof we will identify $W(B_n)$ with its image under $\Psi$ from \Cref{sign permutation are dot permutations}, the subgroup of signed permutations on $I=\{-n,\ldots, -1,1,\ldots,n\}$. This is the natural realization for $W(B_n)$ when thinking about conjugacy classes, since those of the symmetric group $S_n$ are given by cycle types.

\begin{theo} \label{Conjugacy classes via bipartitions}
	There is a bijection
	\begin{align*}
		\left\{\text{bipartitions $(\lambda,\mu)$ of $n$} \right\}
		&\xleftrightarrow{\,\text{1:1}\,}
		\{\text{conjugacy classes of $W(B_n)$}\} \, .
	\end{align*}
	In particular the number of isomorphism classes of irreducible complex representations of $W(B_n)$ is precisely the number of bipartitions.
\end{theo}
\begin{proof}
	Let $w\in W(B_n)$ and consider its decomposition into disjoint cycles in $S_{2n}$. Let $\sigma=(i_1, i_2, \ldots, i_r)\in S_{2n}$ be a cycle in $w$ (not necessarily contained in $W(B_n)$). We denote by $|\sigma|=r$ the length of $\sigma$. For $i\in I$ we write $i\in\sigma$, if $i$ occurs in the cycle $\sigma$. First assume that $-i_1\in \sigma$. Then $\sigma$ is of the form
	\begin{equation} \label{nice type 1 form}
			\sigma=(i_1, i_2, i_3, \ldots, i_{s}, -i_1, -i_2, -i_3, \ldots -i_s).
	\end{equation}
	This follows from iteratively using the rule $w(-j)=-w(j)$ for all $j\in I$. Note that in particular $\sigma\in W(B_n)$ and $|\sigma|=2s$.
	
	It follows that any cycle $\sigma$ appearing in $w$ satisfies exactly one of the following two properties:
	\begin{enumerate}
		\item For all $i\in\sigma$, we have $-i\notin\sigma$. \label{cycle condition 2}
		\item The cycle $\sigma$ is of the form \eqref{nice type 1 form}. \label{cycle condition 1}
	\end{enumerate}
	We call the cycles satisfying the first condition type \ref{cycle condition 2} and the cycles satisfying the second condition type \ref{cycle condition 1} cycles. If $\sigma=(i_1, i_2,\ldots,i_r)$ is of type \ref{cycle condition 2}, then it is not contained in $W(B_n)$. However the cycle $\sigma'=(-i_1, -i_2,\ldots,-i_r)$ must also appear in $w$, by the rule $w(-j)=-w(j)$. The product $\bar{\sigma}=\sigma\sigma'$ is automatically contained in $W(B_n)$. We call every such element $\bar{\sigma}$ a type \ref{cycle condition 2} double cycle. 
	
	Now it follows that every $w\in W(B_n)$ factorizes into the product of two disjoint permutations 
	$w=w_1w_2$, where $w_1$ is a product of pairwise disjoint type \ref{cycle condition 2} double cycles and $w_2$ is a product of disjoint type \ref{cycle condition 1} cycles. In formulas we have
	\begin{equation} \label{expr}
		w=(\underbrace{\sigma_{2,1}\sigma_{2,2}\ldots, \sigma_{2,k}}_{=w_2})\cdot  (\underbrace{\bar{\sigma}_{1,1}\bar{\sigma}_{1,2}\ldots, \bar{\sigma}_{1,l}}_{=w_1})
	\end{equation}
	for certain type \ref{cycle condition 2} cycles $\sigma_{1,1}\sigma_{1,2}\ldots, \sigma_{1,l}$ and certain type \ref{cycle condition 1} cycles $\sigma_{2,1}\sigma_{2,2}\ldots, \sigma_{2,k}$, such that
	\begin{equation*}
		|\sigma_{1,1}|\geq|\sigma_{1,2}|\geq\ldots\geq |\sigma_{1,l}|\geq1
	\end{equation*}
	and
\begin{equation*}
	|\sigma_{2,1}|\geq|\sigma_{2,2}|\geq\ldots\geq|\sigma_{2,k}|\geq2
\end{equation*}
	and all occuring cycles are disjoint. In particular we have
	\[
	  \sum_{i=1}^l |\sigma_{1,i}| + 	\sum_{i=1}^k |\sigma_{2,i}|/2=n,
	\]
	by counting all cycles of length $1$ as factors of $w_1$. 
	This shows that 
	\[
		((|\sigma_{1,1}|, |\sigma_{1,2}|, \ldots, |\sigma_{1,l}|),(|\sigma_{2,1}|/2, \ldots, |\sigma_{2,k}|/2))
	\]
	is a bipartition of $n$. 
	By conjugating with $g\in S_n\subseteq W(B_n)$ or $g=(i, -i)\in W(B_n)$  one can see that the decomposition into type \ref{cycle condition 1} and type \ref{cycle condition 2} double cycles uniquely determines the conjugacy class of $w$. Hence the sizes of the cycles occurring in this decomposition uniquely determine a conjugacy class in $W(B_n)$.
	
	Summarizing we see that the assignment 
	\[
		[w] \mapsto ((|\sigma_{1,1}|, \ldots, |\sigma_{1,l}|),(|\sigma_{2,1}|/2, \ldots, |\sigma_{2,k}|/2))
	\]
	defines a bijection between the set of conjugacy classes of $W(B_n)$ and the set of bipartitions of $n$. It is a well-known consequence of the Artin--Wedderburn theorem and the computation of the center of the group algebra $kG$ for a finite group $G$, that the number of irreducible representations of $G$ over an algebraically closed field of characteristic zero is precisely the number of conjugacy classes.
\end{proof}
Next we give an example for the decomposition into type \ref{cycle condition 2}  and type \ref{cycle condition 1} cycles from the proof of the last theorem.
\begin{beis}
	Consider $W(B_{4})$ and the element
	\begin{equation*}
		s_0s_1s_2s_3J_2=\cbox{
		\begin{tikzpicture}[tldiagram,yscale=0.5]
			\draw \tlcoord{0}{0} \linewave{2}{1};
			\draw \tlcoord{0}{1} \onedot \linewave{2}{1};
			\draw \tlcoord{0}{2} \linewave{2}{1};
			\draw \tlcoord{0}{3} \linewave{2}{-3} \onedot;
		\end{tikzpicture}
	}
	\end{equation*}
	which corresponds to the signed permutation
	\begin{equation*}
		w = \quad \cbox{
			\begin{tikzpicture}[tldiagram]
				%\drawdots{0}{0}{3}
				%\drawdots{0}{5}{8}
				%\drawdots{2}{0}{3}
				%\drawdots{2}{5}{8}
				\draw[dotted] \tlcoord{0}{4} \lineup \lineup;
				\draw \tlcoord{0}{0} \linewave{2}{5};
				\draw \tlcoord{0}{1} \linewave{2}{-1};
				\draw \tlcoord{0}{2} \linewave{2}{5};
				\draw \tlcoord{0}{3} \linewave{2}{-1};
				\draw \tlcoord{0}{5} \linewave{2}{1};
				\draw \tlcoord{0}{6} \linewave{2}{-5};
				\draw \tlcoord{0}{7} \linewave{2}{1};
				\draw \tlcoord{0}{8} \linewave{2}{-5};
			\end{tikzpicture}
		},
	\end{equation*}
	which we read from bottom to top. The decomposition of this permutation of $\{-4,\ldots, 4\}$ into disjoint cycles is given by 
	$w=w_1=(1,2,-3,-4)\cdot (3,4,-1,-2)$. This is precisely a type \ref{cycle condition 2} double cycle and the corresponding conjugacy class corresponds to the bipartition $((4), \varnothing)$.
	For a different example consider the element
	\begin{equation*}
		s_1J_3J_4=\cbox{
			\begin{tikzpicture}[tldiagram]
				\draw \tlcoord{0}{0} \linewave{1}{1};
				\draw \tlcoord{0}{1} \linewave{1}{-1};
				\draw \tlcoord{0}{2} \dlineup;
				\draw \tlcoord{0}{3} \dlineup;
			\end{tikzpicture}
		}
	\end{equation*}
	which corresponds to the signed permutation
	\begin{equation*}
		w' = \quad \cbox{
			\begin{tikzpicture}[tldiagram]
				%\drawdots{0}{0}{3}
				%\drawdots{0}{5}{8}
				%\drawdots{2}{0}{3}
				%\drawdots{2}{5}{8}
				\draw[dotted] \tlcoord{0}{4} \lineup \lineup;
				\draw \tlcoord{0}{0} \linewave{2}{8};
				\draw \tlcoord{0}{1} \lineup\linewave{1}{6};
				\draw \tlcoord{0}{2} \linewave{2}{1};
				\draw \tlcoord{0}{3} \linewave{2}{-1};
				\draw \tlcoord{0}{5} \linewave{2}{1};
				\draw \tlcoord{0}{6} \linewave{2}{-1};
				\draw \tlcoord{0}{7} \lineup \linewave{1}{-6};
				\draw \tlcoord{0}{8} \linewave{2}{-8};
			\end{tikzpicture}
		}.
	\end{equation*}
	The decomposition into disjoint cycles of this element is given by 
	\begin{equation*}
		w'=\underbrace{(1,2)(-1,-2)}_{=w_1}\cdot \underbrace{(3,-3),(4,-4)}_{=w_2}.
	\end{equation*}
	Hence the conjugacy class of this element corresponds to the bipartition $((2),(1,1))$.
\end{beis}
\begin{koro} \label{longet element generates center}
	The center of $W(B_n)$ is equal to the subgroup $\{e, \wo\}$, where $\wo$ is the longest element.
\end{koro}
\begin{proof}
	The center consists of all elements which are their own conjugacy class. By the proof of \Cref{Conjugacy classes via bipartitions} we observe that there are only two bipartitions, which give such a conjugacy class, namely $((1,1,\ldots,1),\varnothing)$ and $(\varnothing, (1,\ldots, 1))$. This follows from the fact that for any other cycle type we can find two distinct entries $i,j\in I$ that give different elements, if replaced via conjugating. The first bipartition corresponds to $e=(1)(-1)(2)(-2)\cdots(n)(-n)$, and the second one to $\wo=(1,-1)(2,-2)\cdots(n,-n)$.
\end{proof}
Next we want to write down the irreducible representations of $W(B_n)$ concretely in terms of idempotents in the group algebra $\C \! W(B_n)$. For this we mimic the construction of Young symmetrizers for the symmetric group $S_n$. 
\begin{defi} \label{defi young symmetrizer}
	Let $(\lambda, \mu)$ a bipartition of $n$. A \textbf{numbering} is a bijection between the boxes of the tuple of Young diagrams corresponding to $(\lambda, \mu)$ with an $n$-element subset of the $2n$ numbers $\{\pm i\mid i\in\{1,\ldots, n\}\}$ such that for each $1\leq i\leq n$ only one of $i$ and $-i$ is used. We imagine a numbering as a filling of the boxes with the corresponding symbols.
	The \textbf{standard numbering} is the numbering with all positive numbers $1,\ldots, n$ by first numbering all the boxes of $\lambda$ and then all the boxes of $\mu$ in reading direction and ascending order. 
	The group $W(B_n)$ \textbf{acts} on the set of numberings $\mathcal{T}_{\lambda, \mu}$ of $(\lambda, \mu)$ as the group of signed permutations with respect to the isomorphism from \Cref{sign permutation are dot permutations}, i.e.\ by permuting the entries of the boxes and possibly flipping the sign. Recall that the Jucys--Murphy element $J_i\in W(B_n)$ flips the numbers $i$ and $-i$ and fixes the rest, while the subgroup $S_n$ permutes the absolute values and ignores the sign. Abstractly the action of $g=(\sigma, \varepsilon)\in G_n\cong W(B_n)$ on a numbering $T_{\lambda, \mu}\in \mathcal{T}_{\lambda, \mu}$ is given by
	\[
	(g\cdot T_{\lambda, \mu})(j)\coloneqq g(T_{\lambda, \mu}(j)) \,
	\]
	for a fixed box $j$ where $g(T_{\lambda, \mu}(j))$ is given by
	\[
		g(T_{\lambda, \mu}(j))= \begin{cases}
						\sigma(i), & \text{if $T_{\lambda, \mu}(j)=i$ and } \varepsilon_j=0, \\
						-\sigma(i), & \text{if $T_{\lambda, \mu}(j)=-i$ and } \varepsilon_j=0, \\
						-\sigma(i), & \text{if $T_{\lambda, \mu}(j)=i$ and } \varepsilon_j=1, \\
						\sigma(i), & \text{if $T_{\lambda, \mu}(j)=-i$ and } \varepsilon_j=1  \\
					\end{cases}
	\]
	for some $i\in\{1,\ldots, n\}$.
	We associate to a numbering $T_{\lambda, \mu}$ the subgroups
	\begin{gather*}
		\begin{aligned}
			R_{T_{\lambda, \mu}}&\coloneqq \left\{ g\in W(B_n) \mid g \text{ fixes all rows of } T_{\lambda, \mu} \text{ up to sign}\right\}\, , \\
			C_{T_{\lambda, \mu}}&\coloneqq \left\{ g\in W(B_n) \mid g \text{ fixes all columns of } T_{\lambda, \mu} \text{ up to sign} \right\}\, .
		\end{aligned}
		\shortintertext{of $W(B_n)$. We call the element}
		q_{T_{\lambda, \mu}}\coloneqq \frac{1}{2^n} r_{T_{\lambda, \mu}}\cdot c_{T_{\lambda, \mu}}
		\shortintertext{where}
		r_{T_{\lambda, \mu}}\coloneqq \sum_{g=(\sigma, \varepsilon)\in R_{T_{\lambda, \mu}}}(-1)^{|\varepsilon|_{\mu}}g \, , \qquad
		c_{T_{\lambda, \mu}}\coloneqq \sum_{g=(\sigma, \varepsilon)\in C_{T_{\lambda, \mu}}}(-1)^{|\varepsilon|_{\mu}}\sgn (\sigma)g
	\end{gather*}
	a type $B$ \textbf{Young symmetrizer} for $(\lambda, \mu)$. Here $|\varepsilon|_{\mu}$ denotes the number of dots, which appear on the bottom of the strands, whose labels appear as entries of the boxes of $\mu$ up to sign.
	We write $q_{\lambda, \mu}$ instead of $q_{T_{\lambda, \mu}}$, when $T_{\lambda, \mu}$ is the standard numbering on $(\lambda,\mu)$. 
	\end{defi}
	In order to understand these definitions consider the following two examples. The first visualizes the action of $W(B_n)$ on numberings and the second is an example of a Young symmetrizer.
	\begin{beis}
		The following two numberings of the bipartition $((2,1),(2))$ are obtained from one another by acting with a group element in $W(B_5)$:
		\begin{center}
			\ytableausetup{nosmalltableaux, boxsize=2em}
			\begingroup
			\setlength{\tabcolsep}{9pt}
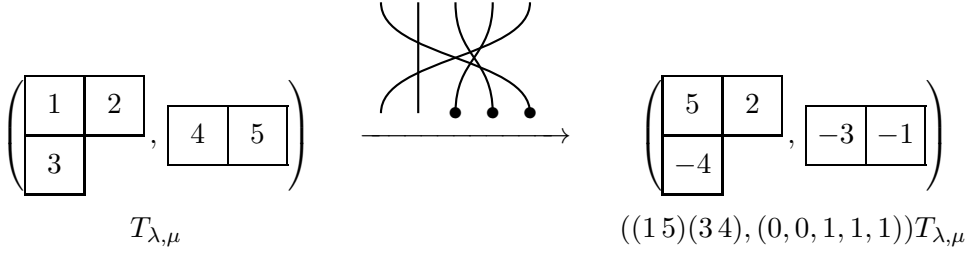
			\begin{tabular}{ccc}
				$\left(\begin{ytableau}
					1 & 2 \\
					3
				\end{ytableau}\, , \, \begin{ytableau}
				4 & 5 
			\end{ytableau}\right)$
				&
				$\xrightarrow{\,\cbox{
					\begin{tikzpicture}[tldiagram, xscale=0.7, yscale=0.7]
						\draw \tlcoord{0}{0} \linewave{2}{4};
						\draw \tlcoord{0}{1} \lineup \lineup;
						\draw \tlcoord{0}{2} \onedot \linewave{2}{1};
						\draw \tlcoord{0}{3} \onedot \linewave{2}{-1};
						\draw \tlcoord{0}{4} \onedot \linewave{2}{-4};
					\end{tikzpicture}
				}\,}$
				&
				$\left(\begin{ytableau}
					5 & 2 \\
					-4
				\end{ytableau}\, , \, \begin{ytableau}
					-3 & -1
				\end{ytableau}\right)$
				\\[2em]
				$T_{\lambda, \mu}$
				&
				{}
				&
				$((1 \, 5) (3 \, 4), (0,0,1,1,1))  T_{\lambda, \mu}$
			\end{tabular}
			\endgroup
			\ytableausetup{smalltableaux}
		\end{center}
		The first numbering is the standard numbering. We have ${|\varepsilon|_\mu}=|(0,0,1,1,1)|_{\mu}$ and hence see that $(-1)^{|\varepsilon|_\mu}=(-1)^2=1$ for the standard numbering, since the strands $4$ and $5$, which are decorated with a dot, appear as entries of the right partition up to sign, while $3$ does not appear. In contrast we have $(-1)^{|\varepsilon|_\mu}=(-1)^1=-1$ for the second numbering, since only $3$appears as entry of a box of the right partition up to sign. 
	\end{beis}
	\begin{beis}
		Let $n=3$ and consider the bipartition $(\lambda, \mu)=((1,1),(1))$ with the standard numbering 
		\[
			T_{\lambda, \mu} = \left(\begin{ytableau}
				1 \\
				2
			\end{ytableau}\, , \, \begin{ytableau}
				3
			\end{ytableau}\right).
		\]
		The row symmetrizer $r_{\lambda, \mu}$ is given by
			\begin{align*}
			\cbox{
				\begin{tikzpicture}[tldiagram, xscale=1/2, yscale=1/2]
					\draw \tlcoord{0}{0} \lineup;
					\draw \tlcoord{0}{1} \lineup;
					\draw \tlcoord{0}{2} \lineup;
				\end{tikzpicture}
			} + \cbox{
				\begin{tikzpicture}[tldiagram, xscale=1/2, yscale=1/2]
					\draw \tlcoord{0}{0} \dlineup;
					\draw \tlcoord{0}{1} \lineup;
					\draw \tlcoord{0}{2} \lineup;
				\end{tikzpicture}
			}
			+ \cbox{
				\begin{tikzpicture}[tldiagram, xscale=1/2, yscale=1/2]
					\draw \tlcoord{0}{0} \lineup;
					\draw \tlcoord{0}{1} \dlineup;
					\draw \tlcoord{0}{2} \lineup;
				\end{tikzpicture}
			}
			- \cbox{
				\begin{tikzpicture}[tldiagram, xscale=1/2, yscale=1/2]
					\draw \tlcoord{0}{0} \lineup;
					\draw \tlcoord{0}{1} \lineup;
					\draw \tlcoord{0}{2} \dlineup;
				\end{tikzpicture}
			} 
			+ \cbox{
				\begin{tikzpicture}[tldiagram, xscale=1/2, yscale=1/2]
					\draw \tlcoord{0}{0} \dlineup;
					\draw \tlcoord{0}{1} \dlineup;
					\draw \tlcoord{0}{2} \lineup;
				\end{tikzpicture}
			} 
			- \cbox{
				\begin{tikzpicture}[tldiagram, xscale=1/2, yscale=1/2]
					\draw \tlcoord{0}{0} \dlineup;
					\draw \tlcoord{0}{1} \lineup;
					\draw \tlcoord{0}{2} \dlineup;
				\end{tikzpicture}
			} 
			- \cbox{
				\begin{tikzpicture}[tldiagram, xscale=1/2, yscale=1/2]
					\draw \tlcoord{0}{0} \lineup;
					\draw \tlcoord{0}{1} \dlineup;
					\draw \tlcoord{0}{2} \dlineup;
				\end{tikzpicture}
			} 
			- \cbox{
				\begin{tikzpicture}[tldiagram, xscale=1/2, yscale=1/2]
					\draw \tlcoord{0}{0} \dlineup;
					\draw \tlcoord{0}{1} \dlineup;
					\draw \tlcoord{0}{2} \dlineup;
				\end{tikzpicture}
			} ,
		\end{align*}
		the column symmetrizer $c_{\lambda, \mu}$ is given by
		\begin{align*}
			&\cbox{
				\begin{tikzpicture}[tldiagram, xscale=1/2, yscale=1/2]
					\draw \tlcoord{0}{0} \lineup;
					\draw \tlcoord{0}{1} \lineup;
					\draw \tlcoord{0}{2} \lineup;
				\end{tikzpicture}
			} + \cbox{
				\begin{tikzpicture}[tldiagram, xscale=1/2, yscale=1/2]
					\draw \tlcoord{0}{0} \dlineup;
					\draw \tlcoord{0}{1} \lineup;
					\draw \tlcoord{0}{2} \lineup;
				\end{tikzpicture}
			}
			+ \cbox{
				\begin{tikzpicture}[tldiagram, xscale=1/2, yscale=1/2]
					\draw \tlcoord{0}{0} \lineup;
					\draw \tlcoord{0}{1} \dlineup;
					\draw \tlcoord{0}{2} \lineup;
				\end{tikzpicture}
			}
			- \cbox{
				\begin{tikzpicture}[tldiagram, xscale=1/2, yscale=1/2]
					\draw \tlcoord{0}{0} \lineup;
					\draw \tlcoord{0}{1} \lineup;
					\draw \tlcoord{0}{2} \dlineup;
				\end{tikzpicture}
			} 
			+ \cbox{
				\begin{tikzpicture}[tldiagram, xscale=1/2, yscale=1/2]
					\draw \tlcoord{0}{0} \dlineup;
					\draw \tlcoord{0}{1} \dlineup;
					\draw \tlcoord{0}{2} \lineup;
				\end{tikzpicture}
			} 
			- \cbox{
				\begin{tikzpicture}[tldiagram, xscale=1/2, yscale=1/2]
					\draw \tlcoord{0}{0} \dlineup;
					\draw \tlcoord{0}{1} \lineup;
					\draw \tlcoord{0}{2} \dlineup;
				\end{tikzpicture}
			} 
			- \cbox{
				\begin{tikzpicture}[tldiagram, xscale=1/2, yscale=1/2]
					\draw \tlcoord{0}{0} \lineup;
					\draw \tlcoord{0}{1} \dlineup;
					\draw \tlcoord{0}{2} \dlineup;
				\end{tikzpicture}
			} 
			- \cbox{
				\begin{tikzpicture}[tldiagram, xscale=1/2, yscale=1/2]
					\draw \tlcoord{0}{0} \dlineup;
					\draw \tlcoord{0}{1} \dlineup;
					\draw \tlcoord{0}{2} \dlineup;
				\end{tikzpicture}
			} \\
			-
			&\cbox{
				\begin{tikzpicture}[tldiagram, xscale=1/2, yscale=1/2]
					\draw \tlcoord{0}{0} \linewave{1}{1};
					\draw \tlcoord{0}{1} \linewave{1}{-1};
					\draw \tlcoord{0}{2} \lineup;
				\end{tikzpicture}
			} - \cbox{
				\begin{tikzpicture}[tldiagram, xscale=1/2, yscale=1/2]
					\draw \tlcoord{0}{0} \onedot \linewave{1}{1};
					\draw \tlcoord{0}{1} \linewave{1}{-1};
					\draw \tlcoord{0}{2} \lineup;
				\end{tikzpicture}
			}
			- \cbox{
				\begin{tikzpicture}[tldiagram, xscale=1/2, yscale=1/2]
					\draw \tlcoord{0}{0} \linewave{1}{1};
					\draw \tlcoord{0}{1} \onedot \linewave{1}{-1};
					\draw \tlcoord{0}{2} \lineup;
				\end{tikzpicture}
			}
			- \cbox{
				\begin{tikzpicture}[tldiagram, xscale=1/2, yscale=1/2]
					\draw \tlcoord{0}{0} \onedot \linewave{1}{1};
					\draw \tlcoord{0}{1} \onedot \linewave{1}{-1};
					\draw \tlcoord{0}{2} \lineup;
				\end{tikzpicture}
			}
			+ \cbox{
				\begin{tikzpicture}[tldiagram, xscale=1/2, yscale=1/2]
					\draw \tlcoord{0}{0} \linewave{1}{1};
					\draw \tlcoord{0}{1} \linewave{1}{-1};
					\draw \tlcoord{0}{2} \dlineup;
				\end{tikzpicture}
			} + \cbox{
				\begin{tikzpicture}[tldiagram, xscale=1/2, yscale=1/2]
					\draw \tlcoord{0}{0} \onedot \linewave{1}{1};
					\draw \tlcoord{0}{1} \linewave{1}{-1};
					\draw \tlcoord{0}{2} \dlineup;
				\end{tikzpicture}
			}
			+ \cbox{
				\begin{tikzpicture}[tldiagram, xscale=1/2, yscale=1/2]
					\draw \tlcoord{0}{0} \linewave{1}{1};
					\draw \tlcoord{0}{1} \onedot \linewave{1}{-1};
					\draw \tlcoord{0}{2} \dlineup;
				\end{tikzpicture}
			}
			+ \cbox{
				\begin{tikzpicture}[tldiagram, xscale=1/2, yscale=1/2]
					\draw \tlcoord{0}{0} \onedot \linewave{1}{1};
					\draw \tlcoord{0}{1} \onedot \linewave{1}{-1};
					\draw \tlcoord{0}{2} \dlineup;
				\end{tikzpicture}
			} \, .
		\end{align*}
		Since in this case $r_{\lambda, \mu}c_{\lambda, \mu}=8c_{\lambda, \mu}$, the Young symmetrizer $q_{\lambda, \mu}$ agrees with the column symmetrizer $c_{\lambda, \mu}$. One can calculate that $q_{\lambda, \mu}^2=16 q_{\lambda, \mu}$ and that $\C \! W(B_3) q_{\lambda, \mu}$ is three-dimensional and spanned by $\{q_{\lambda, \mu}, s_2q_{\lambda, \mu}, s_1s_2 q_{\lambda, \mu} \}$.
	\end{beis}
We can now formulate the classification theorem of the irreducible representations of $W(B_n)$.
\begin{theo}[Classification of irreducible $W(B_n)$ representations]
	For each bipartition $(\lambda, \mu)$ of $n$ the \textbf{Specht module} module $S(\lambda, \mu)\coloneqq \C W(B_n) q_{\lambda, \mu}$ is irreducible. The element $q_{\lambda, \mu}$ is a quasi-idempotent, i.e.\ there exists a non-zero scalar $n_{\lambda, \mu}$ such that $q_{\lambda, \mu}^2=n_{\lambda, \mu}q_{\lambda, \mu}$.
	There is a one-to-one correspondence
	\begin{align*}
		\left\{\text{bipartitions $(\lambda,\mu)$ of $n$} \right\}
		&\xleftrightarrow{\,\text{1:1}\,}
		\{\text{irreducible complex representations of $W(B_n)$}\}/{\cong} \, ,
		\\
		\lambda
		&\xmapsto{\,\phantom{\text{1:1}}\,}
		S(\lambda,\mu) \,.
	\end{align*}
\end{theo}
A complete proof for complex reflection groups can be found in \cite[\S2]{Petrova2017}, where  $G(2,1,n)=W(B_n)$. We give a sketch of the proof.
\begin{proof}[Sketch of Proof]
We sketch how to adapt the proof for type $A$, which can be found in \cite[\S4]{fulton2013}, to type $B$. We additionally give references to \cite[Chapter 2]{wojciechowski19}, since the proofs there contain the proofs for the combinatorial facts used in the proofs in \cite{fulton2013}. 

First one shows an analogue unique characterizing property of $q_{\lambda, \mu}$ as in \cite[Lemma 4.21]{fulton2013}/\cite[Lemma 2.16]{wojciechowski19}. The analogue is the following statement. The Young symmetrizer $q_{\lambda,\mu}$ is the unique (up to taking scalar multiples) element in $\C \! W(B_n)$ which satisfies the properties:
\begin{enumerate}
	\item $x \cdot q_{\lambda,\mu} \cdot \sgn(y) y = q_{\lambda,\mu}$ for all $x\in R_{\lambda, \mu}, y\in C_{\lambda, \mu}$,
	\item $J_i q_{\lambda, \mu}=q_{\lambda, \mu}=q_{\lambda, \mu}J_i$ if $i$ is the entry of a box of $\lambda$ with respect to the chosen numbering,
	\item $J_i q_{\lambda, \mu}=-q_{\lambda, \mu}=q_{\lambda, \mu}J_i$ if $i$ is the entry of a box of $\mu$ with respect to the chosen numbering.
\end{enumerate}
Here $\sgn\colon W(B_n)\rightarrow \{\pm 1\}$ maps a dotted permutation to the sign of the underlying permutation ignoring any dots.  From this property one can conclude that there exists a scalar $n_{\lambda, \mu}$ such that $q_{\lambda, \mu}^2=n_{\lambda, \mu}q_{\lambda, \mu}$ as in \cite[Lemma 4.23]{fulton2013}/\cite[Corollary 2.17]{wojciechowski19}.
In order to conclude that $n_{\lambda, \mu}$ is non-zero, copy the proof of \cite[Lemma 4.26]{fulton2013}/\cite[Lemma 2.21]{wojciechowski19}. However there is one subtlety, namely that the coefficient of $e$ in $q_{\lambda,\mu}$ is really $1$, even though the groups $R_{\lambda, \mu}$ and $C_{\lambda, \mu}$ have non-trivial intersection (which is the subgroup isomorphic to $(\mZ / 2\! \mZ)^n$ generated by the Jucys--Murphy elements), since we normalized $q_{\lambda, \mu}$ in its definition by the factor $\frac{1}{2^n}$. Hence we obtain $n_{\lambda, \mu}=\frac{2^nn!}{\dim S(\lambda, \mu)}$. For versions of \cite[Lemma 4.23]{fulton2013}/\cite[Lemma 2.20]{wojciechowski19} and \cite[Lemma 4.25]{fulton2013}/\cite[Lemma 2.23]{wojciechowski19}, which show that the representations $\C \! W(B_n) q_{\lambda, \mu}$ and $\C \! W(B_n) q_{\lambda', \mu'}$ are not isomorphic one distinguishes two cases. The first case is that the sizes of the bipartitions differ, i.e.\ $|\lambda|\neq|\lambda'|$. Then by the pigeon hole principle there exists a number $i\in\{1,\ldots, n\}$, which occurs (up to sign) both as an entry in $\lambda$ and $\mu'$. In particular we have
	\[
		q_{\lambda, \mu}q_{\lambda', \mu'}=(q_{\lambda, \mu}J_i)q_{\lambda', \mu'}=q_{\lambda, \mu}(J_iq_{\lambda', \mu'})=q_{\lambda, \mu}(-q_{\lambda', \mu'}),
	\]
	which shows that $q_{\lambda, \mu}q_{\lambda', \mu'}=0$. Note that this argument does not depend of the chosen numbering of $(\lambda', \mu')$. In particular $q_{\lambda, \mu}\C\!W(B_n)q_{\lambda', \mu'}=0$, since for any $g\in W(B_n)$ we have
	\[
		q_{\lambda, \mu}\cdot g\cdot q_{T_{\lambda', \mu'}} \cdot g^{-1}=q_{\lambda, \mu}\cdot q_{g\cdot T_{\lambda', \mu'}}=0 .
	\]
	 In the second case assume $|\lambda|=|\lambda'|$ and $|\mu|=|\mu'|$, but either $\lambda\neq\lambda'$ or $\mu\neq \mu'$. Without loss of generality $\lambda\neq \lambda'$ and assume $\lambda>\lambda'$ with respect to the lexicographical order. By the combinatorial argument from the proof of \cite[Lemma 2.20]{wojciechowski19} one finds two indices which are both in the same row in $\lambda$, but in the same column in $\lambda'$. Hence the argument in the proof of \cite[Lemma 4.23]{fulton2013}/\cite[Lemma 2.20]{wojciechowski19} gives that the product $q_{\lambda, \mu}x q_{\lambda', \mu'}$ vanishes and even more generally $q_{\lambda, \mu}x q_{\lambda', \mu'}=0$ for all $x\in W(B_n)$.
	From this we conclude
	\[
		\Hom(S(\lambda, \mu), S(\lambda', \mu'))\cong q_{\lambda, \mu} \C W(B_n) q_{\lambda', \mu'}=0
	\]
	and hence $S(\lambda, \mu)$ and $S(\lambda', \mu')$ are not isomorphic if the bipartitions differ. Finally the Specht modules $S(\lambda, \mu)$ are irreducible by the same proof as \cite[Lemma 4.25]{fulton2013}/\cite[Lemma 2.22]{wojciechowski19} using the above type $B$ version of
	\cite[Lemma 4.21]{fulton2013}/\cite[Lemma 2.16]{wojciechowski19}.
\end{proof}
%This is a type $B_n$ version of the classification theorem of finite dimensional complex representations of the symmetric group $S_n=W(A_{n-1})$. These correspond to partitions $\lambda$ of $n$ and are called $S(\lambda)$, the Specht modules of type $A_{n-1}$.
\begin{beis} \label{n=2 irreducible representations}
	For $n=2$ we have the following five bipartitions:
	\begin{center} \label{n=2 bipartitions}
		\begingroup
		\setlength{\tabcolsep}{15pt}
		\begin{tabular}{ccccc}
			$(\ydiagram{2},\varnothing)$ 
			&
			$(\ydiagram{1,1},\varnothing)$
			&
			$(\ydiagram{1}, \ydiagram{1})$ 
			&
			$(\varnothing ,\ydiagram{1,1})$ 
			&
			$(\varnothing, \ydiagram{2})$ 
			%\\[2.5em]
			%$((2), \varnothing)$
			%&
			%$((1,1), \varnothing)$
			%&
			%$((1), (1))$
			%&
			%$(\varnothing, (1,1))$
			%&
			%$(\varnothing, (2))$
		\end{tabular}
		\endgroup
	\end{center} 
	The five irreducible representations can be constructed elementary as follows. One-dimensional representations correspond to elements in the abelianization
	\[
		W(B_2)/[W(B_2),W(B_2)] = W(B_2)/\{e,s_0s_1s_0s_1\} =\{\overline{e},\overline{s_0},\overline{s_1},\overline{s_0s_1}=\overline{s_1s_0}\}\cong (\mZ/2\mZ)^2.
	\]
	The actions of the generators $s_0$ and $s_1$ must have eigenvalue $1$ or $-1$ on these one-dimensional representations and any choice of such eigenvalues $\varepsilon=(\varepsilon_0,\varepsilon_1)\in\{1,-1\}^2$ defines a unique one-dimensional representation $\C_{\varepsilon}$ such that $s_i$ acts by $\varepsilon_i$ for $i=0,1$. These correspond to the four bipartitions
	\[
		S(\ydiagram{2},\varnothing)=\C_{1,1}\, , \quad S(\ydiagram{1,1},\varnothing)=\C_{1,-1} \, , \quad S(\varnothing,\ydiagram{2})=\C_{-1,1}\, , \quad S(\varnothing,\ydiagram{1,1})=\C_{-1,-1} \, .
	\]
	 Since the group $W(B_2)$ is not abelian, there must exist at least one irreducible representation, which is not one-dimensional. There is an irreducible two-dimensional representation of $W(B_2)$, which is given by $\C^2=\shift{e_1,e_2}$ and the action
	 \[
	 	s_0\cdot e_1=-e_1\, , \quad s_0 \cdot e_2 = e_2\, , \quad s_1(e_1)=e_2\, , \quad s_1(e_2)=e_1.
	 \] 
	We count and see that 
	\[
		\dim \C \! W(B_2)=8=1^2+1^2+2^2+1^2+1^2,
	\]
	so these are all irreducible representations by the Artin-Wedderburn theorem. The two-dimensional irreducible representation we wrote down is isomorphic to $\h^*$, where $\h$ is the Cartan subalgebra of $\lie(B_2)=\so_5$ and hence can be visualized in terms of the type $B_2$ root system. Concretely by setting $e_1\coloneqq\alpha_0, e_2\coloneqq\alpha_0+\alpha_1$ like in \eqref{B2 root system} gives the desired identification.
	%where $\alpha, \beta$ are simple roots with $1=|\alpha_0|<|\alpha_1|=\sqrt{2}$ gives the desired identification,
	The normalized Young symmetrizers $e_{\lambda, \mu}\coloneqq \frac{1}{n_{\lambda, \mu}}q_{\lambda, \mu}$ corresponding to the irreducible representations are given by
	\begin{align*}
		e_{(2),(0)}&=\frac{1}{8} \left(
		\cbox{
			\begin{tikzpicture}[tldiagram, xscale=0.6, yscale=0.7]
				\draw \tlcoord{0}{0} \lineup;
				\draw \tlcoord{0}{1} \lineup;
			\end{tikzpicture}
		} \, + \, 
		\cbox{
			\begin{tikzpicture}[tldiagram, xscale=0.6, yscale=0.7]
				\draw \tlcoord{0}{0} \dlineup;
				\draw \tlcoord{0}{1} \lineup;
			\end{tikzpicture}
		} \, + \, 
		\cbox{
			\begin{tikzpicture}[tldiagram, xscale=0.6, yscale=0.7]
				\draw \tlcoord{0}{0} \lineup;
				\draw \tlcoord{0}{1} \dlineup;
			\end{tikzpicture}
		} \, + \, 
		\cbox{
			\begin{tikzpicture}[tldiagram, xscale=0.6, yscale=0.7]
				\draw \tlcoord{0}{0} \dlineup;
				\draw \tlcoord{0}{1} \dlineup;
			\end{tikzpicture}
		} \, + \, 
		\cbox{
			\begin{tikzpicture}[tldiagram, xscale=0.6, yscale=0.7]
				\draw \tlcoord{0}{0} \linewave{1}{1};
				\draw \tlcoord{0}{1} \linewave{1}{-1};
			\end{tikzpicture}
		} \, + \, 
		\cbox{
			\begin{tikzpicture}[tldiagram, xscale=0.6, yscale=0.7]
				\draw \tlcoord{0}{0} \onedot \linewave{1}{1};
				\draw \tlcoord{0}{1} \linewave{1}{-1};
			\end{tikzpicture}
		} \, + \, 
		\cbox{
			\begin{tikzpicture}[tldiagram, xscale=0.6, yscale=0.7]
				\draw \tlcoord{0}{0} \linewave{1}{1};
				\draw \tlcoord{0}{1} \onedot \linewave{1}{-1};
			\end{tikzpicture}
		} \, + \, 
		\cbox{
			\begin{tikzpicture}[tldiagram, xscale=0.6, yscale=0.7]
				\draw \tlcoord{0}{0} \onedot \linewave{1}{1};
				\draw \tlcoord{0}{1} \onedot \linewave{1}{-1};
			\end{tikzpicture}
		} \right)  \\
		e_{(1,1), (0)}&=\frac{1}{8} \left(
		\cbox{
			\begin{tikzpicture}[tldiagram, xscale=0.6, yscale=0.7]
				\draw \tlcoord{0}{0} \lineup;
				\draw \tlcoord{0}{1} \lineup;
			\end{tikzpicture}
		} \, + \, 
		\cbox{
			\begin{tikzpicture}[tldiagram, xscale=0.6, yscale=0.7]
				\draw \tlcoord{0}{0} \dlineup;
				\draw \tlcoord{0}{1} \lineup;
			\end{tikzpicture}
		} \, + \, 
		\cbox{
			\begin{tikzpicture}[tldiagram, xscale=0.6, yscale=0.7]
				\draw \tlcoord{0}{0} \lineup;
				\draw \tlcoord{0}{1} \dlineup;
			\end{tikzpicture}
		} \, + \, 
		\cbox{
			\begin{tikzpicture}[tldiagram, xscale=0.6, yscale=0.7]
				\draw \tlcoord{0}{0} \dlineup;
				\draw \tlcoord{0}{1} \dlineup;
			\end{tikzpicture}
		} \, - \, 
		\cbox{
			\begin{tikzpicture}[tldiagram, xscale=0.6, yscale=0.7]
				\draw \tlcoord{0}{0} \linewave{1}{1};
				\draw \tlcoord{0}{1} \linewave{1}{-1};
			\end{tikzpicture}
		} \, - \, 
		\cbox{
			\begin{tikzpicture}[tldiagram, xscale=0.6, yscale=0.7]
				\draw \tlcoord{0}{0} \onedot \linewave{1}{1};
				\draw \tlcoord{0}{1} \linewave{1}{-1};
			\end{tikzpicture}
		} \, - \, 
		\cbox{
			\begin{tikzpicture}[tldiagram, xscale=0.6, yscale=0.7]
				\draw \tlcoord{0}{0} \linewave{1}{1};
				\draw \tlcoord{0}{1} \onedot \linewave{1}{-1};
			\end{tikzpicture}
		} \, - \, 
		\cbox{
			\begin{tikzpicture}[tldiagram, xscale=0.6, yscale=0.7]
				\draw \tlcoord{0}{0} \onedot \linewave{1}{1};
				\draw \tlcoord{0}{1} \onedot \linewave{1}{-1};
			\end{tikzpicture}
		} \right) \\
		e_{(1),(1)}&= \frac{1}{4} \left(
		\cbox{
			\begin{tikzpicture}[tldiagram, xscale=0.6, yscale=0.7]
				\draw \tlcoord{0}{0} \lineup;
				\draw \tlcoord{0}{1} \lineup;
			\end{tikzpicture}
		} \, + \, 
		\cbox{
			\begin{tikzpicture}[tldiagram, xscale=0.6, yscale=0.7]
				\draw \tlcoord{0}{0} \dlineup;
				\draw \tlcoord{0}{1} \lineup;
			\end{tikzpicture}
		} \, - \, 
		\cbox{
			\begin{tikzpicture}[tldiagram, xscale=0.6, yscale=0.7]
				\draw \tlcoord{0}{0} \lineup;
				\draw \tlcoord{0}{1} \dlineup;
			\end{tikzpicture}
		} \, - \, 
		\cbox{
			\begin{tikzpicture}[tldiagram, xscale=0.6, yscale=0.7]
				\draw \tlcoord{0}{0} \dlineup;
				\draw \tlcoord{0}{1} \dlineup;
			\end{tikzpicture}
		} \right)\\
		e_{(0),(1,1)}&=\frac{1}{8} \left(
		\cbox{
			\begin{tikzpicture}[tldiagram, xscale=0.6, yscale=0.7]
				\draw \tlcoord{0}{0} \lineup;
				\draw \tlcoord{0}{1} \lineup;
			\end{tikzpicture}
		} \, - \, 
		\cbox{
			\begin{tikzpicture}[tldiagram, xscale=0.6, yscale=0.7]
				\draw \tlcoord{0}{0} \dlineup;
				\draw \tlcoord{0}{1} \lineup;
			\end{tikzpicture}
		} \, - \, 
		\cbox{
			\begin{tikzpicture}[tldiagram, xscale=0.6, yscale=0.7]
				\draw \tlcoord{0}{0} \lineup;
				\draw \tlcoord{0}{1} \dlineup;
			\end{tikzpicture}
		} \, + \, 
		\cbox{
			\begin{tikzpicture}[tldiagram, xscale=0.6, yscale=0.7]
				\draw \tlcoord{0}{0} \dlineup;
				\draw \tlcoord{0}{1} \dlineup;
			\end{tikzpicture}
		} \, - \, 
		\cbox{
			\begin{tikzpicture}[tldiagram, xscale=0.6, yscale=0.7]
				\draw \tlcoord{0}{0} \linewave{1}{1};
				\draw \tlcoord{0}{1} \linewave{1}{-1};
			\end{tikzpicture}
		} \, + \, 
		\cbox{
			\begin{tikzpicture}[tldiagram, xscale=0.6, yscale=0.7]
				\draw \tlcoord{0}{0} \onedot \linewave{1}{1};
				\draw \tlcoord{0}{1} \linewave{1}{-1};
			\end{tikzpicture}
		} \, + \, 
		\cbox{
			\begin{tikzpicture}[tldiagram, xscale=0.6, yscale=0.7]
				\draw \tlcoord{0}{0} \linewave{1}{1};
				\draw \tlcoord{0}{1} \onedot \linewave{1}{-1};
			\end{tikzpicture}
		} \, - \, 
		\cbox{
			\begin{tikzpicture}[tldiagram, xscale=0.6, yscale=0.7]
				\draw \tlcoord{0}{0} \onedot \linewave{1}{1};
				\draw \tlcoord{0}{1} \onedot \linewave{1}{-1};
			\end{tikzpicture}
		} \right)  \\
		e_{(0), (2)} &= \frac{1}{8} \left(
		\cbox{
			\begin{tikzpicture}[tldiagram, xscale=0.6, yscale=0.7]
				\draw \tlcoord{0}{0} \lineup;
				\draw \tlcoord{0}{1} \lineup;
			\end{tikzpicture}
		} \, - \, 
		\cbox{
			\begin{tikzpicture}[tldiagram, xscale=0.6, yscale=0.7]
				\draw \tlcoord{0}{0} \dlineup;
				\draw \tlcoord{0}{1} \lineup;
			\end{tikzpicture}
		} \, - \, 
		\cbox{
			\begin{tikzpicture}[tldiagram, xscale=0.6, yscale=0.7]
				\draw \tlcoord{0}{0} \lineup;
				\draw \tlcoord{0}{1} \dlineup;
			\end{tikzpicture}
		} \, + \, 
		\cbox{
			\begin{tikzpicture}[tldiagram, xscale=0.6, yscale=0.7]
				\draw \tlcoord{0}{0} \dlineup;
				\draw \tlcoord{0}{1} \dlineup;
			\end{tikzpicture}
		} \, + \, 
		\cbox{
			\begin{tikzpicture}[tldiagram, xscale=0.6, yscale=0.7]
				\draw \tlcoord{0}{0} \linewave{1}{1};
				\draw \tlcoord{0}{1} \linewave{1}{-1};
			\end{tikzpicture}
		} \, - \, 
		\cbox{
			\begin{tikzpicture}[tldiagram, xscale=0.6, yscale=0.7]
				\draw \tlcoord{0}{0} \onedot \linewave{1}{1};
				\draw \tlcoord{0}{1} \linewave{1}{-1};
			\end{tikzpicture}
		} \, - \, 
		\cbox{
			\begin{tikzpicture}[tldiagram, xscale=0.6, yscale=0.7]
				\draw \tlcoord{0}{0} \linewave{1}{1};
				\draw \tlcoord{0}{1} \onedot \linewave{1}{-1};
			\end{tikzpicture}
		} \, + \, 
		\cbox{
			\begin{tikzpicture}[tldiagram, xscale=0.6, yscale=0.7]
				\draw \tlcoord{0}{0} \onedot \linewave{1}{1};
				\draw \tlcoord{0}{1} \onedot \linewave{1}{-1};
			\end{tikzpicture}
		} \right).
	\end{align*}
	Together with the element
	\[
		e_{(1),(1)}'= \cbox{
			\begin{tikzpicture}[tldiagram, xscale=0.6, yscale=0.7]
				\draw \tlcoord{0}{0}  \linewave{1}{1};
				\draw \tlcoord{0}{1}  \linewave{1}{-1};
			\end{tikzpicture}
		} \cdot  e_{(1),(1)} \cdot \cbox{
			\begin{tikzpicture}[tldiagram, xscale=0.6, yscale=0.7]
				\draw \tlcoord{0}{0}  \linewave{1}{1};
				\draw \tlcoord{0}{1}  \linewave{1}{-1};
			\end{tikzpicture}
		} =\frac{1}{4} \left(
		\cbox{
			\begin{tikzpicture}[tldiagram, xscale=0.6, yscale=0.7]
				\draw \tlcoord{0}{0} \lineup;
				\draw \tlcoord{0}{1} \lineup;
			\end{tikzpicture}
		} \, - \, 
		\cbox{
			\begin{tikzpicture}[tldiagram, xscale=0.6, yscale=0.7]
				\draw \tlcoord{0}{0} \dlineup;
				\draw \tlcoord{0}{1} \lineup;
			\end{tikzpicture}
		} \, + \, 
		\cbox{
			\begin{tikzpicture}[tldiagram, xscale=0.6, yscale=0.7]
				\draw \tlcoord{0}{0} \lineup;
				\draw \tlcoord{0}{1} \dlineup;
			\end{tikzpicture}
		} \, - \, 
		\cbox{
			\begin{tikzpicture}[tldiagram, xscale=0.6, yscale=0.7]
				\draw \tlcoord{0}{0} \dlineup;
				\draw \tlcoord{0}{1} \dlineup;
			\end{tikzpicture}
		} \right)
	\]
	we obtain a complete collection of pairwise orthogonal primitive idempotents. 
\end{beis} 
\begin{rema}
	The longest element $\wo$ lies in the center of $W(B_n)$ by \Cref{longet element generates center} and hence acts on all irreducible representations by a scalar, which has to be $1$ or $-1$ since $\wo$ is self-inverse. Concretely for $n=2$ the longest element $\wo=s_0s_1s_0s_1$ acts by $1$ on all $1$-dimensional representations and by $-1$ on the $2$-dimensional representation from \Cref{n=2 irreducible representations}. 
	This is a difference to the representation theory of $S_n=W(A_{n-1})$ or $W(D_n)$ (for $n$ odd), where the longest element is in general not central (c.f.\ \Cref{longest element in type D}) and does not need to act by a constant on each irreducible representation.
\end{rema}
The next remark outlines, how to generalize the definition of the Young symmetrizers to the family of complex reflection groups $G(m,1,n)$. While the representation theory for these groups is completely understood (see e.g.\ \cite{ogievetsky2012jucysmurphy} or \cite{Petrova2017}), the author did not find these generalized idempotents in the literature.
\begin{rema}
	One can replace the group $G(1,1,n)\coloneqq W(B_n)=(\mZ/2\!\mZ) \wr S_n$ by the group $G(m,1,n)=(\mZ/r\!\mZ)\wr S_n$. This group consists of all decorated permutation diagrams, which are decorated with up to $r-1$ dots on each strand and with multiplication rule that allows to cancel $r$ dots on the same strand. Bipartitions generalize to $r$-tuples of partitions, which label the irreducible complex representations of $(\mZ/r\!\mZ)\wr S_n$. 
	Instead of numbering the boxes with $I=\{-n,\ldots, n\}=\{-1,1\}\times \{1,\ldots, n\}$, one uses $ \mZ \! / m \! \mZ\times \{1,\ldots, n\}$. The Young symmetrizer is defined in the same way by replacing the scalar $(-1)^{|\varepsilon|_{\mu}}$ by a power of a primitive $n$-th root of unity $\zeta$. 
	%All of the groups $G(m,1,n)$ have Hecke algebra versions, however we focus on $W(B_n)=G(2,1,n)$, since there we have a better understanding using the coideal Schur--Weyl duality.
\end{rema}

\section{Branching rule} \label{Chapter 1 section 3}

In order to get a better feeling for the irreducible Specht modules $S(\lambda, \mu)$ from the last section, we use this section to explain the type $B$ branching rule. This rule determines the entire behavior of induction and restriction functors for representations of $W(B_n)$. Our reference for this section is \cite{ogievetsky2012jucysmurphy}, however this is a non-quantized version.

\begin{defi}
	Let $W(B_n)\hookrightarrow W(B_{m})$ denote the canonical inclusions of parabolic subgroup for $n\leq m$. The \textbf{restriction functor} is defined as
	\begin{align*}
		\RES^{m}_n\colon \Rep (W(B_{m}))&\to \Rep (W(B_n)) \\
		V &\mapsto V.
	\end{align*} 
	The \textbf{induction functor} is defined as
	\begin{align*}
		\IND^{m}_n\colon \Rep W(B_{n})&\to \Rep W(B_m) \\
		V &\mapsto \C\! W(B_m) \otimes_{\C \! W(B_n)} V.
	\end{align*}
\end{defi}
Before we state the branching rule, which gives concrete formulas for the induction and restriction functors, we consider a motivating example and a conceptual way to find the idempotents in $\C \!W(B_2)$ from \Cref{n=2 irreducible representations}.
\begin{beis} \label{example inducing up idempotents}
	Let $(\lambda, \mu)$ a bipartition of $n$. We can view the Young symmetrizer $e_{\lambda, \mu}\in \C \! W(B_n)$ as an element in $W(B_{n+1})$ via the embedding $\C\! W(B_n)\hookrightarrow \C \! W(B_{n+1})$ which diagrammatically adds an additional strand.
	This gives an idempotent in $\C \! W(B_{n+1})$, which is possibly no longer primitive, i.e.\ can be written as sum $e + e'$ for two orthogonal idempotents $e, e'\in \C \! W(B_{n+1})$. Decomposing this idempotent into primitive, pairwise orthogonal idempotents in the algebra $\C \! W(B_{n+1})$ can be interpreted as a decomposition of the unit $1$ in the algebra $e_{\lambda, \mu}\C \! W(B_{n+1}) e_{\lambda, \mu}$ into primitive idempotents. Up to conjugation the idempotents in this decomposition correspond to the irreducible representations of $e_{\lambda, \mu}\C \! W(B_{n+1}) e_{\lambda, \mu}$, which are precisely those irreducible representations $S(\lambda', \mu')$ of $\C \! W(B_{n+1})$ such that $e_{\lambda, \mu}S(\lambda', \mu')\neq 0$, i.e.\ those representations $S(\lambda', \mu')$ of $W(B_{n+1})$ which occur as direct summand in the projective module $\C \! W(B_{n+1})e_{\lambda, \mu}$. This module can be interpreted as the image of the irreducible module $S(\lambda, \mu)$ under the induction functor $\IND^{n+1}_{n}$. Indeed we have 
	\[
	\C \! W(B_{n+1}) e_{\lambda, \mu}=\C \! W(B_{n+1})  \otimes_{\C\!W(B_{n})} \C\!W(B_{n}) e_{\lambda, \mu} = \IND^{n+1}_{n}(S(\lambda, \mu)).
	\]
	Considering now again the case $n=1$ we have the idempotents $e_{(1),(0)}=\frac{1}{2}(e+s_0)$ and $e_{(0),(1)}=\frac{1}{2}(e-s_0)$ in $\C\!W(B_{1})$. Their images in $W(B_2)$ are given diagrammatically by 
	\[
		e_{(1),(0)} = \frac{1}{2} \left( \cbox{
		\begin{tikzpicture}[tldiagram]
			\draw \tlcoord{0}{0} \lineup;
			\draw \tlcoord{0}{1} \lineup;
		\end{tikzpicture}
	} + \cbox{
	\begin{tikzpicture}[tldiagram]
		\draw \tlcoord{0}{0} \dlineup;
		\draw \tlcoord{0}{1} \lineup;
	\end{tikzpicture}
} \right) \, , \quad e_{(0),(1)} = \frac{1}{2} \left( \cbox{
\begin{tikzpicture}[tldiagram]
	\draw \tlcoord{0}{0} \lineup;
	\draw \tlcoord{0}{1} \lineup;
\end{tikzpicture}
} - \cbox{
\begin{tikzpicture}[tldiagram]
	\draw \tlcoord{0}{0} \dlineup;
	\draw \tlcoord{0}{1} \lineup;
\end{tikzpicture}
} \right) \,.
	\]
	These idempotents now decompose precisely as
	\[
		e_{(1),(0)}=e_{(2),(0)}+e_{(1,1),(0)}+e_{(1),(1)} \, , \quad 	e_{(0),(1)}=e_{(0),(2)}+e_{(0),(1,1)}+e_{(1),(1)}'
	\]
	for the idempotents in \Cref{n=2 irreducible representations}.
	This shows that 
	\[
		\IND^{2}_1(S((1),(0)))\cong S((2),(0))\oplus S((1,1),(0)) \oplus S((1),(1)) 
	\]
	and 
	\[ 
		\IND^{2}_1(S((0),(1)))\cong S((0),(2))\oplus S((0),(1,1)) \oplus S((1),(1)) \, .
	\]
	We observe that the bipartitions which occur in these decompositions are precisely the ones obtained by adding an additional box to the existing bipartition $(\lambda, \mu)$. 
\end{beis}
\Cref{example inducing up idempotents} is part of a general behavior of induction and restriction functors. For the symmetric groups this phenomenon is well-known and proven in e.g.\ \cite[\S2.4]{jameskerber81} and for the following type $B$ version see e.g.\ \cite{ogievetsky2012jucysmurphy} and the references therein.
\begin{theo}[Branching rule] \label{theorem induction and restriction behaviour for type B}
	The induction functors $\IND^{n}_{n-1}$ are biadjoint to the restriction functors $\RES^{n}_{n-1}$. Moreover they satisfy the following decomposition rules.
	\begin{enumerate}
		\item The restriction functors satisfy
		\begin{equation*} \label{restriction gives decompositon}
			\RES^{n}_{n-1}(S(\lambda, \mu))\cong\bigoplus_{(\lambda',\mu')}S(\lambda',\mu')
		\end{equation*}
		where $(\lambda',\mu')$ runs over all bipartitions of $n-1$, which can be obtained from $(\lambda,\mu)$ by deletion of a single box in $\lambda$ or $\mu$.
		\item The induction functors satisfy
		\begin{equation*}
			\IND^{n+1}_{n}(S(\lambda, \mu))\cong\bigoplus_{(\lambda',\mu')}S(\lambda',\mu')
		\end{equation*}
		where $(\lambda',\mu')$ runs over all bipartitions of $n+1$, which can be obtained from $(\lambda,\mu)$ by addition of a single box to $\lambda$ or $\mu$. 
	\end{enumerate}
	Visually speaking both procedures correspond to picking out a vertex $(\lambda, \mu)$ in the branching graph of all bipartitions (see below) and either taking the direct sum of the representations corresponding to its parent or child vertices. 
\end{theo}

\begin{beis}
	The first four rows of the branching graph, which contain the labels of the irreducible representations of $W(B_n)$ for $n=0,1,2,3$,  look as follows:
	
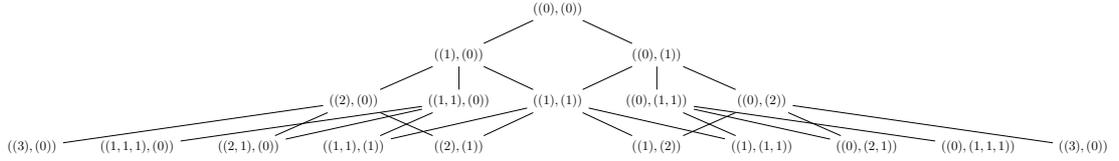
\begin{center}
		\begin{tikzpicture}[baseline= (a).base]
			\node[scale=.45] (a) at (0,0){
				\begin{tikzcd}
					&                 &               &                                                                           &                                                                              & {((0),(0))} \arrow[ld, no head] \arrow[rd, no head]                                           &                                                                              &                                                                           &               &                 &             \\
					&                 &               &                                                                           & {((1),(0))} \arrow[ld, no head] \arrow[d, no head] \arrow[rd, no head]       &                                                                                               & {((0),(1))} \arrow[ld, no head] \arrow[d, no head] \arrow[rd, no head]       &                                                                           &               &                 &             \\
					&                 &               & {((2),(0))} \arrow[llld, no head] \arrow[ld, no head] \arrow[rd, no head] & {((1,1),(0))} \arrow[lld, no head] \arrow[llld, no head] \arrow[ld, no head] & {((1),(1))} \arrow[ld, no head] \arrow[rd, no head] \arrow[rrd, no head] \arrow[lld, no head] & {((0),(1,1))} \arrow[rd, no head] \arrow[rrrd, no head] \arrow[rrd, no head] & {((0),(2))} \arrow[rd, no head] \arrow[rrrd, no head] \arrow[ld, no head] &               &                 &             \\
					{((3),(0))} & {((1,1,1),(0))} & {((2,1),(0))} & {((1,1),(1))}                                                             & {((2),(1))}                                                                  &                                                                                               & {((1),(2))}                                                                  & {((1),(1,1))}                                                             & {((0),(2,1))} & {((0),(1,1,1))} & {((3),(0))}
				\end{tikzcd}
			};
		\end{tikzpicture}
	\end{center}
%	\begin{center}
%		\begin{forest}
%			[{$(\varnothing, \varnothing)$}
%			[{$(\ydiagram{1}, \varnothing)$}, name=T
%			[{$(\ydiagram{2}, \varnothing)$}
%			%[{$(\ydiagram{3}, \varnothing)$}]
%			%[{$(\ydiagram{2,1}, \varnothing)$}, name=A]
%			]
%			[{$(\ydiagram{1,1}, \varnothing)$}, name=B
%			%[{$(\ydiagram{1,1,1}, \varnothing)$}]
%			%[{$(\ydiagram{1,1}, \ydiagram{1})$}]
%			]
%			]
%			[,phantom
%			[{$(\ydiagram{1}, \ydiagram{1})$}, name=S]
%			]
%			[{$(\varnothing,\ydiagram{1})$},name=R
%			[{$(\varnothing, \ydiagram{1,1})$}]
%			[{$(\varnothing, \ydiagram{2})$}]
%			]
%			]
%			%\draw (B)--(A);
%			\draw (R)--(S);
%			\draw (T)--(S);
%		\end{forest}
%	\end{center}
	%It is left to the reader as an exercise to copy this page and add the next row by hand.
	 Because the group $W(B_0)$ is trivial, the category of finite dimensional representations of $W(B_0)$ is just $\kVect_{\C}$. In particular the functor $\RES_0^n$ sends any representation of $W(B_n)$ to its underlying vector space. Using \ref{restriction gives decompositon} we see that
	\begin{align*}
		\RES_0^2(S(\ydiagram{1}, \ydiagram{1}))&=\RES_0^1(\RES_1^2(S(\ydiagram{1}, \ydiagram{1}))) \\
		&\cong\RES_0^1(S(\ydiagram{1}, \varnothing)\oplus S(\varnothing, \ydiagram{1})) \\
		&\cong S(\varnothing, \varnothing)\oplus S(\varnothing, \varnothing)\cong\C^2.
	\end{align*}
	The same argument shows that the other four representations of $W(B_2)$ are $1$-dimensional, because there is a unique path from each of their corresponding bipartitions to $S(\varnothing, \varnothing)$. In particular we recover our concrete example \Cref{n=2 irreducible representations} from the theory. The same procedure allows us to calculate the dimensions of the ten irreducible representations of $W(B_3)$. Hence the third row gives the following Artin--Wedderburn decomposition of the semisimple algebra $\C \! W(B_3)$:
	\[
	\dim \C\!  W(B_3) = 2^3 \cdot 3! = 48 = 1^2 + 1^2 + 2^2 + 3^2 + 3^2 + 3^2 + 3^2 + 2^2 + 1^2 + 1^2.
	\]	
\end{beis}
The method used in the last remark to determine the dimension of $S(\lambda, \mu)$ works in general and allows us to give a combinatorial description using a type $B$ version of ``standard tableaux'' in the following corollary. 
\begin{koro} \label{type B standard tableaux}
	The dimension of the irreducible  representation $S(\lambda, \mu)$ is the number of monotone paths from $(\varnothing, \varnothing)$ to $(\lambda, \mu)$ in the branching graph. Each path corresponds to a numbering of the boxes in $\lambda \sqcup \mu$ with $1,\ldots,n$, such that both the numberings of $\lambda$ and $\mu$ are increasing in the rows to the east and columns to the south.
\end{koro}
\begin{rema} \label{type A hook length}
	%From the last corollary one can conclude that a version of the famous hook length formula holds for type $B$. 
	Recall the \textbf{hook length} of a box $[j]$ in the partition $\lambda$ is the number of all boxes, which are right or below of $j$, including $j$ itself as indicated in the next diagram:
	\begin{center}
		\ydiagram[*(white) \bullet]
		{1+4,1+1,1+1}
		*{5,4,3}
	\end{center} 
		If we insert the hook length for every box into it we obtain the following picture:
	\begin{center}
		\ytableaushort 
		{76531, 5431, 321} *{5,4,3} 
	\end{center} 
	The famous type $A_{n-1}$ formula states the number of standard tablaux of size $\lambda$, which is the dimension of of the Specht module $S(\lambda)$ in type $A$, is given by
	\[
	\dim S(\lambda)
	=
	\frac{n!}{\prod_j(\text{hook length of $j$})} \,.
	\]
	For instance the dimension of the irreducible Specht module of $S_{12}$ corresponding to the above partition is
	\[
		\dim S(\lambda) = \frac{12!}{7 \cdot 6\cdot 5^{2} \cdot 4 \cdot 3^3 \cdot 2 \cdot 1^3}
		=
		2112.
	\]
\end{rema}
We have the following type $B$ version of this formula.
\begin{koro} \label{type B hook length}
	The type $B_n$ \textbf{hook length formula} is
	\[
	\dim S(\lambda,\mu) 
	=
	\frac{n!}{\prod_j(\text{hook length of $j$})} \,,
	\]
	where $j$ runs through all boxes of both $\lambda$ and $\mu$.
\end{koro}
\begin{proof}
	Let $k=|\lambda|$. Using the description of the dimension of $S(\lambda, \mu)$ from \Cref{type B standard tableaux} we see that
	\[
		\dim S(\lambda, \mu)
		=
		\binom{n}{k} \cdot\#\{\text{standard tableaux of $\lambda$}\}\cdot \#\{\text{standard tableaux of $\mu$} \},
	\]
	where the binomial coefficient counts the number of possible orders to delete all boxes of $\lambda$ and $\mu$ ignoring the order inside $\lambda$, respectively $\mu$. Using the type $A$ formula from \Cref{type A hook length} we calculate
	\begin{align*}
		\dim S(\lambda, \mu)
		&=
		\binom{n}{k} \cdot\#\{\text{standard tableaux of $\lambda$}\}\cdot \#\{\text{standard tableaux of $\mu$} \}\\
		&=\binom{n}{k} \cdot \frac{k!}{\prod(\text{hook lengths of $\lambda$})}\cdot \frac{(n-k)!}{\prod(\text{hook lengths of $\mu$})} \\
		&= \frac{n!}{k!(n-k)!}\cdot \frac{k!}{\prod(\text{hook lengths of $\lambda$})}\cdot \frac{(n-k)!}{\prod(\text{hook lengths of $\mu$})} \\
		&= \frac{n!}{\prod(\text{hook lengths of $\lambda$ and $\mu$})}. \qedhere
	\end{align*}
\end{proof}
	\begin{beis}
		For $S(\ydiagram{1}, \ydiagram{1})$ we conclude for the third time that
		\[
		\dim(S(\ydiagram{1}, \ydiagram{1}))
		=
		\frac{2!}{ 1 \cdot 1}
		=
		2 \,.
		\]
		The corresponding ``standard tableaux'' under \Cref{type B standard tableaux} are given by
		\[
			(\ytableaushort{1},\ytableaushort{2}) \qquad \text{and} \qquad (\ytableaushort{2},\ytableaushort{1}).
		\]
	\end{beis}
\begin{rema}
	Using the branching graph or the hook length formula we observe that
	\begin{equation*}
		\dim S(\lambda, \varnothing)=\dim S(\lambda) = \dim S(\varnothing, \lambda)
	\end{equation*}
	where $\lambda$ is a partition of $n$ and $S(\lambda)$ is the Specht module of the symmetric group $S_n$. This corresponds to the two natural ways to turn the irreducible representation $S(\lambda)$ of $S_n$ into a irreducible representation of $W(B_n)=S_2\wr S_n$. Namely we take the action of $S_n$ and let the additional generator $s_0$ act constantly by $1$ or $-1$.
	The first choice gives precisely $S(\lambda, \varnothing)$ and the second one gives $S(\varnothing, \lambda)$.
\end{rema}

\begin{beis}
	Let $\lambda$ a partition of $n$. The irreducible representations $S(\lambda, (1)\,)$ and $S((1), \lambda)$ of $W(B_{n+1})$ have dimension $(n+1)\cdot \dim S(\lambda)$. This can be concluded from the hook length formula or by counting standard tableaux. 
%	The representation $S(\lambda, (1)\,)$ can be constructed concretely from the irreducible Specht module $S(\lambda)$ of $S_n$ by defining 
%	\[
%		S(\lambda, \ydiagram{1}\,)=S(\lambda)\otimes \C^{n+1}
%	\]
%	as vector space. Let $e_1,e_2\ldots, e_{n+1}$ the standard basis of $\C^{n+1}$ and define the action of $W(B_n)$ on $S(\lambda, \ydiagram{1}\,)$ as
%	\begin{align*}
%		s_0 \cdot (v \otimes e_i) &= \begin{cases}
%			-e_1, & \text{if $i=1$,} \\
%			e_i, & \text{for $i=2,\ldots, n+1$} 
%		\end{cases} \\
%	s_1 \cdot (v \otimes e_i) &= \begin{cases}
%		e_2, & \text{if $i=1$,} \\
%		e_1, & \text{if $i=2$,} \\
%		e_i, & \text{for $i=2,\ldots, n+1$} 
%		\end{cases}
%	\end{align*}
\end{beis}
	The next example is not from the literature, but is motivated by the fact that the Specht modules $S(\lambda)$ of the symmetric group $S_n$ can be constructed as subrepresentations of $\C[x_1, \ldots, x_n]$, where $S_n$ acts by permuting the coordinates (see e.g.\ \cite{peel75}). We present a way to define a submodule of the algebra of Laurent polynomials associated to a bipartition.
\begin{beis}
 	Consider $\C[x_1^{\pm 1},\ldots, x_n^{\pm 1}]$, the ring of complex Laurent polynomials in $n$ variables. The group $W(B_n)$ acts by on $\C[x_1^{\pm 1},\ldots, x_n^{\pm}]$ via the formulas
 	\[
 		s_i \cdot x_j = \begin{cases}
 			x_{i+1}, & \text{for } j=i, \\
 			x_{i},&  \text{for } j=i+1\\
 			x_j, & \text{otherwise.}
 		\end{cases} \, , \quad
 		s_0 \cdot x_j = \begin{cases}
 			x_{1}^{-1}, & \text{for } j=1, \\
 			x_{i},&  \text{otherwise.} 
 		\end{cases}
 	\]
 	by algebra homomorphisms. Consider for $i\in\{1,\ldots, n\}$ the elements
 	\[
 		x_i^{+}\coloneqq x_i + x_i^{-1}, \quad x_i^{-}\coloneqq x_i - x_i^{-1}.
 	\]
 	Let $(\lambda, \mu)$ of $n$ a bipartition of $n$ and consider the subrepresentation of $\C[x_1^{\pm 1},\ldots, x_n^{\pm 1}]$ generated by the element
 	\[
 		\prod_{(i,j)\in C_{\lambda}} (x_i^{+}-x_j^{+}) \cdot \prod_{(i, j)\in C_{\mu}}(x_i^{-}-x_j^{-})
 	\]
 	where $C_{\lambda}$ (respectively $C_{\mu}$) denotes the set of all unordered pairs $\{i\neq j\}$ of numbers, which are in the same column in $\lambda$ (respectively $\mu$) in the standard numbering $T_{\lambda, \mu}$ of $(\lambda, \mu)$. We expect that this subrepresentation is closely related to the Specht module $S(\lambda, \mu)$.
\end{beis}

The next family of irreducible representations has funny dimensions and their definition finishes this section.
\begin{defi} \label{special irreducible representations}
	For $0\leq k\leq n $ we define the $k$th \textbf{binomial representation} $S(k)$ of $W(B_n)$ as
	\[
		S(k)\coloneqq S((k),(n-k)).
	\]
	Note that by the hook length formula \ref{type B hook length}
	\begin{equation*}
		\dim S(k)=\frac{n!}{k!(n-k)!}=\binom{n}{k},
	\end{equation*}
	which justifies the name.
\end{defi}
%\begin{beis}
%	The representation $S(1)=S((1),(n-1))$ is $n$-dimensional and isomorphic as a $W(B_n)$-module to $\h^{*}(B_n)$, where $\h(B_n)\subseteq \so_{2n+1}$ is any Cartan subalgebra.
%\end{beis}
%\textbf{Conjecture:} These representations are precisely given by $\bigwedge^kS((1),(n-1))$.
Observe that if we take just the matrix factors of the semisimple algebra $\C \! W(B_n)$, which correspond to the binomial representations $S(k)$, i.e.\ the whole isotypical components, we obtain an algebra of dimension
\[
	\sum_{k=0}^n {n \choose k}^2 = {{2n} \choose n}.
\]
This algebra -- which turns out to be the Temperley--Lieb algebra of type $B_n$ -- is the main topic of the next section. The dimension ${{2n} \choose n}$ has then a diagrammatic meaning.

%\begin{beis}
%	There are two non-isomorphic one-dimensional binomial representations of $W(B_n)$, namely $S(0)$ and $S(n)$. The corresponding idempotents are given by
%	\begin{align*}
%		e_n&=\frac{1}{2^n n!}\left( \sum_{w\in W(D_n)}w + \sum_{w\in W(D_n)}ws_0 \right)\\
%		e_0&=\frac{1}{2^n n!}\left(\sum_{w\in W(D_n)}w - \sum_{w\in W(D_n)}ws_0\right).
%	\end{align*}
%\end{beis}

\section{Temperley--Lieb algebras of type \texorpdfstring{$B$}{B} and \texorpdfstring{$D$}{D}} \label{Chapter 1 section 4}

This section introduces one of the most important objects for this thesis, the Temperley--Lieb algebra $\TL(B_n)$ of type $B$ and its subalgebra $\TL(D_n)$, the Temperley--Lieb algebra of type $D$. 
Besides preparing the notation for the next two chapters, we intend to make \cite{green1998} more accessible by adding more examples. We expect that the reader is already familiar with the usual type $A$ Temperley--Lieb algebra $\TL(A_{n-1})\coloneqq\TL_n$. The basics on this type $A$ Temperley--Lieb algebra can be found for instance in \cite{flath95}, \cite{kauffman1994} or the author's Bachelor thesis \cite{wojciechowski19}, which rephrases the first two sources. We start with the most important definition.
Alternatively \Cref{Chapter 2 section 1} can be used as an introduction to the type $A$, however there we mostly focus on the representation theory aspects involving $\Ug_q(\sl2)$.

\begin{defi} \label{Definition Temperley--Lieb of Type B}
	Let $n\geq1$, $k$ a commutative ring and $\delta\in k$ a scalar. The (specialized) \textbf{type $B$ Temperley--Lieb algebra} $\TL(B_n)\coloneqq\TL(B_n,\delta)$ is the (associative, unital) $k$-algebra with generators $s_0, U_1, \ldots, U_{n-1}$ which are subject to the \textbf{Type $A$ Temperley--Lieb relations}
	\begin{enumerate}[label = {(TL\arabic*)}, align=left]
		\item \label{TL1 relation}
		$U_i^2=\delta U_i$ for all $i=1,\ldots, n-1$,
		\item \label{TL2 relation}
		$U_i U_{i+1} U_i=U_i$ and $U_{i+1} U_{i} U_{i+1}=U_{i+1}$ for all $i=1,\ldots,n-1$,
		\item \label{TL3 relation}
		$U_iU_j=U_jU_i$ for all $|i-j| \geq 2$,
	\end{enumerate}
	and the additional, specialized \textbf{Type $B$ Temperley--Lieb relations}
	\begin{enumerate}[label = {(BTL\arabic*)}, align=left]
		\item $s_0^2=1$, \label{TLdot1 relation}
		\item $U_1 s_0  U_1=0$, \label{TLdot2 relation}
		\item $s_0  U_i = U_i s_0$ for all $i=2\ldots,n-1$ \label{TLdot3 relation}
	\end{enumerate}
	involving the generator $s_0$. 
	We define $\TL(D_n)$, the \textbf{Temperley--Lieb algebra of type} $D$ as the subalgebra of $\TL(B_n)$ generated by $U_0\coloneqq s_0U_1s_0$ and $U_i$ for $1\leq i \leq n$.
\end{defi}
The algebra $\TL(B_n)$ is a special case of the family of algebras, which were defined and systematically treated in \cite{green1998}. Our version is a specialization of their algebra, where their version of $\TL(B_n)$ is defined by the same generators and relations, however \ref{TLdot2 relation} is replaced by $U_1 s_0  U_1=\delta'U_1$ for any $\delta'\in k$. There Green proved a diagrammatic description for $\TL(B_n)$ in terms of the ``blob algebra'', which appeared before in \cite{martin94}. The next definition introduces this diagram algebra description, which we denote by $V_n$.

\begin{defi} \label{diagrammatic definition type B TL}
	Consider the free $k$-module $V_n$ with basis consisting of \textbf{dotted Temperley--Lieb diagrams}. These are Temperley--Lieb $(n,n)$-diagrams (i.e.\ crossingless matchings of $n+n$ points placed equidistantly in two parallel rows) which are decorated with at most one dot per strand, such that every dotted strand can be reached from the left side of the diagram by a curve, which doesn't cross any strand. The vector space $V_n$ can be equipped with a bilinear, associative multiplication which is given on the diagram basis by vertical stacking and applying the local rules
	\[
	\cbox{
		\begin{tikzpicture}[tldiagram, yscale=2/3, xscale=2/3]
			% dots
			%\drawdots{0}{0}{0}
			%\drawdots{2}{2}{2}
			% the line
			\draw \tlcoord{0}{0} \lineup \capright \cupright \lineup;
		\end{tikzpicture}
	}
	\quad
	=
	\quad
	\cbox{
		\begin{tikzpicture}[tldiagram, yscale=2/3, xscale=2/3]
			% dots
			%\drawdots{0}{0}{0}
			%\drawdots{2}{0}{0}
			% the line
			\draw \tlcoord{0}{0} \lineup \lineup;
		\end{tikzpicture}
	}
	\quad
	=
	\quad
	\cbox{
		\begin{tikzpicture}[tldiagram, yscale=2/3, xscale=2/3]
			% dots
			%\drawdots{0}{2}{2}
			%\drawdots{2}{0}{0}
			% the line
			\draw \tlcoord{2}{0} \linedown \cupright \capright \linedown;
		\end{tikzpicture}
	}
	\, ,  \quad \cbox{
		\begin{tikzpicture}[tldiagram, yscale=2/3]
			%dots
			%\drawdots{0}{0}{0}
			%\drawdots{2}{0}{0}
			%nodes
			%\node at \tlcoord{0}{0} [anchor = north] {$\scriptstyle 1$};
			%\node at \tlcoord{2}{0} [anchor = south] {$\scriptstyle 1$};
			
			%lines
			\draw \tlcoord{0}{0} \dlineup \dlineup;
			%big dots
			%\biggerdrawdots{0.66}{0}{0}
			%\biggerdrawdots{1.33}{0}{0}
		\end{tikzpicture}
	} \quad = \quad \cbox{
		\begin{tikzpicture}[tldiagram,yscale=2/3]
			%dots
			%\drawdots{0}{0}{0}
			%\drawdots{2}{0}{0}
			%nodes
			%\node at \tlcoord{0}{0} [anchor = north] {$\scriptstyle 1$};
			%\node at \tlcoord{2}{0} [anchor = south] {$\scriptstyle 1$};
			%lines
			\draw \tlcoord{0}{0} \lineup \lineup;
		\end{tikzpicture}
	} \, , \quad
	\cbox{
		\begin{tikzpicture}[tldiagram]
			\draw \tlcoord{0}{0} \capright \cupleft;
		\end{tikzpicture}
	}
	\,
	=
	\,
	\delta \, , \quad
	\cbox{
		\begin{tikzpicture}[tldiagram]
			\draw \tlcoord{0}{0} \dcapright \cupleft;
		\end{tikzpicture}
	}
	\,
	=
	\,
	0 
\]
to obtain a scalar multiple of a dotted Temperley--Lieb diagram. See \Cref{example temperley-lieb inclusions n=2} and \Cref{example temperley lieb inclusions n=3} for illustrations of the diagram bases of $V_2$ respectively $V_3$.
\end{defi}

\begin{theo} \label{Temperley--Lieb algebra of type B is diagram algebra}
	The assignments
	\begin{align*}
		\TL(B_n) &\rightarrow V_n \\
		U_j &\mapsto \cbox{
			\begin{tikzpicture}[tldiagram]
				% dots\textit{
				%\drawdots{0}{0}{0}
				%\drawdots{0}{2}{5}
				%\drawdots{0}{7}{7}
				%\drawdots{1}{0}{0}
				%\drawdots{1}{2}{5}
				%\drawdots{1}{7}{7}
				% lines
				\draw \tlcoord{0}{0} \lineup;
				\draw \tlcoord{0}{2} \lineup;
				\draw \tlcoord{0}{5} \lineup;
				\draw \tlcoord{0}{7} \lineup;
				% dots
				\makecdots{0}{1}
				\makecdots{0}{6}
				% cap and cup
				\draw \tlcoord{0}{3} \capright;
				\draw \tlcoord{1}{3} \cupright;
				% nodes
				\node at \tlcoord{0}{3} [anchor = north] {$\scriptstyle j$};
				\node at \tlcoord{0}{4} [anchor = north] {$\scriptstyle j+1$};
				\node at \tlcoord{1}{3} [anchor = south] {$\scriptstyle j$};
				\node at \tlcoord{1}{4} [anchor = south] {$\scriptstyle j+1$};
			\end{tikzpicture}
		} \, , \\
		s_0 &\mapsto \cbox{
			\begin{tikzpicture}[tldiagram]
				% dots\textit{
				%\drawdots{0}{0}{1}
				%\drawdots{0}{3}{3}
				%\drawdots{1}{0}{1}
				%\drawdots{1}{3}{3}
				% lines
				\draw \tlcoord{0}{0} \dlineup;
				\draw \tlcoord{0}{1} \lineup;
				\draw \tlcoord{0}{3} \lineup;
				% dots
				\makecdots{0}{2}
				%\biggerdrawdots{0.5}{0}{0}
				% nodes
				%\node at \tlcoord{0}{0} [anchor = north] {$\scriptstyle 1$};
				%\node at \tlcoord{0}{1} [anchor = north] {$\scriptstyle 2$};
				%\node at \tlcoord{1}{0} [anchor = south] {$\scriptstyle 1$};
				%\node at \tlcoord{1}{1} [anchor = south] {$\scriptstyle 2$};
			\end{tikzpicture}
		} \, , \\
	U_0 &\mapsto \cbox{
		\begin{tikzpicture}[tldiagram]
			% dots\textit{
			%\drawdots{0}{0}{0}
			%\drawdots{0}{2}{5}
			%\drawdots{0}{7}{7}
			%\drawdots{1}{0}{0}
			%\drawdots{1}{2}{5}
			%\drawdots{1}{7}{7}
			% lines
			\draw \tlcoord{0}{0} \dcapright;
			\draw \tlcoord{1}{0} \dcupright;
			\draw \tlcoord{0}{2} \lineup;
			\draw \tlcoord{0}{4} \lineup;
			% dots
			\makecdots{0}{3}
			%\makecdots{0}{6}
			% cap and cup
			%\draw \tlcoord{0}{3} \lineup;
			%\draw \tlcoord{1}{3} \cupright;
			% nodes
			%\node at \tlcoord{0}{3} [anchor = north] {$\scriptstyle j$};
			%\node at \tlcoord{0}{4} [anchor = north] {$\scriptstyle j+1$};
			%\node at \tlcoord{1}{3} [anchor = south] {$\scriptstyle j$};
			%\node at \tlcoord{1}{4} [anchor = south] {$\scriptstyle j+1$};
		\end{tikzpicture}
	}
	\end{align*}
	extend to a well-defined $k$-algebra isomorphism $\TL(B_n)\xrightarrow{\cong} V_n$.
\end{theo}
\begin{proof}
	See \cite[Theorem 4.1]{green1998}, where we set the value of the dotted circle to $\delta'=0$. 
\end{proof}
We will from now on identify $\TL(B_n)$ with the diagram algebra $V_n$ from \Cref{diagrammatic definition type B TL}.
As a consequence we obtain a version of \Cref{lemma inclusion of d in b} for Temperley--Lieb algebras.
\begin{koro} \label{inclusions of Temperley--Lieb algebras}
	Let $n\geq 2$. The assignments
		\[
		\begin{tikzcd}[row sep=tiny]
			\phantom{U_i}\TL(A_{n-1}) \arrow[r, hook]   & \TL(D_n) \arrow[r, hook]        & \TL(B_n)\phantom{U_i}    \\
			\phantom{\TL(A_{n-1})}U_i \arrow[r, mapsto] & U_i \arrow[r, mapsto]  & U_i\phantom{\TL(A_{n-1})}       \\
			& U_0  \arrow[r, mapsto] & s_0U_1s_0.
		\end{tikzcd}
		\]
		extend to well-defined $k$-algebra inclusions.
		Under these inclusions $\TL(A_{n-1})$ is spanned by all usual Temperley--Lieb diagrams and  $\TL(D_n)$ is spanned by all dotted Temperley--Lieb diagrams, which are decorated with an even number of dots.
\end{koro}
\begin{proof}
	Since we defined $\TL(D_n)$ as a subalgebra of $\TL(B_n)$, the inclusions are clear from \Cref{Temperley--Lieb algebra of type B is diagram algebra}. The second statement is \cite[Theorem 4.2]{green1998}. The subalgebra $\TL(D_n)$ as we defined it has basis consisting of the diagrams of type ii) from that theorem.
\end{proof}
\begin{rema} \label{explanation green type D vs our type D}
	As a consequence of \cite[Theorem 4.2]{green1998} we have the following description for $\TL(D_n)$ in terms of generators and relations. The generators are $\{U_i \mid 0\leq i\leq n-1 \}$ and the relations are in addition to the type $A$ Temperley--Lieb relations \ref{TL1 relation}, \ref{TL2 relation} and \ref{TL3 relation} given by
	\begin{enumerate}[label = {(DTL\arabic*)}, align=left]
		\item $U_0^2=\delta U_0$, \label{DTLdot1 relation}
		\item $U_1 U_0= 0 = U_0  U_1$, \label{DTLdot2 relation}
		\item $U_0  U_j = U_j U_0$ for all $j=2\ldots,n-1$. \label{DTLdot3 relation}
	\end{enumerate}
	However note that our algebra $\TL(D_n)$ is not the Temperley--Lieb algebra of type $D$ from \cite{green1998}, which we will refer to as $\TL^{\opn{Gr}}(D_n)$, but instead its quotient by the relation \ref{DTLdot2 relation}, which comes from the fact that we chose the value of the dotted circle to be zero. In the algebra $\TL^{\opn{Gr}}(D_n)$ the generators $U_0$ and $U_1$ commute and do not cancel each other. As consequence the exceptional Dynkin diagram isomorphisms $D_2=A_1\times A_1$ and $D_3=A_3$ become
	\[
		\TL(A_1)\times \TL(A_1)\cong\TL^{\opn{Gr}}(D_2)\twoheadrightarrow\TL(D_2)
	\]
	and
	\[
		\TL(A_3)\cong\TL^{\opn{Gr}}(D_3)\twoheadrightarrow \TL(D_3),
	\]
	where the quotient maps are not isomorphisms. The missing part of $\TL^{\opn{Gr}}(D_n)$ in $\TL(D_n)$ is the two-sided ideal generated by $U_0U_1$. This ideal has by \cite[Theorem 4.2]{green1998} a $\C(q)$-basis consisting of so-called type i) diagrams. Type i) diagrams are usual Temperley--Lieb $(n,n)$-diagrams with an additional dotted circle on the very left, excluding the identity diagram. The dotted circle marks the contribution of $U_0U_1=U_1U_0$ in basis elements and Green's diagrammatics make it absorb possible dots on neighbored arcs:
	\[
		\cbox{
		\begin{tikzpicture}[tldiagram]
			\draw \tlcoord{0}{0} \onedot \capright \cupleft;
			\draw \tlcoord{1}{0} \dcupright; 
		\end{tikzpicture}
	} \quad = \quad 
	\cbox{
		\begin{tikzpicture}[tldiagram]
			\draw \tlcoord{0}{0} \onedot \capright \cupleft;
				\draw \tlcoord{1}{0} \cupright;
		\end{tikzpicture}
	} \, , \quad \cbox{
	\begin{tikzpicture}[tldiagram]
		\draw \tlcoord{0}{0} \onedot \capright \cupleft;
		\draw \tlcoord{-1}{2} \lineup \onedot \lineup; 
	\end{tikzpicture}
} \quad = \quad 
\cbox{
\begin{tikzpicture}[tldiagram]
	\draw \tlcoord{0}{0} \onedot \capright \cupleft;
	\draw \tlcoord{-1}{2} \lineup \lineup; 
\end{tikzpicture}
} \, , \quad 
\cbox{
\begin{tikzpicture}[tldiagram]
	\draw \tlcoord{0}{0} \onedot \capright \cupleft;
	\draw \tlcoord{-1}{0} \dcapright; 
\end{tikzpicture}
} \quad = \quad 
\cbox{
\begin{tikzpicture}[tldiagram]
	\draw \tlcoord{0}{0} \onedot \capright \cupleft;
	\draw \tlcoord{-1}{0} \capright;
\end{tikzpicture}
} 
	\]
	For a comparison between $\TL(D_n)$ and $\TL^{\opn{Gr}}(D_n)$ see \Cref{example temperley lieb inclusions n=3} and \Cref{example additional diagrams} thereafter, which shows the additional diagrams.
	Finally note that the $\TL^{\opn{Gr}}(D_n)$, the true Temperley--Lieb algebra of type $D$, is not a subalgebra of $\TL(B_n)$. In contrast to the algebras $\TL(D_n)$ and $\TL(B_n)$, which were categorified in \cite{stroppel2016} by projective functors on parabolic category $\cO$ for $\so_{2n}$, it lacks currently a topological/categorified interpretation. Also the constructions of Jones--Wenzl projectors for $\TL(D_n)$ which we will do later on really require that we work in the quotient by $U_0U_1$.
\end{rema}
%\begin{rema}
%	For now the reader should just accept relation \ref{TLdot2 relation}. 
%	Later we will see how one can come up with these relations. 
%	Also this relation will get more clear, once we discuss the categorification of the Temperley--Lieb algebra from \cite{stroppel2016}. This categorification will use the arc algebra of type $B$. Morphisms between Temperley--Lieb diagrams will come from glueing, closing up and orienting diagrams. A black dot is a condition which reverses a direction. Since there is no way to orient a circle decorated with an orientating reversing dot, the dotted circle has no non-trivial endomorphisms and has to be zero. 
%\end{rema}
%As a consequence of the last remark we get an upper bound for the dimension of Temperley--Lieb of type $B$ and $D$.
%\begin{koro}[]
%	We have
%	\[
%	\dim \TL(B_n)\leq{{2n} \choose n}, \quad \dim \TL(D_n)\leq{{2n} \choose n}/2
%	\]
%\end{koro}
%We will show later that these are equalities. 
The next examples show bases for small Temperley--Lieb algebras of type $B$ and $D$. 
\begin{beis} \label{example temperley-lieb inclusions n=2}
	The algebra inclusions $\TL(A_1)\subseteq \TL(D_2)\subseteq \TL(B_2)\cong V_2$ are
	\[
	\{1,U_1\}\subseteq \{ 1,U_1,U_0\coloneqq s_0U_1s_0\} \subseteq \{ 1,s_0, U_1,s_0U_1, U_1s_0, s_0U_1s_0\}.
	\]
	In terms of diagram bases the inclusions look as follows
	\[
	\left\{
	\cbox{
		\begin{tikzpicture}[tldiagram, xscale=0.6, yscale=0.6]
			\draw \tlcoord{0}{0} \lineup;
			\draw \tlcoord{0}{1} \lineup;
		\end{tikzpicture}
	} \, , \, \cbox{
		\begin{tikzpicture}[tldiagram, xscale=0.6, yscale=0.6]
			\draw \tlcoord{0}{0} \capright;
			\draw \tlcoord{1}{0} \cupright;
		\end{tikzpicture}
	}
	\right\} \subseteq 
	\left\{
	\cbox{
		\begin{tikzpicture}[tldiagram, xscale=0.6, yscale=0.6]
			\draw \tlcoord{0}{0} \lineup;
			\draw \tlcoord{0}{1} \lineup;
		\end{tikzpicture}
	} \, , \, \cbox{
		\begin{tikzpicture}[tldiagram, xscale=0.6, yscale=0.6]
			\draw \tlcoord{0}{0} \capright;
			\draw \tlcoord{1}{0} \cupright;
		\end{tikzpicture}
	} \, ,
	\, \cbox{
		\begin{tikzpicture}[tldiagram, xscale=0.6, yscale=0.6]
			\draw \tlcoord{0}{0} \dcapright;
			\draw \tlcoord{1}{0} \dcupright;
		\end{tikzpicture}
	}
	\right\} \subseteq 
	\left\{
	\cbox{
		\begin{tikzpicture}[tldiagram, xscale=0.6, yscale=0.6]
			\draw \tlcoord{0}{0} \lineup;
			\draw \tlcoord{0}{1} \lineup;
		\end{tikzpicture}
	} \, , \,
	\cbox{
		\begin{tikzpicture}[tldiagram, xscale=0.6, yscale=0.6]
			\draw \tlcoord{0}{0} \dlineup;
			\draw \tlcoord{0}{1} \lineup;
		\end{tikzpicture}
	} \, , \, \cbox{
		\begin{tikzpicture}[tldiagram, xscale=0.6, yscale=0.6]
			\draw \tlcoord{0}{0} \capright;
			\draw \tlcoord{1}{0} \cupright;
		\end{tikzpicture}
	} \, , \, \cbox{
		\begin{tikzpicture}[tldiagram, xscale=0.6, yscale=0.6]
			\draw \tlcoord{0}{0} \capright;
			\draw \tlcoord{1}{0} \dcupright;
		\end{tikzpicture}
	} \, , \, \cbox{
		\begin{tikzpicture}[tldiagram, xscale=0.6, yscale=0.6]
			\draw \tlcoord{0}{0} \dcapright;
			\draw \tlcoord{1}{0} \cupright;
		\end{tikzpicture}
	} \, ,	\, \cbox{
		\begin{tikzpicture}[tldiagram, xscale=0.6, yscale=0.6]
			\draw \tlcoord{0}{0} \dcapright;
			\draw \tlcoord{1}{0} \dcupright;
		\end{tikzpicture}
	}
	\right\}.
	\]
	Observe that $2 \dim \TL(D_2)=\dim \TL(B_2)$ and $3 \dim \TL(A_1)=\dim \TL(B_2)$. 
\end{beis}
Before we discuss the dimension of $\TL(D_n)$ and $\TL(B_n)$ for general $n$, let us consider the example $n=3$.

\begin{beis} \label{example temperley lieb inclusions n=3}
	The diagram basis of $\TL_3=\TL(A_2)$ is given by
	\[
		\cbox{
			\begin{tikzpicture}[tldiagram, xscale=2/3, yscale=2/3]
				% lines
				\draw \tlcoord{-1.5}{0} \lineup;
				\draw \tlcoord{-1.5}{1} \lineup;
				\draw \tlcoord{-1.5}{2} \lineup;
				% boxes
				%\draw \tlcoord{0}{0} \maketlboxgreen{3}{$3$};
				=
			\end{tikzpicture}
		} , \, \cbox{
		\begin{tikzpicture}[tldiagram, xscale=2/3, yscale=2/3]
			% lines
			\draw \tlcoord{0}{0} \capright;
			\draw \tlcoord{1}{0} \cupright;
			\draw \tlcoord{0}{2} \lineup; 
			% boxes
			%\draw \tlcoord{0}{0} \maketlboxgreen{3}{$3$};
			=
		\end{tikzpicture}
	} , \, \cbox{
	\begin{tikzpicture}[tldiagram, xscale=2/3, yscale=2/3 ]
		% lines
		\draw \tlcoord{0}{1} \capright;
		\draw \tlcoord{1}{1} \cupright;
		\draw \tlcoord{0}{0} \lineup; 
		% boxes
		%\draw \tlcoord{0}{0} \maketlboxgreen{3}{$3$};
		=
	\end{tikzpicture}
} , \,
	\cbox{
	\begin{tikzpicture}[tldiagram, xscale=2/3, yscale=2/3]
		% lines
		\draw \tlcoord{0}{0} \linewave{1}{2};
		\draw \tlcoord{0}{1} \capright;
		\draw \tlcoord{1}{1} \cupleft;
		% boxes
		%\draw \tlcoord{0}{0} \maketlboxgreen{3}{$3$};
		=
	\end{tikzpicture}
	} , \, \cbox{
	\begin{tikzpicture}[tldiagram, xscale=2/3, yscale=2/3]
		% lines
		\draw \tlcoord{0}{2} \linewave{1}{-2};
		\draw \tlcoord{0}{0} \capright;
		\draw \tlcoord{1}{2} \cupleft;
		% boxes
		%\draw \tlcoord{0}{0} \maketlboxgreen{3}{$3$};
		=
	\end{tikzpicture}
} .
	\]
	The diagram basis of $\TL(D_3)$ contains additionally the diagrams
	\begin{equation*}
			\cbox{
				\begin{tikzpicture}[tldiagram, xscale=2/3, yscale=2/3]
					% lines
					\draw \tlcoord{0}{0} \dcapright;
					\draw \tlcoord{1}{0} \dcupright;
					\draw \tlcoord{0}{2} \lineup; 
					% boxes
					%\draw \tlcoord{0}{0} \maketlboxgreen{3}{$3$};
					=
				\end{tikzpicture}
			} , \,
			\cbox{
				\begin{tikzpicture}[tldiagram, xscale=2/3, yscale=2/3]
					% lines
					\draw \tlcoord{0}{0} \dlinewave{1}{2};
					\draw \tlcoord{0}{1} \capright;
					\draw \tlcoord{1}{1} \dcupleft;
					% boxes
					%\draw \tlcoord{0}{0} \maketlboxgreen{3}{$3$};
					=
				\end{tikzpicture}
			} , \,
			\cbox{
				\begin{tikzpicture}[tldiagram, xscale=2/3, yscale=2/3]
					% lines
					\draw \tlcoord{0}{2} \dlinewave{1}{-2};
					\draw \tlcoord{0}{0} \dcapright;
					\draw \tlcoord{1}{2} \cupleft;
					% boxes
					%\draw \tlcoord{0}{0} \maketlboxgreen{3}{$3$};
					=
				\end{tikzpicture}
			}  , \,
			\cbox{
				\begin{tikzpicture}[tldiagram, xscale=2/3, yscale=2/3]
					% lines
					\draw \tlcoord{0}{0} \capright;
					\draw \tlcoord{1}{0} \dcupright;
					\draw \tlcoord{0}{2} \dlineup; 
					% boxes
					%\draw \tlcoord{0}{0} \maketlboxgreen{3}{$3$};
					=
				\end{tikzpicture}
			} , \,
			\cbox{
				\begin{tikzpicture}[tldiagram, xscale=2/3, yscale=2/3]
					% lines
					\draw \tlcoord{0}{0} \dcapright;
					\draw \tlcoord{1}{0} \cupright;
					\draw \tlcoord{0}{2} \dlineup; 
					% boxes
					%\draw \tlcoord{0}{0} \maketlboxgreen{3}{$3$};
					=
				\end{tikzpicture}
			}.
	\end{equation*}
	In order to get the diagram basis of $\TL(B_3)\cong V_3$ we need ten more diagrams, which are given by multiplying the diagram basis of $\TL(D_3)$ with $s_0$:
	\begin{align*}
		\cbox{
			\begin{tikzpicture}[tldiagram, xscale=2/3, yscale=2/3]
				% lines
				\draw \tlcoord{-1.5}{0} \dlineup;
				\draw \tlcoord{-1.5}{1} \lineup;
				\draw \tlcoord{-1.5}{2} \lineup;
				% boxes
				%\draw \tlcoord{0}{0} \maketlboxgreen{3}{$3$};
				=
			\end{tikzpicture}
		} &, \, \cbox{
			\begin{tikzpicture}[tldiagram, xscale=2/3, yscale=2/3]
				% lines
				\draw \tlcoord{0}{0} \capright;
				\draw \tlcoord{1}{0} \dcupright;
				\draw \tlcoord{0}{2} \lineup; 
				% boxes
				%\draw \tlcoord{0}{0} \maketlboxgreen{3}{$3$};
				=
			\end{tikzpicture}
		} , \, \cbox{
			\begin{tikzpicture}[tldiagram, xscale=2/3, yscale=2/3 ]
				% lines
				\draw \tlcoord{0}{1} \capright;
				\draw \tlcoord{1}{1} \cupright;
				\draw \tlcoord{0}{0} \dlineup; 
				% boxes
				%\draw \tlcoord{0}{0} \maketlboxgreen{3}{$3$};
				=
			\end{tikzpicture}
		} , \,
		\cbox{
			\begin{tikzpicture}[tldiagram, xscale=2/3, yscale=2/3]
				% lines
				\draw \tlcoord{0}{0} \linewave{1}{2};
				\draw \tlcoord{0}{1} \capright;
				\draw \tlcoord{1}{1} \dcupleft;
				% boxes
				%\draw \tlcoord{0}{0} \maketlboxgreen{3}{$3$};
				=
			\end{tikzpicture}
		} , \, \cbox{
			\begin{tikzpicture}[tldiagram, xscale=2/3, yscale=2/3]
				% lines
				\draw \tlcoord{0}{2} \dlinewave{1}{-2};
				\draw \tlcoord{0}{0} \capright;
				\draw \tlcoord{1}{2} \cupleft;
				% boxes
				%\draw \tlcoord{0}{0} \maketlboxgreen{3}{$3$};
				=
			\end{tikzpicture}
		}  , \, \\
		\cbox{
			\begin{tikzpicture}[tldiagram, xscale=2/3, yscale=2/3]
				% lines
				\draw \tlcoord{0}{0} \dcapright;
				\draw \tlcoord{1}{0} \cupright;
				\draw \tlcoord{0}{2} \lineup; 
				% boxes
				%\draw \tlcoord{0}{0} \maketlboxgreen{3}{$3$};
				=
			\end{tikzpicture}
		} &, \,
		\cbox{
			\begin{tikzpicture}[tldiagram, xscale=2/3, yscale=2/3]
				% lines
				\draw \tlcoord{0}{0} \dlinewave{1}{2};
				\draw \tlcoord{0}{1} \capright;
				\draw \tlcoord{1}{1} \cupleft;
				% boxes
				%\draw \tlcoord{0}{0} \maketlboxgreen{3}{$3$};
				=
			\end{tikzpicture}
		} , \,
		\cbox{
			\begin{tikzpicture}[tldiagram, xscale=2/3, yscale=2/3]
				% lines
				\draw \tlcoord{0}{2} \linewave{1}{-2};
				\draw \tlcoord{0}{0} \dcapright;
				\draw \tlcoord{1}{2} \cupleft;
				% boxes
				%\draw \tlcoord{0}{0} \maketlboxgreen{3}{$3$};
				=
			\end{tikzpicture}
		}  , \,
		\cbox{
			\begin{tikzpicture}[tldiagram, xscale=2/3, yscale=2/3]
				% lines
				\draw \tlcoord{0}{0} \capright;
				\draw \tlcoord{1}{0} \cupright;
				\draw \tlcoord{0}{2} \dlineup; 
				% boxes
				%\draw \tlcoord{0}{0} \maketlboxgreen{3}{$3$};
				=
			\end{tikzpicture}
		} , \,
		\cbox{
			\begin{tikzpicture}[tldiagram, xscale=2/3, yscale=2/3]
				% lines
				\draw \tlcoord{0}{0} \dcapright;
				\draw \tlcoord{1}{0} \dcupright;
				\draw \tlcoord{0}{2} \dlineup; 
				% boxes
				%\draw \tlcoord{0}{0} \maketlboxgreen{3}{$3$};
				=
			\end{tikzpicture}
		}.
	\end{align*}
	Observe that $2 \dim \TL(D_2)=\dim \TL(B_2)$ and $4 \dim \TL(A_2)=\dim \TL(B_3)$.
\end{beis}
\begin{beis} \label{example additional diagrams}
	If one works with the Temperley--Lieb algebra of type $D$ from \cite{green1998}, one obtains the following four additional diagrams, in addition to the first ten diagrams from the last example:
	\begin{gather*}
		U_0U_1 \, = \, \cbox{
			\begin{tikzpicture}[tldiagram, xscale=2/3, yscale=2/3]
				% lines
				\draw \tlcoord{0}{0} \capright;
				\draw \tlcoord{2}{0} \cupright;
				\draw \tlcoord{1}{0} \onedot \capright \cupleft;
				\draw \tlcoord{0}{2} \lineup \lineup; 
				% boxes
				%\draw \tlcoord{0}{0} \maketlboxgreen{3}{$3$};
				=
			\end{tikzpicture}
		} \, , \quad U_2U_0U_1U_2 \, = \, \cbox{
			\begin{tikzpicture}[tldiagram, xscale=2/3, yscale=2/3 ]
				% lines
				\draw \tlcoord{0.5}{-1.5} \onedot \capright \cupleft;
				\draw \tlcoord{0}{1} \capright;
				\draw \tlcoord{1}{1} \cupright;
				\draw \tlcoord{0}{0} \lineup; 
				% boxes
				%\draw \tlcoord{0}{0} \maketlboxgreen{3}{$3$};
				=
			\end{tikzpicture}
		} \, , \\[1em] U_0U_1U_2
		\cbox{
			\begin{tikzpicture}[tldiagram, xscale=2/3, yscale=2/3]
				% lines
				\draw \tlcoord{0.5}{-1.5} \onedot \capright \cupleft;
				\draw \tlcoord{0}{0} \linewave{1}{2};
				\draw \tlcoord{0}{1} \capright;
				\draw \tlcoord{1}{1} \cupleft;
				% boxes
				%\draw \tlcoord{0}{0} \maketlboxgreen{3}{$3$};
				=
			\end{tikzpicture}
		} , \quad U_2U_0U_1 \, = \, \cbox{
			\begin{tikzpicture}[tldiagram, xscale=2/3, yscale=2/3]
				% lines
				\draw \tlcoord{0.5}{-1.5} \onedot \capright \cupleft;
				\draw \tlcoord{0}{2} \linewave{1}{-2};
				\draw \tlcoord{0}{0} \capright;
				\draw \tlcoord{1}{2} \cupleft;
				% boxes
				%\draw \tlcoord{0}{0} \maketlboxgreen{3}{$3$};
				=
			\end{tikzpicture}
		} \, .
	\end{gather*}
	This makes sense since one can identify $\TL^{\opn{Gr}}(D_3)$ with $\TL_4=\TL(A_3)$ by mapping $U_1 \mapsto U_1, U_2 \mapsto U_2, U_0 \mapsto U_3$. Under this identification these four diagrams, which are missing in $\TL(D_3)$, become
	\begin{gather*}
		U_1U_3 \quad = \quad 
		\cbox{
			\begin{tikzpicture}[tldiagram, xscale=2/3, yscale=2/3]
				%dots
				%\drawdots{0}{0}{3}
				%\drawdots{1}{0}{3}
				%lines
				\draw \tlcoord{0}{0} \capright;
				\draw \tlcoord{1}{0} \cupright;
				\draw \tlcoord{0}{2} \capright;
				\draw \tlcoord{1}{2} \cupright;
			\end{tikzpicture}
		}
		\,, \quad U_2U_1U_3U_2 \quad = \quad 
		\cbox{
			\begin{tikzpicture}[tldiagram, xscale=2/3, yscale=2/3]
				%dots
				%\drawdots{0}{0}{3}
				%\drawdots{2.4}{0}{3}
				%lines
				\draw \tlcoord{0}{0} \xcapright{3};
				\draw \tlcoord{0}{1} \capright;
				\draw \tlcoord{2.4}{0} \xcupright{3};
				\draw \tlcoord{2.4}{1} \cupright;
			\end{tikzpicture}
		}
		\, , \\[1em]
		U_1U_3U_2 \quad = \quad 
		\cbox{
			\begin{tikzpicture}[tldiagram, xscale=2/3, yscale=2/3]
				%dots
				%\drawdots{0}{0}{3}
				%\drawdots{1.4}{0}{3}
				%lines
				\draw \tlcoord{1.4}{0} \cupright;
				\draw \tlcoord{1.4}{2} \cupright;
				\draw \tlcoord{0}{0} \xcapright{3};
				\draw \tlcoord{0}{1} \capright;
			\end{tikzpicture}
		}
		\,, \quad U_2U_1U_3 \quad = \quad 
		\cbox{
			\begin{tikzpicture}[tldiagram, xscale=2/3, yscale=2/3]
				%dots
				%\drawdots{0}{0}{3}
				%\drawdots{1.4}{0}{3}
				%lines
				\draw \tlcoord{0}{0} \capright;
				\draw \tlcoord{0}{2} \capright;
				\draw \tlcoord{1.4}{0} \xcupright{3};
				\draw \tlcoord{1.4}{1} \cupright;
			\end{tikzpicture}
		} \, .
	\end{gather*}
	So now we understand what we miss out on in our $10$-dimensional algebra $\TL(D_3)$. In particular observe that we cannot obtain the diagram basis of $\TL^{\opn{Gr}}(D_3)$ by taking the diagram basis of $\TL^{\opn{Gr}}(D_2)$ and adding the rightmost strand or multiplying $U_2$ from the left and right with elements of $\TL^{\opn{Gr}}(D_2)$. This implies that the inductive formula for the Jones--Wenzl projectors of type $D$, which mimics the recursion rule for the traditional Jones--Wenzl projectors, only works for $\TL(D_n)$ and not for $\TL^{\opn{Gr}}(D_n)$. 
\end{beis}
\begin{rema}
	Looking at examples \ref{example temperley-lieb inclusions n=2} and \ref{example temperley lieb inclusions n=3} we observe that
	\begin{equation} \label{dimension Temperley--Lieb von typ B}
		\dim \TL(B_n)=(n+1)\dim \TL(A_{n-1})={{2n} \choose n},
	\end{equation}
	since $\dim \TL(A_{n-1})=\frac{1}{(n+1)}{{2n} \choose n}$ is the $n$-th Catalan number. This holds in general and can be found as dimension for $\TL(B_n)$ from \cite[Lemma 5.7]{green1998}. However even if one knows the basis, it is combinatorially non-trivial to see that decorating the strands with dots increases the total number of diagrams by the factor $n+1$, since the number of possible decorations of any fixed Temperley--Lieb diagram is highly dependent on the diagram one starts with. At least it is clear that the dimension of $\TL(D_n)$, that is the number of diagrams decorated with an even amount of dots, is precisely $\frac{1}{2}\dim \TL(B_n)$ since
	\[
		 \TL(B_n) = \TL(D_n) \oplus s_0\! \TL(D_n) = \TL(D_n) \oplus \TL(D_n) s_0.
	\]
	In contrast the true Temperley--Lieb algebra of type $D$ has dimension
	\[
		\dim \TL^{\opn{Gr}}(D_n)=\frac{1}{2} {{2n} \choose n} + \frac{1}{n+1} {{2n} \choose n} -1 = \frac{n+3}{2(n+1)} {{2n} \choose n}-1,
	\]
	since we had to take the Catalan number minus one many type i) diagrams from \Cref{explanation green type D vs our type D} into account. We can now summarize the dimensions of all the Temperley--Lieb algebras present
	\begin{gather*}
		\dim \TL(A_{n-1})= \frac{1}{n+1} {{2n} \choose n}, \, \dim \TL(D_n)= \frac{1}{2} {{2n} \choose n}, \, \dim \TL(B_n)= {{2n} \choose n}, \\ \dim \TL^{\opn{Gr}}(D_n)= \frac{n+3}{2(n+1)} {{2n} \choose n}-1.
	\end{gather*}
\end{rema}
%\begin{ques}
%	Where else in mathematics do these numbers occur?
%	These dimensions are not precisely the Catalan numbers from \cite[Figure 2.8]{armstrong09}, however in particular $\dim \TL^{\opn{Gr}}(D_n)$ is very close to the type $D_n$ Catalan number, as defined there. Notably $\dim \TL(A_{n-1})$, $\dim \TL(D_n)$ and $\dim \TL(B_n)$ all appear as the number of clusters in certain (truncated) Cluster algebras of their respective type. Is there a connection to (Fuss-)Catalan numbers? 
%\end{ques}
	The next purely combinatorial observation gives a different approach from \cite{green1998} to calculate the type $B$ Catalan numbers using what we call signed Brauer diagrams. 
\begin{obse} \label{lemma number of signed brauer diagrams}
	Consider $\Br_n^{\opn{sign}}$ the set of all signed Brauer diagrams, that is all sign symmetric partitions of size $2$ on the set $\{-n,\ldots, -1,1, \ldots, n\}\times \{ 0,1 \}$ of vertices, such that no vertex $(i,\varepsilon)$ pairs with its vertical mirror image $(-i,\varepsilon)$. Then the number of diagrams is given by
	\begin{equation} \label{number of signed brauer diagrams}
		|\Br_n^{\opn{sign}}|=\prod_{i=0}^{n-1}({4(n-i)-2})=(4n-2)(4n-6)\ldots 6\cdot 2=\prod_{i=0}^{n-1}(4i+2),
	\end{equation}
	which follows from the following inductive argument. The first vertex can be connected to $4n-2$ vertices (i.e.\ all vertices except itself and its mirror image) and any choice of partner vertex determines two pairs of partners, which leaves $4n-4$ vertices without a partner. Proceed with any other vertex, which has then $4n-6=4n-(4+2)$ possible partners and repeat the argument.
	The diagram basis of $\TL(B_n)$ consisting of decorated Temperley--Lieb diagrams can be identified with the set $\TL_n^{\opn{sign}}$ consisting of all sign symmetric diagrams, which can only cross at the middle vertical line and are planar otherwise (i.e.\ an arc can only intersect with its mirror image). In particular we have
	\begin{equation} \label{number of nice temperley--lieb diagrams}
		\dim \TL(B_n)=|\TL_n^{\opn{sign}}|=\frac{\prod_{i=0}^{n-1}(4i+2)}{n!}={{2n} \choose n}
	\end{equation}
	since
	\[
		{{2n} \choose n} =\frac{(2n)!}{(n!)^2}=\frac{\prod_{i=1}^n(2i)}{n!}\cdot \frac{\prod_{i=0}^{n-1}(2i+1)}{n!}=2^n\frac{\prod_{i=0}^{n-1}(2i+1)}{n!}=\frac{\prod_{i=0}^{n-1}(4i+2)}{n!}.
	\]
	For example $\Br_2^{\opn{sign}}$ consists of the $12=6\cdot2$ diagrams
	\begin{gather*}
		\cbox
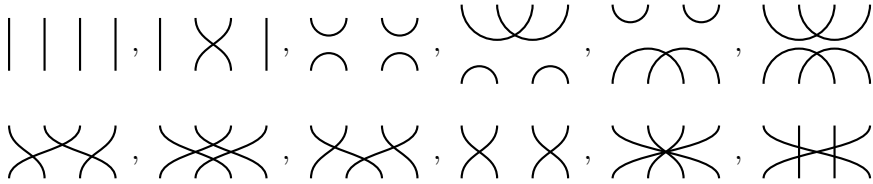
{
		\begin{tikzpicture}[tldiagram, xscale=2/3, yscale=2/3]
			\draw \tlcoord{0}{0} \lineup;
			\draw \tlcoord{0}{1} \lineup;
			\draw \tlcoord{0}{2} \lineup;
			\draw \tlcoord{0}{3} \lineup;
		\end{tikzpicture}
	} \, , \, 
	\cbox{
	\begin{tikzpicture}[tldiagram, xscale=2/3, yscale=2/3]
		\draw \tlcoord{0}{0} \lineup;
		\draw \tlcoord{0}{1} \linewave{1}{1};
		\draw \tlcoord{0}{2} \linewave{1}{-1};
		\draw \tlcoord{0}{3} \lineup;
	\end{tikzpicture}
}\, , \, 
	\cbox{
		\begin{tikzpicture}[tldiagram, xscale=2/3, yscale=2/3]
			\draw \tlcoord{0}{0} \capright;
			\draw \tlcoord{1}{0} \cupright;
			\draw \tlcoord{0}{2} \capright;
			\draw \tlcoord{1}{2} \cupright;
		\end{tikzpicture}
	} \, , \,
\cbox{
	\begin{tikzpicture}[tldiagram, xscale=2/3, yscale=2/3]
		\draw \tlcoord{0}{0} \capright;
		\draw \tlcoord{1.5}{0} \xcupright{2};
		\draw \tlcoord{0}{3} \capleft;
		\draw \tlcoord{1.5}{3} \xcupleft{2};
\end{tikzpicture} }
\, , \,
	\cbox{
		\begin{tikzpicture}[tldiagram, xscale=2/3, yscale=2/3]
			\draw \tlcoord{0}{0} \xcapright{2};
			\draw \tlcoord{1.5}{0} \cupright;
			\draw \tlcoord{0}{3} \xcapleft{2};
			\draw \tlcoord{1.5}{2} \cupright;
		\end{tikzpicture}
	}  \, , \,
\cbox{
	\begin{tikzpicture}[tldiagram, xscale=2/3, yscale=2/3]
		\draw \tlcoord{0}{0} \xcapright{2};
		\draw \tlcoord{1.5}{0} \xcupright{2};
		\draw \tlcoord{0}{3} \xcapleft{2};
		\draw \tlcoord{1.5}{3} \xcupleft{2};
	\end{tikzpicture}
} \\[1em]
	\cbox{
		\begin{tikzpicture}[tldiagram, xscale=2/3, yscale=2/3]
			\draw \tlcoord{0}{0} \linewave{1}{2};
			\draw \tlcoord{0}{1} \linewave{1}{-1};
			\draw \tlcoord{0}{2} \linewave{1}{1};
			\draw \tlcoord{0}{3} \linewave{1}{-2};
		\end{tikzpicture}
	}  \, , \, 
	\cbox{
		\begin{tikzpicture}[tldiagram, xscale=2/3, yscale=2/3]
			\draw \tlcoord{0}{0} \linewave{1}{2};
			\draw \tlcoord{0}{1} \linewave{1}{2};
			\draw \tlcoord{0}{2} \linewave{1}{-2};
			\draw \tlcoord{0}{3} \linewave{1}{-2};
		\end{tikzpicture}
	}  \, , \, 
\cbox{
	\begin{tikzpicture}[tldiagram, xscale=2/3, yscale=2/3]
		\draw \tlcoord{0}{0} \linewave{1}{1};
		\draw \tlcoord{0}{1} \linewave{1}{2};
		\draw \tlcoord{0}{2} \linewave{1}{-2};
		\draw \tlcoord{0}{3} \linewave{1}{-1};
	\end{tikzpicture}
}  \, , \, 
\cbox{
	\begin{tikzpicture}[tldiagram, xscale=2/3, yscale=2/3]
		\draw \tlcoord{0}{0} \linewave{1}{1};
		\draw \tlcoord{0}{1} \linewave{1}{-1};
		\draw \tlcoord{0}{2} \linewave{1}{1};
		\draw \tlcoord{0}{3} \linewave{1}{-1};
	\end{tikzpicture}
}  \, , \, 
\cbox{
	\begin{tikzpicture}[tldiagram, xscale=2/3, yscale=2/3]
		\draw \tlcoord{0}{0} \linewave{1}{3};
		\draw \tlcoord{0}{1} \linewave{1}{1};
		\draw \tlcoord{0}{2} \linewave{1}{-1};
		\draw \tlcoord{0}{3} \linewave{1}{-3};
	\end{tikzpicture}
}  \, , \, 
\cbox{
	\begin{tikzpicture}[tldiagram, xscale=2/3, yscale=2/3]
		\draw \tlcoord{0}{0} \linewave{1}{3};
		\draw \tlcoord{0}{1} \lineup;
		\draw \tlcoord{0}{2} \lineup;
		\draw \tlcoord{0}{3} \linewave{1}{-3};
	\end{tikzpicture}
}
	\end{gather*}
	connecting the vertices $\{-2,-1,1,2\}\times \{0,1\}$ into pairs. The top six of these diagrams correspond one-to-one with the diagram basis of $\TL(B_2)$ from \Cref{example temperley-lieb inclusions n=2} by restricting to the right half of each diagram and translating every arc crossing at the middle vertical line into an arc decorated with a dot.
\end{obse}

\chapter{Tensor product decompositions for a coideal subalgebra} \label{Chapter 2}

This chapter is all about $\sl2$-like decompositions of tensor powers of representations. \Cref{Chapter 2 section 1} is a fast-paced recollection of the $\Ug_q(\sl2)$ story with emphasis on the distinction between type I and type II representations. It is meant as a brief review of the topic. Additionally it contains some philosophical remarks on the subject. \Cref{Chapter 2 section 2} is far more important and contains the new story based on \cite{stroppel2018}, where we explicitly decompose tensor powers of the standard representation of the coideal subalgebra $\Vq\subseteq \Ug_q(\gl_2)$. 

\section{Tensor product decompositions for \texorpdfstring{$\Ug_q(\sl2)$}{U(sl2)}} \label{Chapter 2 section 1}

This section is a review of some important aspects of the representation theory of $\Ug_q(\sl2)$. It is not an optimal way to learn about the subject for the first time, nor a complete collection of all the important facts about tensor powers. However it contains visualizations of representations for $\Ug_q(\sl2)$, a short review of type $A$ Jones--Wenzl projectors and some philosophical discussions on braided monoidal categories or categorical interpretations of the diagrammatic formulas. We do not follow any particular author, since the author learned much of the theory from Catharina Stroppel, but some parts can be found in \cite{kassel2012}, \cite{flath95} or the references listed there.
We start with recalling the definition of the quantum group $\Ug_q(\sl2)$ and its two most important irreducible representations. 

\begin{defi}
	Let $k=\C(q)$. The \textbf{quantum group} $\Ug_q\coloneqq\Ug_q(\sl2)$ is the associative, unital $\C(q)$-algebra with generators $E,F,K,K^{-1}$ subject to the relations $KE=q^2EK$, $KF=q^{-2}FK$, $EF-FK=\frac{K-K^{-1}}{q-q^{-1}}$, and $KK^{-1}=1=K^{-1}K$. The assignments 
	\[
		E\mapsto E\otimes 1 + K \otimes E, \, F\mapsto F\otimes K^{-1} + 1 \otimes F, \, K \mapsto K \otimes K
	\] define a $\C(q)$-algebra homomorphism $\Delta\colon \Ug_q \rightarrow \Ug_q \otimes \Ug_q$ which turns $\Ug_q$ into a Hopf-algebra and hence the category $\Rep(\Ug_q)$ of all finite dimensional $\Ug_q$-modules into a monoidal category with tensor product $\otimes\coloneqq \otimes_k$.
\end{defi}
\begin{defi}
	The \textbf{type I standard representation} of $\Uq$ is the two dimensional vector space $V$ with basis denoted $\{x,y\}$. The action is given diagrammatically by
	
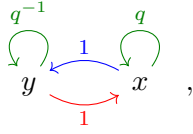
\begin{equation*}
		% https://tikzcd.yichuanshen.de/#N4Igdg9gJgpgziAXAbVABwnAlgFyxMJZABgBpiBdUkANwEMAbAVxiRAE8QBfU9TXfIRQBGclVqMWbAB7dxMKAHN4RUADMAThAC2SMiBwQkokACMYYKEgC0AZn31mrRCGEhqDOuYYAFfngI2DSxFAAscbl4QTR1jakM4kAYICDQiYQAOMjVGOBhxT28-bAChEGCwiOpHKRc3HnUtXUR9BJaPFLSRLNIchjyCrxhff0EgkPD3CSc2a3qomOaTNv1zSxt7aslnVynC4eKBQJcKya4KLiA
		\begin{tikzcd}
			y \arrow[r, "1"', bend right, red] \arrow["q^{-1}"', loop, distance=2em, in=125, out=55, black!50!green] & x \arrow["q"', loop, distance=2em, in=125, out=55, black!50!green] \arrow[l, "1"', bend right, blue]
		\end{tikzcd},
	\end{equation*}
	where the action of $\textcolor{black!50!green}{K}$ is \textcolor{black!50!green}{green}, the action of $\textcolor{red}{E}$ is \textcolor{red}{red} and the action of $\textcolor{blue}{F}$ is \textcolor{blue}{blue}.
	Similarly we define the \textbf{type II standard representation} of $\Uq$, which is given by a two dimensional vector space with basis $\{x,y \}$, where the $\Uq$-action is given by
	
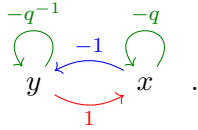
\begin{equation*}
		% https://tikzcd.yichuanshen.de/#N4Igdg9gJgpgziAXAbVABwnAlgFyxMJZABgBpiBdUkANwEMAbAVxiRAE8QBfU9TXfIRQBGclVqMWbAB7dxMKAHN4RUADMAThAC2SMiBwQkokACMYYKEgC0AZn31mrRCGEhqDOuYYAFfngI2DSxFAAscbl4QTR1jakM4kAYICDQiYQAOMjVGOBhxT28-bAChEGCwiOpHKRc3HnUtXUR9BJaPFLSRLNIchjyCrxhff0EgkPD3CSc2a3qomOaTNv1zSxt7aslnVynC4eKBQJcKya4KLiA
		\begin{tikzcd}
			y \arrow[r, "1"', bend right, red] \arrow["-q^{-1}"', loop, distance=2em, in=125, out=55, black!50!green] & x \arrow["-q"', loop, distance=2em, in=125, out=55, black!50!green] \arrow[l, "-1"', bend right, blue]
		\end{tikzcd}.
	\end{equation*} 
\end{defi}
\begin{warn}
	The type I standard representation and the type II standard representation are not isomorphic. Throughout the thesis we will usually refer to the type I standard representation by $V$. However in rare corner cases we will by abuse of notation also denote the type II representation by $V$. We will be careful to not confuse anyone, even though our notation is dangerous.
\end{warn}
Next we will recall the Temperley--Lieb category $\TL$ as defined in \cite{cooper2015}. Afterwards we will recall the monoidal equivalence between $\TL$ and the monoidal subcategory of $\Rep(\Uq)$ generated by $V$. 
\begin{defi}
	Let $k=\C(q)$ and let $\delta\in k$ a scalar. The \textbf{Temperley--Lieb category} $\TL=\TL(\delta)$ is the category with:
	\begin{enumerate}
		\item Objects: non-negative integers $n\in\mN$
		\item Morphisms: $\TL(m,n)=$ the $k$-vector space with basis $(m,n)$-Temperley--Lieb diagrams.
		\item Composition of morphisms: Stacking diagrams on top of each other and replacing each disjoint circle created by the scalar $\delta$.
	\end{enumerate}
	The tensor product $\otimes\colon \TL \times \TL \rightarrow \TL$, given on objects by $m \otimes n\coloneqq m+n$ and defined on Temperley--Lieb diagrams $\lambda\in\TL(a,b), \lambda'\in\TL(a',b')$ as
	\[
		[\lambda\mid\lambda']\in\TL(a+a',b+b'),
	\]
	which is the diagram obtained by placing the Temperley--Lieb diagram $\lambda$ on the left of $\mu$ and extended $k$-bilinearly to all morphisms, turns $\TL$ into a monoidal category. The monoidal unit $\mU$ of $\TL$  is given by the object $0$, whose endomorphism algebra is the $1$-dimensional $k$-algebra spanned by the empty diagram. The type $A$ Temperley--Lieb algebra  $\TL(A_{n-1})=\TL_n=\TL_n(\delta)$ is then by definition the endomorphism-algebra $\TL_n=\End_{\TL}(n)$ of the object $n$. 
\end{defi}
Before we state the connection between $\TL$ and $\Rep(\Uq)$ we recall a fact, which summarizes the Temperley--Lieb category quickly.
\begin{fact} \label{free monoidal?}
	The Temperley--Lieb category $\TL$ is the free, strict monoidal, $k$-linear category with $\End(\mU)=k$ generated by one selfdual object of dimension $\delta$. More precisely the morphisms ${\tlcap}\colon 2 \rightarrow 0$ and ${\tlcup}\colon 0 \rightarrow 2$ generate $\TL$ as a monoidal category and are subject to the minimal set of relations given by
	\[
	\cbox{
		\begin{tikzpicture}[tldiagram, yscale=2/3,xscale=2/3]
			% dots
			%\drawdots{0}{0}{0}
			%\drawdots{2}{2}{2}
			% the line
			\draw \tlcoord{0}{0} \lineup \capright \cupright \lineup;
		\end{tikzpicture}
	}
	\quad
	=
	\quad
	\cbox{
		\begin{tikzpicture}[tldiagram, yscale=2/3]
			% dots
			%\drawdots{0}{0}{0}
			%\drawdots{2}{0}{0}
			% the line
			\draw \tlcoord{0}{0} \lineup \lineup;
		\end{tikzpicture}
	}
	\quad
	=
	\quad
	\cbox{
		\begin{tikzpicture}[tldiagram, yscale=2/3, xscale=2/3]
			% dots
			%\drawdots{0}{2}{2}
			%\drawdots{2}{0}{0}
			% the line
			\draw \tlcoord{2}{0} \linedown \cupright \capright \linedown;
		\end{tikzpicture}
	}
	\qquad \text{and} \qquad 
	\cbox{
		\begin{tikzpicture}[tldiagram]
			\draw (0,0) circle (1);
		\end{tikzpicture}
	}
	\,
	=
	\,
	\delta,
	\]
	where the last relation is interpreted inside $\End_{\TL}(0)=k$.
\end{fact}
\begin{fact}[Graphical calculus, type I] \label{fact graphical sl2 calculus}
	Let $\delta=-[2]=-q-q^{-1}\in k=\C(q)$. The assignments 
	\begin{align*}
		\TL(-[2])&\to \Rep(\Uq) \\
		1 &\mapsto V \\
		\tlcap &\mapsto (\tlcap\colon V^{\otimes 2}\to \mU)\\
		\tlcup &\mapsto (\tlcup\colon \mU \to V^{\otimes 2})
	\end{align*}
	extend to a fully faithful $k$-linear monoidal functor $\TL\to \Rep(\Uq)$. Here $\mU=\C(q)$ is the trivial representation and $V$ is the type I standard representation of $\Uq$. The morphisms $\tlcap\colon V\otimes V \rightarrow \mU$ and $\tlcup\colon \mU \rightarrow V \otimes V$ are defined as
%	\begin{gather*}
%		\tlcap(x \otimes x) = \cap(y \otimes y) = 0 \,,
%		\quad
%		\tlcap(x \otimes y) = -1 \,,
%		\quad
%		\tlcap(y \otimes x) = 1 \,,
%		\shortintertext{and}
%		\tlcup(1) =  x \otimes y - y \otimes x .
%	\end{gather*}
	\begin{gather*}
		\tlcap(x \otimes x) = \cap(y \otimes y) = 0 \,,
		\quad
		\tlcap(x \otimes y) = -q^{-1} \,,
		\quad
		\tlcap(y \otimes x) = 1 \,,
		\shortintertext{and}
		\tlcup(1) =  x \otimes y -  q y \otimes x .
	\end{gather*}
\end{fact}
The important thing to note in the type I story is that the value of $\tlcircle$ is $-[2]$. One can formulate a version of this fact for type II, however there $\tlcircle$ is $[2]$. 
\begin{warn}[Dimensions and braidings] \label{discussion on braidings i}
	Throughout this warning we work for simplicity with representations of the Lie algebra $\sl2$ and its vector representation $\C^2$. This representation is obtained by specializing the type I standard representation $V$ of $\Uq$ at $q=1$. The formulas for $\tlcap$ and $\tlcup$ also specialize to the non-quantized case. We mentioned in \Cref{free monoidal?} that the Temperley--Lieb category is the free monoidal category generated by one self-dual object of dimension $\delta$. In order to define the categorical dimension, one needs at least a braiding on the given monoidal category. However we did not define the braiding of $\TL$, since the existence of such a braiding depends of $\delta$. If we choose the standard symmetric braiding on $\Rep(\sl2)$ defined by 
	\begin{align*}
		W_1 \otimes W_2 &\rightarrow W_2 \otimes W_1 \\
		w_1 \otimes w_2 &\mapsto w_2 \otimes w_1
	\end{align*}
 	(and extended $k$-linearly) the braiding on $\TL$ has to be defined in such a way that the dimension of $1\in \TL$ is also $2$. Such a symmetric braiding exists and is given on $2=1\otimes 1 \rightarrow 1\otimes 1=2$ by the formula
	\begin{equation*} \label{binor identity}
		\cbox{
			\begin{tikzpicture}[tldiagram]
				% dots
				%\drawdots{0}{0}{1}
				%\drawdots{1}{0}{1}
				% strings
				\draw \tlcoord{0}{0} -- \tlcoord{1}{1};
				\draw \tlcoord{0}{1} -- \tlcoord{1}{0};
		\end{tikzpicture}} \quad \coloneqq \quad
		\cbox{
			\begin{tikzpicture}[tldiagram]
				% dots
				%\drawdots{0}{0}{1}
				%\drawdots{1}{0}{1}
				% strings
				\draw \tlcoord{0}{0} \lineup;
				\draw \tlcoord{0}{1} \lineup;
			\end{tikzpicture}
		}
		\quad
		+
		\quad
		\cbox{
			\begin{tikzpicture}[tldiagram]
				% dots
				%\drawdots{0}{0}{1}
				%\drawdots{1}{0}{1}
				% strings
				\draw \tlcoord{0}{0} \capright;
				\draw \tlcoord{1}{0} \cupright;
			\end{tikzpicture}
		},
	\end{equation*}
	which turns the functor $\TL(-2)\to \Rep(\sl2)$ into a braided monoidal functor. Indeed in this case the categorical dimension of $1\in\TL$ with this chosen braiding is
	\begin{equation*}
		\cbox{
			\begin{tikzpicture}[tldiagram, yscale=2/3]
				% dots
				%%\drawdots{0}{0}{1}
				%\drawdots{1}{0}{1}
				% strings
				\draw \tlcoord{0}{0} -- \tlcoord{1}{1} \capleft;
				\draw \tlcoord{0}{0} \cupright -- \tlcoord{1}{0} ;
		\end{tikzpicture}} \quad = \quad
		\cbox{
			\begin{tikzpicture}[tldiagram, yscale=2/3]
				% dots
				%\drawdots{0}{0}{1}
				%\drawdots{1}{0}{1}
				% strings
				\draw \tlcoord{0}{0} \lineup \capright \linedown \cupleft;
			\end{tikzpicture}
		}
		\quad
		+
		\quad
		\cbox{
			\begin{tikzpicture}[tldiagram, yscale=2/3]
				% dots
				%\drawdots{0}{0}{1}
				%\drawdots{1}{0}{1}
				% strings
				\draw \tlcoord{0}{0} \cupright \capleft;
				\draw \tlcoord{1}{0} \cupright \capleft;
			\end{tikzpicture}
		} \quad = \quad -2 + (-2)^2=2.
	\end{equation*}
	If one is not satisfied with the value of the circle being $-2$ instead of $2$ one could consider the monoidal category $\TL(2)$.
	This is again a symmetric monoidal category, where the braiding morphism on $1\otimes 1\rightarrow 1\otimes 1$ is given by
	\begin{equation*} \label{binor identity 2}
		\cbox{
			\begin{tikzpicture}[tldiagram]
				% dots
				%\drawdots{0}{0}{1}
				%\drawdots{1}{0}{1}
				% strings
				\draw \tlcoord{0}{0} -- \tlcoord{1}{1};
				\draw \tlcoord{0}{1} -- \tlcoord{1}{0};
		\end{tikzpicture}} \quad \coloneqq \quad
		\cbox{
			\begin{tikzpicture}[tldiagram]
				% dots
				%\drawdots{0}{0}{1}
				%\drawdots{1}{0}{1}
				% strings
				\draw \tlcoord{0}{0} \lineup;
				\draw \tlcoord{0}{1} \lineup;
			\end{tikzpicture}
		}
		\quad
		-
		\quad
		\cbox{
			\begin{tikzpicture}[tldiagram]
				% dots
				%\drawdots{0}{0}{1}
				%\drawdots{1}{0}{1}
				% strings
				\draw \tlcoord{0}{0} \capright;
				\draw \tlcoord{1}{0} \cupright;
			\end{tikzpicture}
		},
	\end{equation*}
	However the dimension of $1\in\TL(2)$ with respect to this braiding is $-2$, so the category $\TL(2)$ is not equivalent as braided monoidal category to $\TL(-2)$ and there is no hope for a sensible fully faithful braided monoidal functor $\TL(2)\rightarrow \Rep(\sl2)$ mapping $1$ to the standard representation $V$. However note that as algebras $\TL_n(\delta)$ and $\TL_n(-\delta)$ are abstractly isomorphic for any fixed $\delta\in k$ by mapping $U_i\mapsto -U_i$.
	In particular as algebras
	\[
		\TL_n(2)\cong \TL(-2)\cong \End_{\sl2}(V^{\otimes n}),
	\]
	which can be very misleading, since as explained the isomorphism $\TL_n(2)\cong \End_{\sl2}(V^{\otimes n})$ is unnatural. In the quantized setting there are quantized versions of both proposed braidings. The choice of whether one wants to work with the type I or type II standard representation is the choice of whether one wants to have signs in the action of the quantum group or in the formulas for the Temperley--Lieb algebra, i.e.\ the evaluation morphism $\tlcap$ and the coevaluation morphism $\tlcup$.
	Most importantly the value of the circle $\tlcircle$ is $-[2]$ in type I, but $[2]=q+q^{-1}\in\C(q)$ in type II as we will see in \Cref{type II corr formulas}.
	Finally note that the (symmetric) braided monoidal category generated by one self-dual object is the Brauer category (see e.g. \cite{lehrer15}) and not the Temperley--Lieb category, which does not come with inherent braiding morphisms for general $\delta$.
\end{warn}
 
\begin{fact}[Graphical calculus, type II] \label{type II corr formulas}
	Let $\delta=[2]=q+q^{-1}\in \C(q)$. Let $V$ be the type II standard representation of $\Uq$. The assignments 
	\begin{align*}
		\TL([2])&\to \Rep(\Uq) \\
		n &\mapsto V^{\otimes n} \\
		\tlcap &\mapsto (\tlcap\colon V^{\otimes 2}\to \mU)\\
		\tlcup &\mapsto (\tlcup\colon \mU \to V^{\otimes 2})
	\end{align*}
	$\TL\to \Rep(\Uq)$ extend to a fully faithful $\C(q)$-linear monoidal functor. The morphisms $\tlcap$ and $\tlcup$ are defined as
	\begin{gather*}
		\tlcap(x \otimes x) = \cap(y \otimes y) = 0 \,,
		\quad
		\tlcap(x \otimes y) = q^{-1} \,,
		\quad
		\tlcap(y \otimes x) = 1 \,,
		\shortintertext{and}
		\tlcup(1) =  x \otimes y + q  y \otimes x .
	\end{gather*}
\end{fact}
\begin{rema} \label{type A Hecke algebra}
	Recall the Hecke algebra $\Hecke_q(A_{n-1})$, which is the $\C(q)$-algebra with generators $H_1,\ldots, H_{n-1}$ and relations
	\begin{enumerate}
		\item $H_i^2=1+(q^{-1}-q)H_i$ for all $i=1,\ldots, n-1$,
		\item $H_iH_{i+1}H_i=H_{i+1}H_iH_{i+1}$ for all $i=1,\ldots,n-2$,
		\item $H_iH_j=H_jH_i$ for all $|i-j|\geq 2$.
	\end{enumerate}
 	Using the morphisms $\tlcup$ and $\tlcap$ one defines an action of the Hecke algebra $\Hecke_q(A_{n-1})$ on the tensor power of $V^{\otimes n}$ by letting $H_i$ act by $U_i+q^{-1}$, where $U_i$ is the endomorphism $\tlline\otimes\cdots \tlline\otimes\tlcupcap\otimes \tlline \otimes \cdots \tlline$ with $\tlcupcap$ acting on the $i$-th and $(i+1)$-st tensor factors. 
% 	Note that if we define the braiding on $\Rep(\Uq)$ as $H_i$ we have $\dim(V)=1+q^{-2}$. In terms of pure aesthetic it makes more sense to take $qH_i$ as braiding to have $\dim(V)=q+q^{-1}$ and this normalization will appear, when we discuss how to obtain the ``correct'' quotient map $\C(q)B(A_{n-1})\rightarrow \TL_n$ from the braid group to the Temperley--Lieb algebra in \Cref{correct normalization}.
\end{rema}
Before we define the type $A$ Jones--Wenzl projectors, we compare in two examples the tensor product decompositions of $V\otimes V$ for type I and type II. 
\begin{beis}[$V\otimes V$ decomposition for $\Uq$, type I]
	Let $V$ the type I standard representation of $\Uq$.
	The action of $\Uq$ on $V\otimes V$ can be visualized as:
	\begin{equation*}
		\begin{tikzcd}
			& x\otimes y \arrow[ld, "q"', bend right, blue] \arrow[rd, "q", bend right, red] \arrow["q^{0}"', loop, black!50!green, distance=2em, in=125, out=55] &                                                                                                                              \\
			y\otimes y \arrow[ru, "1", bend right, red] \arrow[rd, "q^{-1}"', bend right, red] \arrow["q^{-2}"', loop, distance=2em, in=215, out=145, black!50!green] &                                                                                                                           & x\otimes x \arrow[lu, "1"', bend right, blue] \arrow[ld, "q^{-1}", bend right, blue] \arrow["q^2"', loop, black!50!green, distance=2em, in=35, out=325] \\
			& y\otimes x \arrow[lu, "1", bend right, blue] \arrow[ru, "1"', bend right, red] \arrow["q^0"', loop, distance=2em, in=305, out=235, black!50!green]  &                                                                                                                             
		\end{tikzcd}
	\end{equation*}
	This representation decomposes into the type I representations $L(q^2)$ and $L(q^0)=\mU$, where 
	\begin{equation*}
		L(q^2) \quad = \quad %
		\begin{tikzcd}
			y\otimes y \arrow[r, "1"', bend right, red] \arrow["q^{-2}"', loop, black!50!green, distance=2em, in=215, out=145, black!50!green] & x\otimes y + q^{-1} y\otimes x \arrow[l, "q+q^{-1}"', bend right, blue] \arrow[r, "q+q^{-1}"', bend right, red] \arrow["q^0"', loop, distance=2em, in=125, out=55, black!50!green] & x\otimes x \arrow[l, "1"', bend right, blue] \arrow["q^2"', loop, distance=2em, in=35, out=325, black!50!green]
		\end{tikzcd} \, 
	\end{equation*}
	and $L(q^0)=\shift{x\otimes y - q y\otimes x} \subseteq V\otimes V$. 
\end{beis}
\begin{beis}[$V\otimes V$ decomposition for $\Uq$, type II]
	Let $V$ be the type II standard representation of $\Uq$. The action of $\Uq$ on $V\otimes V$ can be vizualized as:
	\begin{equation*}
		% https://tikzcd.yichuanshen.de/#N4Igdg9gJgpgziAXAbVABwnAlgFyxMJZABgBoBGAXVJADcBDAGwFcYkQBPAHS4jwFt4AAg4gAvqXSZc+QinKli1Ok1bsAHjz5ZBcEeMkgM2PASIKATMoYs2iTloHD1BqSdlELFa6rshNvE56LmLKMFAA5vBEoABmAE4Q-EhkIDgQSAogAEYwYFBIALQAzKk2avbkIDSM9LmMAArSpnIgjDCxOK4gCUmZNOkpNLn5RaU05X4AjtVtdTCNzR728VgRABZdEnGJyYhZg4jFw3kFiCVlvuyFMzXzi+5m9u2d3b17x2kZ+yej5+MqWzXKp3epNR6tVYbLaGd5ILxfIY5U5jS5A+yFEFzMFLJ5tDownZ9RCpQ4IkZnC6ghbgmR4qGbWaTa5TAB6wExYjeu3hA2+nwpqImV0qs1qOIh7AZhJ6PKOfN5yL+VMBFRAbI55C51IedNaLxlcJJCuNbQgEDQngA7GRYkw4DBlOKabjIWtGcL0er2YULFztrLiQdvllGObLfIABy2+2OnW0lpS91dT1qjXEf2wuUIsk1cNEGOMB1O+4J5YgaVMkXe4jc4mfQ6fMMWogATkLxfjrqT0KrXrZFnElDEQA
		\begin{tikzcd}
			& x\otimes y \arrow[ld, "q"', bend right, blue] \arrow[rd, "-q", bend right, red] \arrow["q^{0}"', loop, black!50!green, distance=2em, in=125, out=55] &                                                                                                                              \\
			y\otimes y \arrow[ru, "1", bend right, red] \arrow[rd, "-q^{-1}"', bend right, red] \arrow["q^{-2}"', loop, distance=2em, in=215, out=145, black!50!green] &                                                                                                                           & x\otimes x \arrow[lu, "-1"', bend right, blue] \arrow[ld, "q^{-1}", bend right, blue] \arrow["q^2"', loop, black!50!green, distance=2em, in=35, out=325] \\
			& y\otimes x \arrow[lu, "-1", bend right, blue] \arrow[ru, "1"', bend right, red] \arrow["q^0"', loop, distance=2em, in=305, out=235, black!50!green]  &                                                                                                                             
		\end{tikzcd}
	\end{equation*}
	This representation decomposes into the type I representations $L(q^2)$ and $L(q^0)=\mU$, where 
	\begin{equation*}
		L(q^2) \quad = \quad % https://tikzcd.yichuanshen.de/#N4Igdg9gJgpgziAXAbVABwnAlgFyxMJZABgBpiBdUkANwEMAbAVxiRAE8AdTiPAW3gACdiAC+pdJlz5CKAIzkqtRizYAPbrywC4wwQGphm-kLViJIDNjwEiAJkXV6zVohAaeJ3WdFKYUAHN4IlAAMwAnCD4kMhAcCCQFEAAjGDAoJABaAGZY51U3ORBqBjpUhgAFKRtZEHCsAIALHHMwyOjEJPiY6lT0rNynFVcQO2KQUvKq6xk2eqaW8TaoxOpuxAcUtIzEHLzhtjGSsphK6tm3eebWkAiVjbWEzt7tgf2XNiLjqfPbS4brktbu0enEnrEGBAIGh7AB2MihRhwGBKSanabSP51AEtIYfNyZMZAu4dLpPJKQ6FEOQADgRSJR33Rv1qV1xynxIGINxJSE2602lJhKAAnPSGMjUSczjMsWzxvkRkSKKIgA
		\begin{tikzcd}
			y\otimes y \arrow[r, "-1"', bend right, red] \arrow["q^{-2}"', loop, black!50!green, distance=2em, in=215, out=145, black!50!green] & - x\otimes y + q^{-1} y\otimes x \arrow[l, "-q-q^{-1}"', bend right, blue] \arrow[r, "q+q^{-1}"', bend right, red] \arrow["q^0"', loop, distance=2em, in=125, out=55, black!50!green] & x\otimes x \arrow[l, "1"', bend right, blue] \arrow["q^2"', loop, distance=2em, in=35, out=325, black!50!green]
		\end{tikzcd} \, 
	\end{equation*}
	and $L(q^0)=\shift{x\otimes y + q y\otimes x} \subseteq V\otimes V$.
%	 Note that the image of $\tlcupcap$ is precisely this $1$-dimensional irreducible direct summand.
%	The morphism $U_1=\tlcupcap\in\TL_2$ is given by
%	\begin{equation*}
%		\tlcupcap(x\otimes x)=0=\tlcupcap(y\otimes y), \quad \tlcupcap(x\otimes y)=q^{-1}x\otimes y + y \otimes x , \quad \tlcupcap(y\otimes x)=x\otimes y + q y \otimes x.
%	\end{equation*}
\end{beis}
\begin{warn}
	These two examples are both helpful and misleading. For general $n$ the tensor product decompositions of $V^{\otimes n}$ is described iteratively by the Clebsch--Gordan rule
	\[
		V \otimes L(q^n) \cong L(q^{n+1}) \oplus L(q^{n-1}),
	\]
	where $L(q^n)$ is an $(n+1)$-dimensional irreducible representation of $\Uq$ of type I/type II. In this equality tensoring with the type I standard representation preserves the type, while tensoring with the type II standard representation changes the type.
	This gives
	\begin{equation} \label{Clebsch--Gordan rule nice}
		V^{\otimes n} =L(q^n) \oplus L(q^{n-2})^{a_{n-2}} \oplus L(q^{n-4})^{a_{n-4}} \oplus \cdots 
	\end{equation}
	for certain $a_{n-2}, a_{n-4}, \ldots \in \mZ_{>0}$. In particular if $n$ is odd, the direct summands in this decomposition are either all type I or type II representations, depending on whether $V$ was of type I or of type II. Hence the $n$-th tensor powers of the type I and type II standard representations are not isomorphic if and only if $n$ is odd. However the endomorphism algebras of the tensor powers for type I and type II are always abstractly isomorphic, since they are given by $\TL_n(-[2])$ respectively $\TL_n([2])$, as discussed in \Cref{discussion on braidings i}.
\end{warn}
Now that we discussed tensor powers to a fair extent, we can define the type $A$ Jones--Wenzl projector $a_{n-1}\in \TL_n=\TL(A_{n-1})$ recursively. This projector corresponds to the projection onto the unique $(n+1)$-dimensional irreducible summand of $V^{\otimes n}$. However we give a more general definition based on \cite{lickorish1991}. Afterwards we discuss in \Cref{defis for JW} how these definitions specialize to our setup.

\begin{defi}
	Let $k$ any commutative ring, $\delta\in k$ arbitrary and $U_0,U_1,\ldots,U_{n+1}\in k^{\times}$ which satisfy the recursive identities
	\begin{equation} \label{weird equality}
		U_0=0 \quad \text{and }\quad \delta \cdot U_{m+1} = U_m + U_{m+2}
	\end{equation}
	for all $m\geq 0$. Then we define the \textbf{type $A$ Jones--Wenzl projector}
	\[
		a_{n-1} =\cbox{
			\begin{tikzpicture}[tldiagram, yscale=2/3]
				% lines
				\draw \tlcoord{-1.5}{0} \lineup \lineup \lineup;
				\draw \tlcoord{-1.5}{1} \lineup \lineup \lineup;
				\draw \tlcoord{-1.5}{2} \lineup \lineup \lineup;
				% boxes
				\draw \tlcoord{0}{0} \maketlboxnormal{3}{$n$};
				=
			\end{tikzpicture}
		} \in \TL_n(\delta)=\TL(A_{n-1})
	\]
	recursively via $a_0=1\in \TL_1$ and
	\begin{equation*}
		a_n \quad = \quad \cbox{
			\begin{tikzpicture}[tldiagram, yscale=2/3]
				% lines
				\draw \tlcoord{-1.5}{0} \lineup \lineup \lineup;
				\draw \tlcoord{-1.5}{1} \lineup \lineup \lineup;
				\draw \tlcoord{-1.5}{2} \lineup \lineup \lineup;
				% boxes
				\draw \tlcoord{0}{0} \maketlboxnormal{3}{$n+1$};
				=
			\end{tikzpicture}
		}
		\quad=\quad
		\cbox{
			\begin{tikzpicture}[tldiagram, yscale=2/3]
				% lines
				\draw \tlcoord{-1.5}{0} \lineup \lineup \lineup;
				\draw \tlcoord{-1.5}{1} \lineup \lineup \lineup;
				\draw \tlcoord{-1.5}{2} \lineup \lineup \lineup;
				% boxes
				\draw \tlcoord{0}{0} \maketlboxnormal{2}{$n$};
				=
			\end{tikzpicture}
		}
		\quad-\quad
		\frac{U_n}{U_{n+1}}
		\cdot
		\cbox{
			\begin{tikzpicture}[tldiagram, yscale=2/3]
				% lines
				\draw \tlcoord{-1}{0} \lineup \lineup \lineup \lineup \lineup;
				\draw \tlcoord{3}{1} \lineup;
				\draw \tlcoord{0}{1} \linedown;
				% paths
				\draw \tlcoord{0}{1} \smalllineup \capright \smalllinedown \linedown;
				\draw \tlcoord{3}{1} \smalllinedown \cupright \smalllineup \lineup;
				% boxes
				\draw \tlcoord{0}{0} \maketlboxnormal{2}{$n$};
				\draw \tlcoord{3}{0} \maketlboxnormal{2}{$n$};
			\end{tikzpicture}
		}.
	\end{equation*}
	%The elements $U_i$ are called \textbf{generalised Chebychev polynomials of the second kind}. 
\end{defi}
\begin{rema} \label{defis for JW}
	Note that the $U_i$ are uniquely determined by $U_1$. Moreover if $U_1=1$ we have automatically $U_2=\delta$.
	If we set $\delta=-[2]$ and $U_n= (-1)^n[n]=(-1)^n \frac{q^n-q^{-n}}{q-q^{-1}}\in\C(q)$ these formulas give an idempotent in $\TL(-[2])$, which is the endomorphism algebra of the $n$-th tensor power of the type I standard representation. On the other hand, if we set $\delta=[2]$ and $U_n=[n]$ these formulas give an idempotent in $\TL(-[2])$, which is the endomorphism algebra of the $n$-th tensor power of the type II standard representation. 
	The assumption \eqref{weird equality} in both these cases follows from the identity
	\begin{equation*}
		[2] \cdot [n] = [n+1] \oplus [n-1]
	\end{equation*}
	for quantum numbers.
\end{rema}
\begin{beis}
	Throughout this example $V$ is the type II standard representation of $\Uq$. We list the type $A$ Jones--Wenzl projectors for the cases $n=2$ and $n=3$.
	\begin{enumerate}
%		\item For $n=0$ we have $\TL_0=\End_{\Ug_q}(V^{\otimes 0})=\C(q)$ and $f_0=\id_{\mU}=1$. We don't draw this idempotent.
%		\item For $n=1$ we have $\TL_1=\End_{\Ug_q}(V^{\otimes 1})=\C(q)$ and $f_1=\id_V=1$. Visually speaking:
%		\begin{equation*}
%			\cbox{
%				\begin{tikzpicture}[tldiagram, yscale=2/3]
%					%lines
%					\draw \tlcoord{-1}{0} \lineup \lineup;
%					%box
%					\draw \tlcoord{0}{0} \maketlboxnormal{1}{$1$};
%				\end{tikzpicture}
%			} \quad = \quad 
%			\cbox{
%				\begin{tikzpicture}[tldiagram, yscale=2/3]
%					%dots
%					%\drawdots{0}{0}{0}
%					%\drawdots{1}{0}{0}
%					%lines
%					\draw \tlcoord{0}{0} \lineup;
%				\end{tikzpicture}
%			}
%		\end{equation*}
		\item For $n=2$ we have $\TL_2=\End_{\Ug_q}(V^{\otimes 2})\cong \Hecke_q(S_2)$. The Jones--Wenzl projector is given by
		\begin{equation*}
			a_1=1-\frac{1}{[2]}U_1=\frac{1}{[2]}(q^{-1}-H_1).
		\end{equation*}
		Visually speaking:
		\begin{equation*}
			\cbox{
				\begin{tikzpicture}[tldiagram, yscale=2/3]
					% lines
					\draw \tlcoord{-1}{0} \lineup \lineup;
					\draw \tlcoord{-1}{1} \lineup \lineup;
					% box
					\draw \tlcoord{0}{0} \maketlboxnormal{2}{$2$};
				\end{tikzpicture}
			} \quad = \quad 
			\cbox{
				\begin{tikzpicture}[tldiagram]
					%dots
					%\drawdots{0}{0}{1}
					%\drawdots{1}{0}{1}
					% lines
					\draw \tlcoord{0}{0} \lineup;
					\draw \tlcoord{0}{1} \lineup;
				\end{tikzpicture}
			} \quad - \quad \frac{1}{[2]}  \cbox{
				\begin{tikzpicture}[tldiagram]
					%dots
					%\drawdots{0}{0}{1}
					%\drawdots{1}{0}{1}
					% lines
					\draw \tlcoord{0}{0} \capright;
					\draw \tlcoord{1}{0} \cupright;
				\end{tikzpicture}
			} \, .
		\end{equation*}
		\item For $n=3$ one computes
		\begin{equation*}
			a_2=1-\frac{[2]}{[3]}(U_1+U_2)+\frac{1}{[3]}(U_1U_2+U_2U_1),
		\end{equation*}
		or visually speaking
		\begin{equation*}
			\cbox{
				\begin{tikzpicture}[tldiagram, yscale=1/2, xscale=1/2]
					% lines
					\draw \tlcoord{-1.5}{0} \lineup \lineup \lineup;
					\draw \tlcoord{-1.5}{1} \lineup \lineup \lineup;
					\draw \tlcoord{-1.5}{2} \lineup \lineup \lineup;
					% box
					\draw \tlcoord{0}{0} \maketlboxnormal{3}{$3$};
				\end{tikzpicture}
			}
			\quad = \quad \cbox{
				\begin{tikzpicture}[tldiagram, yscale=2/3, xscale=2/3]
					%dots
					%\drawdots{0}{0}{2}
					%\drawdots{1}{0}{2}
					%lines
					\draw \tlcoord{0}{0} \lineup;
					\draw \tlcoord{0}{1} \lineup;
					\draw \tlcoord{0}{2} \lineup;
				\end{tikzpicture}
			} 
			-\frac{[2]}{[3]}\left( 
			\cbox{
				\begin{tikzpicture}[tldiagram, yscale=2/3, xscale=2/3]
					%dots
					%\drawdots{0}{0}{2}
					%\drawdots{1}{0}{2}
					%lines
					\draw \tlcoord{0}{0} \capright;
					\draw \tlcoord{1}{0} \cupright;
					\draw \tlcoord{0}{2} \lineup;
				\end{tikzpicture}
			} + 
			\cbox{
				\begin{tikzpicture}[tldiagram, yscale=2/3, xscale=2/3]
					%dots
					%\drawdots{0}{0}{2}
					%\drawdots{1}{0}{2}
					%lines
					\draw \tlcoord{0}{0} \lineup;
					\draw \tlcoord{0}{1} \capright;
					\draw \tlcoord{1}{1} \cupright;
				\end{tikzpicture}
			} 
			\right) + \frac{1}{[3]} \left(
			\cbox{
				\begin{tikzpicture}[tldiagram, yscale=2/3, xscale=2/3]
					%\drawdots{0}{0}{2}
					%\drawdots{1}{0}{2}
					%lines
					\draw \tlcoord{1}{0} \cupright;
					\draw \tlcoord{0}{0} \linewave{1}{2};
					\draw \tlcoord{0}{1} \capright;
				\end{tikzpicture}
			}
			+
			\cbox{
				\begin{tikzpicture}[tldiagram, yscale=2/3, xscale=2/3]
					%dots
					%\drawdots{0}{0}{2}
					%\drawdots{1}{0}{2}
					%lines
					\draw \tlcoord{0}{0} \capright;
					\draw \tlcoord{0}{2} \linewave{1}{-2};
					\draw \tlcoord{1}{1} \cupright;
				\end{tikzpicture}
			}
			\right).
		\end{equation*}
	\end{enumerate}
\end{beis}
The next fact is the most important property of the type $A$ Jones--Wenzl projector.
\begin{fact}[Characterizing property of type $A$ Jones--Wenzl projector] \label{char property type A projector}
	The Jones--Wenzl projector $a_{n-1}$ satisfies
	\begin{equation} \label{killing rule}
		a_{n-1}\circ \tlcup_i=0, \quad {\tlcap_i} \circ a_{n-1} =0 
	\end{equation}
	for all $i=1,\ldots, n$. Here $\tlcup_i$ (resp.\ $\tlcap_i$) are defined by applying $\tlcup$ (resp.\ $\tlcap$) to the $i$-th tensor factor:
	\[
		\tlcup_i=\tlline \cdots \tlline \tlcup\tlline \cdots \tlline\colon V^{\otimes (n-2)}\rightarrow V^{\otimes n} \, , \quad \tlcap_i=\tlline \cdots \tlline\tlcap \tlline \cdots \tlline \colon V^{\otimes (n-2)}\rightarrow V^{\otimes n}.
	\]
	It is the unique nonzero idempotent in $\TL_n$, which satisfies \eqref{killing rule} and every other element in $\TL_n$, which satisfies \eqref{killing rule} is a scalar multiple of $a_{n-1}$. 
\end{fact}
Before we end this section reviewing type $A$ we want to advertise the following perspective, which the author gained through the introduction of \cite{wedrich2018}.
\begin{rema}[Grothendieck groups]
	Throughout this remark we work with the non-quantized case for simplicity. We have restriction functors 
	\[
		F\colon \Rep(\sl2)\to \Rep(\h) \quad \text{and} \quad G\colon \Rep(\h)\to \kVect(k),
	\] 
	which are $k$-linear, monoidal and exact. Since all finite dimensional $\sl2$-representations have weight space decompositions with integral weights, we can view $F$ (resp.\ $G$) as a functor mapping to (resp.\ from) the full subcategory $\Rep(\h)^{\nice}$ of $\Rep(\h)$ consisting of all finite dimensional semisimple $\h$-representations with simples given by $1$-dimensional eigenspaces $\{\C_n\vert n\in\mZ\}$ for the action of $h\in\h$. In particular $\Ko(\Rep(\h)^{\nice})=\mZ[q,q^{-1}]$ is a ring of Laurent polynomials in the variable $q=[\C_1]$ and here plays the role of the character ring. By the Clebsch--Gordan formula the Grothendieck ring $\Ko(\Rep(\sl2))$ is a polynomial ring $\mZ[x]$ in one variable $x$, where $x=[V]$ is the class of the standard representation. Of course we have $\Ko(\kVect(k))=\mZ$.
	The functors $F$ and $G$ induce ring homomorphisms
	\begin{equation*}
		f\colon \mZ[x]=\Ko(\Rep(\sl2)) \to \Ko(\Rep(\h)^{\nice})=\mZ[q,q^{-1}]
	\end{equation*} 
	and
	\begin{equation*}
		g\colon \mZ[q,q^{-1}]=\Ko(\Rep(\h)^{\nice}) \to \Ko(\kVect(k))=\mZ.
	\end{equation*}
	The first morphism sends the class of a representation $W$ to its character, while the second sends a representation to its dimension. Concretely $f$ is the inclusion $\mZ[x]\hookrightarrow \mZ[q,q^{-1}]$ given by $x=[V]\mapsto [\C_1]+[\C_{-1}]=q+q^{-1}$ and $g$ is the projection $\mZ[q,q^{-1}]\to \mZ$ given by setting $q=1$. The first Grothendieck group $\Ko(\Rep(\sl2))$ has a basis given by the classes of the simple objects $\{[L(n)] \mid n\in\mN\}$, which satisfy
	\begin{equation*}
		f([L(n)])=q^n+q^{n-2}+\ldots+q^{n-2}+q^{-n}=[n+1]\in\mZ[q,q^{-1}].
	\end{equation*}
	We have $\Rep(\sl2)\simeq\Kar(\Add(\TL(-2)))$ as symmetric monoidal, $k$-linear categories, where $\Kar(\Add(\TL(-2)))$ denotes the idempotent completion of the additive closure of the Temperley--Lieb category $\TL(-2)$ (see e.g.\ \cite{comes12} for these notions). Under this equivalence we have
	\begin{equation*}
		[L(n)]=(n, a_{n-1}),
	\end{equation*}
	where $a_{n-1}\in\TL_n(-2)$ is the Jones--Wenzl Projector. The Clebsch--Gordan formula in the non-quantized setting states that
	\begin{equation*}
		[n]\cdot [2]=[L(n)]\otimes[V]=[L(n+1)]+[L(n-1)]=[n+1]+[n-1],
	\end{equation*}
	which can already be seen in $\Kar(\Add(\TL(-2)))$ via the decomposition of $f_n\otimes \tlline$ as
	\begin{equation*}
		\cbox{
			\begin{tikzpicture}[tldiagram, yscale=2/3]
				% lines
				\draw \tlcoord{-1.5}{0} \lineup \lineup \lineup;
				\draw \tlcoord{-1.5}{1} \lineup \lineup \lineup;
				\draw \tlcoord{-1.5}{2} \lineup \lineup \lineup;
				% boxes
				\draw \tlcoord{0}{0} \maketlboxnormal{2}{$n$};
				=
			\end{tikzpicture}
		}
		\quad=\quad
		-\cbox{
			\begin{tikzpicture}[tldiagram, yscale=2/3]
				% lines
				\draw \tlcoord{-1.5}{0} \lineup \lineup \lineup;
				\draw \tlcoord{-1.5}{1} \lineup \lineup \lineup;
				\draw \tlcoord{-1.5}{2} \lineup \lineup \lineup;
				% boxes
				\draw \tlcoord{0}{0} \maketlboxnormal{3}{$n+1$};
				=
			\end{tikzpicture}
		}
		\quad+\quad
		\frac{n}{n+1}
		\cdot
		\cbox{
			\begin{tikzpicture}[tldiagram, yscale=2/3]
				% lines
				\draw \tlcoord{-1}{0} \lineup \lineup \lineup \lineup \lineup;
				\draw \tlcoord{3}{1} \lineup;
				\draw \tlcoord{0}{1} \linedown;
				% paths
				\draw \tlcoord{0}{1} \smalllineup \capright \smalllinedown \linedown;
				\draw \tlcoord{3}{1} \smalllinedown \cupright \smalllineup \lineup;
				% boxes
				\draw \tlcoord{0}{0} \maketlboxnormal{2}{$n$};
				\draw \tlcoord{3}{0} \maketlboxnormal{2}{$n$};
			\end{tikzpicture}
		}
	\end{equation*}
	into orthogonal idempotents projecting onto $L(n+1)$ and $L(n-1)$.
\end{rema}

\section{Classical and quantum Schur--Weyl duality for \texorpdfstring{$W(B_n)$}{W(Bn)} and \texorpdfstring{$\Ug(\gl_m\times \gl_m)$}{U(glm x glm)}} \label{Chapter 2 section 2} 

In this section we state the quantum coideal version of the Schur--Weyl duality between the group $W(B_n)$ and the Lie algebra $\gl_m\times\gl_m$. The main goal of this section is prove a special case of this Schur--Weyl duality and decompose tensor powers explicitly in \Cref{Theorem decomposition of tensor power of V}. This section is largely influenced by \cite[\S7]{stroppel2018}.

We start by introducing the quantized versions of both participants in the Schur--Weyl duality.
The first participant is the $\C(q)$-algebra $\Hecke_{1,q}(B_n)$, which is a specialization of the $2$-parameter Hecke algebra $\Hecke_{v,q}(B_n)$ defined over $\C(v,q)$, the field of rational functions in two indeterminates.  
\begin{defi} \label{definition Hecke algebra}
	The \textbf{$2$-parameter Hecke algebra} $\Hecke_{v,q}(B_n)$ of type $B_n$ is the associative, unital $\C(v,q)$-algebra with generators $H_0, H_1, \ldots, H_{n-1}$ and relations
	\begin{enumerate}[label = {(H\arabic*)}, align=left]
		\item \label{Hecke1 parameter relation}
		$H_0^2=1 + (v^{-1}-v)H_0$,
		\item \label{Hecke2 parameter relation}
		$H_i^2=1 + (q^{-1}-q)H_i$ for all $1\leq i\leq n-1$,
		\item \label{Hecke type B relation}
		$H_0H_1H_0H_1=H_1H_0H_1H_0$,
		\item \label{Hecke usual Braid relation}
		$H_iH_{i+1}H_i=H_{i+1}H_iH_{i+1}$ for all $1\leq i \leq n-2$,
		\item \label{Hecke usual comm relation}
		$H_iH_j=H_j H_i$ for all $0\leq i,j \leq n-1$ such that $|i-j|\geq 2$.
	\end{enumerate}
	The \textbf{specialized Hecke algebra} $\Hecke_{q}(B_n)\coloneqq\Hecke_{1,q}(B_n)$ is $\C(q)$-algebra, obtained by specializing $\Hecke_{v,q}(B_n)$ by the relation $v=1$. We refer to the image of $H_0$ in $\Hecke_{1,q}(B_n)$ by $s_0$, since it becomes self-inverse.
\end{defi}

Next we introduce the second participant of the Schur--Weyl duality, which is a coideal version of the universal enveloping algebra $U(\gl_m\times \gl_m)$. First recall the definition of the universal enveloping algebra $U_q(\gl_{2m})$. We fix two labeling sets  
\[
I=\{-m+1,-m+2,\ldots, m-2, m-1\}, \quad J=\{-m+\frac{1}{2}, -m+\frac{3}{2}, \ldots, m-\frac{3}{2},m-\frac{1}{2}\}
\]
for a fixed natural number $m\geq 1$.
\begin{defi}
	The \textbf{quantum group} $\Ug_q\coloneqq\Uzwei$ is the (associative, unital) $\C(q)$-algebra with generators $\{E_i, F_i\mid i\in I\} \cup \{ D_j^{\pm1} \mid j\in J\}$ subject to the relations:
	\begin{enumerate}[label = {(U\arabic*)}, align=left]
		\item[\textbf{Cartan relations:}] \label{Commutativity relation}
		The $\{D_j^{\pm 1}\mid j\in J\}$ commute pairwise, moreover $D_jD_j^{-1}=1$ for all $j\in J$.
		\item[\textbf{$\gl_2$-relations:}] \label{gl2 relation}
		\begin{gather*}
			D_jE_iD_j^{-1}=\begin{cases}
				qE_i, & \text{for $j=i+\frac{1}{2}$}, \\
				q^{-1}E_i, & \text{for $j=i-\frac{1}{2}$}, \\
				E_i, & \text{otherwise,}
			\end{cases} \quad D_jF_iD_j^{-1}=\begin{cases}
				q^{-1}F_i, & \text{for $j=i+\frac{1}{2}$}, \\
				qF_i, & \text{for $j=i-\frac{1}{2}$}, \\
				F_i, & \text{otherwise,}
			\end{cases} \\
			[E_i,F_i]=\frac{K_i-K_i^{-1}}{q-q^{-1}}
		\end{gather*}
		for all $i\in I,j\in J$ where $K_i\coloneqq D_{i+\frac{1}{2}}D_{i-\frac{1}{2}}^{-1}$.
		\item[\textbf{Serre relations:}]
		\[
		E_i^2E_{i'}-[2]E_iE_{i'}E_i+ E_{i'}E_i^2=0, \quad F_i^2F_{i'}-[2]F_iF_{i'}F_i+ F_{i'}F_i^2=0
		\]
		for all $i,i'$ such that $i'=i\pm 1$. For all other $i,i'$ we have $E_iE_{i'}=E_{i'}E_i$ and $F_iF_{i'}=F_{i'}F_i$. Here $[2]=q+q^{-1}\in \C(q)$ is the second quantum number.
		%		\item \label{Hecke usual Braid relation}
		%		$H_iH_{i+1}H_i=H_{i+1}H_iH_{i+1}$ for all $1\leq i \leq n-2$,
		%		\item \label{Hecke usual comm relation}
		%		$H_iH_j=H_j H_i$ for all $0\leq i,j \leq n-1$ such that $|i-j|\geq 2$.
	\end{enumerate}
	The assignments
	%\begin{enumerate}
	\begin{align*}
		\Delta\colon \Ug &\to \Ug \otimes \Ug \\
		E_i &\mapsto E_i \otimes 1 + K\otimes E_i,\\
		F_i &\mapsto F_i \otimes K^{-1} + 1 \otimes F_i,\\
		D_j &\mapsto D_j\otimes D_j
	\end{align*}
	define a comultiplication, which additionally determines a counit $\varepsilon\colon \Uzwei \rightarrow \C(q)$ and an antipode $S\colon \Uzwei \rightarrow \Uzwei$. This data turns the algebra into a Hopf algebra. 
	%		\item[\textbf{Counit}]
	%		\begin{align*}
	%			\varepsilon\colon U &\to \C(q) \\
	%			E_i &\mapsto 0, \\
	%			F_i &\mapsto 0, \\
	%			D_i &\mapsto 1
	%		\end{align*}
	%		\item[\textbf{Antipode}]
	%		\begin{align*}
	%			S\colon U &\to U \\
	%			E_i &\mapsto -K_iE_i, \\
	%			F_i &\mapsto -F_iK_i^{-1},
	%			D_i &\mapsto D_i^{-1}
	%		\end{align*}
	%\end{enumerate}
\end{defi}

Now we can define the subalgebra $\Ug^{\iota}\coloneqq\Ug_q'(\gl_m\times \gl_m)$ of $\Ug_q=\Uzwei$. The reader should think of $\Ug^{\iota}$ as a quantized version of a universal enveloping algebra $\Ug(\lie)$ of some lie subalgebra $\lie\subseteq \gl_{2m}$ isomorphic to $\gl_m\times \gl_m$.
\begin{defi}
	The \textbf{coideal subalgebra} $\Ug^{\iota}\coloneqq\Ug_q(\gl_m\times \gl_m)\subseteq\Ug$ is the $\C(q)$-subalgebra of $\Ug$ generated by the elements
	\[
	\{C_j^{\pm 1}\coloneqq (D_jD_{-j})^{\pm1} \mid j\in J\} \cup \{ B_i\coloneqq E_iK_{-i}^{-1}+F_{-i}\mid i\in I\} \cup \{ B_0=q^{-1}E_0K_0^{-1}+F_0\}.
	\]
\end{defi}
This definition is from \cite{stroppel2018}. In \cite{watanabe2021} a braid group action was defined on this algebra as well as crystal bases for its integrable representations. We will not use these tools for the thesis, since we are mostly interested in the case $m=1$, which can be treated more explicitly.

\begin{beis}
	The most important case for the thesis is $\Vq$. In order to simplify the notation we write 
	\[
	E\coloneqq E_0, \quad F\coloneqq F_0, \quad K\coloneqq K_0, \quad D_1\coloneqq D_{\frac{1}{2}},\quad D_2\coloneqq D_{-\frac{1}{2}}.
	\]
	The coideal subalgebra $\Ug^{\iota}=\Vq$ is with this notation generated by  $C\coloneqq D_1D_2$, its inverse and $B\coloneqq q^{-1}EK^{-1}+F$. Since the element  $D_1D_2\in \Ug$ is central we see by the PBW-theorem that $U^{\iota}$ is as a $\C(q)$-algebra isomorphic to the (commutative!) algebra $\C(q)[C,C^{-1},B]$. Our interest in the subalgebra $\Ug^{\iota}$ comes from the following computation, which shows that $\Ug^{\iota}$ is indeed a right coideal of $\Ug$:
	\begin{align*}
		\Delta (B)&=\Delta ( q^{-1} EK^{-1} + F) \\
		&=q^{-1} \Delta (E) \Delta (K^{-1}) + \Delta (F) \\
		&= q^{-1} (E \otimes 1 + K \otimes E) \cdot (K^{-1} \otimes K^{-1}) + F \otimes K^{-1} + 1 \otimes F \\
		&=q^{-1} EK^{-1}\otimes K^{-1} +q^{-1} 1\otimes K^{-1} + F \otimes K^{-1} + 1 \otimes F \\
		&= \underbrace{B \otimes K^{-1}}_{\in \Ug^{\iota}  \otimes \Ug} + \underbrace{1 \otimes B}_{\in \Ug^{\iota}  \otimes \Ug} \in \Ug^{\iota} \otimes \Ug.
	\end{align*}
	This justifies why we refer to $\U^{\iota}$ as ``coideal subalgebra'' and not itself as a quantum group, since it is not a Hopf ideal and hence no Hopf algebra with the restricted comultiplication.
\end{beis}
The next remark comes from a discussion with Bart Vlaar and explains the misleading ``asymmetry'' in the definition of the element $B\in\Vq$ in addition to why one defines this element in this way.
\begin{rema} \label{assymmetry explained}
	When looking at the definition of $\Ug^{\iota}$ one might wonder, why the element $B=q^{-1}EK^{-1}+F$ looks so asymmetrical in $E$ and $F$. The reason is the following. We want that the algebra $\Ug^{\iota}$ is a quantized version of the subalgebra of $\Ug(\gl_2)$ generated by $c\coloneqq E_{1,1}+E_{2,2}$ and $b\coloneqq e+f$. This quantized version should return the classical action of this subalgebra on tensor powers of $V$ if we set $q=1$. However the ``naive'' definition of $B$ given by $B\coloneqq E+F\in \Ug_q(\gl_2)$ is badly behaved, since $E+F$ is not anywhere close to being a primitive element:
	\[
	\Delta(E+F)=E\otimes 1 + K \otimes E + F \otimes K^{-1} + 1 \otimes F. 
	\]
	This is a problem, since $e+f\in \Ug(\gl_2)$ is of course primitive and we would like a quantized version to resemble this behavior. By defining $B\coloneqq q^{-1}EK^{-1} + F$ we obtain $\Delta(B)=B\otimes K^{-1} + 1\otimes B$, i.e.\ $\Delta(B)$ behaves just like $\Delta(F)$. The scalar $q^{-1}$ just ``cancels'' the contribution of the $K^{-1}$ we had to use and makes the action of $B$ on $V$ to something which we would expect. In particular one can define an analogue $B'$ of $B$ in such a way that $\Delta(B')=B'\otimes 1 + K \otimes B'$, i.e.\ its comultiplication behaves in the same way as $E$. In the same spirit as before the sensible definition (with the same scalar philosophy) would be $B'\coloneqq E + q KF$. However note that the subalgebra generated by $D_1D_2$ and $B'$ would then be a left coideal instead of a right coideal.
	The missing symmetry in $E$ and $F$ can be therefore recovered by relating these two elements via an involution of $\Ug_q(\gl_2)$. Indeed the ``quantum transpose'', which is the $q$-antilinear anti-algebra-involution
	\begin{align*}
		\Phi\colon \Ug_q(\gl_2) &\rightarrow \Ug_q(\gl_2) \\
		E &\mapsto F \\
		F &\mapsto E \\
		D_i &\mapsto D_i^{-1} \\
		q &\mapsto q^{-1}, 
	\end{align*}     
	exchanges the roles of $E$ and $F$ and hence the roles of $B=q^{-1}EK^{-1} + F$ and $B'=E + q KF$, which both quantize $e+f$.
\end{rema}
Next we state the quantum Schur--Weyl duality between $\Hecke_{1,q}(B_n)$ and $\Ug_q(\gl_m\times \gl_m)$. In order to state the theorem, we recall the standard representation $V$ of type I of $\Uzwei$ and define an action of $\Hecke_{1,q}(B_n)$ on the tensor power $V^{\otimes n}$.
\begin{defi}
	The standard representation of $\Ug=\Ug_q(\gl_{2n})$ is the $2n$-dimensional vector space $V$ with basis $\{x_j\mid j\in J\}$ and left $\Ug$ action given by the formulas
	\[
	D_j\cdot x_{j'} = q^{\delta_{j,j'}} x_j, E_i \cdot x_j = \delta_{i,j+\frac{1}{2}}x_{j+1}, F_i \cdot x_j = \delta_{i,j-\frac{1}{2}}x_{j-1}.
	\]
\end{defi}
\begin{lemm} \label{right action of Hecke}
	The Hecke algebra $\Hecke_q(B_n)$ acts from the right on $V^{\otimes n}$ via the following rules. The generator $H_i$ for $1\leq i \leq n-1$ acts on the $i$-th and $(i+1)$-th tensor factor via the formulas
	\[
	(x_j \otimes x_{j'})\cdot H_i = \begin{cases}
		x_{j'}\otimes x_j, & \text{if $j<j'$,} \\
		x_{j'}\otimes x_j + (q^{-1}-q) x_j \otimes x_{j'},& \text{if $j>j'$,}\\
		q^{-1} x_{j'} \otimes x_j, & \text{if $j=j'$,}
	\end{cases} 
	\]
	and the generator $s_0=H_0$ acts on the first tensor-factor by $v_j \mapsto v_{-j}$.
\end{lemm}
\begin{proof}
	This is \cite[Lemma 10.3]{stroppel2018}. We provide a part of the proof later in the case $m=1$, i.e.\ where $V$ is two-dimensional.
	%	We have to show that these formulas respect the defining relations of the Hecke algebra from \Cref{definition Hecke algebra}. The type $A$ relation \ref{Hecke usual Braid relation} is well-known for these formulas. It is clear that the action of $H_0$ commutes with the action of $H_i$ for $i\geq 2$, since $H_0$ only changes the first tensor factor, while $H_i$ only changes tensor factors, which are far away from the first one. Also note the self-inverse element $s_0=H_0$ acts by a self-inverse map. Hence we only to show that the type $B$ relation \ref{Hecke type B relation} is satisfied. We only compute this relation in the case $m=1$, i.e.\ $V=\shift{x,y}$. The  
	%	We compute on the first two tensor factors. For this let $j,j'\in J$. We calculate:
	%	\begin{align*}
	%		(v_j \otimes v_j')H_0H_1H_0H_1 &= (v_{-j} \otimes v_j')H_1H_0H_1\\
	%		 &= \begin{cases}
	%			(v_{j'}\otimes v_{-j})H_0H_1, & \text{if $-j<j'$,} \\
	%			(v_{j'}\otimes v_{-j} + (q^{-1}-q) v_{-j} \otimes v_{j'})H_0H_1,& \text{if $-j>j'$,}\\
	%			(q^{-1} v_{j'} \otimes v_{-j})H_0H_1, & \text{if $-j=j'$,}
	%		\end{cases} \\
	%	&= \begin{cases}
	%		(v_{-j'}\otimes v_{-j})H_1, & \text{if $-j<j'$,} \\
	%		(v_{-j'}\otimes v_{-j} + (q^{-1}-q) v_{j} \otimes v_{j'})H_0H_1,& \text{if $-j>j'$,}\\
	%		(q^{-1} v_{-j'} \otimes v_{-j})H_0H_1, & \text{if $-j=j'$,}
	%	\end{cases} \\
	%	&= 
	%	\end{align*}
	%	
\end{proof}

The next example visualizes the action of $\Vq$ on $V \otimes V$ and is of great importance.
\begin{beis} \label{example standard representation m=1}
	Let $V$ the type I standard representation of $\Ug_q(\gl_2)$.
	We simplify again the notation and redefine $x\coloneqq x_{\frac{1}{2}}, y\coloneqq x_{-\frac{1}{2}}$.
	The representation $V$ is $2$-dimensional and is usually drawn as
	\begin{equation*}
		% https://tikzcd.yichuanshen.de/#N4Igdg9gJgpgziAXAbVABwnAlgFyxMJZABgBpiBdUkANwEMAbAVxiRAE8QBfU9TXfIRQBGclVqMWbAB7dxMKAHN4RUADMAThAC2SMiBwQkokACMYYKEgC0AZn31mrRCGEhqDOuYYAFfngI2DSxFAAscbl4QTR1jakM4kAYICDQiYQAOMjVGOBhxT28-bAChEGCwiOpHKRcAR0j1LV1EfQTWjxS0kSzSHIY8gq8YX39BIJDw9wknNms3HibYxBN2-XNLG3tqyWdXacKR4oFAlwqpxejmvXijDqSuomzc-I9h0ZLxs8mqmdqQBqXGItVZ3EzJVJPPovIZFManco-aY1PYLChcIA
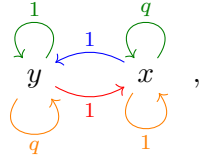
		\begin{tikzcd}
			y \arrow[r, "1"', bend right, red] \arrow["1"', loop, distance=2em, in=125, out=55, black!50!green] \arrow["q"', loop, distance=2em, in=305, out=235, orange] & x \arrow["q"', loop, distance=2em, in=125, out=55, black!50!green] \arrow[l, "1"', bend right, blue] \arrow["1"', loop, distance=2em, in=305, out=235, orange]
		\end{tikzcd},
	\end{equation*} 	
	where the action of $\textcolor{black!50!green}{D_1}$ is visualized \textcolor{black!50!green}{green}, the action of $\textcolor{orange}{D_2}$ is \textcolor{orange}{orange}, the action of $\textcolor{red}{E}$ is \textcolor{red}{red} and the action of $\textcolor{blue}{F}$ is \textcolor{blue}{blue}. 
	In this picture we omitted the actions of $D_1,D_2$.
	When restricting to the action of $\Ug^{\iota}=\Vq$ these two pictures become
	
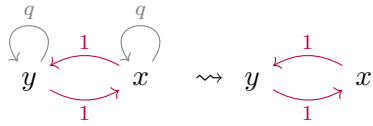
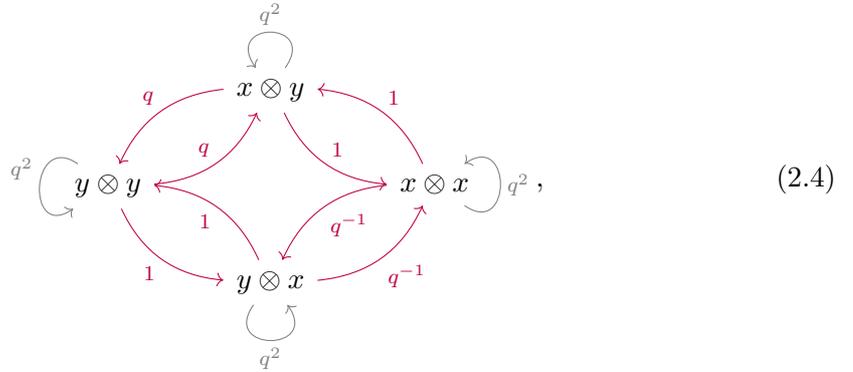
\[
	\begin{tikzcd}
		y \arrow[r, "1"', bend right, purple] \arrow["q"', loop, distance=2em, in=125, out=55, gray]  & x \arrow["q"', loop, distance=2em, in=125, out=55, gray] \arrow[l, "1"', bend right, purple] 
	\end{tikzcd} \rightsquigarrow
\begin{tikzcd}
y \arrow[r, "1"', bend right, purple]   & x  \arrow[l, "1"', bend right, purple] 
\end{tikzcd}
	\]
	and
	\begin{equation} \label{action on second tensor power for coideal}
		\begin{tikzcd}
			& x\otimes y \arrow[ld, "q"', bend right, purple] \arrow[rd, "1", bend right, purple] \arrow["q^{2}"', loop, gray, distance=2em, in=125, out=55] &                                                                                                                              \\
			y\otimes y \arrow[ru, "q", bend right, purple] \arrow[rd, "1"', bend right, purple] \arrow["q^{2}"', loop, distance=2em, in=215, out=145, gray] &                                                                                                                           & x \otimes x \arrow[lu, "1"', bend right, purple] \arrow[ld, "q^{-1}", bend right, purple] \arrow["q^2"', loop, gray, distance=2em, in=35, out=325] \\
			& y \otimes x \arrow[lu, "1", bend right, purple] \arrow[ru, "q^{-1}"', bend right, purple] \arrow["q^2"', loop, distance=2em, in=305, out=235, gray]  &                                                                                                                             
		\end{tikzcd},
	\end{equation} 
	where the action of $\textcolor{gray}{C}$ is visualized \textcolor{gray}{gray} and the action of $\textcolor{purple}{B}$ is \textcolor{purple}{purple}. %(We kept the loops in the second picture, since the author thinks that light-gray and purple have nice contrast.) 
	By looking carefully one sees that $V\otimes V$ has a non-trivial eigenbasis for $B$ and $C$ given by
	\begin{equation} \label{eigen basis in second tensor power}
		\begin{array}{c|c}
			\text{basis vector} & \text{eigenvalue of } B  \\ \hline
			x\otimes x + y \otimes x + q x \otimes y + q y \otimes y & q+q^{-1} \\
			x\otimes x + y \otimes x - q^{-1} x \otimes y - q^{-1} y \otimes y & 0 \\
			x\otimes x - y \otimes x + q^{-1} x \otimes y - q^{-1} y \otimes y & 0 \\
			x\otimes x - y \otimes x - q x \otimes y + q y \otimes y & -q-q^{-1}.
		\end{array}
	\end{equation}
	If we compare picture \eqref{action on second tensor power for coideal} with the action of $B'=E+qKF$ and $C^{-1}$ in the spirit of \Cref{assymmetry explained} we see
	
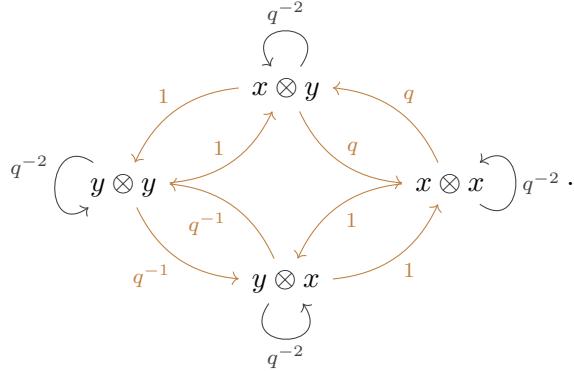
\begin{equation*}
		\begin{tikzcd}
			& x\otimes y \arrow[ld, "1"', bend right, brown] \arrow[rd, "q", bend right, brown] \arrow["q^{-2}"', loop, darkgray, distance=2em, in=125, out=55] &                                                                                                                              \\
			y\otimes y \arrow[ru, "1", bend right, brown] \arrow[rd, "q^{-1}"', bend right, brown] \arrow["q^{-2}"', loop, distance=2em, in=215, out=145, darkgray] &                                                                                                                           & x \otimes x \arrow[lu, "q"', bend right, brown] \arrow[ld, "1", bend right, brown] \arrow["q^{-2}"', loop, darkgray, distance=2em, in=35, out=325] \\
			& y \otimes x \arrow[lu, "q^{-1}", bend right, brown] \arrow[ru, "1"', bend right, brown] \arrow["q^{-2}"', loop, distance=2em, in=305, out=235, darkgray]  &                                                                                                                             
		\end{tikzcd}.
	\end{equation*}
	This action is the same up to a $q$-antilinear automorphism of $V$, which flips $x$ and $y$, so there is nothing new going on here. In particular we can insert the eigenbasis \eqref{eigen basis in second tensor power} into the automorphism to get the corresponding eigenbasis. An interesting observation, which one obtains from this calculation, is that the eigenbasis for $B$ is already an eigenbasis for $B'$, since the automorphism of $V\otimes V$ rescales the basis \eqref{eigen basis in second tensor power}. 
\end{beis}

\begin{proof}[Proof of \Cref{right action of Hecke} for $m=1$]
	With the notation from \Cref{example standard representation m=1} the action of $\Hecke_{1,q}(B_n)$ on $V^{\otimes n}$ looks as follows. The element $H_i$ acts on the $i$th and $(i+1)$st tensor factors by
	\begin{align*} 
		(x \otimes x)\cdot H_i = q^{-1} x\otimes x &, \quad (x \otimes y)\cdot H_i = y \otimes x\\
		(y \otimes x)\cdot H_i = x\otimes y + (q^{-1} - q) x \otimes y &, \quad (y\otimes y) \cdot H_i = q^{-1} y \otimes y.  
	\end{align*}
	The element $s_0=H_0$ acts on the first tensor factor by swapping the basis vectors $x$ and $y$.
	We have to show that these formulas respect the defining relations of the Hecke algebra from \Cref{definition Hecke algebra}. The type $A$ relation \ref{Hecke usual Braid relation} is well-known. It is clear that the action of $H_0$ commutes with the action of $H_i$ for $i\geq 2$, since $H_0$ only changes the first tensor factor, while $H_i$ only changes tensor factors, which are far away from the first one. Also note the self-inverse element $s_0=H_0$ acts by a self-inverse map. Hence we only to show that the type $B$ relation \ref{Hecke type B relation} is satisfied. This is a easy case-by-case distinction. 
	% This is wrong!!
	%	Indeed eight easy calculations give
	%	\begin{align*}
	%		(x \otimes x) \cdot H_0H_1H_0H_1 &= q^{-1}y\otimes y = (x \otimes x) \cdot H_1H_0H_1H_0 \\
	%		(x \otimes y)  \cdot H_0H_1H_0H_1 &= q^{-1}(y \otimes x + (q^{-1}-q)x\otimes y) = (x \otimes y)  \cdot H_1H_0H_1H_0 \\
	%		(y \otimes x)  \cdot H_0H_1H_0H_1 &= q^{-1} x\otimes y = (y \otimes x)  \cdot H_1H_0H_1H_0 \\
	%		(y \otimes y)  \cdot H_0H_1H_0H_1 &=  q^{-1}(x \otimes x + (q^{-1}-q)y\otimes y) = (y \otimes y)  \cdot H_1H_0H_1H_0.
	%	\end{align*}
	%	This proves that $H_0H_1H_0H_1$ and $H_1H_0H_1H_0$ act in the same way on a basis of $V^{\otimes n}$ and hence act by the same linear map. 
\end{proof}
\begin{rema}
	One can check that the element $C$ acts on $V^{\otimes n}$ by $q^n$ in general. In particular note that $\Hom_{\Ug^{\iota}}(V^{\otimes n}, V^{\otimes m})=0$ for $n\neq m$ just like for $\Ug_q(\gl_2)$. 
\end{rema}
\begin{rema}[Challenge of working with coideals]
	When working with $\Ug^{\iota}$ one will run eventually into the following observation. Although $s_0\otimes \id \colon V\otimes V\rightarrow V\otimes V$ commutes with the action of $B$ on $V\otimes V$ and is hence a morphism of $\Ug^{\iota}$ modules, the map $\id_V\otimes s_0 \colon V \rightarrow V$ is \textbf{not} a morphism of $\Ug^{\iota}$-modules. Indeed we have
	\[
	(B \circ \id_V\otimes s_0) (x\otimes x)= B \cdot x\otimes y = x \otimes x + q y\otimes y
	\]
	which is not equal to 
	\[
	(\id_V\otimes s_0 \circ B) (x\otimes x) = (\id_V\otimes s_0)(x\otimes y + q^{-1}y\otimes x)=x\otimes x +q^{-1}y\otimes y.
	\]
	In particular finite dimensional $\Ug^{\iota}$-modules do \textbf{not} form a monoidal category with the usual tensor product. Using the comultiplication one just obtains the following implications:
	\begin{enumerate}
		\item If $V$ is $\Ug^{\iota}$-module and $W$ is a $\Ug$-module, then $V\otimes W$ becomes a $\Ug^{\iota}$-module via
		\[
			% https://tikzcd.yichuanshen.de/#N4Igdg9gJgpgziAXAbVABwnAlgFyxMJZABgBpiBdUkANwEMAbAVxiRAB12BVAcwD1gnfDjoBfThDwBbeAAIAarInS5AdRCjS6TLnyEUARnJVajFm068BQyWOVYZcJdx7PJDuYvuPZ6zduw8AiIAJmNqemZWRA4Xa3ZhO3Z3Hy9klSdLV281DS0QDEC9IgBmcNMotnkcpz8TGCgeeCJQADMAJwgpJDIQHAgkIwrzGM4AERgGETcM5ywoPLbO7sQh-qQw4ejYgGMCHkWQDq6N6nXEEtEKUSA
			\begin{tikzcd}
				\Ug^{\iota}\otimes V \otimes W \arrow[r, "\Delta \otimes \id"] & \Ug^{\iota}\otimes \Ug \otimes V \otimes W \arrow[r, "\cong"] & \Ug^{\iota}\otimes V \otimes \Ug \otimes W \arrow[r, "\varphi_V \otimes \varphi_W"] & V\otimes W
			\end{tikzcd}
		\]
		where $\varphi_V\colon \Ug^{\iota} \otimes V \rightarrow V$ and $\varphi_W\colon \Ug \otimes W \rightarrow W$ are the structure maps turning $V$ and $W$ into left modules.
		\item If $V_1\rightarrow V_2$ is a morphism of $\Ug^{\iota}$-modules and $g\colon W_1 \rightarrow W_2$ is a morphism of $\Ug$-modules, then $f\otimes g$ is a morphism of $\Ug^{\iota}$-modules.
		\item The trivial representation $\mathbb{1}=\C(q)$ of $\Ug$ becomes a representation of $\Ug^{\iota}$ via restriction of the counit to $\Ug^{\iota}$. 
		The tensor product of this representation with an arbitrary representation $V$ of $\Ug^{\iota}$ still makes sense
		and we have $\mathbb{1}\otimes V\cong V$ as representations of $\Ug^{\iota}$. 
	\end{enumerate}
	In $2$-categorical language this means the category  $\modd(\Ug^{\iota})$ of left modules $\Ug^{\iota}$ is a module category of the monoidal category $\modd(\Ug)$ in the sense of \cite{ostrik03}. This abstract setup is studied in \cite{kolb20} and the references therein.
\end{rema}

As representation theorists we ask ourselves how these representations decompose into irreducible/indecomposable subrepresentations. This will be the content of the next theorem after we proved \Cref{right action of Hecke}.
\begin{rema}
	If we work classically, i.e.\ set $q=1$, pass from $\C(q)$ to $\C$ and $\Hecke_q(B_n)$ to $W(B_n)$ the type $B$ relation \begin{equation*} %pictures hinschreiben
		s_0s_1s_0s_1= \cbox{
			\begin{tikzpicture}[tldiagram]
				\draw \tlcoord{0}{0}   \linewave{1}{1} \linewave{1}{-1} \onedot;
				\draw \tlcoord{0}{1} \linewave{1}{-1} \onedot \linewave{1}{1};
			\end{tikzpicture}
		} \quad = 
		\cbox{
			\begin{tikzpicture}[tldiagram]
				\draw \tlcoord{-1}{0} \lineup \onedot \lineup;
				\draw \tlcoord{-1}{1} \lineup \onedot \lineup;
			\end{tikzpicture}
		}
		\quad 
		= 
		\quad 
		\cbox{
			\begin{tikzpicture}[tldiagram]
				\draw \tlcoord{0}{0} \onedot  \linewave{1}{1} \linewave{1}{-1} ;
				\draw \tlcoord{0}{1} \linewave{1}{-1}  \onedot \linewave{1}{1} ;
			\end{tikzpicture}
		}
		=s_1s_0s_1s_0
	\end{equation*}
	has a nice interpretation for the action of $W(B_n)$ on $V^{\otimes n}$. Here $H_i=s_i$ becomes the usual swap of tensor factors and hence both sides of the relation map a pure tensor $v_1\otimes v_2$ to 
	\[
	Tv_1\otimes Tv_2,
	\]
	where $T$ is the linear endomorphism of $V$ given by $x \mapsto y, y\mapsto x$. In particular note that the action of the Jucys--Murphy element on a pure tensor $v_1 \otimes \cdots \otimes v_n$ is just given by
	\[
	v_1 \otimes \cdots \otimes v_i \otimes \cdots \otimes v_n \cdot J_i= v_1 \otimes \cdots \otimes Tv_i \otimes \cdots \otimes v_n.
	\]
\end{rema}

Now that both the Hecke algebra and the coideal subalgebra are established we can state the following theorem.
\begin{theo}[Coideal Schur--Weyl duality] \label{Schur-Weyl duality}
	The actions of $\Hecke_{q}(B_n)$ and $\Ug_q(\gl_m \times \gl_m)$ on $V^{\otimes n}$ commute and centralize each other. 
\end{theo}
This theorem was proven in \cite[Theorem 10.4]{stroppel2018}. Since we want to study analogues of Temperley--Lieb algebras, we just care about the case $m=1$, for which we provide a full proof. The proof will use \Cref{Theorem decomposition of tensor power of V}, which gives a full description of the representation  $V^{\otimes n}$ of $\Vq$. Beforehand we formulate the classical Schur--Weyl duality and do one longer example which gives an overview of the non-quantized case for $m=1$.
\begin{theo}[Schur--Weyl duality]
	Let $V=\C^{2m}$ the standard representation of $\gl_{2m}$ and consider the subalgebra $\gl_m\times \gl_m\subset \gl_{2m}$, which consists of all block matrices of the form $\begin{psmallmatrix}
		A & B \\
		B & A
	\end{psmallmatrix}$, where $A,B\in\gl_m$.
	The actions of $\Ug^{\iota}=\Ug(\gl_m\times \gl_m)$ and $W(B_n)$ on $V^{\otimes n}$ commute and centralize each other. This Schur--Weyl duality gives rise to the decomposition
	\begin{equation*}
		V^{\otimes n}\cong\bigoplus_{(\lambda, \mu)}L(\lambda, \mu) \boxtimes S(\lambda, \mu)
	\end{equation*}
	where the direct sum runs over all bipartitions $(\lambda, \mu)$ of $n$, such that $\lambda$ and $\mu$ have at most $m$ rows. Here $S(\lambda, \mu)$ is the Specht module corresponding to $(\lambda, \mu)$ and $L(\lambda, \mu)=L(\lambda)\boxtimes L(\mu)$, where $L(\lambda), L(\mu)$ are irreducible Schur modules of $\gl_m$. 
	%	More concretely we have
	%	\begin{equation*}
	%		V^{\otimes n}=L(n)\boxtimes S(n)\oplus L(n-1)\boxtimes S(n-1) \oplus \cdots \oplus L(0)\boxtimes S(0).
	%	\end{equation*}
	%	In this decomposition $L(k)$ is the one-dimensional irreducible module of $U^{\iota}$ on which $C$ acts by $n$ and $D$ acts by $n-2k$, and $S(k)$ is the ${n}\choose{k}$-dimensional binomial representation of $W(B_n)$.
\end{theo}
%\begin{beis}
%	For $n=1$ we have the decomposition 
%	\begin{equation*}
%		V={<}v_{+}{>} \oplus {<}v_{-}{>}\cong L(1)\boxtimes S(1) \oplus L(-1)\boxtimes S(-1)
%	\end{equation*} into one-dimensional subrepresentations where $v_{+}=\frac{1}{2}(e_1+e_2)$ and $v_{-}=\frac{1}{2}(e_1-e_2)$. Here $S(1)$ is the trivial representation of $W(B_1)\cong S_2$ and $S(-1)$ is the sign representation.
%\end{beis}
\begin{beis}
	Consider the lie subalgebra $\mathfrak{g}\coloneqq\gl_1\times \gl_1\subseteq \gl_2$, which consists of all $2\times 2$ matrices of the form $\begin{psmallmatrix}
		\lambda & \mu \\
		\mu & \lambda
	\end{psmallmatrix}$ for $\lambda,\mu\in k$. The Lie algebra $\lie$ has a natural choice of basis given by the matrices $c=\begin{psmallmatrix}
		1 & 0 \\
		0 & 1
	\end{psmallmatrix}$ and $b=\begin{psmallmatrix}
		0 & 1 \\
		1 & 0
	\end{psmallmatrix}=e+f$, where
	\begin{equation*}
		e = \begin{psmallmatrix}
			0 & 1 \\
			0 & 0
		\end{psmallmatrix}, \quad h=  \begin{psmallmatrix}
			1 & 0 \\
			0 & -1
		\end{psmallmatrix}, \quad f= \begin{psmallmatrix}
			0 & 0 \\
			1 & 0
		\end{psmallmatrix} , \quad c=\begin{psmallmatrix}
			1 & 0 \\
			0 & 1
		\end{psmallmatrix}
	\end{equation*}
	is our preferred choice of basis of $\gl_2$. We think of the coidal $\Ug^{\iota}$ as a quantized version of the Lie algebra $\lie$, respectively its universal enveloping algebra $\Ug(\lie)$. 
	Observe that $\lie$ is the fix point subalgebra $(\gl_2)^\iota$ under the Lie algebra involution $\iota\colon \gl_2 \rightarrow \gl_2$ given by
	\begin{equation*}
		\begin{pmatrix}
			x_{11} & x_{12} \\
			x_{21} & x_{22}
		\end{pmatrix} \mapsto  \begin{pmatrix}
			0 & 1 \\
			1 & 0
		\end{pmatrix} \cdot  \begin{pmatrix}
			x_{11} & x_{12} \\
			x_{21} & x_{22}
		\end{pmatrix} \cdot \begin{pmatrix}
			0 & 1 \\
			1 & 0
		\end{pmatrix} = \begin{pmatrix}
			x_{22} & x_{21} \\
			x_{12} & x_{11}
		\end{pmatrix}.
	\end{equation*} 
	The universal property of the universal enveloping algebra $U(\gl_2)$ turns the Lie algebra involution $\iota$ into a (associative) algebra involution $\iota\colon U(\gl_2) \to U(\gl_2)$. 
	Just like its quantum version the universal enveloping algebra $\Ug(\lie)=\Ug((\gl_2)^\iota)$ is a (commutative) polynomial ring in the elements $c$ and $b=e+f$ and can be identified with a subalgebra of $\Ug(\gl_2)$ by the PBW theorem. However one should notice that this subalgebra $\Ug(\lie)=\Ug(\lie^{\iota})$ is not the fix-point subalgebra $\Ug(\gl_2)^{\iota}$ of $\Ug(\gl_2)$. For instance the linear combination of PBW basis elements $e^2+f^2$ lives in $\Ug(\gl_2)^{\iota}$, but not in $\Ug(\lie)$.We conclude from this that our quantized version $\Ug^{\iota}$ is not a fixed point subalgebra of $\Ug_q(\gl_2)$, but rather a subalgebra of some fixed point algebra. The notation $\Ug^{\iota}$ for our coideal subalgebra is hence a little bit misleading.
	
	Next let us consider tensor powers of $V=\C^2$, the standard representation of $\gl_2$, viewed as representation of $\lie=\gl_1\times \gl_1$. 
	First observe that $V$ decomposes as representation of $\lie$ into two one-dimensional representations
	\begin{align*}
		V= \shift{e_1+e_2} \oplus \shift{e_1-e_2}=\C_1 \oplus \C_{-1},
	\end{align*}
	where $\C_{\lambda}$ is a one-dimensional representation of $\lie$ such that $b$ acts by the scalar $\lambda\in\C$. Since $d\in U(\lie)$ is primitive, we know $\C_{\lambda}\otimes \C_{\mu}=\C_{\lambda+\mu}$. In particular $V^{\otimes n}$ decomposes as 
	\[
	V^{\otimes n} = \bigoplus_{k=0}^{n} {n \choose k} \C_{n-2k}.
	\]
	The next theorem will show a non-trivial quantized version of this decomposition, where one has to tensor twisted versions of $v_0=e_1+e_2$ and $w_0=e_1-e_2$ together. 
	%		\begin{equation}
	%			\begin{tikzcd}
	%				& x\otimes y \arrow[ld, "1"', bend right, purple] \arrow[rd, "1", bend right, purple] \arrow["1"', loop, black!50!green, distance=2em, in=125, out=55] &                                                                                                                              \\
	%				y\otimes y \arrow[ru, "1", bend right, purple] \arrow[rd, "1"', bend right, purple] \arrow["1"', loop, distance=2em, in=215, out=145, black!50!green] &                                                                                                                           & x \otimes x \arrow[lu, "1"', bend right, purple] \arrow[ld, "1", bend right, purple] \arrow["1"', loop, black!50!green, distance=2em, in=35, out=325] \\
	%				& y \otimes x \arrow[lu, "1", bend right, purple] \arrow[ru, "1"', bend right, purple] \arrow["1"', loop, distance=2em, in=305, out=235, black!50!green]  &                                                                                                                             
	%			\end{tikzcd}.
	%		\end{equation}
\end{beis}

Throughout the next theorem $V$ denotes the type I representation of $\Ug_q(\gl_2)$.

\begin{theo}[Coideal tensor product decomposition in type I] \label{Theorem decomposition of tensor power of V}
	Consider for $n\in \mZ$ the special elements in $V$ given by
	\begin{equation*}
		v_{n}= x + q^{n}y, \quad w_{n}=x-q^{-n}y.
	\end{equation*}
	Then the following statements hold:
	\begin{enumerate}
		\item \label{part i of gl1 x gl1 decomp}The representation $V$ decomposes in two $1$-dimensional representations
		\begin{equation*}
			V = \shift{v_{0}} \oplus \shift{w_{0}},
		\end{equation*}
		which are eigenspaces for the action of $B$ with eigenvalues $1$ respectively $-1$.
		\item \label{part ii of gl1 x gl1 decomp} Let $z$ an eigenvector for the action of $B$ with eigenvalue $[n]$ for $n\in\mZ$, where $[n]=\frac{q^n-q^{-n}}{q-q^{-1}}$ is a quantum integer. Then $z\otimes v_{n}$ is an eigenvector of $B$ with eigenvalue $[n+1]$.
		\item \label{part iii of gl1 x gl1 decomp} Let $z$ an eigenvector for the action of $B$ with eigenvalue $[n]$ for $n\in \mZ$. Then $z\otimes w_{n}$ is an eigenvector of $B$ with eigenvalue $[n-1]$.
		\item \label{part iv of gl1 x gl1 decomp} The $n$-th tensor power decomposes as 
		\begin{equation*}
			V^{\otimes n} = \bigoplus_{k=0}^{n} {n \choose k} L([n-2k])
		\end{equation*}
		Here $L([n-2k])$ is a a one-dimensional representation of $\Vq$, where $C$ acts as $q^n$ and $B$ acts by the multiplication with the quantum integer $[n-2k]$.
	\end{enumerate}
\end{theo}
\begin{proof} \leavevmode
	\begin{enumerate} 
		\item Since the action of $B$ flips $x$ and $y$ we see that $v_{0}=x+y$ and $w_{0}=x-y$ are eigenvectors for the action of $B$ with eigenvalue $1$ respectively $-1$. The additional generator $C$ of $\Vq$ acts by $q$ on all of $V$, so this is really a decomposition into $1$-dimensional subrepresentations.
		\item 
		Since $\Delta(B)=B\otimes K^{-1}+1\otimes B$ we see that
		\begin{align*}
			B (z\otimes (x+q^ny))
			&=Bz \otimes K^{-1}(x+q^ny) + z \otimes B(x+q^n y) \\
			&= [n] z \otimes (q^{-1} x + q^{n+1}y) + z \otimes (y + q^n x) \\
			&=[n+1] z \otimes (x+q^ny),
		\end{align*}
		where in the last equality we use the identities for quantum integers
		\[
		q^{-1}[n] + q^n =[n+1], \quad  q^{n+1}[n]+1=q^n [n+1].
		\]
		Indeed by multiplying both proposed equalities with $q-q^{-1}$ these are equivalent to
		\[
		q^{-1}(q^{n}-q^{-n}) + q^n(q-q^{-1}) =q^{n+1}-q^{-n-1}
		\]
		and
		\[
		q^{n+1}(q^n-q^{-n})+q-q^{-1}=q^n (q^{n+1}-q^{-n-1}),
		\]
		which hold true.
		\item This calculation is very similar to \ref{part ii of gl1 x gl1 decomp}. Just like in the proceeding calculation we see that
		\begin{align*}
			B (z\otimes (x-q^{-n}y))
			&=Bz \otimes K^{-1}(x-q^{-n}y) + z \otimes B(x-q^{-n} y) \\
			&= [n] z \otimes (q^{-1} x - q^{-n+1}y) + z \otimes (y - q^{-n} x) \\
			&=[n-1] z \otimes (x-q^ny). 
		\end{align*}
		In the last equality we used the identities
		\[
		q^{-1}[n]-q^{-n}=[n-1], \quad -q^{-{n+1}}[n]+1=-q^{-n}[n-1].
		\]
		for quantum integers. These are equivalent after multiplying with $q-q^{-1}$ to
		\[
		q^{-1}(q^n-q^{-n})-q^{-n}=q^{n-1}-q^{-n+1}
		\]
		and
		\[
		-q^{-{n+1}}(q^n-q^{-n})+q-q^{-1}=-q^{-n}(q^{n-1}-q^{-n+1}),
		\]
		which again hold.
		\item We proceed by induction over $n$. Part \ref{part i of gl1 x gl1 decomp} gives the decomposition $V=L(1) \oplus L(-1)$, which is the case $n=1$. 
		We calculate the induction step using \ref{part ii of gl1 x gl1 decomp} and \ref{part iii of gl1 x gl1 decomp}:
		\begin{align*}
			V^{\otimes n} \otimes V &=   \left(\bigoplus_{k=0}^{n} {n \choose k} L([n-2k]) \right) \otimes V  \\
			&= \bigoplus_{k=0}^{n}  {n \choose k} (L([n-2k]) \otimes V) \\
			&=  \bigoplus_{k=0}^{n} {n \choose k} (L([n-2k]) \otimes (\shift{v_{n-2k}} \oplus \shift{w_{n-2k}})) \\
			&= \bigoplus_{k=0}^{n} {n \choose k} (L([n-2k+1] \oplus L([n-2k-1])) \\
			&= \bigoplus_{k=0}^{n+1} {n+1 \choose k}  L([n+1-2k]),
		\end{align*}
		where we used that $V$ has basis $\{v_{j}, w_{j} \}$ for any fixed $j=n-2k$ and reordered the summands in the last step. \qedhere
	\end{enumerate}
\end{proof}
Next we look at a small example to understand the notation from \Cref{Theorem decomposition of tensor power of V}. It also introduces some slight notation, where we label the one-dimensional summands by elements $\varepsilon\in \{1,-1\}^{n}$.
\begin{beis} \label{example with epsilon notation}
	The concrete decompositions into subrepresentations of $V^{\otimes 2}$ and $V^{\otimes 3}$ are
	\begin{align*}
		V^{\otimes 2}&=L([2])\oplus L(0)^2 \oplus L(-[2]) \\
		&= L(1,1) \oplus \left(L(1,-1) \oplus L(-1,1) \right) \oplus L(-1,-1)\\
		&=\shift {v_{0}\otimes v_1} \oplus \shift{v_{0}\otimes w_1, w_{0}\otimes v_{-1}} \oplus \shift{w_{0}\otimes w_{-1}}
	\end{align*}
	and
	\begin{align*}
		V^{\otimes 3}&=L([3])\oplus L([1])^3 \oplus L(-[1])^3 \oplus L(-[3]) \\
		&= L(1,1,1)\oplus \begin{pmatrix}
			L(1,1,-1) \\
			L(1,-1,1) \\
			L(-1,1,1)
		\end{pmatrix} \oplus 
		\begin{pmatrix}
			L(1,-1,-1) \\
			L(-1,1,-1) \\
			L(-1,-1,1)
		\end{pmatrix} 
		\oplus 
		L(-1,-1,-1)  \\
		&=\shift {v_{0}\otimes v_1 \otimes v_2} \oplus \begin{matrix}
			\langle v_0\otimes v_1 \otimes w_2,\phantom{\rangle} \\
			\phantom{\langle}v_0 \otimes w_1 \otimes v_0,\phantom{\rangle} \\
			\phantom{\langle}w_0 \otimes v_{-1} \otimes v_0
			\rangle
		\end{matrix} \oplus 
		\begin{matrix}
			\langle v_0 \otimes w_1 \otimes w_0,\phantom{\rangle} \\
			\phantom{\langle}w_0\otimes v_{-1} \otimes w_0,\phantom{\rangle} \\
			\phantom{\langle}w_0 \otimes w_{-1} \otimes v_{-2}
			\rangle
		\end{matrix} 
		\oplus 
		\shift{w_{0}\otimes w_{-1} \otimes w_{-2}}.
	\end{align*}
	As indicated any $\{1,-1\}$ sequence of length $n$ determines a direct summand in $L_\varepsilon \coloneqq L(\varepsilon)$ in $V^{\otimes n}$, which is one of the $L(|\varepsilon|)$ for $|\varepsilon|=|\varepsilon_1+\cdots+\varepsilon_n|$.
\end{beis}
The next example compares the classical Clebsch--Gordan decomposition to the decomposition of $V^{\otimes n}$.
\begin{beis}[Visualized tensor product decompositions] \label{vizualised tensor prod sl2 decompo}
	One can visualize the Clebsch--Gordan decomposition for $\Ug_q(\sl2)$ as
	\begin{equation*}
		% https://tikzcd.yichuanshen.de/#N4Igdg9gJgpgziAXAbVABwnAlgFyxMJZARgBoAGAXVJADcBDAGwFcYkRiQBfU9TXfIRTlSxanSat2nHn2x4CRAEyjxDFm0QduvEBnmCiZJWsmbtsvfwVDkAZlIma6qVpm79AxcNJ3TG6R05L1sVP2czdiUgqwNvZAAWX39XCw9rQxQyBJTzOxjPGyIHHIiArXzLQszkAFZSUolytOCin1rc9kr0uNDSDrLUgDYCjPikgabUhNHeoiH+zrducRgoAHN4IkouIA
		\begin{array}{c|ccccccccc}
			& & L(q^0) & L(q^1) & L(q^2) & L(q^3) & L(q^4) &L(q^5) & L(q^6) \\ \hline
			V^{\otimes 0} & & 1 &  &   &   &   &   &   \\
			V^{\otimes 1} & & & 1 &   &   &   &   &   \\
			V^{\otimes 2} & & 1 &   & 1 &   &   &   &   \\
			V^{\otimes 3} & & & 2 &   & 1 &   &   &  \\
			V^{\otimes 4} & & 2 &   & 3 &   & 1 &   &   \\
			V^{\otimes 5} & & & 5 &   & 4 &   & 1 &   \\
			V^{\otimes 6} & & 5 &   & 9 &   & 5 &   & 1.  \\
		\end{array}
	\end{equation*}
	This picture is constructed layer by layer just like Pascal's triangle, but bounded to the left. 
	In contrast the decomposition numbers for $V^{\otimes n}$ for $\Vq$ are much easier and given by the usual Pascal's triangle
	\begin{equation*}
		\begin{array}{c|cccccccc}
			& & L(-[3]) & L(-[2])	& L(-[1]) & L([0]) & L([1]) & L([2]) & L([3])									  \\ \hline
			V^{\otimes 0} & & &   &   & 1 &   &   &   \\
			V^{\otimes 1} & & &	  & 1 &   & 1 &   &   \\
			V^{\otimes 2} & & & 1 &   & 2 &   & 1 &   \\
			V^{\otimes 3} & & 1 &   & 3 &   & 3 &   & 1.
		\end{array}
	\end{equation*}
	The direct summands in the decomposition of $V^{\otimes n}$ for $\Ug_q(\sl2)$ are labelled by Dyck paths, i.e.\ $\{1,-1\}$ sequences $\varepsilon\in \{-1,1\}^n$, such that the sum of the first $k$ entries is always non-negative for all $k\leq n$, while the direct summands in the coideal decomposition correspond to all $\{1,-1\}$ sequences as explained in \Cref{example with epsilon notation}.
	Finally note that the decomposition for the coideal is omitting the action of $C=D_1D_2\in \Ug^{\iota}$, which acts constantly by $q^n$ on $V^{\otimes n}$. In particular the submodule $L(-[1])$ inside $V^{\otimes 1}$ is not isomorphic to any submodule of $V^{\otimes 3}$, since $C$ acts by $q$ on $V$ and by $q^3$ on $V^{\otimes 3}$. In order to get a correct notion of type $B$ intertwiners between different tensor powers, one should only consider morphisms which commute with $B=q^{-1}EK^{-1}+F$.
\end{beis}
Now with the full representation theoretic picture from \Cref{Theorem decomposition of tensor power of V} and the last remark in mind, we can prove the Schur--Weyl duality between the algebras $\Vq$ and $\Hecke_{1,q}(B_n)$.
\begin{proof}[Proof of Schur--Weyl duality for $m=1$]
	First we need to check that the action of $\Hecke_q(B_n)$ and the action of $\Vq$ on $V^{\otimes n}$ are semisimple.
	For the latter this is \Cref{Theorem decomposition of tensor power of V} and for the first note that $\Hecke_{1,q}(B_n)$ is a semisimple $\C(q)$-algebra. This follows by the following specialization argument for $\Hecke_{1,q}(B_n)$ at $q=1$. Let $x$ be contained in the Jacobson radical of $\Hecke_{1,q}(B_n)$, write it in terms of the standard basis and assume without loss of generality that all coefficients are in $\C[q]$ (else multiply $x$ with the product of all denominators in each coefficient). By evaluating at $q=1$ one obtains an element in the radical of $\C \! W(B_n)$, which is zero, since this algebra is semisimple by Maschke's theorem. Therefore $x$ is divisible by $1-q$ and hence all coefficients of $x$ are divisible by $1-q$. The same argument can be applied to $x/(1-q)$, to conclude that all coefficients of $x$ in the standard basis are divisible by $(1-q)^2$ and even by $(1-q)^n$ for all $n$ by induction. This shows that $x$ is zero, since we assumed that all coefficients are polynomials.
	
	In order to check that the actions commute, observe that the action of the subalgebra $\Hecke_q(A_{n-1})=\Hecke_q(S_n)$ of $\Hecke_q(B_n)$ generated by $H_1, \ldots, H_{n-1}$ inside $W(B_n)$ commutes with the action of $\Ug_q(\gl_2)$ as part of type $A$ Schur--Weyl duality. Hence we only need to check that $C$ and $B$ commute with the action of the additional generator $H_0=s_0$ of $\Hecke_q(B_n)$. For $C$ this is clear since it acts by $q^n$ on all of $V^{\otimes n}$.  For the generator $B$ we have $\Delta(B)=B\otimes K^{-1}+1\otimes B$, which gives by induction
	\[
	\Delta^{(n)}(B)=\sum_{i=1}^{n+1} 1\otimes \cdots \otimes \underset{i-1}{1} \otimes \underset{i}{B} \otimes \underset{i+1}{K^{-1}} \otimes \cdots \otimes K^{-1}\in \Ug^{\otimes {n+1}},
	\]
	where $\Delta^{(n)}$ is defined recursively via $\Delta^{(1)}=\Delta$ and
	\[
	\Delta^{(k+1)}=(\Delta \otimes \id)\circ (\Delta^{(k)})=(\id \otimes \Delta)\circ (\Delta^{(k)}).
	\] 
	Using this formula
	we calculate for a pure tensor $x=v_1\otimes \cdots \otimes v_n$
	\begin{align*}
		B(xH_0)&=B\left( Tv_1\otimes v_2 \otimes \cdots \otimes v_n \right)\\
		&= (T^2v_1)\otimes K^{-1}v_2 \otimes \cdots \otimes K^{-1}v_n + Tv_1 \otimes \biggl( \sum_{i=2}^n v_2 \otimes \cdots  \otimes  Tv_i\otimes \cdots\otimes K^{-1}v_n \biggr) \\
		&= \biggl( \sum_{k=1}^n v_1 \otimes \cdots  \otimes  Tv_k\otimes \cdots \otimes v_n \biggr) H_0=(Bx) H_0,
	\end{align*}
	where $T\colon V\to V, x\mapsto y, y\mapsto x$ is both the action of $H_0$ on the leftmost tensor factor as well as the action of $B$ on $V$. 
	We will show next that the map 
	$\Hecke_q(B_n)\rightarrow \End_{\Ug^{\iota}}(V^{\otimes n})$ is surjective. By the double centralizer theorem this would imply that the actions of $\Hecke_q(B_n)$ and $\Ug^{\iota}$ centralize each other.
	For the surjectivity we argue in the following way. First we observe that the action of $\Hecke_{1,q}(B_n)$ on $V^{\otimes n}$ factors through the Temperley--Lieb algebra $\TL(B_n)$ via the algebra homomorphism
	\begin{align*}
		\Hecke_{1,q}(B_n) &\rightarrow \TL(B_n) \\
		s_0 &\mapsto s_0 \\
		H_i &\mapsto U_i + q^{-1}
	\end{align*}
	which already appeared in type $A$ in \Cref{type A Hecke algebra}. Here $s_0\in \TL(B_n)$ acts in the same way like $s_0\in \Hecke_{1,q}(B_n)$ and the action of $U_i$ is given by the formulas in \Cref{fact graphical sl2 calculus}. Next we want to show that $\TL(B_n)$ acts faithfully on $V^{\otimes n}$. For this consider the following commutative diagram
	\[
		\begin{tikzcd}
			& (\TL(B_n))^{\op} \arrow[ld, "f"'] \arrow[rd, "g"]                &                               \\
			\End_{\Ug^{\iota}}(V^{\otimes n}) \arrow[r, hook] & \End_{\C(q)}(V^{\otimes n}) \arrow[r, "\cong"'] & {\Hom_{\C(q)}(\C(q), V^{\otimes 2n})}.
		\end{tikzcd}
	\]
	This diagram consists of the following ingredients. The left diagonal map $f$ is the algebra homomorphism, which defines the right module action of $\TL(B_n)$ on $V^{\otimes n}$. The bottom isomorphism is adjunction using an appropriate identification of $V$ with $V^{*}=\Hom_{\C(q)}(V, \C(q))$. The right diagonal morphism $g$ is the bending morphism, i.e.\ the morphism obtained by mapping a type $B$ Temperley--Lieb diagram $\lambda$ to the map $\C(q)\rightarrow V^{\otimes (2n)}$ defined by the bended diagram $\lambda'$. This bended diagram is obtained by moving all bottom vertices $\{1,\ldots, n\}\times \{0\}$ to the right of the top vertices, keeping planarity and moving possible dot on each arc to the left of the new bent arc. For instance 
	\[
		\lambda = \cbox
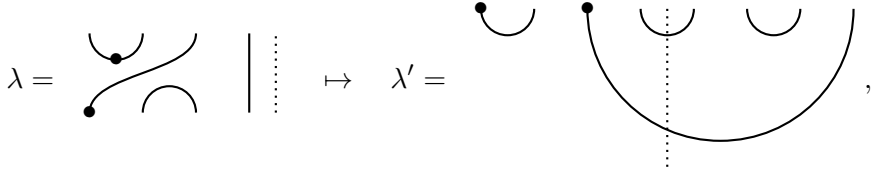
{
		\begin{tikzpicture}[tldiagram]
			\draw \tlcoord{0}{0} \onedot \linewave{1}{2};
			\draw \tlcoord{0}{1} \capright;
			\draw \tlcoord{1}{0} \dcupright;
			\draw \tlcoord{0}{3} \lineup;
			\draw[dotted] \tlcoord{0}{3.5} \lineup;
		\end{tikzpicture}
	} \quad \mapsto \quad \lambda' = 
	\cbox{
	\begin{tikzpicture}[tldiagram]
		\draw \tlcoord{1}{0} \onedot \cupright;
		\draw \tlcoord{1}{2} \onedot \xcupright{5};
		\draw \tlcoord{1}{3} \cupright;
		\draw \tlcoord{1}{5} \cupright;
		\draw[dotted] \tlcoord{1}{3.5} \linedown \linedown;
	\end{tikzpicture}
	},
	\]
	where the dotted line indicates where we are bending the diagram. Such a bent diagram can be now interpreted as a $\C(q)$-linear map $\C(q)\rightarrow V^{\otimes (2n)}$ by the formulas from \Cref{fact graphical sl2 calculus}. Importantly the image of $1\in \C(q)$ under $\lambda'$ allows to recover the diagram $\lambda$ one started with, since it is given by a $\C(q)$-linear combination of standard basis vectors in $V^{\otimes (2n)}$, such that one standard basis vector has coefficient $1$ and the others have coefficients in $\mZ[q]$. By incorporating an order on the $2^{2n}$ basis factors of $V^{\otimes (2n)}$ this argument shows, that $g$ is injective and hence $f$ is injective, since the diagram commutes. Then using \Cref{Theorem decomposition of tensor power of V} we have 
	\[
	\End_{\Ug^{\iota}}(V^{\otimes n})\cong \bigoplus_{k=0}^n\bigoplus_{j=0}^n \Hom(L([n-2k]),L([n-2j]))\cong \bigoplus_{k=0}^n \End (L([n-2k]))
	\]
	by Schur's lemma, since the quantum integers $[n-2k], [n-2j]\in \C(q)$ are different for $k\neq j$. Hence the dimension of the endomorphism algebra is
	\[
	\dim\End_{\Ug^{\iota}}(V^{\otimes n})= \sum_{k=0}^{n} {n \choose k}^2 =  {{2n} \choose n},
	\]
	which is the dimension of $\TL(B_n)$. This proves that $f$ is an isomorphism and hence the composition $\Hecke_{1,q}(B_n)\rightarrow \End_{\Ug^{\iota}}(V^{\otimes n})$ is surjective. This completes the proof.
\end{proof}

The next and final remark of the chapter sketches the connection between Schur--Weyl duality for type $B$ and classical skew Howe duality. The advantage of this point of view is the symmetry of the picture, in contrast to having two different algebras $\Ug_q'(\gl_n\times \gl_n)$ and $\Hecke_{1,q}(B_n)$ in the Schur--Weyl picture. It is based on [\cite[Theorem F]{stroppel2018}. Loosely speaking it can be summarized as type $B$ Schur--Weyl duality being squeezed in between two type $A$ Schur--Weyl dualities. Thinking like this makes makes it less surprising that the representation theory of $\Hecke_{1,q}(B_n)$ or the Temperley--Lieb algebra $\TL(B_n)$ resembles so much the type $A$ story. 
\begin{rema}
	Let $n\geq1, d\geq1$ and consider the quantum groups $\Ug_q(\gl_{2n})$ and $\Ug_q(\gl_{2d})$. Let $V=\C(q)^{2n}$ the standard representation of $\Ug_q(\gl_{n})$ and $W=\C(q)^m$ the standard representation of $\Ug_q(\gl_{2d})$.
	The quantum group $\Ug_q(\gl_{2n})$ contains the coideal subalgebra $\Ug_q'(\gl_{n}\times \gl_{n})$, which contains a regular quantum group $\Ug_q(\gl_{n})$. Similarly $\Ug_q(\gl_{2d})$ contains the coideal subalgebra $\Ug_q'(\gl_{d}\times \gl_{d})$, which contains again a regular quantum group $\Ug_q(\gl_{d})$. These algebras fit in to the following pictures, where the actions algebras on the left and right in each row commute and centralize each other:
	\begin{equation} \label{skew howe duality diagram}
		\begin{tikzcd}
			\Ug_q(\gl_{2n})                         &[-25pt] \adjustbox{scale={2}{2}}{$\circlearrowright$} &[-25pt] \bigwedge^d(\C(q)^{2n}\boxtimes \C(q)^d) \arrow[d, "\cong"]    &[-25pt] \adjustbox{scale={2}{2}}{$\circlearrowleft$} &[-25pt] \Ug_q(\gl_{d}) \arrow[d, hook]             \\
			\Ug_q'(\gl_{n}\times \gl_{n}) \arrow[u, hook] & \adjustbox{scale={2}{2}}{$\circlearrowright$} & \bigwedge^d((\C(q)^{n}\boxtimes \C(q)^d)^2) \arrow[d, "\cong"] & \adjustbox{scale={2}{2}}{$\circlearrowleft$} & \Ug_q'(\gl_{d}\times \gl_{d}) \arrow[d, hook] \\
			\Ug_q(\gl_{n})\arrow[u, hook]             & \adjustbox{scale={2}{2}}{$\circlearrowright$} & \bigwedge^d(\C(q)^{n}\boxtimes \C(q)^{2d})                     & \adjustbox{scale={2}{2}}{$\circlearrowleft$} & \Ug_q(\gl_{2d})                       
		\end{tikzcd}
	\end{equation}
	In the top and bottom row we have the quantum skew Howe duality for two quantum groups of $\gl_n$. It gives that the representation $\bigwedge^d(\C(q)^{2n}\boxtimes \C(q)^d)$ decomposes when viewed as representation of $\Ug_q(\gl_{2n})$ as
	\begin{equation*}
		\bigwedge^d(\C(q)^{2n}\boxtimes \C(q)^d)=\bigoplus_{
			\begin{array}{c}
				\underline{a}=(a_1,\ldots, a_d)\\
				\sum a_i=d
			\end{array}
		}\bigwedge^{a_1}V\otimes \cdots \otimes \bigwedge^{a_d}V.
	\end{equation*}
	When restricting the action of $\Ug_q(\gl_{2n})$ to $\Ug_q'(\gl_n\times \gl_n)$ and to the direct summand 
	\[
	V^{\otimes d}=\bigwedge^1 V\otimes\cdots\otimes\bigwedge^1 V
	\] this recovers the Schur--Weyl duality between $\Ug_q'(\gl_n\times \gl_n)$ and $\Hecke_{1,q}(B_d)$, where $\Hecke_{1,q}(B_d)$ can be realized as a subalgebra $\Ug_q'(\gl_d\times \gl_d)$. This picture inspired the recently published paper \cite{tubbenhauer2022}, which can be understood as giving a web calculus for the middle row. However in contrast to this thesis they work with Levi subalgebras and not coideal subalgebras of the quantum group.
\end{rema}

	\chapter{Jones--Wenzl projectors of type \texorpdfstring{$B$}{B} and \texorpdfstring{$D$}{D}} \label{Chapter 3}

%One should fix this!
In this chapter we discuss generalizations of Jones--Wenzl projectors from type $A$, which we briefly mentioned in \Cref{Chapter 2}, to types $B$ and $D$. In \Cref{Chapter 3 section 1} we define the type $B$ Jones--Wenzl projectors -- a positive type $B$ projector $\bp$ and a negative projector $\bn$ -- recursively, and interpret them as idempotent projections onto $L([n])$ (respectively $L(-[n])$) of $V^{\otimes n}$ in the decomposition from \Cref{Theorem decomposition of tensor power of V}. We prove a algebraic criterion which characterizes these idempotents inherently in the algebra $\TL(B_n)\cong\End_{\Vq}(V^{\otimes n})$, which is the key for this interpretation. In \Cref{Chapter 3 section 2} we define the Jones--Wenzl projectors $d_n$ of type $D$, which can be interpreted as projections onto $L([n])\oplus L(-[n])$ and satisfy nicer integrality properties. \Cref{Chapter 3 section on higher projectors} treats the so-called higher Jones--Wenzl projectors, which will be defined recursively and can be interpreted as the projections onto all the irreducible direct summands in the decomposition of $V^{\otimes n}$. After we finished these three most important sections, we discussion Reidemeister-like relations for the type $B$ Temperley--Lieb algebra in \Cref{section on Reidemeister moves}, and end the chapter with a comparison of type $B$ Temperley--Lieb diagrams with affine diagrams from \cite{wedrich2018} in \Cref{affine section}.  

\section{Type \texorpdfstring{$B$}{B} Jones--Wenzl projectors} \label{Chapter 3 section 1}

In this section we define Jones--Wenzl projectors in the Temperley--Lieb algebra $\TL(B_n)$ for $n\geq 1$ and prove their basic properties. Recall that our ground field is $\C(q)$ and the value of the undotted circle is $\tlcircle=[2]=q+q^{{-}1}$.

\begin{defi}[Projectors of type $B$] \label{defi type B projectors}
	For $n=1$ we set
	\begin{equation*}
		b_{1,{+}}=\frac{1}{2}(1+s_0), \quad b_{1,{-}}=\frac{1}{2}(1-s_0).
	\end{equation*}
	Assume that $b_{n, {+}}$ and its counter-part $b_{n, {-}}$ were already defined. Then we define
	\begin{equation*}
		b_{n+1, {+}}=b_{n,{+}} - \frac{q^{n-1} + q^{{-}(n-1)}}{q^n + q^{{-}n}}b_{n,{+}} U_n b_{n,{+}},
	\end{equation*}
	which we draw diagrammatically as
	\[
	\cbox{
		\begin{tikzpicture}[tldiagram, yscale=1/2]
			% lines
			\draw \tlcoord{-2.5}{0} \lineup \lineup \lineup \lineup \lineup;
			\draw \tlcoord{-2.5}{1} \lineup \lineup \lineup \lineup \lineup;
			\draw \tlcoord{-2.5}{2} \lineup \lineup \lineup \lineup \lineup;
			% boxes
			\draw \tlcoord{0}{0} \maketlboxred{3}{$n+1, {+}$};
			=
		\end{tikzpicture}
	}
	\quad=\quad
	\cbox{
		\begin{tikzpicture}[tldiagram, yscale=1/2]
			% lines
			\draw \tlcoord{-2.5}{0} \lineup \lineup \lineup \lineup \lineup;
			\draw \tlcoord{-2.5}{1} \lineup \lineup \lineup \lineup \lineup;
			\draw \tlcoord{-2.5}{2} \lineup \lineup \lineup \lineup \lineup;
			% boxes
			\draw \tlcoord{0}{0} \maketlboxred{2}{$n, {+}$};
			=
		\end{tikzpicture}
	}
	\quad-\quad
	\frac{q^{n-1} + q^{{-}(n-1)}}{q^n + q^{{-}n}}
	\cbox{
		\begin{tikzpicture}[tldiagram, yscale=1/2]
			% lines
			\draw \tlcoord{-1}{0} \lineup \lineup \lineup \lineup \lineup;
			\draw \tlcoord{3}{1} \lineup;
			\draw \tlcoord{0}{1} \linedown;
			% paths
			\draw \tlcoord{0}{1} \smalllineup \capright \smalllinedown \linedown;
			\draw \tlcoord{3}{1} \smalllinedown \cupright \smalllineup \lineup;
			% boxes
			\draw \tlcoord{0}{0} \maketlboxred{2}{$n, {+}$};
			\draw \tlcoord{3}{0} \maketlboxred{2}{$n, {+}$};
		\end{tikzpicture}
	} .
	\]
	Analogously, we define
	\begin{equation*}
		b_{n+1, {-}}=\bn - \frac{q^{n-1} + q^{{-}(n-1)}}{q^n + q^{{-}n}}b_{n,{-}} U_n b_{n,{-}},
	\end{equation*}
	which we draw diagrammatically as
	\[
	\cbox{
		\begin{tikzpicture}[tldiagram, yscale=1/2]
			% lines
			\draw \tlcoord{-2.5}{0} \lineup \lineup \lineup \lineup \lineup;
			\draw \tlcoord{-2.5}{1} \lineup \lineup \lineup \lineup \lineup;
			\draw \tlcoord{-2.5}{2} \lineup \lineup \lineup \lineup \lineup;
			% boxes
			\draw \tlcoord{0}{0} \maketlboxblue{3}{$n+1, {-}$};
			=
		\end{tikzpicture}
	}
	\quad=\quad
	\cbox{
		\begin{tikzpicture}[tldiagram, yscale=1/2]
			% lines
			\draw \tlcoord{-2.5}{0} \lineup \lineup \lineup \lineup \lineup;
			\draw \tlcoord{-2.5}{1} \lineup \lineup \lineup \lineup \lineup;
			\draw \tlcoord{-2.5}{2} \lineup \lineup \lineup \lineup \lineup;
			% boxes
			\draw \tlcoord{0}{0} \maketlboxblue{2}{$n, {-}$};
			=
		\end{tikzpicture}
	}
	\quad-\quad
	\frac{q^{n-1} + q^{{-}(n-1)}}{q^n + q^{{-}n}}
	\cbox{
		\begin{tikzpicture}[tldiagram, yscale=1/2]
			% lines
			\draw \tlcoord{-1}{0} \lineup \lineup \lineup \lineup \lineup;
			\draw \tlcoord{3}{1} \lineup;
			\draw \tlcoord{0}{1} \linedown;
			% paths
			\draw \tlcoord{0}{1} \smalllineup \capright \smalllinedown \linedown;
			\draw \tlcoord{3}{1} \smalllinedown \cupright \smalllineup \lineup;
			% boxes
			\draw \tlcoord{0}{0} \maketlboxblue{2}{$n, {-}$};
			\draw \tlcoord{3}{0} \maketlboxblue{2}{$n, {-}$};
		\end{tikzpicture}
	} .
	\]
	 We call $b_{n,{+}}$ (resp.\ $b_{n,{-}}$) the \textbf{positive} (respectively \textbf{negative}) \textbf{Jones--Wenzl projector of type B}.
\end{defi}

\begin{warn} %add a reference here
	Recall that we work with $V=L(q)$, the type I standard representation of $\Ug_q(\gl_2)$. One has to be careful when considering these formulas as endomorphisms of $V^{\otimes n}$, since
	\[
		\End_{\Vq}(V^{\otimes n})\cong \TL(B_n,\tlcircle=-[2]),
	\]
	i.e.\ the value of the undotted circle for the type I representation $V$ is $-[2]$. If one wants to interpret the formulas inside $\End_{\Vq}(V^{\otimes n})$, one needs to first use the isomorphism
	\begin{align*}
		\TL(B_n, \tlcircle=[2]) &\rightarrow \TL(B_n,\tlcircle=-[2]) \\
		s_0 &\mapsto s_0 \\
		U_i &\mapsto -U_i \text{ for } i=1,\ldots, n-1
	\end{align*}
	of $\C(q)$-algebras. In particular the second summand in the recursive formulas for $\bp$ respectively $\bn$ has coefficient $\frac{q^{n-1}+q^{-(n-1)}}{q^n+q^{-n}}$ instead of $-\frac{q^{n-1}+q^{-(n-1)}}{q^n+q^{-n}}$. A better solution then using the proposed isomorphism of Temperley--Lieb algebras is to work with tensor powers of the type II standard representation of $\Ug_q(\gl_2)$. Concretely it is given by enhancing the action of $\Ug_q(\sl2)$ from \Cref{type II corr formulas} to
	\begin{equation*}
		% https://tikzcd.yichuanshen.de/#N4Igdg9gJgpgziAXAbVABwnAlgFyxMJZABgBpiBdUkANwEMAbAVxiRAE8QBfU9TXfIRQBGclVqMWbAB7dxMKAHN4RUADMAThAC2SMiBwQkokACMYYKEgC0AZn31mrRCGEhqDOuYYAFfngI2DSxFAAscbl4QTR1jakM4kAYICDQiYQAOMjVGOBhxT28-bAChEGCwiOpHKRcAR0j1LV1EfQTWjxS0kSzSHIY8gq8YX39BIJDw9wknNms3HibYxBN2-XNLG3tqyWdXacKR4oFAlwqpxejmvXijDqSuomzc-I9h0ZLxs8mqmdqQBqXGItVZ3EzJVJPPovIZFManco-aY1PYLChcIA
		\begin{tikzcd}
			y \arrow[r, "1"', bend right, red] \arrow["-1"', loop, distance=2em, in=125, out=55, green] \arrow["q"', loop, distance=2em, in=305, out=235, orange] & x \arrow["-q"', loop, distance=2em, in=125, out=55, green] \arrow[l, "-1"', bend right, blue] \arrow["1"', loop, distance=2em, in=305, out=235, orange]
		\end{tikzcd},
	\end{equation*} 	
	where the action of $\textcolor{green}{D_1}$ is visualized \textcolor{green}{green}, the action of $\textcolor{orange}{D_2}$ is \textcolor{orange}{orange}, the action of $\textcolor{red}{E}$ is \textcolor{red}{red} and the action of $\textcolor{blue}{F}$ is \textcolor{blue}{blue}. One can adjust the results from \Cref{Chapter 2} to tensor powers of the type II representations. The formulas for the Jones--Wenzl projectors above are correct with respect to this notion. For more on this topic see \Cref{discussion on braidings i} and the discussion afterwards.
\end{warn}

The next theorem is the reason why we call the defined elements Jones--Wenzl projectors. It gives just like \Cref{char property type A projector} in type $A$, an internal characterization of the elements $\bp$ and $\bn$ in the algebra $\TL(B_n)$. Loosely speaking as the unique idempotents which kill all ``cup-cap'' generators of the Temperley--Lieb algebra and are killed by multiplication with $1-s_0$ respectively $1+s_0$. Such uniqueness theorems are useful to recognize these idempotents in different contexts including categorification. The proof is heavily inspired by the proof of \cite[Lemma 2]{lickorish1991}.
\begin{theo}[Characterizing properties of type $B$ Jones--Wenzl projectors] \label{Theorem on Existence of JW projectors of type B}
 \leavevmode
	\begin{enumerate}
		\item The element $b_{n,{+}}$ is the unique non-zero idempotent in $\dotTL$ which satisfies
		\[
			s_0b_{n,{+}}=b_{n,{+}}=b_{n,{+}}s_0 \quad \text{and} \quad U_ib_{n,{+}}=0=b_{n,{+}}U_i \, \text{ for all } i=1,\ldots, n-1 .
		\]
		\label{defining properties of positive type B JW} 
		\item The element $b_{n,{-}}$ is the unique non-zero idempotent in $\dotTL$ which satisfies
		\[
		s_0b_{n,{-}}=-b_{n,{-}}=b_{n,{-}}s_0 \quad \text{and} \quad U_ib_{n,{-}}=0=b_{n,{-}}U_i \, \text{ for all } i=1,\ldots, n-1.
		\] \label{defining properties of negative type B JW}
	\end{enumerate}
\end{theo}
\begin{proof}
	We just prove \ref{defining properties of positive type B JW}, because the proof for \ref{defining properties of negative type B JW} is completely analogous if one replaces the generator $s_0$ by ${-}s_0$. 
	We start by showing the uniqueness of $b_{n,{+}}$.
	For this observe that as a vector-space we have
	\begin{equation} \label{nicuu direct sum decomposition}
		\dotTL={\shift{1, s_0}} \oplus (U_1,\ldots,U_{n-1}),
	\end{equation}
	 where $(U_1,\ldots,U_{n-1})$ denotes the two-sided ideal generated by the $U_i$. This is a consequence of the defining relations of the Temperley--Lieb algebra or alternatively follows from \Cref{Temperley--Lieb algebra of type B is diagram algebra}. Now assume $x,x'\in\dotTL$ are two elements which satisfy the properties from \ref{defining properties of positive type B JW} and write 
	 \begin{equation*}
	 	x=\lambda \cdot 1 + \mu s_0 + r, \quad x'=\lambda' \cdot 1 + \mu' s_0 + r'
	 \end{equation*}
 	for some $\lambda,\lambda', \mu, \mu'\in k$ and $r,r'\in(U_1,\ldots,U_{n-1})$. Because $x$ is fixed under multiplication with $s_0$ and $s_0$ is self-inverse, we obtain that
	\begin{equation*}
		\lambda \cdot 1 + \mu s_0 + r= x = s_0 x = \mu \cdot 1 + \lambda s_0 + s_0 r .
	\end{equation*}
	Hence by comparing coefficients we obtain $\lambda=\mu$ and $s_0 r= r$ as well as $\lambda'=\mu'$ and $s_0 r'= r'$. Now because $x,x'$ annihilate all $U_i$ and are fixed by $s_0$ from either side, they lie automatically in the center of our algebra $\dotTL$. Hence $x,x'$ annihilate all elements in the two-sided ideal $(U_1,\ldots,U_{n-1})$, in particular both $r$ and $r'$, which lie inside this ideal. Hence we observe that
	\begin{equation*}
		x\cdot x'= (\lambda \cdot 1 + \lambda s_0 + r)\cdot x' =\lambda (1 + s_0)\cdot x' = 2\lambda x'.
	\end{equation*} 
	Hence, again by using the centrality of $x$ and $x'$, we obtain
	\begin{equation} \label{niccuuu step}
		2 \lambda' x = x'\cdot x =x \cdot x'= 2 \lambda x'.
	\end{equation}
	By \eqref{nicuu direct sum decomposition} we know that $\lambda, \lambda'\neq 0$, since $x, x'$ are both non-zero idempotents by assumption. Hence \eqref{niccuuu step} shows that $\lambda=\frac{1}{2}=\lambda'$. From this we conclude $x=x'$ vy using \eqref{niccuuu step}, which completes the proof of the uniqueness.
	
	Next we show that $\bp$ satisfies the two properties from \ref{defining properties of positive type B JW}. By construction it is clear that $\bp$ is of the form $\frac{1}{2}(1+s_0) + r$ where $r\in (U_1,\ldots,U_{n-1})$. In particular $\bp$ is non-zero. Additionally it is fixed by $s_0$ for any $n$, since this holds for $b_{1,{+}}$ and $\bp$ lives in $b_{1,{+}}\TL(B_n)b_{1,{+}}$ by the recursive formula. In order to show that $\bp$ satisfies the other desired properties, we do a simultaneous induction over the following four statements:
	\begin{enumerate}[label = (\Roman*\textsubscript{$n$}), leftmargin=4\parindent]
		\item $b_n\coloneqq b_{n,{+}}$ is an idempotent.
		\item $b_n$ annihilates $U_j$, if $1\leq j \leq n-1$.
		\item $(U_{n}b_n)^2=\frac{q^{n}+ q^{{-}n}}{q^{n-1}+ q^{{-}(n-1)}}U_{n}b_n$ .
		\item $(b_nU_n)^2=\frac{q^{n}+ q^{{-}n}}{q^{n-1}+ q^{{-}(n-1)}}b_nU_n$.
	\end{enumerate} 
	We first check the base case $n=1$.
	Property I\textsubscript{$1$} is immediate from $s_0$ being self-inverse and II\textsubscript{$1$} holds vacuously. Property III\textsubscript{$1$} follows from \ref{TL1 relation} and \ref{TLdot2 relation} as follows:
	\begin{align*}
		(U_1 b_1)^2&= U_1 \left(\frac{1}{2}(1 + s_0)\right)\cdot U_1 \left(\frac{1}{2}(1 + s_0)\right) \\
		&= \frac{1}{4}U_1^2 + \frac{1}{4}U_1s_0U_1 + \frac{1}{4}U_1U_1s_0 + \frac{1}{4}U_1s_0U_1s_0 \\
		&= \frac{q+q^{{-}1}}{4}U_1 + 0 +\frac{q+q^{{-}1}}{4}U_1s_0 + 0 \\
		&= \frac{q+q^{{-}1}}{2}U_1 b_1 \\
		&= \frac{{q+q^{{-}1}}}{q^0 + q^{-0}} U_1 b_1 .
	\end{align*}
	Property IV\textsubscript{$1$} is obtained in a similar fashion. Assume that properties I\textsubscript{$n$}, II\textsubscript{$n$}, III\textsubscript{$n$}, IV\textsubscript{$n$} all hold for a fixed $n$.
	We show I$_{n+1}$ first. We see that
	\begin{align*}
		b_{n+1}^2
		&= \left( b_n-\frac{q^{n-1} - q^{{-}(n-1)}}{q^n + q^{{-}n}} b_nU_{n}b_n \right)^2 \\
		&= b_n -2\frac{q^{n-1} + q^{{-}(n-1)}}{q^n + q^{{-}n}}b_nU_{n}b_n+\left( \frac{q^{n-1} + q^{{-}(n-1)}}{q^n + q^{{-}n}}  \right)^2 b_n(U_nb_n)(U_nb_n)
		\tag{I\textsubscript{$n$}} \\
		&=
		b_n-2\frac{q^{n-1} + q^{{-}(n-1)}}{q^n + q^{{-}n}}b_nU_{n}b_n+\frac{q^{n-1} + q^{{-}(n-1)}}{q^n + q^{{-}n}}b_nU_{n}b_n
		\tag{III\textsubscript{$n$}} \\
		&=b_{n+1} \,.
	\end{align*}
	Next we show II$_{n+1}$. Consider $U_j$ for an index $j=1,\ldots, n$. For $j<n$ we can use ~II\textsubscript{$n$} and observe that
	\[
	U_jb_{n+1} = U_jb_n+\frac{q^{n-1} - q^{{-}(n-1)}}{q^n + q^{{-}n}}U_jb_nU_{n}b_n=0-0=0
	\]
	and analogously $b_{n+1}U_j=0$. For $j=n$ we have
	\[
	U_nb_{n+1} = U_nb_n-\frac{q^{n-1} + q^{{-}(n-1)}}{q^n + q^{{-}n}}(U_nb_n)(U_{n}b_n) 
	\xlongequal{\!\!\text{III\textsubscript{$n$}}\!\!}0
	\]
	and analogously we obtain $b_{n+1}U_n=0$ by IV\textsubscript{$n$}. Now we show III$_{n+1}$. We calculate
	\begin{align*}
		(U_{n+1}b_{n+1})^2
		&=U_{n+1}\left( b_n-\frac{q^{n-1} + q^{{-}(n-1)}}{q^n + q^{{-}n}}b_nU_{n}b_n \right)U_{n+1}b_{n+1} \\
		&=U_{n+1}b_{n}U_{n+1}b_{n+1}-\frac{q^{n-1} + q^{{-}(n-1)}}{q^n + q^{{-}n}}U_{n+1}b_nU_nb_nU_{n+1}b_{n+1} \\
		&=U_{n+1}^2b_{n+1}-\frac{q^{n-1} + q^{{-}(n-1)}}{q^n + q^{{-}n}}b_nU_{n+1}U_nU_{n+1}b_{n+1}
		\tag{TL3, BTL3}\\
		&=(q+q^{-1}) U_{n+1}b_{n+1}-\frac{q^{n-1} + q^{{-}(n-1)}}{q^n + q^{{-}n}}U_{n+1}b_{n+1}
		\tag{TL1,TL2}\\
		&=\left( (q+q^{-1}) - \frac{q^{n-1} + q^{{-}(n-1)}}{q^n + q^{{-}n}} \right)U_{n+1}b_{n+1} \\
		&=\frac{q^{n+1}+q^{n-1} + q^{{-}(n-1)}+q^{-(n+1)} - q^{n-1} - q^{{-}(n-1)}}{q^n + q^{{-}n}}U_{n+1}b_{n+1} \\
		&=\frac{q^{n+1} + q^{{-}(n+1)}}{q^n + q^{{-}n}}U_{n+1}b_{n+1} \,,
	\end{align*}
	We have used that the elements $b_n$ und $U_{n+1}$ commute by \ref{TL3 relation} and \ref{TLdot3 relation} and that $b_n\cdot b_{n+1}=b_{n+1}$ by I\textsubscript{$n$} and the definition of $b_{n+1}$.
	%				Bei der letzten Gleichheit haben wir verwendet, dass 
	%				\begin{equation}
	%          \label{quantengleichheit}
	%					\delta+\frac{T_n}{T_{n+1}}=-2+\frac{q^{n-1} + q^{{-}(n-1)}}{q^n + q^{{-}n}}=-\frac{n+2}{n+1}=-\frac{T_{n+2}}{T_{n+1}}
	%				\end{equation}
	Property IV$_{n+1}$ follows completely analogously.
	\qedhere
\end{proof}
\begin{rema}
	Just as for the usual Jones--Wenzl projector the proof of uniqueness in \Cref{Theorem on Existence of JW projectors of type B} gives an even stronger statement. Any element which kills the generators $U_1,\ldots, U_{n-1}$ of $\TL(B_n)$ and is killed by $1-s_0$ respectively $1+s_0$ is a scalar multiple of the Jones--Wenzl projector $\bp$ respectively $\bn$. The elements $\bp$ respectively $\bn$ are the unique non-zero idempotents in these $1$-parameter families of central elements.
\end{rema}
The recursive formulas for the type $B$ projectors in \Cref{defi type B projectors} appeared without a Lie theoretic interpretation in \cite{tomdieck1998} and in a non-quantized form in \cite{wedrich2018} as idempotents projecting onto the maximal and lowest weight space of the standard representation of $\sl2$. The next corollary yields a lift of this interpretation to our quantum setting and gives in particular a Lie theoretic interpretation of the formulas in \cite{tomdieck1998}.
\begin{koro}
	Consider the tensor product decomposition of $V^{\otimes n}$ over $\Vq$ in \Cref{Theorem decomposition of tensor power of V}. The positive Jones--Wenzl projector $\bp$ can be identified with the idempotent projection $V^{\otimes n}\rightarrow L([n])\hookrightarrow V^{\otimes n}$. Similarly the negative Jones--Wenzl projector $\bn$ can be identified with the idempotent projection $V^{\otimes n}\rightarrow L(-[n])\hookrightarrow V^{\otimes n}$.
\end{koro}
\begin{proof}
	Throughout this proof we identify $\End_{\Vq}(V^{\otimes n})$ and $\TL(B_n)$.
	By the explicit decomposition in \Cref{Theorem decomposition of tensor power of V} the maximal weight submodule $L([n])$ is spanned by $v_0\otimes \cdots \otimes v_{n-1}$, while the lowest weight submodule is spanned by $w_0\otimes \cdots\otimes w_{-(n-1)}$, where $v_0=x+y$ and $w_0=x-y$. In particular $s_0$, which acts by swapping $x$ and $y$ in the first tensor factor, acts on by $1$ on $L([n])$ and by $-1$ on $L([-n])$. Moreover $[n]$ and $-[n]$ do not appear as eigenvalues of $B=q^{-1}EK^{-1}+F$ acting on $V^{\otimes l}$ by \Cref{Theorem decomposition of tensor power of V} for all $l$ strictly smaller then $n$. This shows that all morphisms $V^{\otimes n}\rightarrow V^{\otimes l}$ commuting with the action of $B=q^{-1}EK^{-1}+F$ are zero, when composed with the projection onto $L(-[n])$ respectively $L(-[n])$. In particular this applies to the morphisms $U_i=\tlline \cdots \tlline \tlcupcap \tlline \cdots \tlline\colon V^{\otimes n}\rightarrow V^{\otimes (n-2)}\rightarrow V^{\otimes n}$. This proves that the projections onto $L([n])$ respectively $L(-[n])$ satisfy the characterizing properties of $\bp$ respectively $\bn$ and hence agree with these elements by the uniqueness in \Cref{Theorem on Existence of JW projectors of type B}.
\end{proof}
Next we consider an example to see, what idempotents we just constructed.
\begin{beis} \label{n=2 type B projectors}
	The projectors for $n=1$ are given diagrammatically by
	\[
		\cbox{
			\begin{tikzpicture}[tldiagram, yscale=4/7]
				\draw \tlcoord{-1.5}{0} \lineup \lineup \lineup;
				\draw \tlcoord{0}{0} \maketlboxred{1}{$1, {+} $};
			\end{tikzpicture}
		} = \frac{1}{2} \left( \cbox{
			\begin{tikzpicture}[tldiagram, yscale=1/3]
				\draw \tlcoord{-1.5}{0} \lineup \lineup \lineup;
				%\draw \tlcoord{0}{0} \maketlboxred{1}{$1 $};
			\end{tikzpicture}
		} + \cbox{
			\begin{tikzpicture}[tldiagram, yscale=1/3]
				\draw \tlcoord{-1.5}{0} \lineup \dlineup \lineup;
				%\draw \tlcoord{0}{0} \maketlboxred{1}{$1 $};
			\end{tikzpicture}
		} \right) \, \text{and} \quad \cbox{
		\begin{tikzpicture}[tldiagram, yscale=4/7]
			\draw \tlcoord{-1.5}{0} \lineup \lineup \lineup;
			\draw \tlcoord{0}{0} \maketlboxblue{1}{$1, {-} $};
		\end{tikzpicture}
	} = \frac{1}{2} \left( \cbox{
	\begin{tikzpicture}[tldiagram, yscale=1/3]
		\draw \tlcoord{-1.5}{0} \lineup \lineup \lineup;
		%\draw \tlcoord{0}{0} \maketlboxred{1}{$1 $};
	\end{tikzpicture}
} - \cbox{
\begin{tikzpicture}[tldiagram, yscale=1/3]
	\draw \tlcoord{-1.5}{0} \lineup \dlineup \lineup;
	%\draw \tlcoord{0}{0} \maketlboxred{1}{$1 $};
\end{tikzpicture}
} \right) 
	\]
	Using the recursive formula we obtain that the positive Jones--Wenzl projector $b_{2,+}$ is given by
	\begin{align*}
		b_{2,+}&=\frac{1}{2}\left( \cbox{
			\begin{tikzpicture}[tldiagram]
				%dots
				%\drawdots{0}{0}{1}
				%\drawdots{1}{0}{1}
				%lines
				\draw \tlcoord{0}{0} \lineup;
				\draw \tlcoord{0}{1} \lineup;
			\end{tikzpicture}
		} + \cbox{
			\begin{tikzpicture}[tldiagram]
				%dots
				%\drawdots{0}{0}{1}
				%\drawdots{1}{0}{1}
				%lines
				\draw \tlcoord{0}{0} \dlineup;
				\draw \tlcoord{0}{1} \lineup;
			\end{tikzpicture}
		}\right) -	\frac{1}{2[2]} \left( 
		\cbox{
			\begin{tikzpicture}[tldiagram]
				%dots
				%\drawdots{0}{0}{1}
				%\drawdots{1}{0}{1}
				%lines
				\draw \tlcoord{0}{0} \capright;
				\draw \tlcoord{1}{0} \cupright;
			\end{tikzpicture}
		} + \cbox{
			\begin{tikzpicture}[tldiagram]
				%dots
				%\drawdots{0}{0}{1}
				%\drawdots{1}{0}{1}
				%lines
				\draw \tlcoord{0}{0} \capright;
				\draw \tlcoord{1}{0} \dcupright;
			\end{tikzpicture}
		}
		+ \cbox{
			\begin{tikzpicture}[tldiagram]
				%dots
				%\drawdots{0}{0}{1}
				%\drawdots{1}{0}{1}
				%lines
				\draw \tlcoord{0}{0} \dcapright;
				\draw \tlcoord{1}{0} \cupright;
			\end{tikzpicture}
		}
		+ \cbox{
			\begin{tikzpicture}[tldiagram]
				%dots
				%\drawdots{0}{0}{1}
				%\drawdots{1}{0}{1}
				%lines
				\draw \tlcoord{0}{0} \dcapright;
				\draw \tlcoord{1}{0} \dcupright;
			\end{tikzpicture}
		}
		\right).
	\end{align*}
	Similarly one can calculate that the negative projector is given by
	\begin{align*}
		b_{2,-}=\frac{1}{2}\left( \cbox{
			\begin{tikzpicture}[tldiagram]
				%dots
				%\drawdots{0}{0}{1}
				%\drawdots{1}{0}{1}
				%lines
				\draw \tlcoord{0}{0} \lineup;
				\draw \tlcoord{0}{1} \lineup;
			\end{tikzpicture}
		} - \cbox{
			\begin{tikzpicture}[tldiagram]
				%dots
				%\drawdots{0}{0}{1}
				%\drawdots{1}{0}{1}
				%lines
				\draw \tlcoord{0}{0} \dlineup;
				\draw \tlcoord{0}{1} \lineup;
			\end{tikzpicture}
		}\right) -	\frac{1}{2[2]} \left( 
		\cbox{
			\begin{tikzpicture}[tldiagram]
				%dots
				%\drawdots{0}{0}{1}
				%\drawdots{1}{0}{1}
				%lines
				\draw \tlcoord{0}{0} \capright;
				\draw \tlcoord{1}{0} \cupright;
			\end{tikzpicture}
		} - \cbox{
			\begin{tikzpicture}[tldiagram]
				%dots
				%\drawdots{0}{0}{1}
				%\drawdots{1}{0}{1}
				%lines
				\draw \tlcoord{0}{0} \capright;
				\draw \tlcoord{1}{0} \dcupright;
			\end{tikzpicture}
		}
		- \cbox{
			\begin{tikzpicture}[tldiagram]
				%dots
				%\drawdots{0}{0}{1}
				%\drawdots{1}{0}{1}
				%lines
				\draw \tlcoord{0}{0} \dcapright;
				\draw \tlcoord{1}{0} \cupright;
			\end{tikzpicture}
		}
		+ \cbox{
			\begin{tikzpicture}[tldiagram]
				%dots
				%\drawdots{0}{0}{1}
				%\drawdots{1}{0}{1}
				%lines
				\draw \tlcoord{0}{0} \dcapright;
				\draw \tlcoord{1}{0} \dcupright;
			\end{tikzpicture}
		}
		\right).
	\end{align*}
	Notably this is the same expression as for $b_{2,+}$ with the difference that the coefficient in front of each diagram basis element $\lambda$ is multiplied by $(-1)^{\#\text{dots on } \lambda}$, giving a nice symmetry between $\bp$ and $\bn$. We phrase this symmetry conceptually via the existence of the following involution of $\TL(B_n)$ in the next lemma.
\end{beis}
	 %Sollten das als Involution formulieren von Darstellungen von S(\lambda, \mu)
	\begin{lemm} \label{Type B automorphism s_o flips sign} % Type B Involution 
		There is a well-defined $k$-linear algebra involution of $\dotTL$ defined by
		\begin{align*}
			\Phi\colon \dotTL &\to \dotTL \\
			s_0 &\to -s_0 \\
			U_i &\to U_i \quad  \text{ where } i=1,\ldots,n-1,
		\end{align*}
		which interchanges the roles of $\bp$ and $\bn$. The fixpoint subalgebra of $\TL(B_n)^{\Phi}$ under this involution is precisely $\TL(D_n)$.
	\end{lemm}
	\begin{proof}
		One can easily check that $\Phi$ is well-defined. It is clearly an involution. The projectors are interchanged by this involution by their characterization as the unique idempotents which kill all generators $U_i$ and satisfy $s_0\bp=\bp$ respectively $s_0\bn=-\bn$. The concrete description of the diagram bases of $\TL(B_n)$ and $\TL(D_n)$ from \Cref{diagrammatic definition type B TL} and \Cref{inclusions of Temperley--Lieb algebras} gives that the fixed point subalgebra under $\Phi$ is precisely $\TL(D_n)$.
	\end{proof}
As a consequence of the recursive formula and \Cref{Theorem on Existence of JW projectors of type B} we have the following absorption formulas for Jones--Wenzl projectors.
\begin{koro}[Box absorption] \label{box absorption type B}
	Let $n,m\geq 1$ such that $n>m$. Then we have
	\[
		\bp b_{m,{+}} = \bp =  b_{m,{+}} \bp \quad \text{ and } \quad \bn b_{m,{-}} = \bn =  b_{m,{-}}\bn.
	\]
	We indicate these formulas by the following diagrams
	\[
	\cbox
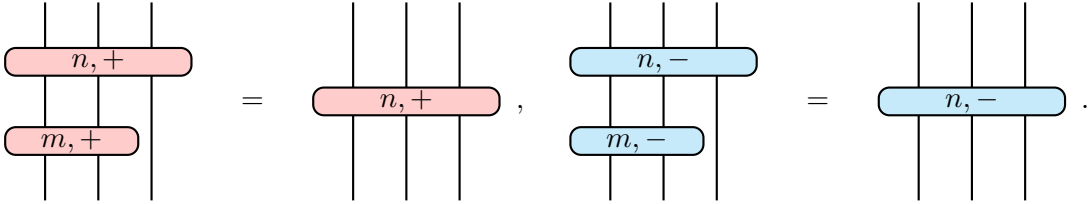
{
		\begin{tikzpicture}[tldiagram, yscale=1/2]
			% lines
			\draw \tlcoord{-2.5}{0} \lineup \lineup \lineup \lineup \lineup;
			\draw \tlcoord{-2.5}{1} \lineup \lineup \lineup \lineup \lineup;
			\draw \tlcoord{-2.5}{2} \lineup \lineup \lineup \lineup \lineup;
			% boxes
			\draw \tlcoord{1}{0} \maketlboxred{3}{$n, {+}$};
			\draw \tlcoord{-1}{0} \maketlboxred{2}{$m, {+}$};
			=
		\end{tikzpicture}
	}
	\quad=\quad
	\cbox{
		\begin{tikzpicture}[tldiagram, yscale=1/2]
			% lines
			\draw \tlcoord{-2.5}{0} \lineup \lineup \lineup \lineup \lineup;
			\draw \tlcoord{-2.5}{1} \lineup \lineup \lineup \lineup \lineup;
			\draw \tlcoord{-2.5}{2} \lineup \lineup \lineup \lineup \lineup;
			% boxes
			\draw \tlcoord{0}{0} \maketlboxred{3}{$n, {+}$};
		\end{tikzpicture}
	}
	\, , \quad
		\cbox{
			\begin{tikzpicture}[tldiagram, yscale=1/2]
				% lines
				\draw \tlcoord{-2.5}{0} \lineup \lineup \lineup \lineup \lineup;
				\draw \tlcoord{-2.5}{1} \lineup \lineup \lineup \lineup \lineup;
				\draw \tlcoord{-2.5}{2} \lineup \lineup \lineup \lineup \lineup;
				% boxes
				\draw \tlcoord{1}{0} \maketlboxblue{3}{$n, {-}$};
				\draw \tlcoord{-1}{0} \maketlboxblue{2}{$m, {-}$};
				=
			\end{tikzpicture}
		}
		\quad=\quad
		\cbox{
			\begin{tikzpicture}[tldiagram, yscale=1/2]
				% lines
				\draw \tlcoord{-2.5}{0} \lineup \lineup \lineup \lineup \lineup;
				\draw \tlcoord{-2.5}{1} \lineup \lineup \lineup \lineup \lineup;
				\draw \tlcoord{-2.5}{2} \lineup \lineup \lineup \lineup \lineup;
				% boxes
				\draw \tlcoord{0}{0} \maketlboxblue{3}{$n, {-}$};
			\end{tikzpicture}
		} \, .
	\]
\end{koro}
\begin{proof}
	This is a direct consequence of the recursive nature of \Cref{defi type B projectors} and \Cref{Theorem on Existence of JW projectors of type B}. 
\end{proof}
\begin{koro}[Red-blue orthogonality] \label{red-blue orthogonality}
	The positive and negative type $B$ projectors are orthogonal, i.e.\ we have
	\[
	\bp b_{m,{-}}=0=b_{m,{-}}\bp \text{ for all } n,m\geq 1,
	\]
	which we indicate diagrammatically as
	\[
		\cbox{
			\begin{tikzpicture}[tldiagram, yscale=1/2]
				% lines
				\draw \tlcoord{-2.5}{0} \lineup \lineup \lineup \lineup \lineup;
				\draw \tlcoord{-2.5}{1} \lineup \lineup \lineup \lineup \lineup;
				% boxes
				\draw \tlcoord{1}{0} \maketlboxred{2}{$n, {+}$};
				\draw \tlcoord{-1}{0} \maketlboxblue{2}{$n, {-}$};
				=
			\end{tikzpicture}
		} \quad = \quad  0 \quad = \quad \cbox{
		\begin{tikzpicture}[tldiagram, yscale=1/2]
			% lines
			\draw \tlcoord{-2.5}{0} \lineup \lineup \lineup \lineup \lineup;
			\draw \tlcoord{-2.5}{1} \lineup \lineup \lineup \lineup \lineup;
			% boxes
			\draw \tlcoord{1}{0} \maketlboxblue{2}{$n, {-}$};
			\draw \tlcoord{-1}{0} \maketlboxred{2}{$n, {+}$};
			=
		\end{tikzpicture}
	} \, .
	\]
\end{koro}
\begin{proof}
	 Using \Cref{box absorption type B} we have
	\[
	\bp b_{m,{-}} =\bp b_{1,{+}} b_{1,{-}} b_{m,{-}} =  \bp \cdot 0 \cdot b_{m,{-}} = 0
	\]
	since $(1+s_0)(1-s_0)=1-s_0^2=0$ and similarly $b_{m,{-}}\bp=0$. 
\end{proof}
\begin{rema} \label{Nice coefficients remark}
	As seen in \Cref{n=2 type B projectors} the elements $\bp$ and $\bn$ have at first sight surprising coefficients in front of the diagram basis. These coefficients are not quantum fractions (i.e.\ element in $\mZ[[q]][q^{-1}]$) as in the type $A$ case! Instead every basis element has some contribution of $\frac{1}{2}$. This looks demotivating from a categorification point of view, since it is very difficult to categorify elements in $\mQ$. However instead of giving up, observe that by taking the sum of $b_{2,{+}}$ and $b_{2,{-}}$ we obtain the element
	\[
		d_2 \coloneqq b_{2,{+}} + b_{2,{-}} \, = \, \cbox{
			\begin{tikzpicture}[tldiagram]
				%dots
				%\drawdots{0}{0}{1}
				%\drawdots{1}{0}{1}
				%lines
				\draw \tlcoord{0}{0} \lineup;
				\draw \tlcoord{0}{1} \lineup;
			\end{tikzpicture}
		} - \frac{1}{[2]} \cbox{
			\begin{tikzpicture}[tldiagram]
				%dots
				%\drawdots{0}{0}{1}
				%\drawdots{1}{0}{1}
				%lines
				\draw \tlcoord{0}{0} \capright;
				\draw \tlcoord{1}{0} \cupright;
			\end{tikzpicture}
		} - \frac{1}{[2]} \cbox{
			\begin{tikzpicture}[tldiagram]
				%dots
				%\drawdots{0}{0}{1}
				%\drawdots{1}{0}{1}
				%lines
				\draw \tlcoord{0}{0} \capright;
				\draw \tlcoord{1}{0} \cupright;
				%big dots
				\biggerdrawdots{0.3}{0.5}{0.5}
				\biggerdrawdots{0.7}{0.5}{0.5}
			\end{tikzpicture}
		}
	\]
	whose coefficients live in $\mZ[[q]][q^{-1}]$, which is what we want! Moreover this idempotent lies not just in the algebra $\TL(B_n)$, but instead its subalgebra $\TL(D_n)$ from the inclusion in \Cref{inclusions of Temperley--Lieb algebras}. This remark is the starting point for the next section.
\end{rema}

%---------------------------------------------------------------------------

\section{Type \texorpdfstring{$D$}{D} Jones--Wenzl projectors} \label{Chapter 3 section 2}

In this important section we define Jones--Wenzl projectors of type $D_n$ and explore some of their properties, which resemble the type $A$ projectors in more than one way.

\begin{defi}[Projectors of type $D$] \label{definition of projectors of type d} %Jones--Wenzl Projector of type D 
	We define the Jones--Wenzl Projector $d_n\in \TL(D_n)$ of type $D_n$ as $d_n=\bp+\bn\in \TL(B_n)$. We visualize the type $D_n$ projector as a green box on $n$ strands and this definition as
	\begin{equation*}
		\cbox
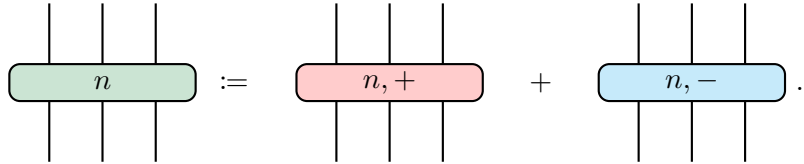
{
			\begin{tikzpicture}[tldiagram, yscale=2/3]
				% lines
				\draw \tlcoord{-1.5}{0} \lineup \lineup  \lineup;
				\draw \tlcoord{-1.5}{1} \lineup \lineup  \lineup;
				\draw \tlcoord{-1.5}{2} \lineup \lineup  \lineup;
				% boxes
				\draw \tlcoord{0}{0} \maketlboxgreen{3}{$n$};
				=
			\end{tikzpicture}
		} \, \coloneqq \quad \cbox{
		\begin{tikzpicture}[tldiagram, yscale=2/3]
			% lines
			\draw \tlcoord{-1.5}{0} \lineup \lineup \lineup;
			\draw \tlcoord{-1.5}{1} \lineup  \lineup \lineup;
			\draw \tlcoord{-1.5}{2} \lineup \lineup \lineup;
			% boxes
			\draw \tlcoord{0}{0} \maketlboxred{3}{$n, {+}$};
			=
		\end{tikzpicture}
	} \quad + \quad \cbox{
	\begin{tikzpicture}[tldiagram, yscale=2/3]
		% lines
		\draw \tlcoord{-1.5}{0} \lineup \lineup \lineup;
		\draw \tlcoord{-1.5}{1} \lineup \lineup  \lineup;
		\draw \tlcoord{-1.5}{2} \lineup \lineup \lineup;
		% boxes
		\draw \tlcoord{0}{0} \maketlboxblue{3}{$n, {-}$};
		=
	\end{tikzpicture}
}.
	\end{equation*}
\end{defi}
The following corollary is proven totally similarly to \Cref{Theorem on Existence of JW projectors of type B}.
\begin{koro} \label{characterizing properties of type B projector}
	The type $D$ projector $d_n$ is the unique non-zero idempotent in $\TL(D_n)$, which satisfies
	\[
		d_n U_i= 0 = U_i d_n 
	\]
	for all $i=0,1,\ldots, n-1$.
\end{koro}
\begin{proof}
	First observe that $d_n$ is really contained in the subalgebra $\TL(D_n)\subseteq \TL(B_n)$ by \Cref{Type B automorphism s_o flips sign}. It is non-zero, since the coefficient in front of $1\in \TL(D_n)$ in $d_n$ is $1$. Moreover it is is idempotent by \Cref{red-blue orthogonality} and since $\bp$ and $\bn$ are both idempotent.
	The uniqueness follows in the same way as the uniqueness in the proof \Cref{Theorem on Existence of JW projectors of type B} using $\TL(D_n)=\shift{1}\oplus (U_0,U_1,\ldots, U_{n-1})$.
\end{proof}

%Note that as a easy consequence of \Cref{red-blue orthogonality} we have
%	\[
%		d_n\bp=\bp =\bp d_n  \text{ and } d_n\bn = \bn = \bn d_n.
%	\]
%which one could indicate by red respectively blue boxes absorbing every green box.
\Cref{definition of projectors of type d} is unfortunately not very useful, since one has to compute $\bp$ (and $\bn$ for that matter) first. However the following theorem gives a recursive formula for the type $D$ projectors.

\begin{theo} \label{type D projector recursive description}
	The type $D$ projectors $d_n$ are described by the following recursive formulas:
	\begin{equation*}
		d_1=\cbox{
		\begin{tikzpicture}[tldiagram]
			%dots
			%\drawdots{0}{0}{0}
			%\drawdots{1}{0}{0}
			%lines
			\draw \tlcoord{0}{0} \lineup;
	\end{tikzpicture}} \, , \qquad d_2=\cbox{
	\begin{tikzpicture}[tldiagram]
		%dots
		%\drawdots{0}{0}{1}
		%\drawdots{1}{0}{1}
		%lines
		\draw \tlcoord{0}{0} \lineup;
		\draw \tlcoord{0}{1} \lineup;
	\end{tikzpicture}
} - \frac{1}{[2]} \cbox{
\begin{tikzpicture}[tldiagram]
	%dots
	%\drawdots{0}{0}{1}
	%\drawdots{1}{0}{1}
	%lines
	\draw \tlcoord{0}{0} \capright;
	\draw \tlcoord{1}{0} \cupright;
\end{tikzpicture}
} - \frac{1}{[2]} \cbox{
\begin{tikzpicture}[tldiagram]
	%dots
	%\drawdots{0}{0}{1}
	%\drawdots{1}{0}{1}
	%lines
	\draw \tlcoord{0}{0} \capright;
	\draw \tlcoord{1}{0} \cupright;
	%big dots
	\biggerdrawdots{0.3}{0.5}{0.5}
	\biggerdrawdots{0.7}{0.5}{0.5}
\end{tikzpicture}
} \, ,
	\end{equation*}
	Assume $d_n$ was already calculated for $n\geq2$, then $d_{n+1}$ is given by the formula
	\[
		d_{n+1}=d_n-\frac{q^{n-1}+q^{-(n-1)}}{q^n+q^{-n}}d_nU_nd_n,
	\]
	which we visualize as
	\[
	\cbox{
		\begin{tikzpicture}[tldiagram, yscale=1/2]
			% lines
			\draw \tlcoord{-2.5}{0} \lineup \lineup \lineup \lineup \lineup;
			\draw \tlcoord{-2.5}{1} \lineup \lineup \lineup \lineup \lineup;
			\draw \tlcoord{-2.5}{2} \lineup \lineup \lineup \lineup \lineup;
			% boxes
			\draw \tlcoord{0}{0} \maketlboxgreen{3}{$n+1$};
			=
		\end{tikzpicture}
	}
	\quad=\quad
	\cbox{
		\begin{tikzpicture}[tldiagram, yscale=1/2]
			% lines
			\draw \tlcoord{-2.5}{0} \lineup \lineup \lineup \lineup \lineup;
			\draw \tlcoord{-2.5}{1} \lineup \lineup \lineup \lineup \lineup;
			\draw \tlcoord{-2.5}{2} \lineup \lineup \lineup \lineup \lineup;
			% boxes
			\draw \tlcoord{0}{0} \maketlboxgreen{2}{$n$};
			=
		\end{tikzpicture}
	}
	\quad-\quad
	\frac{q^{n-1} + q^{{-}(n-1)}}{q^n + q^{{-}n}}
	\cbox{
		\begin{tikzpicture}[tldiagram, yscale=1/2]
			% lines
			\draw \tlcoord{-1}{0} \lineup \lineup \lineup \lineup \lineup;
			\draw \tlcoord{3}{1} \lineup;
			\draw \tlcoord{0}{1} \linedown;
			% paths
			\draw \tlcoord{0}{1} \smalllineup \capright \smalllinedown \linedown;
			\draw \tlcoord{3}{1} \smalllinedown \cupright \smalllineup \lineup;
			% boxes
			\draw \tlcoord{0}{0} \maketlboxgreen{2}{$n$};
			\draw \tlcoord{3}{0} \maketlboxgreen{2}{$n$};
		\end{tikzpicture}
	} .
	\]
\end{theo}
\begin{proof}
	The description for the first two type $D$ projectors $d_1$ and $d_2$ follows immediately from adding up the type $B$ projectors described in \Cref{n=2 type B projectors}. The recursive formula follows by the calculation
	\begin{align*}
		d_{n+1}&=b_{n+1,{+}}+b_{n+1,{-}} \\
		&=\left(\bp - \frac{q^{n-1}+q^{-(n-1)}}{q^n+q^{-n}} \bp U_n \bp\right) + \left(\bn - \frac{q^{n-1}+q^{-(n-1)}}{q^n+q^{-n}} \bn U_n \bn\right) \\
		&= d_n - \frac{q^{n-1}+q^{-(n-1)}}{q^n+q^{-n}} (\bp U_n \bp + \bn U_n \bn) \\
		&= d_n - \frac{q^{n-1}+q^{-(n-1)}}{q^n+q^{-n}} (\bp+\bn) U_n (\bp +\bn)  \\
		&= d_n + \frac{q^{n-1}+q^{-(n-1)}}{q^n+q^{-n}} d_{n+1} U_n d_{n+1} ,
	\end{align*}
where we used in the non-trivial step that $
\bn U_n \bp = 0 = \bp U_n \bn $ since
\[ \cbox{
	\begin{tikzpicture}[tldiagram, yscale=1/2, xscale=2/3]
		% lines
		\draw \tlcoord{-1}{0} \lineup \lineup \lineup \lineup \lineup;
		\draw \tlcoord{3}{1} \lineup;
		\draw \tlcoord{0}{1} \linedown;
		% paths
		\draw \tlcoord{0}{1} \smalllineup \capright \smalllinedown \linedown;
		\draw \tlcoord{3}{1} \smalllinedown \cupright \smalllineup \lineup;
		% boxes
		\draw \tlcoord{0}{0} \maketlboxred{2}{$n, {+}$};
		\draw \tlcoord{3}{0} \maketlboxblue{2}{$n, {-}$};
	\end{tikzpicture}
}\, = \, \cbox{
\begin{tikzpicture}[tldiagram, yscale=2/3, xscale=3/2]
	% lines, 
	\draw \tlcoord{-1}{0} \lineup \lineup \lineup \lineup \lineup;
	\draw \tlcoord{3}{1} \lineup;
	\draw \tlcoord{0}{1} \linedown;
	% paths
	\draw \tlcoord{0}{1} \smalllineup \capright \smalllinedown \linedown;
	\draw \tlcoord{3}{1} \smalllinedown \cupright \smalllineup \lineup;
	% boxes
	\draw \tlcoord{0}{0} \maketlboxred{2}{$n, {+}$};
	\draw \tlcoord{1}{0} \maketlboxred{1}{$n-1, {+}$};
	\draw \tlcoord{2}{0} \maketlboxblue{1}{$n-1, {-}$};
	\draw \tlcoord{3}{0} \maketlboxblue{2}{$n, {-}$};
\end{tikzpicture}
} \,  = \, 0  \, = \, \cbox{
\begin{tikzpicture}[tldiagram, yscale=2/3, xscale=3/2]
	% lines
	\draw \tlcoord{-1}{0} \lineup \lineup \lineup \lineup \lineup;
	\draw \tlcoord{3}{1} \lineup;
	\draw \tlcoord{0}{1} \linedown;
	% paths
	\draw \tlcoord{0}{1} \smalllineup \capright \smalllinedown \linedown;
	\draw \tlcoord{3}{1} \smalllinedown \cupright \smalllineup \lineup;
	% boxes
	\draw \tlcoord{0}{0} \maketlboxblue{2}{$n, {-}$};
	\draw \tlcoord{1}{0} \maketlboxblue{1}{$n-1, {-}$};
	\draw \tlcoord{2}{0} \maketlboxred{1}{$n-1, {+}$};
	\draw \tlcoord{3}{0} \maketlboxred{2}{$n, {+}$};
\end{tikzpicture}
}\, = \, \cbox{
\begin{tikzpicture}[tldiagram, yscale=1/2, xscale=2/3]
	% lines
	\draw \tlcoord{-1}{0} \lineup \lineup \lineup \lineup \lineup;
	\draw \tlcoord{3}{1} \lineup;
	\draw \tlcoord{0}{1} \linedown;
	% paths
	\draw \tlcoord{0}{1} \smalllineup \capright \smalllinedown \linedown;
	\draw \tlcoord{3}{1} \smalllinedown \cupright \smalllineup \lineup;
	% boxes
	\draw \tlcoord{0}{0} \maketlboxblue{2}{$n, {-}$};
	\draw \tlcoord{3}{0} \maketlboxred{2}{$n, {+}$};
\end{tikzpicture}
},\] 
by box absorption, commutativity of $b_{n-1,{+}}$ and $b_{n-1, {-}}$ with $U_n$, and red-blue orthogonality.
\end{proof}
From this recursive formulas we can conclude that the coefficients for the type $D$ projectors are nice in the sense of \Cref{Nice coefficients remark} for general $n$.

\begin{koro} \label{nice coefficients corollary}
	The Jones--Wenzl projector $d_n$ of Type $D_n$ has coefficients in $\mZ[[q]][q^{-1}]$ with respect to the diagram basis.
\end{koro}
\begin{proof}
	This clearly holds for $n=1$. For $n=2$ the only non-trivial power series is 
	\[
		\frac{1}{[2]}=\frac{1}{q+q^{-1}}=q-q^3+q^5-q^7+\ldots = \sum_{l=0}^{\infty}(-1)^lq^{2l+1} \in \mZ[[q]][q^{-1}].
	\] 
	The recursive description of $d_n$ in \Cref{type D projector recursive description} gives the statement for general $n$, since all coefficients of $d_n$ are sums of products of $\delta=q+q^{-1}$, its inverse and fractions of the form
	\begin{align*}
		\frac{q^{k}+q^{-k}}{q^{k+1}+q^{-(k+1)}} &= (q^k+q^{-k})\frac{q^{k+1}}{1-(-q^{2(k+1)})} \\
		&=	(q^{2k+1}+q)\frac{q^{k+1}}{1-(-q^{2(k+1)})} \\
		&= (q^{2k+1}+q)(1- q^{2(k+1)} + q^{4(k+1)} - q^{6(k+1)} + \ldots )  \\
		&= q + q^{2k+1} - q^{2k+3} -q^{4k+3} +q^{4k+5} +q^{6k+5}- q^{6k+7} + \dots  \\
		&= q + \sum_{l=1}^{\infty}(-1)^l(q^{2l(k+1)+1}-q^{2l(k+1)-1}) \in \mZ[[q]][q^{-1}] .
	\end{align*}
	Since all of the above ``building pieces'' lie in $\mZ[[q]][q^{-1}]$ this holds for all coefficients of $d_n$.
\end{proof}
We end this section with a look at the coefficients $\frac{q^{k}+q^{-k}}{q^{k+1}+q^{-(k+1)}}$, which appear in the recursive definition.
\begin{obse}
	The power series for $\frac{q^{k}+q^{-k}}{q^{k+1}+q^{-(k+1)}}$, which we calculated in the proof of \Cref{nice coefficients corollary} is $q+r_k$, where $r_k\in \mZ[[q]]$ is a expression divisible by $q^{2k+1}$. This implies that
	\[
		\frac{q^{k}+q^{-k}}{q^{k+1}+q^{-(k+1)}} \xrightarrow{k\rightarrow \infty} q
	\]
	with respect to the $q$-adic norm on $\mZ[[q]][q^{-1}]$. If we would have worked with $q=1$ we would not have noticed any of this, since the left side is already equal to $1$, which is $q$. For further $q$-adic discussions and the proper definitions, see \Cref{Chapter on Infinite Braids}. 
\end{obse}
\begin{rema} \label{quantum integers or not}
	The fraction $\frac{q^k+q^{-k}}{q^{k+1}+q^{-(k+1)}}$ can be written as 
	\[
	\frac{q^{k}+q^{-k}}{q^{k+1}+q^{-(k+1)}}=\frac{[k+1]-[k-1]}{[k+2]-[k]},
	\]
	which is a fraction, in which both the numerator and denominator are differences of quantum integers. However sometimes this expression is already a quotient of products of quantum numbers and not just their differences. For instance observe that
	\[
	\frac{q+q^{-1}}{q^2+q^{-2}}=\frac{[2]}{[3]-[1]}=\frac{[2][2]}{[4]} 
	\]
	since we have
	\[
		[4]=q^3+q+q^{-1}+q^{-3}=(q+q^{1})(q^2+q^{-2})=[2]([3]-1).
	\]
	We do not know whether this is a general phenomenon. The author stumbled upon this equality when calculating $d_3$ in two ways: The left fraction is part of the recursive expression for $d_3$, while the right hand side appears, when computing the Jones--Wenzl projector $a_3$ of type $A$  living in $\TL_4=\TL(A_3)$. The coefficients in front of the diagram basis of $d_3$ appear already as coefficients of $a_3$ since $\TL(D_3)$ is a quotient of $\TL(A_3)$ as explained in \Cref{explanation green type D vs our type D} by making the four diagrams from \Cref{example additional diagrams} zero.
\end{rema}
\section{Higher Jones--Wenzl projectors} \label{Chapter 3 section on higher projectors}
In this section we introduce the idempotents, which project to all the direct summands of $V^{\otimes n}$ over the coideal $\Vq$. By \Cref{Theorem decomposition of tensor power of V} we know all the direct summands, however we still need formulas in the Temperley--Lieb algebra for their projections. Fortunately the recursive formulas for the type $B$ projectors in \Cref{defi type B projectors} have a representation theoretic interpretation. The first summand in the recursive formula can be understood as the idempotent projection in $V^{\otimes (n+1)}$ onto $L([n])\otimes V$, while the second summand $\bp U_n \bp$ can be interpreted as projecting onto $L([n-1])$. Throughout this section we work with morphisms going between arbitrary tensor powers of $V$, i.e.\ $\TL_B(l,k)\coloneqq \Hom_{\Ug_q'(\gl_1)}(V^{\otimes l}, V^{\otimes k})$, where $\Ug_q'(\gl_1)$ is the subalgebra of $\Vq$ generated by $B=q^{-1}EK^{-1}+F$. For the reason, why we should consider these morphism spaces, see the discussion at the end of \Cref{vizualised tensor prod sl2 decompo}. 

\begin{defi}
	Let $\varepsilon\in\{-1,1\}^n$. We want to define a quasi-idempotent element $q_{\varepsilon}\in \TL_B(n,n)\coloneqq \TL(B_n)$ projecting onto $L_\varepsilon$. First we need to define certain elements $t_\varepsilon\in\TL_B(k,n)\coloneqq \Hom_{\Ug_q'(\gl_1)}(V^{\otimes k}, V^{\otimes n}) $, where $k=|\varepsilon|=|\varepsilon_1+\cdots+\varepsilon_n|$.
	We define $t_{(1)}=\bp$ and $t_{(-1)}=\bn$. Assume that $t_{\varepsilon}$ is already defined. We define
	\begin{align*}
		t_{\varepsilon*1}&\coloneqq\begin{cases}
			\cbox{
				\begin{tikzpicture}[tldiagram, yscale=2/3]
					% lines
					\draw \tlcoord{-1}{0} \lineup \lineup \lineup;
					\draw \tlcoord{-1}{1} \lineup \lineup \lineup;
					% boxes
					\draw \tlcoord{0}{0} \maketlboxred{2}{$k+1, {+}$};
					\draw \tlcoord{1.3}{0} \maketlellipse{1}{$t_\varepsilon$};
				\end{tikzpicture}
			} , & \text{if $|\varepsilon|>0$,} \\
			& \\
			\cbox{
				\begin{tikzpicture}[tldiagram, yscale=2/3]
					% lines
					\draw \tlcoord{-1}{0} \lineup \lineup \lineup;
					\draw \tlcoord{1}{1} \lineup;
					% path
					\draw \tlcoord{1}{1} \smalllinedown \cupright \smalllineup \lineup;
					% boxes
					\draw \tlcoord{0}{0} \maketlboxblue{1}{$k-1$};
					\draw \tlcoord{1.3}{0} \maketlellipse{2}{$t_\varepsilon$};
				\end{tikzpicture}
			}, & \text{if $|\varepsilon|<0$,} \\
			& \\
			\cbox{
				\begin{tikzpicture}[tldiagram, yscale=2/3]
					% lines
					\draw \tlcoord{0}{0}  \lineup;
					\draw \tlcoord{-1}{2} \lineup \lineup;
					% boxes
					\draw \tlcoord{0}{2} \maketlboxred{1}{${1}$};
					\draw \tlcoord{0}{0} \maketlellipse{1}{$t_\varepsilon$};
				\end{tikzpicture}
			}, & \text{if $|\varepsilon|=0$,}
		\end{cases} \qquad 
		t_{\varepsilon*(-1)}&\coloneqq\begin{cases}
			\cbox{
				\begin{tikzpicture}[tldiagram, yscale=2/3]
					% lines
					\draw \tlcoord{-1}{0} \lineup \lineup \lineup;
					\draw \tlcoord{-1}{1} \lineup \lineup \lineup;
					% boxes
					\draw \tlcoord{0}{0} \maketlboxblue{2}{$k+1, {-}$};
					\draw \tlcoord{1.3}{0} \maketlellipse{1}{$t_\varepsilon$};
				\end{tikzpicture}
			} , & \text{if $|\varepsilon|<0$,} \\
			& \\
			\cbox{
				\begin{tikzpicture}[tldiagram, yscale=2/3]
					% lines
					\draw \tlcoord{-1}{0} \lineup \lineup \lineup;
					\draw \tlcoord{1}{1} \lineup;
					% path
					\draw \tlcoord{1}{1} \smalllinedown \cupright \smalllineup \lineup;
					% boxes
					\draw \tlcoord{0}{0} \maketlboxred{1}{$k-1$};
					\draw \tlcoord{1.3}{0} \maketlellipse{2}{$t_\varepsilon$};
				\end{tikzpicture}
			}, & \text{if $|\varepsilon|>0$,} \\
			& \\
			\cbox{
				\begin{tikzpicture}[tldiagram, yscale=2/3]
					% lines
					\draw \tlcoord{0}{0}  \lineup;
					\draw \tlcoord{-1}{2} \lineup \lineup;
					% boxes
					\draw \tlcoord{0}{2} \maketlboxblue{1}{${1}$};
					\draw \tlcoord{0}{0} \maketlellipse{1}{$t_\varepsilon$};
				\end{tikzpicture}
			}, & \text{if $|\varepsilon|=0$.}
		\end{cases}
	\end{align*}
	The quasi-idempotent $q_\varepsilon\in \TL(B_n)=\TL_B(n,n)$ is defined as the product 
	$q_\varepsilon=t_\varepsilon \circ \bar{t_\varepsilon}$,
	where $\bar{t}_\varepsilon\in\TL_B(n,k)$ is the vertical mirror image of $t_\varepsilon \in\TL_B(k,n)$.
\end{defi}
These definitions are heavily inspired by the definitions of the higher Jones--Wenzl projectors from \cite{cooper2015}. We obtained the perspective from the beginning of the section through reading the introduction of \cite{wedrich2018}, which helped us to understand the construction from \cite{cooper2015} more deeply and generalize the formulas to our setup. The next theorem is a 
\begin{theo}
	The elements $q_{\varepsilon}$ are quasi-idempotent, i.e.\ there exist $n_\mu\in\C(q)$ such that 
	$q_{\varepsilon}^2=n_{\varepsilon}q_{\varepsilon}$. The elements $\{e_{\varepsilon}\coloneqq\frac{1}{n_{\varepsilon}}q_{\varepsilon}\mid \varepsilon\in\{-1,1\}^n \}$ define a complete set of pairwise orthogonal primitive idempotents of $\TL(B_n)$, called \textbf{higher type $B$ Jones--Wenzl projectors}. Viewed as elements in $\End_{\Vq}(V^{\otimes n})$ they project onto the one-dimensional summands $L_{\varepsilon}$ of $V^{\otimes n}$.
\end{theo}
\begin{proof}
	The proof is totally analogous to the proof for the type $A$ case, see e.g.\ \cite[Theorem 2.20]{cooper2015}. The above definitions are all chosen in such a way to fit the recursive formulas for $\bp$ and $\bn$. The elements $q_{\varepsilon}$ are appearing as summand of $\id_{V^{\otimes 1}}, \id_{V^{\otimes 2}}, \ldots$ when applying the recursive formulas from \Cref{defi type B projectors} iteratively using that $\id_{V^{\otimes 1}}=b_{1,{+}}+b_{1,{-}}$, whenever the trivial rep $L(0)$ appears as tensor factor.
\end{proof}
\begin{beis} \label{n=2 B higher projectors}
	The higher type $B$ projectors for $n=2$ are given by
	\begin{gather*}
		e_{1,1}=\cbox{
			\begin{tikzpicture}[tldiagram, yscale=2/3]
				% lines
				\draw \tlcoord{-1.5}{0} \lineup \lineup  \lineup;
				\draw \tlcoord{-1.5}{1} \lineup \lineup  \lineup;
				% boxes
				\draw \tlcoord{0}{0} \maketlboxred{2}{$2, +$};
			\end{tikzpicture}
		}, \quad e_{1,-1} = \frac{2}{q+q^{-1}}\cbox{
			\begin{tikzpicture}[tldiagram, yscale=2/3]
				% lines
				\draw \tlcoord{-1.5}{0} \lineup \capright \linedown;
				\draw \tlcoord{1.5}{0} \linedown \cupright  \lineup;
				% boxes
				\draw \tlcoord{-1}{0} \maketlboxred{1}{$1, {+}$};
				\draw \tlcoord{1}{0} \maketlboxred{1}{$1, {+}$};
			\end{tikzpicture}
		}
		\,, \\[1em]
		e_{-1,1}=\frac{2}{q+q^{-1}} \cbox{
			\begin{tikzpicture}[tldiagram, yscale=2/3]
				% lines
				\draw \tlcoord{-1.5}{0} \lineup \capright \linedown;
				\draw \tlcoord{1.5}{0} \linedown \cupright  \lineup;
				% boxes
				\draw \tlcoord{-1}{0} \maketlboxblue{1}{$1, {-}$};
				\draw \tlcoord{1}{0} \maketlboxblue{1}{$1, {-}$};
			\end{tikzpicture}
		} \, , \qquad e_{-1,-1} = \cbox{
			\begin{tikzpicture}[tldiagram, yscale=2/3]
				% lines
				\draw \tlcoord{-1.5}{0} \lineup \lineup  \lineup;
				\draw \tlcoord{-1.5}{1} \lineup \lineup  \lineup;
				% boxes
				\draw \tlcoord{0}{0} \maketlboxblue{2}{$2, +$};
			\end{tikzpicture}
		} 
		\,.
	\end{gather*}
	In terms of the diagram basis the idempotents $e_{1,-1}$ and $e_{-1,1}$ are given by
	\begin{align*}
		e_{1,-1} &= \frac{2}{q+q^{-1}}\cbox{
			\begin{tikzpicture}[tldiagram, yscale=2/3]
				% lines
				\draw \tlcoord{-1.5}{0} \lineup \capright \linedown;
				\draw \tlcoord{1.5}{0} \linedown \cupright  \lineup;
				% boxes
				\draw \tlcoord{-1}{0} \maketlboxred{1}{$1, {+}$};
				\draw \tlcoord{1}{0} \maketlboxred{1}{$1, {+}$};
			\end{tikzpicture}
		}= \frac{1}{2[2]} \left( 
		\cbox{
			\begin{tikzpicture}[tldiagram]
				%dots
				%\drawdots{0}{0}{1}
				%\drawdots{1}{0}{1}
				%lines
				\draw \tlcoord{0}{0} \capright;
				\draw \tlcoord{1}{0} \cupright;
			\end{tikzpicture}
		} + \cbox{
			\begin{tikzpicture}[tldiagram]
				%dots
				%\drawdots{0}{0}{1}
				%\drawdots{1}{0}{1}
				%lines
				\draw \tlcoord{0}{0} \capright;
				\draw \tlcoord{1}{0} \dcupright;
			\end{tikzpicture}
		}
		+ \cbox{
			\begin{tikzpicture}[tldiagram]
				%dots
				%\drawdots{0}{0}{1}
				%\drawdots{1}{0}{1}
				%lines
				\draw \tlcoord{0}{0} \dcapright;
				\draw \tlcoord{1}{0} \cupright;
			\end{tikzpicture}
		}
		+ \cbox{
			\begin{tikzpicture}[tldiagram]
				%dots
				%\drawdots{0}{0}{1}
				%\drawdots{1}{0}{1}
				%lines
				\draw \tlcoord{0}{0} \dcapright;
				\draw \tlcoord{1}{0} \dcupright;
			\end{tikzpicture}
		}
		\right), \\
		e_{-1,1} &= \frac{2}{q+q^{-1}}\cbox{
			\begin{tikzpicture}[tldiagram, yscale=2/3]
				% lines
				\draw \tlcoord{-1.5}{0} \lineup \capright \linedown;
				\draw \tlcoord{1.5}{0} \linedown \cupright  \lineup;
				% boxes
				\draw \tlcoord{-1}{0} \maketlboxblue{1}{$1, {-}$};
				\draw \tlcoord{1}{0} \maketlboxblue{1}{$1, {-}$};
			\end{tikzpicture}
		}= \frac{1}{2[2]} \left( 
		\cbox{
			\begin{tikzpicture}[tldiagram]
				%dots
				%\drawdots{0}{0}{1}
				%\drawdots{1}{0}{1}
				%lines
				\draw \tlcoord{0}{0} \capright;
				\draw \tlcoord{1}{0} \cupright;
			\end{tikzpicture}
		} - \cbox{
			\begin{tikzpicture}[tldiagram]
				%dots
				%\drawdots{0}{0}{1}
				%\drawdots{1}{0}{1}
				%lines
				\draw \tlcoord{0}{0} \capright;
				\draw \tlcoord{1}{0} \dcupright;
			\end{tikzpicture}
		}
		- \cbox{
			\begin{tikzpicture}[tldiagram]
				%dots
				%\drawdots{0}{0}{1}
				%\drawdots{1}{0}{1}
				%lines
				\draw \tlcoord{0}{0} \dcapright;
				\draw \tlcoord{1}{0} \cupright;
			\end{tikzpicture}
		}
		+ \cbox{
			\begin{tikzpicture}[tldiagram]
				%dots
				%\drawdots{0}{0}{1}
				%\drawdots{1}{0}{1}
				%lines
				\draw \tlcoord{0}{0} \dcapright;
				\draw \tlcoord{1}{0} \dcupright;
			\end{tikzpicture}
		}
		\right).
	\end{align*}
	For the elements $e_{1,1}=b_{2,{+}}$ and $e_{-1,-1}=b_{2,{-}}$ see \Cref{n=2 type B projectors}.
	We have
	\[
	1=e_{1,1}+e_{1,-1}+e_{-1,1}+e_{-1,-1}
	\]
	as complete set of primitive, pairwise orthogonal idempotents. We saw non-quantized version of these idempotents already in \Cref{n=2 irreducible representations}. Concretely under the surjective algebra homomorphism defined by
	\begin{align*}
		\varphi\colon \C \! W(B_2) &\rightarrow \TL(B_2)|_{q=1} \\
		s_0 &\mapsto s_0 \\
		s_1 &\mapsto 1-U_1
	\end{align*}
	we have $e_{1,1}=\varphi(e_{(2),(0)}),  e_{1,-1}=\varphi (e_{(1),(1)}), e_{-1,1}=\varphi (s_1e_{(1),(1)}s_1),  e_{-1,-1}=\varphi(e_{(0),(2)})$.
	The kernel of this morphism is the $2$-dimensional subspace of $\C \! W(B_2)$ spanned by $e_{(1,1),(0)}$ and $e_{(0),(1,1)}$. 
\end{beis}
The next remark explains why we need to do develop all this machinery to define the higher projectors. Shortly said the problem is that the category of f.d.\ $\Vq$ representations is not monoidal, i.e.\ we cannot tensor the one-dimensional modules.
\begin{rema}
	If we work in the non-quantized setting, i.e.\ replace $\Vq\subseteq \Ug_q(\gl_2)$ by $\Ug(\gl_1\times \gl_1)\subseteq \Ug(\gl_2)$ the interesting formulas for the higher Jones--Wenzl projectors still work, but can be replaced by much easier ones. Since $\Ug(\gl_1\times \gl_1)$ is a Hopf subalgebra of $\Ug(\gl_2)$, the category of finite dimensional $\Ug(\gl_1\times \gl_1)$-modules is monoidal. This allows us to tensor idempotents with each other to create new ones. In particular the $q=1$ versions of the idempotents $e_\varepsilon$ in
	$\End_{\gl_1\times \gl_1}(V^{\otimes n})$
	can be constructed as
	\[
		e_\varepsilon=e_{\varepsilon_1}\otimes e_{\varepsilon_2} \otimes \cdots \otimes e_{\varepsilon_n}
	\]
	where $e_1=\frac{1}{2}(1+s_0)$ and $e_{-1}=\frac{1}{2}(1-s_0)$. This decomposition is very easy compared to the Jones--Wenzl formulas. For example the idempotents from \Cref{n=2 B higher projectors} are under this identification just given by $e_1 \otimes e_1$, $e_1\otimes e_{-1}$, $e_{-1}\otimes e_{1}$ and $e_{-1}\otimes e_{-1}$. We think of the behavior of the coideal as a feature, which is reflected in the type $B$ diagrammatics, where one cannot tensor arbitrary type $B$ diagrams with each other, but just type $B$ diagram from the right with a type $A$ diagram.
\end{rema}
\begin{defi}
	Let $\varepsilon\in\{-1,1\}^n$. Define the \textbf{type $D$ higher projector} as
	\[
	f_{\varepsilon}=e_{\varepsilon}+e_{-\varepsilon}\in \TL(B_n).
	\]
	Note that by construction of $e_{\epsilon}$ the involution $\Phi$ from \Cref{Type B automorphism s_o flips sign} swaps $e_{\varepsilon}$ and $e_{-\varepsilon}$. In particular $f_{\varepsilon}$ is fixed under $\Psi$ and hence is contained in $\TL(D_n)$.
\end{defi}

\begin{beis}
	The higher type $D$ projectors for $n=2$ are given by
	\begin{gather*}
	f_{1,1}=e_{1,1}+e_{-1,-1}=  \, \cbox{
		\begin{tikzpicture}[tldiagram, yscale=2/3]
			% lines
			\draw \tlcoord{-1.5}{0} \lineup \lineup  \lineup;
			\draw \tlcoord{-1.5}{1} \lineup \lineup  \lineup;
			% boxes
			\draw \tlcoord{0}{0} \maketlboxgreen{2}{$2$};
		\end{tikzpicture}
	} \,  =  \,\cbox{\begin{tikzpicture}[tldiagram]
			%dots
			%\drawdots{0}{0}{1}
			%\drawdots{1}{0}{1}
			%lines
			\draw \tlcoord{0}{0} \lineup;
			\draw \tlcoord{0}{1} \lineup;
	\end{tikzpicture}}
	 \, -  \,\frac{1}{[2]} \cbox{
		\begin{tikzpicture}[tldiagram]
			%dots
			%\drawdots{0}{0}{1}
			%\drawdots{1}{0}{1}
			%lines
			\draw \tlcoord{0}{0} \capright;
			\draw \tlcoord{1}{0} \cupright;
		\end{tikzpicture}
	} \, - \, \frac{1}{[2]} \cbox{
		\begin{tikzpicture}[tldiagram]
			%dots
			%\drawdots{0}{0}{1}
			%\drawdots{1}{0}{1}
			%lines
			\draw \tlcoord{0}{0} \capright;
			\draw \tlcoord{1}{0} \cupright;
		\end{tikzpicture}
	} \, , \\[2pt]
	f_{1,-1}=e_{1,-1}+e_{-1,1}= \frac{1}{[2]} \cbox{
		\begin{tikzpicture}[tldiagram]
			%dots
			%\drawdots{0}{0}{1}
			%\drawdots{1}{0}{1}
			%lines
			\draw \tlcoord{0}{0} \capright;
			\draw \tlcoord{1}{0} \cupright;
		\end{tikzpicture}
	} + \frac{1}{[2]} \cbox{
		\begin{tikzpicture}[tldiagram]
			%dots
			%\drawdots{0}{0}{1}
			%\drawdots{1}{0}{1}
			%lines
			\draw \tlcoord{0}{0} \capright;
			\draw \tlcoord{1}{0} \cupright;
			%big dots
			\biggerdrawdots{0.3}{0.5}{0.5}
			\biggerdrawdots{0.7}{0.5}{0.5}
		\end{tikzpicture}
	}.
	\end{gather*}
\end{beis}
\begin{beis}
	The higher type $B$ projectors for $n=3$ are given by
	\begin{gather*}
		e_{1,1,1} = \cbox{
			\begin{tikzpicture}[tldiagram, yscale=2/3]
				% lines
				\draw \tlcoord{-1.5}{0} \lineup \lineup  \lineup;
				\draw \tlcoord{-1.5}{1} \lineup \lineup  \lineup;
				\draw \tlcoord{-1.5}{2} \lineup \lineup  \lineup;
				% boxes
				\draw \tlcoord{0}{0} \maketlboxred{3}{$3, {+}$};
				=
			\end{tikzpicture}
		}\, , \qquad
		e_{-1,-1,-1} = \cbox{
			\begin{tikzpicture}[tldiagram, yscale=2/3]
				% lines
				\draw \tlcoord{-1.5}{0} \lineup \lineup  \lineup;
				\draw \tlcoord{-1.5}{1} \lineup \lineup  \lineup;
				\draw \tlcoord{-1.5}{2} \lineup \lineup  \lineup;
				% boxes
				\draw \tlcoord{0}{0} \maketlboxblue{3}{$3, {-}$};
				=
			\end{tikzpicture}
		}\, , \\[1em]
		e_{1,1,-1} = \frac{q+q^{-1}}{q^2+q^{-2}} \cbox{
			\begin{tikzpicture}[tldiagram, yscale=2/3]
				% lines
				\draw \tlcoord{-1.5}{0} \lineup \capright \linedown;
				\draw \tlcoord{1.5}{0} \linedown \cupright  \lineup;
				\draw \tlcoord{-1.5}{-1} \lineup \lineup \lineup;
				% boxes
				\draw \tlcoord{-1}{-1} \maketlboxred{2}{$2, {+}$};
				\draw \tlcoord{1}{-1} \maketlboxred{2}{$2, {+}$};
				%\draw \tlcoord{0}{2} \maketlboxred{1}{$1, {+}$};
			\end{tikzpicture}
		}\, , \qquad
		e_{-1,-1,1} = \frac{q+q^{-1}}{q^2+q^{-2}} \cbox{
			\begin{tikzpicture}[tldiagram, yscale=2/3]
				% lines
				\draw \tlcoord{-1.5}{0} \lineup \capright \linedown;
				\draw \tlcoord{1.5}{0} \linedown \cupright  \lineup;
				\draw \tlcoord{-1.5}{-1} \lineup \lineup \lineup;
				% boxes
				\draw \tlcoord{-1}{-1} \maketlboxblue{2}{$2, {-}$};
				\draw \tlcoord{1}{-1} \maketlboxblue{2}{$2, {-}$};
				%\draw \tlcoord{0}{2} \maketlboxred{1}{$1, {+}$};
			\end{tikzpicture}
		}\, , \\[1em]
		e_{1,-1,1} = \frac{2}{q+q^{-1}}\cbox{
			\begin{tikzpicture}[tldiagram, yscale=2/3]
				% lines
				\draw \tlcoord{-1.5}{0} \lineup \capright \linedown;
				\draw \tlcoord{1.5}{0} \linedown \cupright  \lineup;
				\draw \tlcoord{-1.5}{2} \lineup \lineup \lineup;
				% boxes
				\draw \tlcoord{-1}{0} \maketlboxred{1}{$1, {+}$};
				\draw \tlcoord{1}{0} \maketlboxred{1}{$1, {+}$};
				\draw \tlcoord{0}{2} \maketlboxred{1}{$1, {+}$};
			\end{tikzpicture}
		}\, , \qquad e_{-1,1,-1} =  \frac{2}{q+q^{-1}}\cbox{
			\begin{tikzpicture}[tldiagram, yscale=2/3]
				% lines
				\draw \tlcoord{-1.5}{0} \lineup \capright \linedown;
				\draw \tlcoord{1.5}{0} \linedown \cupright  \lineup;
				\draw \tlcoord{-1.5}{2} \lineup \lineup \lineup;
				% boxes
				\draw \tlcoord{-1}{0} \maketlboxblue{1}{$1, {-}$};
				\draw \tlcoord{1}{0} \maketlboxblue{1}{$1, {-}$};
				\draw \tlcoord{0}{2} \maketlboxblue{1}{$1, {-}$};
			\end{tikzpicture}
		}
		\, ,  \\[1em]
		e_{1,-1,-1} = \frac{2}{q+q^{-1}}\cbox{
			\begin{tikzpicture}[tldiagram, yscale=2/3]
				% lines
				\draw \tlcoord{-1.5}{0} \lineup \capright \linedown;
				\draw \tlcoord{1.5}{0} \linedown \cupright  \lineup;
				\draw \tlcoord{-1.5}{2} \lineup \lineup \lineup;
				% boxes
				\draw \tlcoord{-1}{0} \maketlboxred{1}{$1, {+}$};
				\draw \tlcoord{1}{0} \maketlboxred{1}{$1, {+}$};
				\draw \tlcoord{0}{2} \maketlboxblue{1}{$1, {-}$};
			\end{tikzpicture}
		} \, , \qquad e_{-1,1,1} = \frac{2}{q+q^{-1}}\cbox{
			\begin{tikzpicture}[tldiagram, yscale=2/3]
				% lines
				\draw \tlcoord{-1.5}{0} \lineup \capright \linedown;
				\draw \tlcoord{1.5}{0} \linedown \cupright  \lineup;
				\draw \tlcoord{-1.5}{2} \lineup \lineup \lineup;
				% boxes
				\draw \tlcoord{-1}{0} \maketlboxblue{1}{$1, {-}$};
				\draw \tlcoord{1}{0} \maketlboxblue{1}{$1, {-}$};
				\draw \tlcoord{0}{2} \maketlboxred{1}{$1, {+}$};
			\end{tikzpicture}
		}
		\, .
	\end{gather*}
	We summarize the corresponding $D_3$ projectors in a table:
	\[
	\arraycolsep=1.4pt\def\arraystretch{2.2}
	\begin{array}{c|c|c|c|c|c|c|c|c|c|c}
		 & \cbox{
		 	\begin{tikzpicture}[tldiagram, yscale=1/2, xscale=1/2]
		 		% lines
		 		\draw \tlcoord{-1.5}{0} \lineup;
		 		\draw \tlcoord{-1.5}{1} \lineup;
		 		\draw \tlcoord{-1.5}{2} \lineup;
		 		% boxes
		 		%\draw \tlcoord{0}{0} \maketlboxgreen{3}{$3$};
		 		=
		 	\end{tikzpicture}
		 } & \cbox{
		 \begin{tikzpicture}[tldiagram, yscale=1/2, xscale=1/2, ]
		 	% lines
		 	\draw \tlcoord{0}{0} \capright;
		 	\draw \tlcoord{1}{0} \cupright;
		 	\draw \tlcoord{0}{2} \lineup; 
		 	% boxes
		 	%\draw \tlcoord{0}{0} \maketlboxgreen{3}{$3$};
		 	=
		 \end{tikzpicture}
	 }& \cbox{
	 \begin{tikzpicture}[tldiagram, yscale=1/2, xscale=1/2, ]
	 	% lines
	 	\draw \tlcoord{0}{0} \dcapright;
	 	\draw \tlcoord{1}{0} \dcupright;
	 	\draw \tlcoord{0}{2} \lineup; 
	 	% boxes
	 	%\draw \tlcoord{0}{0} \maketlboxgreen{3}{$3$};
	 	=
	 \end{tikzpicture}
 } & \cbox{
 \begin{tikzpicture}[tldiagram, yscale=1/2, xscale=1/2]
 	% lines
 	\draw \tlcoord{0}{0} \linewave{1}{2};
 	\draw \tlcoord{0}{1} \capright;
 	\draw \tlcoord{1}{1} \cupleft;
 	% boxes
 	%\draw \tlcoord{0}{0} \maketlboxgreen{3}{$3$};
 	=
 \end{tikzpicture}
} & 	\cbox{
\begin{tikzpicture}[tldiagram, yscale=1/2, xscale=1/2]
	% lines
	\draw \tlcoord{0}{2} \linewave{1}{-2};
	\draw \tlcoord{0}{0} \capright;
	\draw \tlcoord{1}{2} \cupleft;
	% boxes
	%\draw \tlcoord{0}{0} \maketlboxgreen{3}{$3$};
	=
\end{tikzpicture}
}& \cbox{
\begin{tikzpicture}[tldiagram, yscale=1/2, xscale=1/2]
	% lines
	\draw \tlcoord{0}{0} \dlinewave{1}{2};
	\draw \tlcoord{0}{1} \capright;
	\draw \tlcoord{1}{1} \dcupleft;
	% boxes
	%\draw \tlcoord{0}{0} \maketlboxgreen{3}{$3$};
	=
\end{tikzpicture}
} & 	\cbox{
\begin{tikzpicture}[tldiagram, yscale=1/2, xscale=1/2]
	% lines
	\draw \tlcoord{0}{2} \dlinewave{1}{-2};
	\draw \tlcoord{0}{0} \dcapright;
	\draw \tlcoord{1}{2} \cupleft;
	% boxes
	%\draw \tlcoord{0}{0} \maketlboxgreen{3}{$3$};
	=
\end{tikzpicture}
}& 	\cbox{
\begin{tikzpicture}[tldiagram, yscale=1/2, xscale=1/2, ]
	% lines
	\draw \tlcoord{0}{1} \capright;
	\draw \tlcoord{1}{1} \cupright;
	\draw \tlcoord{0}{0} \lineup; 
	% boxes
	%\draw \tlcoord{0}{0} \maketlboxgreen{3}{$3$};
	=
\end{tikzpicture}
} & \cbox{
\begin{tikzpicture}[tldiagram, yscale=1/2, xscale=1/2, ]
	% lines
	\draw \tlcoord{0}{0} \capright;
	\draw \tlcoord{1}{0} \dcupright;
	\draw \tlcoord{0}{2} \dlineup; 
	% boxes
	%\draw \tlcoord{0}{0} \maketlboxgreen{3}{$3$};
	=
\end{tikzpicture}
} & \cbox{
\begin{tikzpicture}[tldiagram, yscale=1/2, xscale=1/2, ]
	% lines
	\draw \tlcoord{0}{0} \dcapright;
	\draw \tlcoord{1}{0} \cupright;
	\draw \tlcoord{0}{2} \dlineup; 
	% boxes
	%\draw \tlcoord{0}{0} \maketlboxgreen{3}{$3$};
	=
\end{tikzpicture}
} \\ \hline
		f_{1,1,1} & 1 & -\frac{[3]}{[4]} & -\frac{[3]}{[4]} & \frac{[2]}{[4]} & \frac{[2]}{[4]} & \frac{[2]}{[4]} & \frac{[2]}{[4]} & \frac{[2][2]}{[4]}& -\frac{[1]}{[4]}& -\frac{[1]}{[4]}\\ \hline
		f_{1,1,-1} & 0 & \frac{[1]}{[4]}&\frac{[1]}{[4]} & -\frac{[2]}{[4]}& -\frac{[2]}{[4]} & -\frac{[2]}{[4]}& -\frac{[2]}{[4]}& -\frac{[2][2]}{[4]}& \frac{[1]}{[4]}& \frac{[1]}{[4]}\\ \hline 
		f_{1,-1,1} & 0 & \frac{[1]}{2[2]}&\frac{[1]}{2[2]} & 0 & 0 & 0 & 0 & 0 & - \frac{[1]}{2[2]}& \frac{[1]}{2[2]}\\ \hline
		f_{-1,1,1} & 0 & \frac{[1]}{2[2]}& \frac{[1]}{2[2]}& 0 & 0 & 0 & 0& 0 & \frac{[1]}{2[2]} & -\frac{[1]}{2[2]}
	\end{array}
	\]
	Note that the sum of all coefficients in each column is zero, except for the identity. In particular a miraculous identity appears in the coefficient of $U_1$ respectively $U_0$. We have
	\[
	%\frac{[2]+[4]}{[2][4]}=
	\frac{[3]}{[4]}=\frac{1}{[4]}+\frac{1}{[2]}
	\]
	which is quite surprising. Also note that the coefficient $\frac{[2][2]}{[4]}$ in front of $U_2$ is not equal to $1$, however if we set $q=1$ it is indeed $1$. Observe that the coefficients for the other idempotents are not elements in $\mZ[[q]][q^{-1}]$ in contrast to $d_3=f_{1,1,1}$. For an explanation why the denominator $[4]$ appears in the first place, see the discussion in \Cref{quantum integers or not}.
\end{beis}

\section{Reidemeister moves} \label{section on Reidemeister moves}

This section is a collection of Reidemeister moves, which hold in the Temperley--Lieb algebra $\TL(B_n)$. Just like in type $A$ many geometric relations, which one would expect, do not hold directly, but rather include ``error powers'' of $q$. These $q$-powers play a very important role for the theory, since they are the main tool to realize Jones--Wenzl projectors as infinite braids, i.e.\ limits of images of braids with respect to the $q$-adic norm on $\TL(B_n)$. First we recall the Kauffman--Skein bracket relations, which define crossings as elements in the Temperley--Lieb algebra.

\begin{defi} \label{definitions of crossings}
	The \textbf{negative crossing} is defined as the linear combination
	\begin{equation*} 
		\cbox{
			\begin{tikzpicture}[tldiagram]
				\negcrossing{0}{0}
			\end{tikzpicture}
		} \quad = \quad \cbox{
			\begin{tikzpicture}[tldiagram]
				\draw \tlcoord{0}{0} \lineup;
				\draw \tlcoord{0}{1} \lineup;
			\end{tikzpicture}
		} \quad - \quad q \cdot \cbox{
			\begin{tikzpicture}[tldiagram]
				\draw \tlcoord{0}{0} \capright;
				\draw \tlcoord{1}{0} \cupright;
			\end{tikzpicture}
		} \in \TL(A_1).
	\end{equation*} 
	Similarly the \textbf{positive crossing} is defined as the linear combination 
	\begin{equation*}
		\cbox{
			\begin{tikzpicture}[tldiagram]
				\poscrossing{0}{0}
			\end{tikzpicture}
		} \quad = \quad \cbox{
			\begin{tikzpicture}[tldiagram]
				\draw \tlcoord{0}{0} \lineup;
				\draw \tlcoord{0}{1} \lineup;
			\end{tikzpicture}
		} \quad - \quad q^{-1} \cdot \cbox{
			\begin{tikzpicture}[tldiagram]
				\draw \tlcoord{0}{0} \capright;
				\draw \tlcoord{1}{0} \cupright;
			\end{tikzpicture}
		}\in \TL(A_1).
	\end{equation*}
\end{defi}

The first lemma of the section gives the type $A$ Reidemeister moves, which hold for our normalization of the Kauffman bracket. We only show the versions, which hold for negative crossings. If we replace each negative crossing by a positive crossing, the same identities hold with $q$ replaced by $q^{-1}$.

\begin{lemm}[Type $A$ Reidemeister moves]
	The following local relations hold in $\TL(A_{n-1}):$
	\begin{enumerate}[label = {(R\arabic*)}, align=left]
		\item
		\quad $\cbox{
				\begin{tikzpicture}[tldiagram, yscale=2/3]
					%dots 
					%\drawdots{0}{0}{1}
					%lines
					 \draw  \tlcoord{0}{0} \linewave{1}{1};
					 \draw[white, double=black] \tlcoord{0}{1} \linewave{1}{-1} \capright;
					%crossing
			\end{tikzpicture}} 
		\quad = \quad -q^{2}\cdot \cbox{
			\begin{tikzpicture}[tldiagram]
				%dots 
				%\drawdots{0}{0}{1}
				%lines
				\draw \tlcoord{0}{0} \capright;
			\end{tikzpicture}
		}$, 
		\item[(R1)'] \quad 
			$\cbox{
				\begin{tikzpicture}[tldiagram, yscale=2/3]
					%dots 
					%\drawdots{0}{0}{0}
					%\drawdots{3}{0}{0}
					%lines
					\draw \tlcoord{1}{1} \cupright \lineup \capleft;
					\draw \tlcoord{0}{0} \lineup;
					\draw \tlcoord{2}{0} \lineup;
					%crossing
					\draw  \tlcoord{1}{0} \linewave{1}{1};
					\draw[white, double=black] \tlcoord{1}{1} \linewave{1}{-1};
			\end{tikzpicture}} 
			\quad = \quad q^{-1}\cdot \cbox{
				\begin{tikzpicture}[tldiagram]
					%dots 
					%\drawdots{0}{0}{0}
					%\drawdots{1}{0}{0}
					%lines
					\draw \tlcoord{0}{0} \lineup;
				\end{tikzpicture}
			}$,
		\item \quad $
		\cbox{
			\begin{tikzpicture}[tldiagram, yscale=2/3]
				%dots 
				%\drawdots{0}{0}{0}
				%\drawdots{3}{0}{0}
				%lines
				\negcrossing{0}{0} \negcrossing{1}{1}
				\draw \tlcoord{1}{0} \lineup \capright;
				\draw \tlcoord{0}{2} \lineup;
		\end{tikzpicture}} 
		\quad = \quad -q\cdot \cbox{
			\begin{tikzpicture}[tldiagram]
				%dots 
				%\drawdots{0}{0}{0}
				%\drawdots{1}{0}{0}
				%lines
				\draw \tlcoord{0}{0} \lineup;
				\draw \tlcoord{0}{1} \capright;
			\end{tikzpicture}
		} \, , \quad \cbox{
		\begin{tikzpicture}[tldiagram, yscale=2/3, xscale=-1]
			%dots 
			%\drawdots{0}{0}{0}
			%\drawdots{3}{0}{0}
			%lines
			\poscrossing{0}{0} \poscrossing{1}{1}
			\draw \tlcoord{1}{0} \lineup \capright;
			\draw \tlcoord{0}{2} \lineup;
	\end{tikzpicture}}
		\quad = \quad -q\cdot \cbox{
			\begin{tikzpicture}[tldiagram]
				%dots 
				%\drawdots{0}{0}{0}
				%\drawdots{1}{0}{0}
				%lines
				\draw \tlcoord{0}{2} \lineup;
				\draw \tlcoord{0}{0} \capright;
			\end{tikzpicture}
		}$,
		\item[(R2)'] \label{Reidemeister 2 move, other version}
			\quad $\cbox{
				\begin{tikzpicture}[tldiagram, yscale=2/3]
					%dots 
					%\drawdots{0}{0}{0}
					%\drawdots{3}{0}{0}
					%lines
					\poscrossing{0}{0} \negcrossing{1}{0}
			\end{tikzpicture}} 
			\quad = \quad \cbox{
				\begin{tikzpicture}[tldiagram]
					%dots 
					%\drawdots{0}{0}{0}
					%\drawdots{1}{0}{0}
					%lines
					\draw \tlcoord{0}{0} \lineup;
					\draw \tlcoord{0}{1} \lineup;
				\end{tikzpicture}
			} 
			\quad = \quad \cbox{
				\begin{tikzpicture}[tldiagram, yscale=2/3]
					%dots 
					%\drawdots{0}{0}{0}
					%\drawdots{1}{0}{0}
					%lines
					\negcrossing{0}{0} \poscrossing{1}{0}
				\end{tikzpicture}
			}$,
		\item \label{Reidemeister 3 move} \quad $
			\cbox{
				\begin{tikzpicture}[tldiagram, yscale=2/3]
					%dots 
					%\drawdots{0}{0}{0}
					%\drawdots{3}{0}{0}
					%lines
					\negcrossing{0}{0} \negcrossing{1}{1} \negcrossing{2}{0}
					\draw \tlcoord{0}{2} \lineup;
					\draw \tlcoord{1}{0} \lineup;
					\draw \tlcoord{2}{2} \lineup;
			\end{tikzpicture}} 
			\quad = \quad \cbox{
				\begin{tikzpicture}[tldiagram, yscale=2/3, scale=-1]
					%dots 
					%\drawdots{0}{0}{0}
					%\drawdots{3}{0}{0}
					%lines
					\negcrossing{0}{0} \negcrossing{1}{1} \negcrossing{2}{0}
					\draw \tlcoord{0}{2} \lineup;
					\draw \tlcoord{1}{0} \lineup;
					\draw \tlcoord{2}{2} \lineup;
			\end{tikzpicture}}$.
	\end{enumerate}
\end{lemm}
\begin{proof}
	The proof is a straight-forward calculation using \Cref{definitions of crossings}
	and the Temperley--Lieb relations, most notably $\tlcircle=q+q^{-1}$. The relation \ref{Reidemeister 3 move} uses the relations
	\[
		U_1U_2U_1=U_1, \quad U_2U_1U_2=U_2. \qedhere
	\]
\end{proof}
As a direct consequence of \ref{Reidemeister 2 move, other version} and \ref{Reidemeister 3 move} we obtain the following corollary, which states that the Temperley--Lieb algebra is a quotient of the braid group.
\begin{koro} \label{Quotient of the braid group is Temperley--Lieb}
	The assignments 
	\begin{align*}
		\C(q) \! B(A_{n-1}) &\longrightarrow \TL(A_{n-1})=\TL_n(\delta) \\
		\sigma_i &\longmapsto 1-qU_i \\
		\sigma_i^{-1} &\longmapsto 1-q^{-1}U_i
	\end{align*}
	extend to a well-defined surjective $\C(q)$-algebra homomorphism from the group algebra of the braid group to the Temperley--Lieb algebra .
\end{koro}
The next remark explains how the last algebra homomorphism factors through the well-known morphism $\Hecke_q(A_{n-1})\rightarrow \TL(A_{n-1})$, which maps the Kazdhan--Lusztig generators to the Temperley--Lieb algebra generators. 
\begin{rema}[Correct normalization] \label{correct normalization}
	The algebra homomorphism from \Cref{Quotient of the braid group is Temperley--Lieb} is the composition of four algebra morphisms 
	\[
		\C(q)B(A_{n-1})\overset{\varphi_1}{\longrightarrow} \C(q)B(A_{n-1}) \overset{\varphi_2}{\longrightarrow} \Hecke_q(A_{n-1})\overset{\varphi_3}{\longrightarrow} \Hecke_q(A_{n-1})\overset{\varphi_4}{\longrightarrow}\TL(A_{n-1}).
	\]
	The first morphism $\varphi_1$ is the $q$-antilinear morphism, which maps every generator $\sigma_i$ to its inverse. The second morphism $\varphi_2$ is $q$-linear and maps the generator $\sigma_i$ to $qH_i$, where $H_i$ is the standard generator of the Hecke algebra.
	The third morphism $\varphi_3$ is the $q$-antilinear algebra involution of $\Hecke_q(A_{n-1})$, which maps $H_i$ to $-H_i$. The final morphism $\varphi_4$ is $q$-linear and maps $H_i$ to $U_i-q$. In particular note that $\varphi_4$ maps the Kazdhan--Lusztig generator $C_i\coloneqq H_i+q$ to $U_i$. Some authors use the definition of $\varphi_4$ for their definition of the Kauffmann bracket
	\[
		\left( \cbox{
			\begin{tikzpicture}[tldiagram]
				\poscrossing{0}{0}
			\end{tikzpicture}
		} \right)' \quad \coloneqq \quad  \cbox{
			\begin{tikzpicture}[tldiagram]
				\draw \tlcoord{0}{0} \capright;
				\draw \tlcoord{1}{0} \cupright;
			\end{tikzpicture}
		} \quad - q \cdot \cbox{
		\begin{tikzpicture}[tldiagram]
			\draw \tlcoord{0}{0} \lineup;
			\draw \tlcoord{0}{1} \lineup;
		\end{tikzpicture}
	},
	\]
	so this remark relates our and their definitions.
\end{rema}

The next lemma gives new Reidemeister moves for type $B$. It consists of the type $B$ braid relation as well as interesting type $B$ versions of the Reidemeister I relations.
\begin{lemm}[Type $B$ Reidemeister moves] \label{Type B Reidemeister moves lemma}
	The following local relations hold in $\TL(B_n)$:
	\begin{enumerate}
		\item[(B0)] \label{Reidemeister type B} \quad $
			\cbox{
				\begin{tikzpicture}[tldiagram, yscale=2/3]
					\uppernegcrossing{0}{0} \uppernegcrossing{1}{0}
			\end{tikzpicture}} 
			\quad = \quad \cbox{
				\begin{tikzpicture}[tldiagram, yscale=2/3]
					\lowernegcrossing{0}{0} \lowernegcrossing{1}{0}
			\end{tikzpicture}} $,
		\item[(B1)] \quad $
			\cbox{
				\begin{tikzpicture}[tldiagram, yscale=2/3]
					%dots 
					%\drawdots{0}{0}{1}
					%lines
					\draw \tlcoord{1}{0} \dcapright;
					%crossing
					\negcrossing{0}{0}
			\end{tikzpicture}} 
			\quad = \quad \cbox{
				\begin{tikzpicture}[tldiagram]
					%dots 
					%\drawdots{0}{0}{1}
					%lines
					\draw \tlcoord{0}{0} \dcapright;
				\end{tikzpicture}
			}$,
		\item[(B1)'] \quad $
			\cbox{
				\begin{tikzpicture}[tldiagram, yscale=2/3, scale=-1]
					%dots 
					%\drawdots{0}{0}{0}
					%\drawdots{3}{0}{0}
					%lines
					\draw \tlcoord{1}{1} \cupright \dlineup \capleft;
					\draw \tlcoord{0}{0} \lineup;
					\draw \tlcoord{2}{0} \lineup;
					%crossing
					\negcrossing{1}{0}
			\end{tikzpicture}} 
			\quad = \quad -q\cdot \cbox{
				\begin{tikzpicture}[tldiagram]
					%dots 
					%\drawdots{0}{0}{0}
					%\drawdots{1}{0}{0}
					%lines
					\draw \tlcoord{0}{0} \dlineup;
				\end{tikzpicture}
			} $.
	\end{enumerate}
\end{lemm}
\begin{proof}
	These are straight-forward calculations using that a dotted circle is zero.
\end{proof}
\begin{warn}
	There is no special type $B$ analogue of Reidemeister II, for instance the relation
	\[
		\cbox{
			\begin{tikzpicture}[tldiagram, yscale=2/3]
				%dots 
				%\drawdots{0}{0}{0}
				%\drawdots{3}{0}{0}
				%lines
				\negcrossing{0}{0} \negcrossing{1}{1}
				\draw \tlcoord{1}{0} \lineup \dcapright;
				\draw \tlcoord{0}{2} \lineup;
		\end{tikzpicture}} 
		\quad = \quad \lambda \cdot \cbox{
			\begin{tikzpicture}[tldiagram]
				%dots 
				%\drawdots{0}{0}{0}
				%\drawdots{1}{0}{0}
				%lines
				\draw \tlcoord{0}{0} \lineup;
				\draw \tlcoord{0}{1} \dcapright;
			\end{tikzpicture}
		}
	\]
	does not make sense for any $\lambda\in \C(q)$, since the right side is not a Temperley--Lieb diagram of type $B$. %However if we work over $\C$ and set $q=1$ this relation exists and makes sense using images of the Jucys--Murphys elements $J_i\in \C \! W(B_n)$ we encountered as \eqref{image of Jucys-Murphy element in Type B Weyl Group} in $\TL(B_n)$.
\end{warn}
As a consequence of \Cref{Type B Reidemeister moves lemma} we obtain a type $B$ version of \Cref{Quotient of the braid group is Temperley--Lieb}.
\begin{koro} \label{well-defined map from B1 to TL (B_n)}
	The assignments in \eqref{definitions of crossings} with $s_0\mapsto s_0$ define a surjective algebra morphism $\C(q)B_1(B_n)\to \TL(B_n)$.
\end{koro}
\begin{proof}
	The well-definedness follows from the relations \ref{Reidemeister type B}, \ref{Reidemeister 2 move, other version} and \ref{Reidemeister 3 move}. Surjectivity is clear, since all the generators $s_0, U_1, \ldots, U_{n-1}$ are contained in the image.
\end{proof}

\section{An affine Temperley--Lieb algebra quotient} \label{affine section}

This section gives a different viewpoint on the algebra $\TL(B_n)$, namely as a quotient of the extended affine Temperley--Lieb algebra. The section serves two purposes: The first purpose is to adjust the notions from \cite{wedrich2018} to fit with our quantum setting. The second purpose is to give an immediate application of the Reidemeister moves from \Cref{section on Reidemeister moves}. We fix $k=\C(q)$ and start with the definition of the extended affine Temperley--Lieb algebra. 
\begin{defi}
	Let $\delta=q+q^{-1}\in k$. The \textbf{extended affine Temperley--Lieb algebra} $\ATL_n$ is the (associative, unital) $k$-algebra with generators $U_1, \ldots, U_{n-1}, D, D^{-1}$ and relations
	\begin{enumerate}[label = {(ATL\arabic*)}, align=left]
		\item $U_i^2=\delta U_{i}$ for all $i\in\mZn$, where
		$U_0\coloneqq U_n\coloneqq DU_{1}D^{-1}$, \label{affine relation i}
		\item $U_iU_{i+1}U_i=U_i$ and $U_{i+1}U_{i}U_{i+1}=U_{i+1}$ for all $i\in\mZn$,  \label{affine relation ii}
		\item $U_iU_j=U_jU_i$ for all $i,j\in\mZn$ which satisfy $i-j\neq \pm1$, \label{affine relation iii}
		\item $D D^{-1} = 1=D^{-1} D$, \label{affine relation iv}
		\item $U_iD=DU_{i+1}$ for all $i\in\mZn$. \label{affine relation v}
	\end{enumerate}
\end{defi}

A basis of this algebra can be identified with certain Temperley--Lieb diagrams in the annulus $S^1\times [0,1]$, which connect $n$ vertices $\{1,\ldots, n\}\times \{1\}$ positioned on the inner boundary with $n$ vertices $\{1,\ldots, n\}\times \{0\}$ on the outer boundary, and contain no contractible loops, i.e.\ cycles avoiding the hole in the annulus. However they can have connected components going around the loop. The composition of two such diagrams $\lambda$ and $\mu$ is obtained by gluing $\lambda$ along its outer boundary to the inner boundary of $\mu$, such that the outer $n$ vertices are glued to their corresponding inner vertices. Every closed loop that arises in this procedure and does not go around the inner circle, is replaced by the scalar $\delta$. 
\begin{beis} \label{n=3 example for affine Temperley--Lieb}
	The generators of $\ATL_3$ are given by
	%\definecolor{ffqqqq}{rgb}{1.0,0.0,0.0}
	%\definecolor{qqqqff}{rgb}{0.0,0.0,1.0}
	%\definecolor{qqccqq}{rgb}{0.0,0.8,0.0}
	%\definecolor{uuuuuu}{rgb}{0.26666666666666666,0.26666666666666666,0.26666666666666666}
	\definecolor{ffqqqq}{rgb}{0.5,0.5,0.5}
	\definecolor{qqqqff}{rgb}{0.5,0.5,0.5}
	\definecolor{qqccqq}{rgb}{0.5,0.5,0.5}
	\definecolor{uuuuuu}{rgb}{0.26666666666666666,0.26666666666666666,0.26666666666666666}
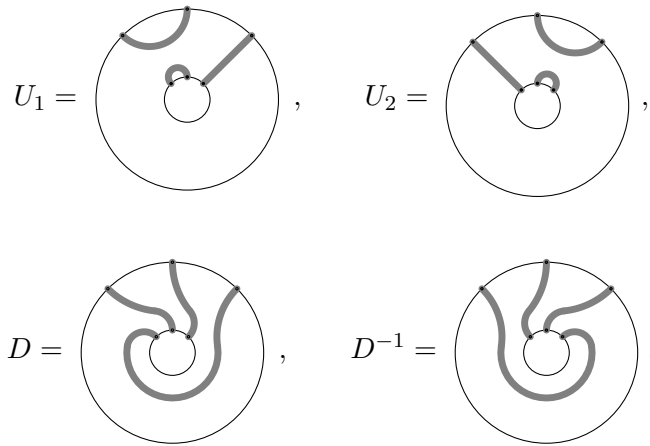
	\begin{gather*}
		U_1 = \cbox{
			\begin{tikzpicture}[line cap=round,line join=round,>=triangle 45,x=1.0cm,y=1.0cm, yscale=0.3, xscale=0.3]
				%\clip(-9.669038064820429,-5.005708867037201) rectangle (5.918346826603462,5.6978687489367035);
				\draw(0.0,0.0) circle (1.0cm);
				\draw(0.0,0.0) circle (4.0cm);
				\draw [shift={(-0.41421356237309503,1.0)},line width=2.6pt,color=qqccqq] plot[domain=0.0:3.9269908169872414,variable=\t]({1.0*0.41421356237309503*cos(\t r)+-0.0*0.41421356237309503*sin(\t r)},{0.0*0.41421356237309503*cos(\t r)+1.0*0.41421356237309503*sin(\t r)});
				\draw [shift={(-1.6568542494923801,4.0)},line width=2.6pt,color=qqqqff] plot[domain=-2.356194490192345:0.0,variable=\t]({1.0*1.6568542494923801*cos(\t r)+-0.0*1.6568542494923801*sin(\t r)},{0.0*1.6568542494923801*cos(\t r)+1.0*1.6568542494923801*sin(\t r)});
				\draw [line width=2.6pt,color=ffqqqq] (2.8284271247461903,2.82842712474619)-- (0.7071067811865476,0.7071067811865475);
				\begin{scriptsize}
					\draw [fill=black] (0.7071067811865476,0.7071067811865475) circle (1.5pt);
					\draw [fill=black] (0.0,1.0) circle (1.5pt);
					\draw [fill=black] (-0.7071067811865475,0.7071067811865476) circle (1.5pt);
					\draw [fill=uuuuuu] (2.8284271247461903,2.82842712474619) circle (1.5pt);
					\draw [fill=uuuuuu] (0.0,4.0) circle (1.5pt);
					\draw [fill=uuuuuu] (-2.82842712474619,2.8284271247461903) circle (1.5pt);
				\end{scriptsize}
			\end{tikzpicture}
		} \, , \qquad 
		U_2 =  \! \! \! \! \! \! \! \! \!  \cbox{
			\begin{tikzpicture}[line cap=round,line join=round,>=triangle 45,x=1.0cm,y=1.0cm, xscale=-0.3, yscale=0.3]
				\clip(-9.669038064820429,-5.005708867037201) rectangle (5.918346826603462,5.6978687489367035);
				\draw(0.0,0.0) circle (1.0cm);
				\draw(0.0,0.0) circle (4.0cm);
				\draw [shift={(-0.41421356237309503,1.0)},line width=2.6pt,color=qqccqq] plot[domain=0.0:3.9269908169872414,variable=\t]({1.0*0.41421356237309503*cos(\t r)+-0.0*0.41421356237309503*sin(\t r)},{0.0*0.41421356237309503*cos(\t r)+1.0*0.41421356237309503*sin(\t r)});
				\draw [shift={(-1.6568542494923801,4.0)},line width=2.6pt,color=qqqqff] plot[domain=-2.356194490192345:0.0,variable=\t]({1.0*1.6568542494923801*cos(\t r)+-0.0*1.6568542494923801*sin(\t r)},{0.0*1.6568542494923801*cos(\t r)+1.0*1.6568542494923801*sin(\t r)});
				\draw [line width=2.6pt,color=ffqqqq] (2.8284271247461903,2.82842712474619)-- (0.7071067811865476,0.7071067811865475);
				\begin{scriptsize}
					\draw [fill=black] (0.7071067811865476,0.7071067811865475) circle (1.5pt);
					\draw [fill=black] (0.0,1.0) circle (1.5pt);
					\draw [fill=black] (-0.7071067811865475,0.7071067811865476) circle (1.5pt);
					\draw [fill=uuuuuu] (2.8284271247461903,2.82842712474619) circle (1.5pt);
					\draw [fill=uuuuuu] (0.0,4.0) circle (1.5pt);
					\draw [fill=uuuuuu] (-2.82842712474619,2.8284271247461903) circle (1.5pt);
				\end{scriptsize}
			\end{tikzpicture}
		}  \! \! \! \! \! \! \! \! \!  \! \! \! \! \! \! \! \! \! \! \! \! \! \! \! \! \!, \\[1em]
		%--------------------------------------
		D=  \cbox{
			\begin{tikzpicture}[line cap=round,line join=round,>=triangle 45,x=1.0cm,y=1.0cm, xscale=0.3, yscale=0.3]
				%\clip(-10.091981629918333,-5.737312563324368) rectangle (8.487639250700614,7.0209798350683155);
				\draw(0.0,0.0) circle (1.0cm);
				\draw(0.0,0.0) circle (4.0cm);
				\draw [shift={(0.10451473109487833,1.309698831278217)},line width=2.6pt,color=qqqqff] plot[domain=-0.7853981633974483:0.6833346315911222,variable=\t]({1.0*0.8521938498178464*cos(\t r)+-0.0*0.8521938498178464*sin(\t r)},{0.0*0.8521938498178464*cos(\t r)+1.0*0.8521938498178464*sin(\t r)});
				\draw [shift={(3.4087753992713847,4.0)},line width=2.6pt,color=qqqqff] plot[domain=3.141592653589793:3.8249272851809155,variable=\t]({1.0*3.4087753992713843*cos(\t r)+-0.0*3.4087753992713843*sin(\t r)},{0.0*3.4087753992713843*cos(\t r)+1.0*3.4087753992713843*sin(\t r)});
				\draw [shift={(-0.4180589243795132,5.238795325112868)},line width=2.6pt,color=ffqqqq] plot[domain=3.9269908169872414:4.610325448578363,variable=\t]({1.0*3.4087753992713847*cos(\t r)+-0.0*3.4087753992713847*sin(\t r)},{0.0*3.4087753992713847*cos(\t r)+1.0*3.4087753992713847*sin(\t r)});
				\draw [shift={(-0.8521938498178466,1.0000000000000002)},line width=2.6pt,color=ffqqqq] plot[domain=0.0:1.46873279498857,variable=\t]({1.0*0.8521938498178467*cos(\t r)+-0.0*0.8521938498178467*sin(\t r)},{0.0*0.8521938498178467*cos(\t r)+1.0*0.8521938498178467*sin(\t r)});
				\draw [shift={(-1.2374368670764582,0.17677669529663698)},line width=2.6pt,color=qqccqq] plot[domain=0.7853981633974483:2.999695598985629,variable=\t]({1.0*0.75*cos(\t r)+-0.0*0.75*sin(\t r)},{0.0*0.75*cos(\t r)+1.0*0.75*sin(\t r)});
				\draw [shift={(4.949747468305832,0.7071067811865477)},line width=2.6pt,color=qqccqq] plot[domain=2.356194490192345:3.283489708193957,variable=\t]({1.0*2.9999999999999987*cos(\t r)+-0.0*2.9999999999999987*sin(\t r)},{0.0*2.9999999999999987*cos(\t r)+1.0*2.9999999999999987*sin(\t r)});
				\draw [shift={(0.0,-0.0)},line width=2.6pt,color=qqccqq] plot[domain=-3.283489708193957:0.14189705460416396,variable=\t]({1.0*2.0000000000000004*cos(\t r)+-0.0*2.0000000000000004*sin(\t r)},{0.0*2.0000000000000004*cos(\t r)+1.0*2.0000000000000004*sin(\t r)});
				\begin{scriptsize}
					\draw [fill=qqqqff] (0.7071067811865476,0.7071067811865475) circle (1.5pt);
					\draw [fill=qqqqff] (6.123233995736766E-17,1.0) circle (1.5pt);
					\draw [fill=qqqqff] (-0.7071067811865475,0.7071067811865476) circle (1.5pt);
					\draw [fill=qqqqff] (2.8284271247461903,2.82842712474619) circle (1.5pt);
					\draw [fill=qqqqff] (2.4492935982947064E-16,4.0) circle (1.5pt);
					\draw [fill=qqqqff] (-2.82842712474619,2.8284271247461903) circle (1.5pt);
				\end{scriptsize}
			\end{tikzpicture}
		} \, , \qquad
		D^{-1} =  \cbox{
			\begin{tikzpicture}[line cap=round,line join=round,>=triangle 45,x=1.0cm,y=1.0cm,, xscale=0.3, yscale=0.3]
				%\clip(-9.669038064820429,-5.005708867037201) rectangle (5.918346826603462,5.6978687489367035);
				\draw(0.0,0.0) circle (1.0cm);
				\draw(0.0,0.0) circle (4.0cm);
				\draw [shift={(-0.10451473109487842,1.3096988312782165)},line width=2.6pt,color=qqccqq] plot[domain=2.4582580219986703:3.926990816987241,variable=\t]({1.0*0.8521938498178462*cos(\t r)+-0.0*0.8521938498178462*sin(\t r)},{0.0*0.8521938498178462*cos(\t r)+1.0*0.8521938498178462*sin(\t r)});
				\draw [shift={(-3.4087753992713847,4.0)},line width=2.6pt,color=qqccqq] plot[domain=-0.6833346315911228:0.0,variable=\t]({1.0*3.4087753992713847*cos(\t r)+-0.0*3.4087753992713847*sin(\t r)},{0.0*3.4087753992713847*cos(\t r)+1.0*3.4087753992713847*sin(\t r)});
				\draw [shift={(0.8521938498178464,1.0)},line width=2.6pt,color=ffqqqq] plot[domain=1.6728598586012229:3.141592653589793,variable=\t]({1.0*0.8521938498178463*cos(\t r)+-0.0*0.8521938498178463*sin(\t r)},{0.0*0.8521938498178463*cos(\t r)+1.0*0.8521938498178463*sin(\t r)});
				\draw [shift={(0.41805892437951225,5.238795325112868)},line width=2.6pt,color=ffqqqq] plot[domain=4.814452512191016:5.497787143782138,variable=\t]({1.0*3.4087753992713856*cos(\t r)+-0.0*3.4087753992713856*sin(\t r)},{0.0*3.4087753992713856*cos(\t r)+1.0*3.4087753992713856*sin(\t r)});
				\draw [shift={(-4.949747468305832,0.7071067811865488)},line width=2.6pt,color=qqqqff] plot[domain=-0.14189705460416402:0.7853981633974483,variable=\t]({1.0*2.999999999999999*cos(\t r)+-0.0*2.999999999999999*sin(\t r)},{0.0*2.999999999999999*cos(\t r)+1.0*2.999999999999999*sin(\t r)});
				\draw [shift={(1.2374368670764584,0.17677669529663675)},line width=2.6pt,color=qqqqff] plot[domain=0.14189705460416382:2.356194490192345,variable=\t]({1.0*0.7499999999999999*cos(\t r)+-0.0*0.7499999999999999*sin(\t r)},{0.0*0.7499999999999999*cos(\t r)+1.0*0.7499999999999999*sin(\t r)});
				\draw [shift={(0.0,0.0)},line width=2.6pt,color=qqqqff] plot[domain=-3.283489708193957:0.1418970546041638,variable=\t]({1.0*2.0000000000000004*cos(\t r)+-0.0*2.0000000000000004*sin(\t r)},{0.0*2.0000000000000004*cos(\t r)+1.0*2.0000000000000004*sin(\t r)});
				\begin{scriptsize}
					\draw [fill=uuuuuu] (0.0,1.0) circle (1.5pt);
					\draw [fill=uuuuuu] (0.7071067811865476,0.7071067811865476) circle (1.5pt);
					\draw [fill=uuuuuu] (-0.7071067811865476,0.7071067811865476) circle (1.5pt);
					\draw [fill=uuuuuu] (2.8284271247461903,2.8284271247461903) circle (1.5pt);
					\draw [fill=uuuuuu] (0.0,4.0) circle (1.5pt);
					\draw [fill=uuuuuu] (-2.8284271247461903,2.8284271247461903) circle (1.5pt);
				\end{scriptsize}
			\end{tikzpicture}
		}.
	\end{gather*}
	The special element $U_0$ is
	\[
	U_0 = \cbox
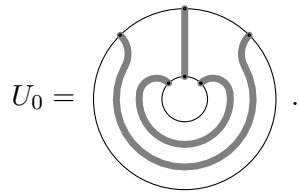
{
		\begin{tikzpicture}[line cap=round,line join=round,>=triangle 45,x=1.0cm,y=1.0cm, xscale=0.3, yscale=0.3]
			%\clip(-9.669038064820429,-5.005708867037201) rectangle (5.918346826603462,5.6978687489367035);
			\draw(0.0,0.0) circle (1.0cm);
			\draw(0.0,0.0) circle (4.0cm);
			\draw [line width=2.6pt,color=ffqqqq] (0.0,4.0)-- (6.123233995736766E-17,1.0);
			\draw [shift={(3.6533850361304974,2.0034692133618828)},line width=2.6pt,color=qqqqff] plot[domain=2.356194490192345:3.643196707778913,variable=\t]({1.0*1.166666666666669*cos(\t r)+-0.0*1.166666666666669*sin(\t r)},{0.0*1.166666666666669*cos(\t r)+1.0*1.166666666666669*sin(\t r)});
			\draw [shift={(-3.6533850361304974,2.003469213361883)},line width=2.6pt,color=qqqqff] plot[domain=-0.5016040541891202:0.7853981633974481,variable=\t]({1.0*1.1666666666666674*cos(\t r)+-0.0*1.1666666666666674*sin(\t r)},{0.0*1.1666666666666674*cos(\t r)+1.0*1.1666666666666674*sin(\t r)});
			\draw [shift={(0.0,0.0)},line width=2.6pt,color=qqqqff] plot[domain=-3.643196707778913:0.5016040541891199,variable=\t]({1.0*3.0*cos(\t r)+-0.0*3.0*sin(\t r)},{0.0*3.0*cos(\t r)+1.0*3.0*sin(\t r)});
			\draw [shift={(1.2374368670764582,0.17677669529663673)},line width=2.6pt,color=qqccqq] plot[domain=0.1418970546041638:2.356194490192345,variable=\t]({1.0*0.7500000000000001*cos(\t r)+-0.0*0.7500000000000001*sin(\t r)},{0.0*0.7500000000000001*cos(\t r)+1.0*0.7500000000000001*sin(\t r)});
			\draw [shift={(-1.2374368670764582,0.17677669529663698)},line width=2.6pt,color=qqccqq] plot[domain=0.7853981633974483:2.999695598985629,variable=\t]({1.0*0.75*cos(\t r)+-0.0*0.75*sin(\t r)},{0.0*0.75*cos(\t r)+1.0*0.75*sin(\t r)});
			\draw [shift={(0.0,0.0)},line width=2.6pt,color=qqccqq] plot[domain=-3.283489708193957:0.1418970546041638,variable=\t]({1.0*2.0000000000000004*cos(\t r)+-0.0*2.0000000000000004*sin(\t r)},{0.0*2.0000000000000004*cos(\t r)+1.0*2.0000000000000004*sin(\t r)});
			\begin{scriptsize}
				\draw [fill=black] (0.7071067811865476,0.7071067811865475) circle (1.5pt);
				\draw [fill=black] (6.123233995736766E-17,1.0) circle (1.5pt);
				\draw [fill=black] (-0.7071067811865475,0.7071067811865476) circle (1.5pt);
				\draw [fill=uuuuuu] (2.8284271247461903,2.82842712474619) circle (1.5pt);
				\draw [fill=uuuuuu] (0.0,4.0) circle (1.5pt);
				\draw [fill=uuuuuu] (-2.82842712474619,2.8284271247461903) circle (1.5pt);
			\end{scriptsize}
		\end{tikzpicture}
	} \, .
	\]
\end{beis}
The main statement of this section is that $\TL(B_n)$ is a quotient of the algebra $\ATL_n$. In this way the algebra $\TL(B_n)$ is nothing else then a quantum version of $\ATL^{\text{ess}}$ from \cite{wedrich2018}. It can be seen as a consequence of the Reidemeister moves from \Cref{section on Reidemeister moves}.
\begin{theo} \label{Theorem affine quotient}
	The assignments 
	\begin{align*}
		\varphi\colon \ATL_n&\longrightarrow\TL(B_n) \\
		U_i &\longmapsto U_i,  \quad   (1\leq i\leq n-1) \\
		D & \longmapsto [\sigma_{n-1}\sigma_{n-2}\cdots \sigma_1]s_0\\
		D^{-1} & \longmapsto s_0[\sigma_1^{-1} \cdots \sigma_{n-2}^{-1}\sigma_{n-1}^{-1}].
	\end{align*} 
	extend to a well-defined surjective algebra homomorphism $\ATL_n\twoheadrightarrow\TL(B_n)$. 
	Here $[\sigma]\in \TL(B_n)$ denotes the class of an element $\sigma \in B(B_n)$ under the quotient map $\C(q)B_1(B_n) \twoheadrightarrow \TL(B_n)$ from \Cref{well-defined map from B1 to TL (B_n)}.
\end{theo}
\begin{proof}
	This is a special case of \cite[\S2]{iohara18}, where a family of quotients of the affine Temperley--Lieb algebra depending on an additional parameter $Q$ were constructed, if we set $Q=1$. Alternatively one can use the Reidemeister moves from \Cref{section on Reidemeister moves} to check the well-definedness. Clearly if $\varphi$ is well-defined, it is surjective, since all generators $s_0, U_1, \ldots, U_{n-1}$ of $\TL(B_n)$ lie in the image. 
\end{proof}
\begin{beis}
	The image of $D$ is best understood via a side-by-side comparison of $D$ and its image. In the $n=3$ case, we have
	\begin{equation}
		\definecolor{ffqqqq}{rgb}{0.5,0.5,0.5}
		\definecolor{qqqqff}{rgb}{0.5,0.5,0.5}
		\definecolor{qqccqq}{rgb}{0.5,0.5,0.5}
		%	\definecolor{ffqqqq}{rgb}{1.0,0.0,0.0}
		%	\definecolor{qqccqq}{rgb}{0.0,0.8,0.0}
		%	\definecolor{qqqqff}{rgb}{0.0,0.0,1.0}
		\definecolor{uuuuuu}{rgb}{0.26666666666666666,0.26666666666666666,0.26666666666666666}
		D = \cbox{
			\begin{tikzpicture}[line cap=round,line join=round,>=triangle 45,x=1.0cm,y=1.0cm, xscale=0.3, yscale=0.3]
				%\clip(-10.091981629918333,-5.737312563324368) rectangle (8.487639250700614,7.0209798350683155);
				\draw(0.0,0.0) circle (1.0cm);
				\draw(0.0,0.0) circle (4.0cm);
				\draw [shift={(0.10451473109487833,1.309698831278217)},line width=2.6pt,color=qqqqff] plot[domain=-0.7853981633974483:0.6833346315911222,variable=\t]({1.0*0.8521938498178464*cos(\t r)+-0.0*0.8521938498178464*sin(\t r)},{0.0*0.8521938498178464*cos(\t r)+1.0*0.8521938498178464*sin(\t r)});
				\draw [shift={(3.4087753992713847,4.0)},line width=2.6pt,color=qqqqff] plot[domain=3.141592653589793:3.8249272851809155,variable=\t]({1.0*3.4087753992713843*cos(\t r)+-0.0*3.4087753992713843*sin(\t r)},{0.0*3.4087753992713843*cos(\t r)+1.0*3.4087753992713843*sin(\t r)});
				\draw [shift={(-0.4180589243795132,5.238795325112868)},line width=2.6pt,color=ffqqqq] plot[domain=3.9269908169872414:4.610325448578363,variable=\t]({1.0*3.4087753992713847*cos(\t r)+-0.0*3.4087753992713847*sin(\t r)},{0.0*3.4087753992713847*cos(\t r)+1.0*3.4087753992713847*sin(\t r)});
				\draw [shift={(-0.8521938498178466,1.0000000000000002)},line width=2.6pt,color=ffqqqq] plot[domain=0.0:1.46873279498857,variable=\t]({1.0*0.8521938498178467*cos(\t r)+-0.0*0.8521938498178467*sin(\t r)},{0.0*0.8521938498178467*cos(\t r)+1.0*0.8521938498178467*sin(\t r)});
				\draw [shift={(-1.2374368670764582,0.17677669529663698)},line width=2.6pt,color=qqccqq] plot[domain=0.7853981633974483:2.999695598985629,variable=\t]({1.0*0.75*cos(\t r)+-0.0*0.75*sin(\t r)},{0.0*0.75*cos(\t r)+1.0*0.75*sin(\t r)});
				\draw [shift={(4.949747468305832,0.7071067811865477)},line width=2.6pt,color=qqccqq] plot[domain=2.356194490192345:3.283489708193957,variable=\t]({1.0*2.9999999999999987*cos(\t r)+-0.0*2.9999999999999987*sin(\t r)},{0.0*2.9999999999999987*cos(\t r)+1.0*2.9999999999999987*sin(\t r)});
				\draw [shift={(0.0,-0.0)},line width=2.6pt,color=qqccqq] plot[domain=-3.283489708193957:0.14189705460416396,variable=\t]({1.0*2.0000000000000004*cos(\t r)+-0.0*2.0000000000000004*sin(\t r)},{0.0*2.0000000000000004*cos(\t r)+1.0*2.0000000000000004*sin(\t r)});
				\begin{scriptsize}
					\draw [fill=qqqqff] (0.7071067811865476,0.7071067811865475) circle (1.5pt);
					\draw [fill=qqqqff] (6.123233995736766E-17,1.0) circle (1.5pt);
					\draw [fill=qqqqff] (-0.7071067811865475,0.7071067811865476) circle (1.5pt);
					\draw [fill=qqqqff] (2.8284271247461903,2.82842712474619) circle (1.5pt);
					\draw [fill=qqqqff] (2.4492935982947064E-16,4.0) circle (1.5pt);
					\draw [fill=qqqqff] (-2.82842712474619,2.8284271247461903) circle (1.5pt);
				\end{scriptsize}
			\end{tikzpicture}
		}
		\quad \text{and} \quad
		\varphi(D) \quad = \quad \cbox{
			\begin{tikzpicture}[tldiagram]
				%dots
				%\drawdots{0}{0}{2}
				%\drawdots{2}{0}{2}
				%crossings
				\negcrossing{0}{0}
				\negcrossing{1}{1}
				%lines
				\draw \tlcoord{0}{2} \lineup;
				\draw \tlcoord{1}{0} \lineup;
				\draw \tlcoord{0}{0} \onedot;
				%big dots
				%\biggerdrawdots{0}{0}{0}
			\end{tikzpicture}
		}.
	\end{equation}
	We think of the contribution of $s_0$ in $\varphi(D)$ as symbolizing that we went around the hole once. The relation $s_0^2=1$ comes from our specialization of the $2$-parameter Hecke algebra. 
\end{beis}

\begin{rema}
	If we set $q=1$ in all formulas, the theorem gives precisely a surjective algebra morphism $\ATL_n\to \TL(B_n)$, where $\TL(B_n)=\ATL^{\text{ess}}_n$ from \cite{wedrich2018}.
	This gives a nice interpretation of the element $D^n$ as the image of the longest element $\wo\in \C W(B_n)$. For example we have
	\definecolor{ffqqqq}{rgb}{0.5,0.5,0.5}
	\definecolor{qqccqq}{rgb}{0.5,0.5,0.5}
	\definecolor{qqqqff}{rgb}{0.5,0.5,0.5}
	\definecolor{uuuuuu}{rgb}{0.26666666666666666,0.26666666666666666,0.26666666666666666}
	\begin{equation*}
		D^3 =  \cbox{
			\begin{tikzpicture}[line cap=round,line join=round,>=triangle 45,x=1.0cm,y=1.0cm, xscale=0.3, yscale=0.3]4.037378280492656) -- (0.0,4.524011808403494);
				
				%\draw[color=black] (0pt,-10pt) node[right] {\footnotesize $0$};
				%\clip(-8.241756579503134,-4.037378280492656) rectangle (4.321987396242472,4.524011808403494);
				\draw(0.0,0.0) circle (1.0cm);
				\draw(0.0,0.0) circle (4.0cm);
				\draw [shift={(0.0,-0.0)},line width=2.6pt,color=qqccqq] plot[domain=-3.141592653589793:0.0,variable=\t]({1.0*1.75*cos(\t r)+-0.0*1.75*sin(\t r)},{0.0*1.75*cos(\t r)+1.0*1.75*sin(\t r)});
				\draw [shift={(0.0,-0.0)},line width=2.6pt,color=ffqqqq] plot[domain=-3.141592653589793:0.0,variable=\t]({1.0*2.5*cos(\t r)+-0.0*2.5*sin(\t r)},{0.0*2.5*cos(\t r)+1.0*2.5*sin(\t r)});
				\draw [shift={(0.0,-0.0)},line width=2.6pt,color=qqqqff] plot[domain=-3.141592653589793:0.0,variable=\t]({1.0*3.25*cos(\t r)+-0.0*3.25*sin(\t r)},{0.0*3.25*cos(\t r)+1.0*3.25*sin(\t r)});
				\draw [shift={(-0.9888356558433666,0.0)},line width=2.6pt,color=qqccqq] plot[domain=1.191648621797656:3.141592653589793,variable=\t]({1.0*0.7611643441566334*cos(\t r)+-0.0*0.7611643441566334*sin(\t r)},{0.0*0.7611643441566334*cos(\t r)+1.0*0.7611643441566334*sin(\t r)});
				\draw [shift={(-1.05,0.0)},line width=2.6pt,color=ffqqqq] plot[domain=0.7610127542247297:3.141592653589793,variable=\t]({1.0*1.45*cos(\t r)+-0.0*1.45*sin(\t r)},{0.0*1.45*cos(\t r)+1.0*1.45*sin(\t r)});
				\draw [shift={(-1.2082691381318578,0.0)},line width=2.6pt,color=qqqqff] plot[domain=0.3536530910402287:3.141592653589793,variable=\t]({1.0*2.0417308618681425*cos(\t r)+-0.0*2.0417308618681425*sin(\t r)},{0.0*2.0417308618681425*cos(\t r)+1.0*2.0417308618681425*sin(\t r)});
				\draw [shift={(-1.1181625305160388,-0.0)},line width=2.6pt,color=qqqqff] plot[domain=0.0:0.44396197191690606,variable=\t]({1.0*4.363233773641306*cos(\t r)+-0.0*4.363233773641306*sin(\t r)},{0.0*4.363233773641306*cos(\t r)+1.0*4.363233773641306*sin(\t r)});
				\draw [shift={(3.8115342547371,2.3447100955368265)},line width=2.6pt,color=qqqqff] plot[domain=2.684342308803695:3.5855546255066995,variable=\t]({1.0*1.0956650005298587*cos(\t r)+-0.0*1.0956650005298587*sin(\t r)},{0.0*1.0956650005298587*cos(\t r)+1.0*1.0956650005298587*sin(\t r)});
				\draw [shift={(-0.8059962725622805,0.0)},line width=2.6pt,color=ffqqqq] plot[domain=0.0:0.9336368717723645,variable=\t]({1.0*3.311821685279578*cos(\t r)+-0.0*3.311821685279578*sin(\t r)},{0.0*3.311821685279578*cos(\t r)+1.0*3.311821685279578*sin(\t r)});
				\draw [shift={(-0.7253796609475539,0.0)},line width=2.6pt,color=qqccqq] plot[domain=0.0:1.6174999879637495,variable=\t]({1.0*2.491959243256882*cos(\t r)+-0.0*2.491959243256882*sin(\t r)},{0.0*2.491959243256882*cos(\t r)+1.0*2.491959243256882*sin(\t r)});
				\draw [shift={(3.60844757659975,5.964352872578287)},line width=2.6pt,color=ffqqqq] plot[domain=3.640107894129397:4.0752295253621575,variable=\t]({1.0*4.108476131252946*cos(\t r)+-0.0*4.108476131252946*sin(\t r)},{0.0*4.108476131252946*cos(\t r)+1.0*4.108476131252946*sin(\t r)});
				\draw [shift={(-1.0614355236888346,7.190260467855036)},line width=2.6pt,color=qqccqq] plot[domain=4.327491254759953:4.759092641553543,variable=\t]({1.0*4.706150170921378*cos(\t r)+-0.0*4.706150170921378*sin(\t r)},{0.0*4.706150170921378*cos(\t r)+1.0*4.706150170921378*sin(\t r)});
				\begin{scriptsize}
					\draw [fill=qqqqff] (0.7071067811865476,0.7071067811865476) circle (1.5pt);
					\draw [fill=qqqqff] (0.0,1.0) circle (1.5pt);
					\draw [fill=qqqqff] (-0.7071067811865476,0.7071067811865476) circle (1.5pt);
					\draw [fill=qqqqff] (2.8284271247461903,2.8284271247461903) circle (1.5pt);
					\draw [fill=qqqqff] (0.0,4.0) circle (1.5pt);
					\draw [fill=qqqqff] (-2.8284271247461903,2.8284271247461903) circle (1.5pt);
					%\draw [fill=qqqqff] (1.1642549531406337,2.662005443778098) circle (1.5pt);
					%\draw [fill=qqqqff] (-0.8417209757599243,2.489241966647907) circle (1.5pt);
				\end{scriptsize}
			\end{tikzpicture}
		}
		\quad \text{and} \quad
		\varphi(D^3)  = \rho(\wo) =  \quad \rho \left( \cbox{
			\begin{tikzpicture}[tldiagram]
				%dots
				%\drawdots{0}{0}{2}
				%\drawdots{1}{0}{2}
				%crossings
				%\negcrossing{0}{1}
				%\negcrossing{1}{0}
				%lines
				\draw \tlcoord{0}{0} \lineup;
				\draw \tlcoord{0}{1} \lineup;
				%\draw \tlcoord{1}{2} \lineup \lineup;
				%\draw \tlcoord{2}{0} \lineup;
				\draw \tlcoord{0}{2} \lineup;
				%dots
				\biggerdrawdots{0.5}{0}{2}
			\end{tikzpicture}
		} \right),
	\end{equation*}
	for $n=3$, where $\rho\colon \C \! W(B_n) \twoheadrightarrow \TL(B_n)$ is the quotient map. We will see in \Cref{Chapter on Infinite Braids} that the element $\wo\in W(B_n)$ is closely related to what we call type $B$ full-twist. The diagram for $D^n\in\ATL_n$ justifies the name since it literally twists around the hole in the annulus.
\end{rema}

\chapter{Infinite braids approximating projectors} \label{Chapter on Infinite Braids}

In this chapter of the thesis we discuss generalizations of the type $A$ full twist, which was used in \cite{rozansky2010} to categorify Jones--Wenzl projectors in type $A$, to other types. The goal is to use our generalization to obtain such a categorification of the type $D$ Jones--Wenzl projectors from the last chapter. We do not follow any particular source, however the chapter was influenced by discussions with the author's advisors as well as talks of Ben Elias and Matthew Hogancamp about their paper \cite{eliasdiag2}. 
Instead of reading this chapter as a strict mathematical text, we intend that the reader joins us on a philosophical journey to discover the full twist(s). 

\section{Full twist of type \texorpdfstring{$A$}{A}}

The type $A$ full twist is the special element
\[
	\alpha_{n-1}\coloneqq(\sigma_1\sigma_2\cdots\sigma_{n-1})^n\in B(A_{n-1})
\]
of the braid group on $n$ strands. Visually speaking $\alpha_{n-1}$ rotates the $n$ strands clockwise in a full circle (i.e.\ by $2\pi$). It can be written in many different ways. For example we have
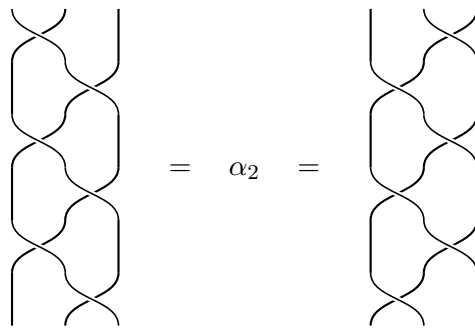
\[
	 \cbox{
	\begin{tikzpicture}[tldiagram, yscale=2/3]
		\negcrossing{0}{1}
		\negcrossing{1}{0}
		\negcrossing{2}{1}
		\negcrossing{3}{0}
		\negcrossing{4}{1}
		\negcrossing{5}{0}
		\draw \tlcoord{0}{0} \lineup;
		\draw \tlcoord{1}{2} \lineup;
		\draw \tlcoord{2}{0} \lineup;
		\draw \tlcoord{3}{2} \lineup;
		\draw \tlcoord{4}{0} \lineup;
		\draw \tlcoord{5}{2} \lineup;
	\end{tikzpicture}
} \quad  = \quad \alpha_2 \quad = \quad \cbox{
\begin{tikzpicture}[tldiagram, yscale=2/3]
	\negcrossing{0}{0}
	\negcrossing{1}{1}
	\negcrossing{2}{0}
	\negcrossing{3}{1}
	\negcrossing{4}{0}
	\negcrossing{5}{1}
	\draw \tlcoord{0}{2} \lineup;
	\draw \tlcoord{1}{0} \lineup;
	\draw \tlcoord{2}{2} \lineup;
	\draw \tlcoord{3}{0} \lineup;
	\draw \tlcoord{4}{2} \lineup;
	\draw \tlcoord{5}{0} \lineup;
\end{tikzpicture}
} 
\]
It satisfies many special properties, most notably we have
\[
	\sigma_i\alpha_{n-1}=\alpha_{n-1}\sigma_i \quad \text{ for all } i=1,\ldots, n-1.
\] 
This implies that $\alpha_{n-1}$ lives in the center of the braid group. One can even show that it generates its center for $n\geq3$ (see e.g.\ \cite[Theorem 4.2]{gonzales11}). The most important property from our point of view is the following fact. Consider the Temperley--Lieb algebra $\TL(A_{n-1})=\TL_n(q+q^{-1})$ defined over $k=\mZ[[q]][q^{-1}]$.
Then the image $[\alpha_{n-1}]$ of $\alpha_{n-1}$ under the surjection 
\[
	\mZ[[q]][q^{-1}]B(A_{n-1})\rightarrow \TL(A_{n-1})
\]
satisfies 
\begin{equation} \label{Type A full twist converges to jones--wenzl}
	\lim_{m\to\infty}[\alpha_{n-1}]^m=a_{n-1}
\end{equation}
where $a_{n-1}$ is the Jones--Wenzl projector and the limit is taken with respect to the $q$-adic norm. This norm is first defined on $\mZ[[q]][q^{-1}]$, where it is given by
\begin{align*}
	|\blank|_q\colon\mZ[[q]][q^{-1}] &\rightarrow \mR_{+} \\
	p=\sum_{k\geq n(p)}a_kt^k &\mapsto |p(t)|_q\coloneqq 2^{-n(p)}
\end{align*}
where $n(\blank)$ is the \textbf{$q$-adic valuation} on $\mZ[[q]][q^{-1}]$, that is $n(p)$ is the smallest index $k\in\mZ$ such that the coefficient $a_{k}\neq 0$ of $p$ is non-zero, using the convention $|0|_q=2^{-\infty}=0$. Since $\TL_n$ is free of finite rank as a $k$-module, we can define a norm $|\blank|$ on $\TL_n$ as the uniform norm
\begin{align*}
	|\blank|=|\blank|_{\infty}\colon \TL_n &\rightarrow \mR_{+} \\
	x=\sum_{\lambda}p_{\lambda}\lambda &\mapsto |x|_{\infty}\coloneqq\sup_{\lambda}|p_{\lambda}|_q
\end{align*}
with respect to the diagram basis $\{\lambda \mid \lambda \text{ is a Temperley--Lieb (n,n)-diagram}\}$ of $\TL_n$. The convergence \eqref{Type A full twist converges to jones--wenzl} was translated in \cite{rozansky2010} into a categorical setting via interpreting $[\alpha_{n-1}]$ as a complex in Bar-Natan's cobordism category describing Khovanov homology from \cite{barnatan2015}, and proved using the homological algebra methods, which we present in \Cref{chapter 5}. We try to motivate \eqref{Type A full twist converges to jones--wenzl} via a small example, where we consider $\TL_2$. 
\begin{beis} \label{Full Twist a1 example}
	Recall the type $A_1$ Hecke-algebra from \Cref{labellist}, which is the $\C(q)$-algebra on one generator $H_s\coloneqq H_1$, subject to the quadratic relation, which is equivalent to the inverse of $H_s$ being given by $H_s + (q - q^{-1})$.
	There are two isomorphism classes of irreducible $\Hecke_q(A_1)$ representations -- the trivial representation $V_{\opn{triv}}$ and the sign representation $V_{\opn{sign}}$, which are both $1$-dimensional $\C(q)$ vector spaces. The generator $H_s$ acts by $q^{-1}$ on $V_{\opn{triv}}$ and by $-q$ on $V_{\opn{sign}}$.
	The Temperley--Lieb algebra $\TL(A_1)=\TL_2(q+q^{-1})$ is isomorphic to $\Hecke_q(A_1)$ via the identification 
	\[
		\tlcupcap=U_1\coloneqq C_s \coloneqq H_s + q 
	\]
	where $C_s$ is the Kazdhan-Lusztig generator of $\Hecke_{q}(A_1)$. This gives two non-trivial idempotents of $\Hecke_{q}(A_1)$ namely the Jones--Wenzl projector
	\[
		a_1=\tlline\tlline-\frac{1}{q+q^{-1}}\tlcupcap = 1-\frac{1}{q+q^{-1}}(H_s+q) = \frac{1}{q+q^{-1}} (-H_s + q^{-1}).
	\] 
	and its orthogonal counter part
	\[
		\quad 1-a_1 =\frac{1}{q+q^{-1}}\tlcupcap=\frac{1}{q+q^{-1}}(H_s+q)
	\]
	We have isomorphisms
	\[
		V_{\opn{sign}}\cong \Hecke_q(A_1)a_1, \quad V_{\opn{triv}}\cong \Hecke_q(A_1)(1-a_1),
	\]
	since
	\[
	H_s (-H_s + q^{-1}) = -1 +(q-q^{-1})H_s +q^{-1} H_s = -1 + q H_s = -q (-H_s + q^{-1})
	\]
	and
	\[
		H_s (H_s + q) = 1 +(q^{-1}-q)H_s + qH_s =q^{-1} (H_s + q).
	\]
	The full twist of type $A_1$ is given by
	\[
	\alpha_1=\sigma_1^2 = \cbox{
	\begin{tikzpicture}[tldiagram, yscale=2/3]
		\negcrossing{0}{0}
		\negcrossing{1}{0}
	\end{tikzpicture}
} \, .
	\]
	The image of $\sigma_1\in \C(q)B(A_1)$ under the composition 
	\[
		[ - ] = \varphi_3 \circ \varphi_2 \circ \varphi_1\colon \C(q)B(A_1)\rightarrow \Hecke_q(A_1)
	\]
	from \Cref{correct normalization} is -- as explained there -- not $H_s$, but instead
	\[
		[\sigma_1]= 1-q(H_s+q)=\tlline\tlline-q\tlcupcap.
	\]
	This element acts by $1$ on the sign representation and by $-q^{2}$ on the trivial representation, which implies that the image of the full twist $f(\alpha_1)=f(\sigma_1^2)$ acts on the sign representation by $1$ and on the trivial representation by $q^4$. If we consider the action of 
	$\alpha_1^m$ for $m=1,2,\ldots$ we see that it fixes the sign representation and acts by higher and higher $q$-powers on the trivial representation. This already gives us a representation theoretic reason that powers of $\alpha_1^m$ should approximate $a_1$ and is backed up by the calculation
	\begin{align*}
		[\sigma_1]&=\tlline\tlline -q \tlcupcap, \\
		[\sigma_1^2] &= \tlline\tlline + (-q+q^3) \tlcupcap, \\
		[\sigma_1^3] &= \tlline\tlline + (-q+q^3-q^5) \tlcupcap, \\
		[\sigma_1^4] &= \tlline\tlline + (-q+q^3-q^5+q^7) \tlcupcap, 
	\end{align*}
	which gives
	\[
		\lim_{m\to\infty}[\sigma_1^m]=\tlline\tlline + (-q+q^3-q^5+q^7 + \cdots) \tlcupcap = \tlline\tlline - \frac{1}{q+q^{-1}} \tlcupcap = a_{1}
	\]
	where we view the fraction $\frac{1}{q+q^{-1}}$ inside $\mZ[[q]][q^{-1}]$. In particular
	\[
	\lim_{m\to\infty}[\alpha_1^m]=\lim_{m\to\infty}[\sigma_1^2]^m=a_{1}.
	\]
	Alternatively if one wants to phrase this behavior in terms of matrices, the action of $[\sigma_1]$ on $\TL_2$ with respect to the basis $1, U_1$ is given by the matrix $\begin{psmallmatrix}
		1 & 0 \\
		-q & -q^2
	\end{psmallmatrix}$
	and the powers of this matrix converge towards 
	$\begin{psmallmatrix}
		1 & 0 \\
		-\frac{1}{[2]} & 0
	\end{psmallmatrix}$ with respect to the $q$-adic norm on matrices. This point of view in terms of matrices will be the key to our proof of a type $D$ version of the statement. 
\end{beis}
\begin{rema}
	Convergence of powers of matrices appears prominently in the Perron--Frobenius theorem. This theorem states that each $n\times n$-matrix $A$ with strictly positive $\mR$-valued entries had a unique maximal eigenvalue $\lambda\in\mR_{>0}$, which is real and has a one-dimensional eigenspace. The projection onto this eigenspace is given as the limit of the matrix $(A/\lambda)^n$ as $n$ goes to infinity. For the statement and proof see e.g. \cite{smyth02}. However recently in \cite{costa16} a version of the Perron--Frobenius theorem was proven for $p$-adic numbers $\mQ_p$ and the methods they use apply to general discrete valuation rings/fields like $\C[[q]]$ or $\C[[q]][q^{-1}]$. We expect one can use this version of Perron--Frobenius to rephrase this theory of the Jones--Wenzl projector and possibly use this theorem as a general tool to construct idempotents for Hecke algebras, their quotients and $q$-deformations of other finite dimensional algebras appearing in Lie theory. For instance note that the eigenvalues of $[\sigma_1]$ in \Cref{Full Twist a1 example} are $1$ and $-q^2$, where $1$ is the maximal eigenvalue in the $q$-adic norm. This fits perfectly with Perron--Frobenius and the powers of $[\sigma_1]$ indeed converge towards the projection onto the eigenspace.
\end{rema}
\begin{rema}
	In \Cref{Full Twist a1 example} we could consider instead of the full twist $\alpha_1$ its inverse $\alpha_1'\coloneqq\alpha_1^{-1}$, which is given by
	\[
	\alpha_1' = \cbox{
		\begin{tikzpicture}[tldiagram, yscale=2/3]
			\poscrossing{0}{0}
			\poscrossing{1}{0}
		\end{tikzpicture}
	} \, .
	\]
	The image of this element in the Hecke algebra using our conventions is 
	\[
		[\sigma_1^{-2}]=q^{-2}H_s^2 = q^{-2}(1+(q^{-1}-q)H_s)=q^{-2}+(q^{-3}-q^{-1})H_s=1+(-q^{-1}+q^{-3})(H_s+q)
	\]
	It acts on $V_{\opn{sign}}$ by $1$ and on $V_{\opn{triv}}$ by $q^{-4}$. The powers of $\sigma_1^{-2}$ converge, however here we have
	\[
		\lim_{m\to\infty}[\sigma_1^m]=1 + (-q^{-1}+q^{-3}-q^{-5}+q^{-7} + \cdots) U_1 = 1 - \frac{1}{q+q^{-1}} U_1 = a_{1},
	\]
	where we view $\frac{1}{q+q^{-1}}$ as an element of $\mZ[[q^{-1}]][q]$. 
\end{rema}

When writing the thesis, the author stumbled upon the following problem. How can/should one define a full twist of the type $B$ or $D$ full twist? The answer lies in the so-called half twist.
For instance in \Cref{Full Twist a1 example} we considered the full twist of type $A_1$, which is not $\sigma_1$, but instead $\sigma_1^2$. This does not happen by accident, but is a general feature. The full twist can always be written as the square of the half-twist, which we define next.
\begin{defi}[Half twist] \label{defi half-twist}
	The type $A$ \textbf{half twist} $h_{n-1}\in B(A_{n-1})$ is the element
		\[
		h_{n-1}\coloneqq\sigma_1(\sigma_2\sigma_1)(\sigma_3\sigma_2\sigma_1)\cdots(\sigma_{n-1}\cdots\sigma_1).
		\]
	Visually this element corresponds to the rotation of the $n$ strands by $\pi$ in clockwise direction.
\end{defi}
\begin{lemm} \label{half-twist lemma}
	The half twist $h_{n-1}\in B(A_{n-1})$ satisfies $h_{n-1}^2=\alpha_{n-1}$. It is the image of the longest element $\wo\in S_n=W(A_{n-1})$ under the Matsumoto section 
	 \begin{align*}
		\iota\colon W(A_{n-1}) &\rightarrow B(A_{n-1})\\
		w=s_{i_1}\cdots s_{i_r}&\mapsto \sigma_{i_1}\cdots\sigma_{i_r},
	\end{align*}
	which is defined element-wise on $w\in W(A_{n-1})$ by mapping every element to the product $\sigma_{i_1}\cdots \sigma_{i_r}$ of braid generators, where $w=s_{i_1}\cdots s_{i_r}$ is a reduced expression in the Coxeter group.
\end{lemm}
\begin{proof}
	This first statement is clear from the geometric description of the half and full twist. See e.g.\ \cite{gonzales11} for the details.
	 For the second statement one checks inductively that this expression gives the permutation $\wo=(1,n)(2,n-1)(3,n-2)\cdots$, which is known to be the longest element. The expression is reduced, since the number of positive roots of $\sln$, which is ${n \choose 2}=\sum_{k=1}^{n-1}k$, agrees with the number of Coxeter generators in the expression. Note that the Matsumoto section is well-defined by Matsumoto's lemma.
\end{proof}
\section{Full twists of types \texorpdfstring{$B$}{B} and \texorpdfstring{$D$}{D}}

We start by considering the example $D_3$. This example is particularly interesting by the exceptional isomorphism $D_3\cong A_3$ of Dynkin diagrams, which allows us to translate the full twist of type $A$, which we already know, to receive a first full twist of type $D$.
\begin{beis}
	The half-twist $h_3\in B(D_3)\cong B(A_3)$ is given by the image
	\[
		h_3\quad =\quad \cbox
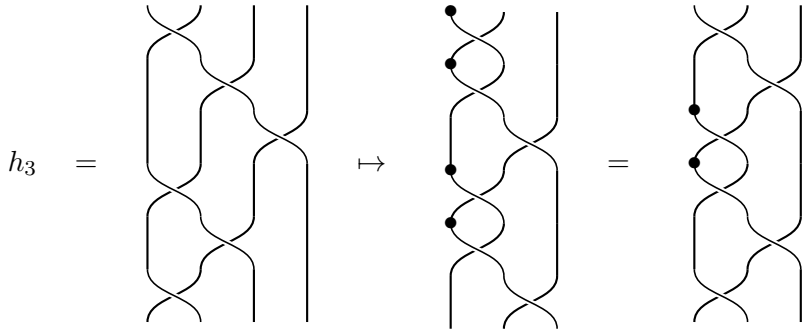
{
		\begin{tikzpicture}[tldiagram, yscale=2/3]
			\negcrossing{0}{0}
			\negcrossing{1}{1}
			\negcrossing{2}{0}
			\negcrossing{3}{2}
			\negcrossing{4}{1}
			\negcrossing{5}{0}
			\draw \tlcoord{0}{3} \lineup \lineup \lineup;
			\draw \tlcoord{0}{2} \lineup;
			\draw \tlcoord{1}{0} \lineup;
			\draw \tlcoord{2}{2} \lineup;
			\draw \tlcoord{3}{0} \lineup \lineup;
			\draw \tlcoord{3}{1} \lineup;
			\draw \tlcoord{4}{3} \lineup \lineup;
			\draw \tlcoord{5}{2} \lineup;
		\end{tikzpicture}
	} \quad \mapsto\quad  \cbox{
\begin{tikzpicture}[tldiagram, yscale=2/3]
	\negcrossing{0}{1}
	\negcrossing{1}{0}
	\dnegcrossing{2}{0}
	\negcrossing{3}{1}
	\negcrossing{4}{0}
	\dnegcrossing{5}{0}
	\draw \tlcoord{0}{0} \lineup;
	\draw \tlcoord{1}{2} \lineup;
	\draw \tlcoord{2}{2} \lineup;
	\draw \tlcoord{3}{0} \lineup;
	\draw \tlcoord{4}{2} \lineup;
	\draw \tlcoord{5}{2} \lineup;
\end{tikzpicture}
} \quad = \quad 
	\cbox{
		\begin{tikzpicture}[tldiagram, yscale=2/3]
			\negcrossing{0}{0}
			\negcrossing{1}{1}
			\negcrossing{2}{0}
			\dnegcrossing{3}{0}
			\negcrossing{4}{1}
			\negcrossing{5}{0}
			\draw \tlcoord{0}{2} \lineup;
			\draw \tlcoord{1}{0} \lineup;
			\draw \tlcoord{2}{2} \lineup;
			\draw \tlcoord{4}{0} \lineup;
			\draw \tlcoord{3}{2} \lineup;
			\draw \tlcoord{5}{2} \lineup;
		\end{tikzpicture}
	}
	\]
	where the diagrams on the right are visualizing the image of $h_3\in B(A_3)$ under the composition $B(A_3)\cong B(D_3)\rightarrow B_1(B_3)$, where the second map is from \Cref{braid group of type d inside type b}. In the last equality we used the type $D$ relations from \eqref{Dynkin diagrams} to obtain
	\[
		(\sigma_0'\sigma_1\sigma_2)^2=\sigma_1\sigma_0'\sigma_2\sigma_0'\sigma_1\sigma_2=\sigma_1\sigma_2\sigma_0'\sigma_2\sigma_1\sigma_2=\sigma_1\sigma_2\sigma_0'\sigma_1\sigma_2\sigma_1,
	\]
	where we used that $\sigma_0'$ and $\sigma_1$ commute.
	Note that the result coincides with the formula given for the longest element in $W(D_3)$ in \Cref{longest element in type D} or rather the image of $\wo\in W(D_3)$ under the Matsumoto section $W(D_3)\hookrightarrow B(D_3)$. Taking the square of this type $D_3$ half-twist gives a candidate for a full-twist of type $D_3$.
\end{beis}
Motivated by this example we propose a generalization of both \Cref{defi half-twist} and \Cref{half-twist lemma} to other finite Coxeter groups.
\begin{defi} \label{definition full twist}
	Let $C$ a Coxeter matrix, such that the corresponding Coxeter group $W(C)$ is finite. Let $\wo=s_{i_1}\cdots s_{i_r}$ be a reduced expression of the longest element $\wo\in W$.
	The \textbf{half twist} $h$ is the image of $\wo$ under the Matsumoto section in $W(C)\to B(C)$. The \textbf{full twist} is defined as $h^2$, the square of the half-twist. If $C=A_{n-1},B_n,D_n$ we denote the corresponding full twist by $\alpha_{n-1}, \beta_n$ and $\delta_n$, respectively.
\end{defi}

Since we want to generalize \eqref{Type A full twist converges to jones--wenzl} to type $B$ and $D$, we are interested in $\beta_n$ and $\delta_n$ specifically.
\begin{beis}[Type $B$ full twist]
	Let $n\geq2$. The longest element $\wo$ of $W(B_n)$ has reduced expression
	\begin{equation*}
		\wo=(s_{n-1}\cdots s_2s_1s_0s_1s_2 \dots s_{n-1})(s_{n-2}\cdots s_2s_1s_0s_1s_2 \dots s_{n-2})\cdots (s_1s_0s_1)s_0
	\end{equation*}
	as we saw in \Cref{longest element in type B}. Hence the corresponding full twist is given by
		\begin{align*}
		\beta_n&=((\sigma_{n-1}\cdots \sigma_2\sigma_1\sigma_0\sigma_1\sigma_2 \dots \sigma_{n-1})(\sigma_{n-2}\cdots \sigma_2\sigma_1\sigma_0\sigma_1\sigma_2 \dots \sigma_{n-2})\cdots (\sigma_1\sigma_0\sigma_1)\sigma_0)^2 \\
		&= J_{n}J_{n-1}\cdots J_1,
	\end{align*}
	where the $J_i$'s are the Jucys--Murphy elements from \Cref{defi mult jucys-murphys}.
	For a visualization, see \Cref{vizualisation of full twists} below.
\end{beis}
\begin{obse}
	Just like for the full twist of type $A_{n-1}$ the full twist of type $B_n$ is invariant under the action of $\mZ/ n\mZ$ permuting the standard generators. More precisely we have
	\begin{equation*}
		\beta_n=(\sigma_{\pi(n-1)}\sigma_{\pi(n-2)}\ldots\sigma_{\pi(0)})^{2n}
	\end{equation*}
	for all $\pi\in \mZ/ n\mZ$.
\end{obse}
\begin{beis}[Type $D$ full twist] \label{definition type D full twist}
	The full twist $\delta_n\in B(D_n)$ of type $D_n$ is
	\begin{align*}
		\delta_n=(\sigma_0'\sigma_1\sigma_2\sigma_3\ldots \sigma_{n-1})^{2(n-1)}
	\end{align*}
	by \Cref{longest element in type D}.
	In particular its image under the map $B(D_n)\subseteq B_{1}(B_n)$ from \Cref{Three families of groups} is given by
	\[
		((s_0\sigma_1s_0)\sigma_1\sigma_2\sigma_3\ldots \sigma_{n-1})^{2(n-1)}\in B_1(B_n).
	\]
\end{beis}
\begin{beis} \label{vizualisation of full twists}
	We picture $\beta_3$ and $\delta_3$ diagrammatically as
	\[
	\beta_3 = (J_3 J_2 J_1)^2 \quad =\quad \left( \cbox{
		\begin{tikzpicture}[tldiagram, yscale=1/2, xscale=2/3]
			%first stack
			\draw \tlcoord{0}{0.5} \lineup;
			\draw[white, double=black] \tlcoord{0}{1} \linewave{1}{-1} \linewave{1}{1};
			\draw[white, double=black] \tlcoord{1}{0.5} \lineup;
			%\draw \tlcoord{0}{2} \lineup \lineup;
			%\draw \tlcoord{0}{3} \lineup \lineup;
			%second stack
			\draw \tlcoord{2}{0.5} \lineup;
			\draw \tlcoord{2}{1} \lineup;
			\draw[white, double=black] \tlcoord{0}{2} \linewave{3}{-2} \linewave{2}{2};
			\draw[white, double=black] \tlcoord{3}{0.5} \lineup;
			\draw[white, double=black] \tlcoord{3}{1} \lineup;
			%\draw \tlcoord{2}{3} \lineup \lineup;
			%third stack
			\draw \tlcoord{4}{0.5} \lineup;
			\draw \tlcoord{4}{1} \lineup;
			\draw[white, double=black] \tlcoord{0}{3} \linewave{5}{-3} \linewave{1}{3};
			\draw[white, double=black] \tlcoord{5}{0.5} \lineup;
			\draw[white, double=black] \tlcoord{5}{1} \lineup;
			\draw[white, double=black] \tlcoord{5}{2} \lineup;
%			%foruth stack
%			\draw \tlcoord{0}{0.5} \lineup;
%			\draw[white, double=black] \tlcoord{0}{1} \linewave{1}{-1} \linewave{1}{1};
%			\draw[white, double=black] \tlcoord{1}{0.5} \lineup;
%			%\draw \tlcoord{0}{2} \lineup \lineup;
%			%\draw \tlcoord{0}{3} \lineup \lineup;
%			%second stack
%			\draw \tlcoord{2}{0.5} \lineup;
%			\draw \tlcoord{2}{1} \lineup;
%			\draw[white, double=black] \tlcoord{0}{2} \linewave{3}{-2} \linewave{2}{2};
%			\draw[white, double=black] \tlcoord{3}{0.5} \lineup;
%			\draw[white, double=black] \tlcoord{3}{1} \lineup;
%			%\draw \tlcoord{2}{3} \lineup \lineup;
%			%third stack
%			\draw \tlcoord{4}{0.5} \lineup;
%			\draw \tlcoord{4}{1} \lineup;
%			\draw[white, double=black] \tlcoord{0}{3} \linewave{5}{-3} \linewave{1}{3};
%			\draw[white, double=black] \tlcoord{5}{0.5} \lineup;
%			\draw[white, double=black] \tlcoord{5}{1} \lineup;
%			\draw[white, double=black] \tlcoord{5}{2} \lineup;
		\end{tikzpicture}
	} \right)^2\, , \quad \delta_3 \quad = \left( \cbox{
	\begin{tikzpicture}[tldiagram, yscale=1/2, xscale=2/3]
		\negcrossing{0}{1}
		\negcrossing{1}{0}
		\dnegcrossing{2}{0}
		\negcrossing{3}{1}
		\negcrossing{4}{0}
		\dnegcrossing{5}{0}
		\draw \tlcoord{0}{0} \lineup;
		\draw \tlcoord{1}{2} \lineup;
		\draw \tlcoord{2}{2} \lineup;
		\draw \tlcoord{3}{0} \lineup;
		\draw \tlcoord{4}{2} \lineup;
		\draw \tlcoord{5}{2} \lineup;
	\end{tikzpicture}
} \right)^2 \, .
	\]
	Technically speaking the diagrams are images of $\beta_3$ and $\delta_3$ under $B(B_n)\rightarrow B(A_n)$ from \Cref{embedding bn into an}, respectively $B(D_n)\rightarrow B_1(B_n)$ from \Cref{Three families of groups}.
\end{beis}

\section{Type \texorpdfstring{$D$}{D} approximation theorem}

The main goal of this section is to prove the approximation theorem, which generalizes \eqref{Type A full twist converges to jones--wenzl} from the type $A$ full twist to type $D$. Throughout the next definition consider $\TL(D_n)\subseteq \TL(B_n)$, where $\TL(B_n)$ is defined over the ring $k=\mZ[[q]][q^{-1}]$, c.f.\ \Cref{Definition Temperley--Lieb of Type B}.

\begin{theo}[Infinite type $D$ braid] \label{infinite type D braid}
	Consider the image $[\delta_n]\in\TL(D_n)$ of the type $D$ full twist $\delta_n$ under the quotient map
	\begin{align*}
		\mZ[[q]][q^{-1}]B(D_n)&\twoheadrightarrow \TL(D_n)\subseteq\TL(B_n), \\
		\sigma_i &\mapsto [\sigma_i]=1-qU_i, \text{ for } i=1,\ldots, n-1 \\
		\sigma_0' &\mapsto [\sigma_0']=1-qU_0=1-qs_0U_1s_0.
	\end{align*}
	Then powers of $[\delta_n]$ converge in the $q$-adic norm on $\TL(D_n)$ towards the Jones--Wenzl projector $d_n$. In formulas we have
	\[
		\lim_{m\to\infty}[\delta_n]^m= \quad \cbox{
			\begin{tikzpicture}[tldiagram, yscale=2/3]
				% lines
				\draw \tlcoord{-1.5}{0} \lineup \lineup \lineup;
				\draw \tlcoord{-1.5}{1} \lineup \lineup \lineup;
				\draw \tlcoord{-1.5}{2} \lineup  \lineup \lineup;
				% boxes
				\draw \tlcoord{0}{0} \maketlboxgreen{3}{$n$};
			\end{tikzpicture}
		}.
	\]
\end{theo}
Before we proof the theorem let us consider the two smallest examples.
\begin{beis} \label{full twist type D2 example}
	The image of the full twist $[\delta_2]\in\TL(B_2)$ is given by
	\[
		[\delta_2]  \quad = \quad \left( \cbox{
		\begin{tikzpicture}[tldiagram, yscale=2/3]
			\negcrossing{0}{0}
			\dnegcrossing{1}{0}
		\end{tikzpicture}
	} \right)^2 \quad = \quad \cbox{\begin{tikzpicture}[tldiagram]
			%dots
			%\drawdots{0}{0}{1}
			%\drawdots{1}{0}{1}
			%lines
			\draw \tlcoord{0}{0} \lineup;
			\draw \tlcoord{0}{1} \lineup;
	\end{tikzpicture}} 
	+(-q+q^3) \cdot \left( \cbox{
		\begin{tikzpicture}[tldiagram]
			%dots
			%\drawdots{0}{0}{1}
			%\drawdots{1}{0}{1}
			%lines
			\draw \tlcoord{0}{0} \capright;
			\draw \tlcoord{1}{0} \cupright;
		\end{tikzpicture}
	} + \cbox{
		\begin{tikzpicture}[tldiagram]
			%dots
			%\drawdots{0}{0}{1}
			%\drawdots{1}{0}{1}
			%lines
			\draw \tlcoord{0}{0} \dcapright;
			\draw \tlcoord{1}{0} \dcupright;
			%big dots
		\end{tikzpicture}
	} \right) .
	\]
	When we square this element we obtain by the type $B$ Temperley-Lieb relations
	\[
	[\delta_2]^2  \quad = \quad \left( \cbox{
		\begin{tikzpicture}[tldiagram, yscale=2/3]
			\negcrossing{0}{0}
			\dnegcrossing{1}{0}
		\end{tikzpicture}
	} \right)^4 \quad = \quad \cbox{\begin{tikzpicture}[tldiagram]
			%dots
			%\drawdots{0}{0}{1}
			%\drawdots{1}{0}{1}
			%lines
			\draw \tlcoord{0}{0} \lineup;
			\draw \tlcoord{0}{1} \lineup;
	\end{tikzpicture}} 
	+(-q+q^3-q^5+q^7) \cdot \left( \cbox{
		\begin{tikzpicture}[tldiagram]
			%dots
			%\drawdots{0}{0}{1}
			%\drawdots{1}{0}{1}
			%lines
			\draw \tlcoord{0}{0} \capright;
			\draw \tlcoord{1}{0} \cupright;
		\end{tikzpicture}
	} + \cbox{
		\begin{tikzpicture}[tldiagram]
			%dots
			%\drawdots{0}{0}{1}
			%\drawdots{1}{0}{1}
			%lines
			\draw \tlcoord{0}{0} \dcapright;
			\draw \tlcoord{1}{0} \dcupright;
			%big dots
		\end{tikzpicture}
	} \right).
	\]
	The same calculation shows that for $m\geq1$ we have
	\[
	[\delta_2]^m  \quad = \quad \left( \cbox{
		\begin{tikzpicture}[tldiagram, yscale=2/3]
			\negcrossing{0}{0}
			\dnegcrossing{1}{0}
		\end{tikzpicture}
	} \right)^m \quad = \quad \cbox{\begin{tikzpicture}[tldiagram]
			%dots
			%\drawdots{0}{0}{1}
			%\drawdots{1}{0}{1}
			%lines
			\draw \tlcoord{0}{0} \lineup;
			\draw \tlcoord{0}{1} \lineup;
	\end{tikzpicture}} 
	+(\sum_{k=1}^m (-1)^k q^{2k-1}) \cdot \left( \cbox{
		\begin{tikzpicture}[tldiagram]
			%dots
			%\drawdots{0}{0}{1}
			%\drawdots{1}{0}{1}
			%lines
			\draw \tlcoord{0}{0} \capright;
			\draw \tlcoord{1}{0} \cupright;
		\end{tikzpicture}
	} + \cbox{
		\begin{tikzpicture}[tldiagram]
			%dots
			%\drawdots{0}{0}{1}
			%\drawdots{1}{0}{1}
			%lines
			\draw \tlcoord{0}{0} \dcapright;
			\draw \tlcoord{1}{0} \dcupright;
			%big dots
		\end{tikzpicture}
	} \right).
	\]
	The sequence $\sum_{k=1}^m (-1)^k q^{2k-1}$ converges for $m\to\infty$ to the fraction $-\frac{1}{q+q^{-1}}=-\frac{1}{[2]}$, which shows the theorem for the full twist $\delta_2$. 
	\end{beis}  
\begin{beis} \label{full twist type D3 example}
	All examples of \Cref{infinite type D braid} become very difficult to check for $n\geq3$. Already to compute the image $[\delta_3]\in\TL(D_n)$ one needs to perform $2^{12}=4096$ multiplications, since $\delta_3$ consists of $12$ crossings. When calculating there are several miraculous cancellations going on during the process, so there is a benefit to doing the calculations oneself instead of using a computer. 
	%For example when computing the first quarter of $\delta_3=(\sigma_0'\sigma_2\sigma_1)^4$, which is given by $\sigma_0'\sigma_2\sigma_0'$, the coefficient in front of the Temperley--Lieb basis elements
%	\[
%		\cbox{
%			\begin{tikzpicture}[tldiagram]
%				% lines
%				\draw \tlcoord{0}{0} \linewave{1}{2};
%				\draw \tlcoord{0}{1} \capright;
%				\draw \tlcoord{1}{1} \cupleft;
%				% boxes
%				%\draw \tlcoord{0}{0} \maketlboxgreen{3}{$3$};
%				=
%			\end{tikzpicture}
%		} \qquad \text{and} \qquad \cbox{
%		\begin{tikzpicture}[tldiagram]
%			% lines
%			\draw \tlcoord{0}{2} \linewave{1}{-2};
%			\draw \tlcoord{0}{0} \capright;
%			\draw \tlcoord{1}{2} \cupleft;
%			% boxes
%			%\draw \tlcoord{0}{0} \maketlboxgreen{3}{$3$};
%			=
%		\end{tikzpicture}
%	}
%	\] is given by $q^2$. However in a direct calculation both coefficients appear in a different fashion. The first one appears directly, while the coefficient of the right diagram appears as the expression
%	\[
%		q^2+q^4-q^3[2]+q^2,
%	\]
%	which also equals $q^2$.
	We just give the result of the computation for the full twist $\delta_3$. The following table is a side-by-side comparison of the Jones--Wenzl projector $d_3\in\TL(D_3)$ and the images of the first two powers of the full twist: 
	\[
	%\arrayrwsep=1.4pt\def\arraystretch{2.2}
	\begin{array}{c|c|c|c}
		\text{basis of } \TL(D_3) & d_3 & [\delta_3] & [\delta_3^2]  \\[4pt] \hline
		\cbox{
			\begin{tikzpicture}[tldiagram, xscale=1/2, yscale=1/2]
				% lines
				\draw \tlcoord{-1.5}{0} \lineup;
				\draw \tlcoord{-1.5}{1} \lineup;
				\draw \tlcoord{-1.5}{2} \lineup;
				% boxes
				%\draw \tlcoord{0}{0} \maketlboxgreen{3}{$3$};
				=
			\end{tikzpicture}
		} & 1& 1&1 \\[4pt] \hline
		\cbox{
			\begin{tikzpicture}[tldiagram, xscale=1/2, yscale=1/2, ]
				% lines
				\draw \tlcoord{0}{0} \capright;
				\draw \tlcoord{1}{0} \cupright;
				\draw \tlcoord{0}{2} \lineup; 
				% boxes
				%\draw \tlcoord{0}{0} \maketlboxgreen{3}{$3$};
				=
			\end{tikzpicture}
		} &-\frac{[3]}{[4]} &-q+q^7 &-q+q^7-q^9+q^{15} \\[4pt] \hline
		\cbox{
			\begin{tikzpicture}[tldiagram, xscale=1/2, yscale=1/2, ]
				% lines
				\draw \tlcoord{0}{0} \dcapright;
				\draw \tlcoord{1}{0} \dcupright;
				\draw \tlcoord{0}{2} \lineup; 
				% boxes
				%\draw \tlcoord{0}{0} \maketlboxgreen{3}{$3$};
				=
			\end{tikzpicture}
		} & -\frac{[3]}{[4]} &-q+q^7 & -q+q^7-q^9+q^{15}\\[4pt] \hline 
		\cbox{
			\begin{tikzpicture}[tldiagram, xscale=1/2, yscale=1/2]
				% lines
				\draw \tlcoord{0}{0} \linewave{1}{2};
				\draw \tlcoord{0}{1} \capright;
				\draw \tlcoord{1}{1} \cupleft;
				% boxes
				%\draw \tlcoord{0}{0} \maketlboxgreen{3}{$3$};
				=
			\end{tikzpicture}
		} & \frac{[2]}{[4]} & q^2-q^6& q^2-q^6+q^{10}-q^{14}\\[4pt] \hline
		\cbox{
			\begin{tikzpicture}[tldiagram, xscale=1/2, yscale=1/2]
				% lines
				\draw \tlcoord{0}{2} \linewave{1}{-2};
				\draw \tlcoord{0}{0} \capright;
				\draw \tlcoord{1}{2} \cupleft;
				% boxes
				%\draw \tlcoord{0}{0} \maketlboxgreen{3}{$3$};
				=
			\end{tikzpicture}
		} & \frac{[2]}{[4]} & q^2-q^6& q^2-q^6+q^{10}-q^{14}\\[4pt] \hline
		\cbox{
			\begin{tikzpicture}[tldiagram, xscale=1/2, yscale=1/2]
				% lines
				\draw \tlcoord{0}{0} \dlinewave{1}{2};
				\draw \tlcoord{0}{1} \capright;
				\draw \tlcoord{1}{1} \dcupleft;
				% boxes
				%\draw \tlcoord{0}{0} \maketlboxgreen{3}{$3$};
				=
			\end{tikzpicture}
		} & \frac{[2]}{[4]} & q^2-q^6& q^2-q^6+q^{10}-q^{14}\\[4pt] \hline 
		\cbox{
			\begin{tikzpicture}[tldiagram, xscale=1/2, yscale=1/2]
				% lines
				\draw \tlcoord{0}{2} \dlinewave{1}{-2};
				\draw \tlcoord{0}{0} \dcapright;
				\draw \tlcoord{1}{2} \cupleft;
				% boxes
				%\draw \tlcoord{0}{0} \maketlboxgreen{3}{$3$};
				=
			\end{tikzpicture}
		} & \frac{[2]}{[4]} & q^2-q^6 & q^2-q^6+q^{10}-q^{14}\\[4pt] \hline 
		\cbox{
			\begin{tikzpicture}[tldiagram, xscale=1/2, yscale=1/2, ]
				% lines
				\draw \tlcoord{0}{1} \capright;
				\draw \tlcoord{1}{1} \cupright;
				\draw \tlcoord{0}{0} \lineup; 
				% boxes
				%\draw \tlcoord{0}{0} \maketlboxgreen{3}{$3$};
				=
			\end{tikzpicture}
		} & \frac{[2][2]}{[4]}& -q-q^3+q^5+q^7&
		-q-q^3+q^5+q^7
		-q^9-q^{11}+q^{13}+q^{15}\\[4pt] \hline
		\cbox{
			\begin{tikzpicture}[tldiagram, xscale=1/2, yscale=1/2, ]
				% lines
				\draw \tlcoord{0}{0} \capright;
				\draw \tlcoord{1}{0} \dcupright;
				\draw \tlcoord{0}{2} \dlineup; 
				% boxes
				%\draw \tlcoord{0}{0} \maketlboxgreen{3}{$3$};
				=
			\end{tikzpicture}
		} & -\frac{1}{[4]}&-q^3+q^5 & -q^3+q^5-q^{11}+q^{13} \\[4pt] \hline
		\cbox{
			\begin{tikzpicture}[tldiagram, xscale=1/2, yscale=1/2, ]
				% lines
				\draw \tlcoord{0}{0} \dcapright;
				\draw \tlcoord{1}{0} \cupright;
				\draw \tlcoord{0}{2} \dlineup; 
				% boxes
				%\draw \tlcoord{0}{0} \maketlboxgreen{3}{$3$};
				=
			\end{tikzpicture}
		} & -\frac{1}{[4]}& -q^3+q^5-q^{11}+q^{13}& -q^3+q^5-q^{11}+q^{13}
	\end{array}
	\]
	Again observe the coefficients in front of the powers of $\delta_n$ approximate the corresponding fractions of quantum integers in the Jones--Wenzl projector. For example the fraction $\frac{[2]}{[4]}$ is given as power series by
	\begin{align*}
		\frac{[2]}{[4]}=\frac{q^2-q^{-2}}{q^4-q^{-4}}=\frac{q^{2}-q^6}{1-q^8}&=(q^{2}-q^6)(1+q^8+q^{16}+q^{24}+\cdots) \\
		&=q^2-q^6+q^{10}-q^{14}+q^{18}-q^{24}+\ldots
	\end{align*}
	Consider the ordered set
	\[
		d_3, U_1, U_0, U_1U_2, U_2U_1, U_0U_2, U_2U_0, U_2, U_0U_2U_1, U_1U_2U_0,
	\] 
	 which is a a basis of $\TL(D_3)$, since it is just the diagram basis, but instead of the identity element we chose $d_3\in 1 + (U_0,U_1,\ldots,U_{n-11})$. One can compute that the representing matrix of $[\delta_3]$ acting on $\TL(D_3)$ with respect to this basis is of the form
	\[
	([\delta_3]\cdot \blank)= \left( \begin{array}{c|c}
		1 & \mathbf{0} \\ \hline
		\mathbf{0} & A
	\end{array} \right)
	\]
	where $A$ is a matrix of size $9\times 9=(\dim \TL(D_3)-1)^2$ with entries in $q\mZ[q]$. In particular in the $q$-adic norm we have
	\[
		([\delta_3]\cdot \blank)^m \xrightarrow{m\rightarrow \infty} \left( \begin{array}{c|c}
			1 & \mathbf{0} \\ \hline
			\mathbf{0} & \mathbf{0}
		\end{array} \right) = ([d_3] \cdot \blank),
	\]
	which is the representing matrix of multiplication with the Jones--Wenzl projector $d_3$, since $d_3$ is idempotent and $d_3U_i=0$ for $i=0,1,2$.
\end{beis}
We come to \Cref{most crucial lemma}, which is a diagrammatic statement,  generalizing \cite[Lemma 4.5]{rozansky2010} to type $D$. It will serve as the main  step in the proof of \Cref{infinite type D braid}.
\begin{lemm}[Approximative Cup-Cap-Killing] \label{most crucial lemma}
	The type $D$ full twist satisfies the following computational properties
	\[
		\dtldcap_1 \cdot [\delta_n] = q^{4(n-1)}[\delta_{n-2}]\cdot \dtldcap_1, \quad \tlcap_i \cdot [\delta_n] = q^{4(n-1)}[\delta_{n-2}]\cdot \tlcap_i
	\]
	for all $i=1,\ldots, n-1$.
	In particular
	\[
	\dtldcap_1 \cdot [\delta_n]^m = q^{4m(n-1)}[\delta_{n-2}]\cdot \dtldcap_1, \quad \tlcap_i \cdot [\delta_n]^m = q^{4m(n-1)}[\delta_{n-2}]\cdot \tlcap_i
	\]
	for all $i=1,\ldots, n$ and $m\geq1$. Here $\tlcap_i=\tlline \cdots \tlline \tlcap \tlline \cdots \tlline$ denotes the diagram given by a cap connecting the $i$-th and $(i+1)$-st vertices and the other vertices by a vertical strand, respectively. Similarly $\dtldcap_1\in \TL_{B}(n,n-2)$ denotes the $\TL_{B}(n,n-2)$ diagram given by a dotted cap connecting the first and second vertex and the other vertices by a vertical strand respectively.
\end{lemm}
\begin{proof}
	This is proven by induction using the Reidemeister moves from \Cref{section on Reidemeister moves}. 
\end{proof}
\begin{beis}
	Its best to understand \Cref{most crucial lemma} via a concrete example:
	\begin{align*}
		\tlcap_1 \cdot \delta_3\, &= \, \cbox{
			\begin{tikzpicture}[tldiagram, yscale=1/3, xscale=1/2]
				\negcrossing{0}{1}
				\negcrossing{1}{0}
				\dnegcrossing{2}{0}
				\negcrossing{3}{1}
				\negcrossing{4}{0}
				\dnegcrossing{5}{0}
				\negcrossing{6}{1}
				\negcrossing{7}{0}
				\dnegcrossing{8}{0}
				\negcrossing{9}{1}
				\negcrossing{10}{0}
				\dnegcrossing{11}{0}
				\draw \tlcoord{0}{0} \lineup;
				\draw \tlcoord{3}{0} \lineup;
				\draw \tlcoord{6}{0} \lineup;
				\draw \tlcoord{9}{0} \lineup;
				\draw \tlcoord{1}{2} \lineup;
				\draw \tlcoord{2}{2} \lineup;
				\draw \tlcoord{4}{2} \lineup;
				\draw \tlcoord{5}{2} \lineup;
				\draw \tlcoord{7}{2} \lineup;
				\draw \tlcoord{8}{2} \lineup;
				\draw \tlcoord{10}{2} \lineup;
				\draw \tlcoord{11}{2} \lineup;
				\draw \tlcoord{12}{0} \capright;
			\end{tikzpicture}
		}
		\, = \, \cbox{
			\begin{tikzpicture}[tldiagram, yscale=1/3, xscale=1/2]
				\negcrossing{0}{1}
				\negcrossing{1}{0}
				\dnegcrossing{2}{0}
				\negcrossing{3}{1}
				\negcrossing{4}{0}
				\dnegcrossing{5}{0}
				\negcrossing{6}{1}
				\negcrossing{7}{0}
				\dnegcrossing{8}{0}
				\negcrossing{9}{1}
				\uppernegcrossing{10}{0}
				\draw \tlcoord{0}{0} \lineup;
				\draw \tlcoord{3}{0} \lineup;
				\draw \tlcoord{6}{0} \lineup;
				\draw \tlcoord{9}{0} \lineup;
				\draw \tlcoord{1}{2} \lineup;
				\draw \tlcoord{2}{2} \lineup;
				\draw \tlcoord{4}{2} \lineup;
				\draw \tlcoord{5}{2} \lineup;
				\draw \tlcoord{7}{2} \lineup;
				\draw \tlcoord{8}{2} \lineup;
				\draw \tlcoord{10}{2} \lineup;
				\draw \tlcoord{11}{2} \lineup;
				\draw \tlcoord{11}{0} \dcapright;
			\end{tikzpicture}
		}
		 \, = \, -q^2
		 \cbox{
		 	\begin{tikzpicture}[tldiagram, yscale=1/3, xscale=1/2]
		 		\negcrossing{0}{1}
		 		\negcrossing{1}{0}
		 		\dnegcrossing{2}{0}
		 		\negcrossing{3}{1}
		 		\negcrossing{4}{0}
		 		\dnegcrossing{5}{0}
		 		\negcrossing{6}{1}
		 		\negcrossing{7}{0}
		 		\dnegcrossing{8}{0}
		 		\negcrossing{9}{1}
		 		\draw \tlcoord{0}{0} \lineup;
		 		\draw \tlcoord{3}{0} \lineup;
		 		\draw \tlcoord{6}{0} \lineup;
		 		\draw \tlcoord{9}{0} \lineup;
		 		\draw \tlcoord{1}{2} \lineup;
		 		\draw \tlcoord{2}{2} \lineup;
		 		\draw \tlcoord{4}{2} \lineup;
		 		\draw \tlcoord{5}{2} \lineup;
		 		\draw \tlcoord{7}{2} \lineup;
		 		\draw \tlcoord{8}{2} \lineup;
		 		\draw \tlcoord{10}{2} \lineup;
		 		\draw \tlcoord{10}{0} \capright;
		 	\end{tikzpicture}
		 } \, = \, -q^2
	 \cbox{
	 \begin{tikzpicture}[tldiagram, yscale=1/3, xscale=1/2]
	 	\negcrossing{0}{1}
	 	\negcrossing{1}{0}
	 	\dnegcrossing{2}{0}
	 	\negcrossing{3}{1}
	 	\negcrossing{4}{0}
	 	\dnegcrossing{5}{0}
	 	\negcrossing{6}{1}
	 	\dnegcrossing{7}{0}
	 	\negcrossing{8}{0}
	 	\negcrossing{9}{1}
	 	\draw \tlcoord{0}{0} \lineup;
	 	\draw \tlcoord{3}{0} \lineup;
	 	\draw \tlcoord{6}{0} \lineup;
	 	\draw \tlcoord{9}{0} \lineup;
	 	\draw \tlcoord{1}{2} \lineup;
	 	\draw \tlcoord{2}{2} \lineup;
	 	\draw \tlcoord{4}{2} \lineup;
	 	\draw \tlcoord{5}{2} \lineup;
	 	\draw \tlcoord{7}{2} \lineup;
	 	\draw \tlcoord{8}{2} \lineup;
	 	\draw \tlcoord{10}{2} \lineup;
	 	\draw \tlcoord{10}{0} \capright;
	 \end{tikzpicture}
 }  \, = \, q^3 \cbox{
\begin{tikzpicture}[tldiagram, yscale=1/3, xscale=1/2]
	\negcrossing{0}{1}
	\negcrossing{1}{0}
	\dnegcrossing{2}{0}
	\negcrossing{3}{1}
	\negcrossing{4}{0}
	\dnegcrossing{5}{0}
	\negcrossing{6}{1}
	\dnegcrossing{7}{0}
	\draw \tlcoord{0}{0} \lineup;
	\draw \tlcoord{3}{0} \lineup;
	\draw \tlcoord{6}{0} \lineup;
	\draw \tlcoord{1}{2} \lineup;
	\draw \tlcoord{2}{2} \lineup;
	\draw \tlcoord{4}{2} \lineup;
	\draw \tlcoord{5}{2} \lineup;
	\draw \tlcoord{7}{2} \lineup;
	\draw \tlcoord{8}{0} \lineup;
	\draw \tlcoord{8}{1} \capright;
\end{tikzpicture}
} \\
&=\, q^3 \cbox{
\begin{tikzpicture}[tldiagram, yscale=1/3, xscale=1/2]
	\negcrossing{0}{1}
	\negcrossing{1}{0}
	\dnegcrossing{2}{0}
	\negcrossing{3}{1}
	\negcrossing{4}{0}
	\lowernegcrossing{5}{0}
	\negcrossing{6}{1}
	\negcrossing{7}{0}
	\draw \tlcoord{0}{0} \lineup;
	\draw \tlcoord{3}{0} \lineup;
	\draw \tlcoord{6}{0} \lineup;
	\draw \tlcoord{1}{2} \lineup;
	\draw \tlcoord{2}{2} \lineup;
	\draw \tlcoord{4}{2} \lineup;
	\draw \tlcoord{5}{2} \lineup;
	\draw \tlcoord{7}{2} \lineup;
	\draw \tlcoord{8}{0} \dlineup;
	\draw \tlcoord{8}{1} \capright;
\end{tikzpicture}
} \, = \, -q^4 \cbox{
\begin{tikzpicture}[tldiagram, yscale=1/3, xscale=1/2]
	\negcrossing{0}{1}
	\negcrossing{1}{0}
	\dnegcrossing{2}{0}
	\negcrossing{3}{1}
	\negcrossing{4}{0}
	\lowernegcrossing{5}{0}
	\draw \tlcoord{0}{0} \lineup;
	\draw \tlcoord{3}{0} \lineup;
	\draw \tlcoord{1}{2} \lineup;
	\draw \tlcoord{2}{2} \lineup;
	\draw \tlcoord{4}{2} \lineup;
	\draw \tlcoord{5}{2} \dlineup;
	\draw \tlcoord{6}{0} \capright;
\end{tikzpicture}
} \, = \, q^6 \cbox{
\begin{tikzpicture}[tldiagram, yscale=1/3, xscale=1/2]
	\negcrossing{0}{1}
	\negcrossing{1}{0}
	\dnegcrossing{2}{0}
	\negcrossing{3}{1}
	\negcrossing{4}{0}
	\draw \tlcoord{0}{0} \lineup;
	\draw \tlcoord{3}{0} \lineup;
	\draw \tlcoord{1}{2} \lineup;
	\draw \tlcoord{2}{2} \lineup;
	\draw \tlcoord{4}{2} \lineup;
	\draw \tlcoord{5}{2} \dlineup;
	\draw \tlcoord{5}{0} \dcapright;
\end{tikzpicture}
} \, = \, \cbox{
\begin{tikzpicture}[tldiagram, yscale=1/3, xscale=1/2]
	\negcrossing{0}{1}
	\negcrossing{1}{0}
	\dnegcrossing{2}{0}
	\negcrossing{3}{1}
	\draw \tlcoord{0}{0} \lineup;
	\draw \tlcoord{3}{0} \lineup;
	\draw \tlcoord{1}{2} \lineup;
	\draw \tlcoord{2}{2} \lineup;
	\draw \tlcoord{4}{2} \dlineup;
	\draw \tlcoord{4}{0} \dcapright;
\end{tikzpicture}
} \, = \, q^6
\cbox{
	\begin{tikzpicture}[tldiagram, yscale=1/3, xscale=1/2]
		\negcrossing{0}{1}
		\negcrossing{1}{0}
		\lowernegcrossing{2}{0}
		\negcrossing{3}{1}
		\draw \tlcoord{0}{0} \lineup;
		\draw \tlcoord{3}{0} \lineup;
		\draw \tlcoord{1}{2} \lineup;
		\draw \tlcoord{2}{2} \lineup;
		\draw \tlcoord{4}{2} \dlineup;
		\draw \tlcoord{4}{0} \capright;
	\end{tikzpicture}
} \\
&= \, -q^7 \cbox{
\begin{tikzpicture}[tldiagram, yscale=1/3, xscale=1/2]
	\negcrossing{0}{1}
	\uppernegcrossing{1}{0}
	\draw \tlcoord{0}{0} \lineup;
	\draw \tlcoord{1}{2} \lineup;
	\draw \tlcoord{2}{0} \dlineup;
	\draw \tlcoord{2}{1} \capright;
\end{tikzpicture}
} \, = \, q^8 \cbox{
\begin{tikzpicture}[tldiagram, yscale=1/3, xscale=1/2]
	\draw \tlcoord{0}{0} \capright;
	\draw \tlcoord{0}{2} \lineup;
\end{tikzpicture}
} \, = q^{8} (\delta_1 \cdot \tlcap_1).
	\end{align*}
	Here we used first type $B$ Reidemeister I, the relation $s_0^2=1$ (= BTL1), Reidemeister I, the type $B$ braid relation, Reidemeister II, the relation $s_0U_2=U_2s_0$ (=BTL3), Reidemeister II, Reidemeister I, type $B$ Reidemeister I, BTL3, Reidemeister 2, BTL1, and finally Reidemeister II. 
\end{beis}
Now we can use \Cref{most crucial lemma} to prove the main theorem. The proof is inspired by the proof of Theorem 7 in \cite{costa16} and the categorical ideas from \cite{rozansky2010}. The flavor of the proof is more functional analytical than algebraic and the key idea is the matrix point of view at the end of \Cref{full twist type D3 example}.
\begin{proof}[Proof of \Cref{infinite type D braid}]
	Let $k=\C[[q]][q^{-1}]$ and consider the regular representation
	\begin{align*}
		\TL(D_n)&\hookrightarrow \End(\TL(D_n))\cong\M_{r\times r}(k) \\
		x &\mapsto (x \cdot \blank \colon y \mapsto xy),
	\end{align*}
	with respect to a chosen basis, where $r=\dim \TL(D_n)=\frac{1}{2}{{2n} \choose {n}}$. Concretely we choose the basis $(d_n,\lambda_1,\ldots, \lambda_{r-1})$ where $d_n$ is the type $D_n$ Jones--Wenzl projector and $\lambda_1,\ldots\lambda_{r-1}$ is the diagram basis of the two-sided ideal $(U_0,U_1,\ldots, U_{n-1})\subseteq \TL(D_n)$ in some chosen order, which consists of all type $D_n$ Temperley--Lieb diagrams except the identity diagram.
	By \Cref{most crucial lemma} we have
	\[
		[\delta_n]\lambda_i\in q\mZ[[q]]\{\lambda_1,\ldots \lambda_{r-1}\},
	\]
	which implies for all $m>0$ that
	\[
		[\delta_n]^m\lambda_i\in q^{m}\mZ[[q]]\{\lambda_1,\ldots \lambda_{r-1}\}.
	\]
	Moreover by definition of $[\sigma_0']=1-qU_0$ and $[\sigma_i]=1-qU_i$ we have
	\[
		[\delta_n]\in 1+ \mZ[[q]]\{\lambda_1,\ldots \lambda_{r-1}\}.
	\]
	This implies that $\delta_n d_n=d_n$ since the Jones--Wenzl projector $d_n$ satisfies $d_n\lambda_i=0$ for all $i=1,\ldots, n-1$. Hence the representing matrix of $[\delta_n]^m$ is given by
	\[
	([\delta_n]^m\cdot \blank)= \left( \begin{array}{c|c}
		1 & \mathbf{0} \\ \hline
		\mathbf{0} & A
	\end{array} \right)
	\]
	where $A$ is a $(r-1)\times (r-1)$ matrix with entries in $q^{m}\mZ[[q]]$ for all $m>0$.
	Hence 	
	\[
	([\delta_n]\cdot \blank)^m \xrightarrow{m\rightarrow \infty} \left( \begin{array}{c|c}
		1 & \mathbf{0} \\ \hline
		\mathbf{0} & \mathbf{0}
	\end{array} \right) = ([d_n] \cdot \blank),
	\]
	where the limit on the left is taken with respect to the $q$-adic uniform norm on $M_{r\times r}(k)$. We conclude that
	$\lim_{m\rightarrow\infty}([\delta_n]^m)=[d_n]$ by the bounded inverse theorem (see e.g.\ \cite[Corollary 2.7]{brezis11}) since the map $\TL(D_n)\rightarrow \M_{r\times r}(k)$ is continuous, injective, and $\TL(D_n)$ and $\M_{r\times r}(k)$ are complete normed spaces with respect to their $q$-adic norms.
\end{proof}
We conclude that \Cref{infinite type D braid} gives a different combinatorial way to construct the type $D$ Jones--Wenzl projectors by approximating them with powers of the full twist. This paves the way to a categorification of the type $D$ Jones--Wenzl projectors with the homological tools from \cite{rozansky2010}. These tools will be the main topic of the next chapter.

	\backmatter
	
	\ihead{}
	\chead{} 
	\ohead{Literature}  % override previous behavior for the header
	\printbibliography
\end{document}